\documentclass[a4paper,11pt]{article}

\usepackage[T1]{fontenc}
\usepackage{lmodern}

\usepackage{dsfont}

\usepackage[latin1]{inputenc}
\usepackage{amsmath}
\usepackage{amsthm}
\usepackage{amssymb}
\usepackage{mathrsfs}
\usepackage{graphicx}
\usepackage[all]{xy}
\usepackage{hyperref}

\usepackage{makeidx}

\usepackage{stmaryrd}
\usepackage{caption}

\usepackage{abstract} 

\newtheorem{thm}{Theorem}[section]
\newtheorem{cor}[thm]{Corollary}
\newtheorem{claim}[thm]{Claim}
\newtheorem{fact}[thm]{Fact}

\newtheorem{lemma}[thm]{Lemma}
\newtheorem{prop}[thm]{Proposition}

\theoremstyle{definition}
\newtheorem{definition}[thm]{Definition}

\newtheorem{remark}[thm]{Remark}
\newtheorem{question}[thm]{Question}
\newtheorem{problem}[thm]{Problem}
\newtheorem{conj}[thm]{Conjecture}

\newcommand{\QM}{\mathrm{QM}(\Gamma, \mathcal{G})}
\newcommand{\ST}{\mathrm{ST}(\Gamma, \mathcal{G})}
\newcommand{\crossing}{\mathrm{T}(\Gamma, \mathcal{G})}

\def\rquotient#1#2{%
	\makeatletter
	\raise.3ex\hbox{$#1$}/\lower.3ex\hbox{$#2$}%
	\makeatother
}	

\makeatletter
\newcommand{\subjclass}[2][2010]{%
	\let\@oldtitle\@title%
	\gdef\@title{\@oldtitle\footnotetext{#1 \emph{Mathematics subject classification.} #2}}%
}
\newcommand{\keywords}[1]{%
	\let\@@oldtitle\@title%
	\gdef\@title{\@@oldtitle\footnotetext{\emph{Key words and phrases.} #1.}}%
}
\makeatother

\newcommand{\Address}{{
		\bigskip
		\small
		
		\textsc{D\'epartement de Math\'ematiques B\^atiment 307, Facult\'e des Sciences d'Orsay, Universit\'e Paris-Sud, F-91405 Orsay Cedex, France.}\par\nopagebreak
		\textit{E-mail address}: \texttt{anthony.genevois@math.u-psud.fr}
		
}}

\makeindex

\title{Automorphisms of graph products of groups and acylindrical hyperbolicity}
\date{\today}
\author{Anthony Genevois}


\begin{document}

\maketitle

\begin{abstract}
This article is dedicated to the study of the acylindrical hyperbolicity of automorphism groups of graph products of groups. Our main result is that, if $\Gamma$ is a finite graph which contains at least two vertices and is not a join and if $\mathcal{G}$ is a collection of finitely generated irreducible groups, then either $\Gamma \mathcal{G}$ is infinite dihedral or $\mathrm{Aut}(\Gamma \mathcal{G})$ is acylindrically hyperbolic. This theorem is new even for right-angled Artin groups and right-angled Coxeter groups. Various consequences are deduced from this statement and from the techniques used to prove it. For instance, we show that the automorphism groups of most graph products verify vastness properties such as being SQ-universal; we show that many automorphism groups of graph products do not satisfy Kazhdan's property (T); we solve the isomorphism problem between graph products in some cases; and we show that a graph product of coarse median groups, as defined by Bowditch, is coarse median itself. 
\end{abstract}

\tableofcontents


\section{Introduction}

\noindent
Given a simplicial graph $\Gamma$ and a collection of groups $\mathcal{G}= \{G_u \mid u \in V(\Gamma) \}$ indexed by the vertices of $\Gamma$, the \emph{graph product} $\Gamma \mathcal{G}$ is defined as the quotient
$$\left( \underset{u \in V(\Gamma)}{\ast} G_u \right) / \langle \langle [g,h], \ g \in G_u, h\in G_v, \{u,v\} \in E(\Gamma) \rangle \rangle$$
where $E(\Gamma)$ denotes the edge-set of $\Gamma$. In other words, we take the free product of the groups in $\mathcal{G}$, referred to as \emph{vertex-groups}, and we force two adjacent factors to commute.

\medskip \noindent
Notice that, when $\Gamma$ does not have any edge, then the graph product coincides with the free product of $\mathcal{G}$. In the opposite direction, if $\Gamma$ is a complete graph (i.e., if any two vertices in $\Gamma$ are linked by an edge), then the graph product coincides with the direct sum of $\mathcal{G}$. Therefore, graph products of groups provide a natural interpolation between two very different regions of group theory: free products and direct sums. In order to understand the differences between these two regions, it is interesting to look at, along our interpolation, when properties stop being satisfied or when they begin to be satisfied.

\medskip \noindent
Another motivation in the study of graph products comes from the observation that they include several classical families of groups, providing a common formalism to investigate them and to understand the similarities between different statements available in the literature which were proved independently for several classes of groups. Classical examples of graph products include \emph{right-angled Artin groups}, corresponding to graph products of infinite cyclic groups; \emph{right-angled Coxeter groups}, corresponding to graph products of cyclic groups of order two; and, more generally, graph products of finite groups such as \emph{Bourdon groups} \cite{Bourdon}, which are closely related to locally finite right-angled buildings.

\medskip \noindent
In this article, we focus our attention on automorphisms of general graph products. Many articles have been dedicated to the study of the automorphism groups of right-angled Artin and Coxeter groups, but, beyond these two cases, the automorphism groups of general graph products of groups are poorly understood. Most of the literature on this topic imposes very strong conditions on the graph products involved, either on the underlying simplicial graph (as in the case of the automorphisms groups of free products \cite{OutSpaceFreeProduct, HorbezHypGraphsForFreeProducts, HorbezTitsAlt}) or on the vertex-groups (most of the case, they are required to be abelian or even cyclic  \cite{AutGPabelianSet, AutGPabelian, AutGPSIL, RuaneWitzel}). In \cite{ConjAut}, we initiated a geometric study of automorphisms of general graph products based on the formalism introduced in \cite{Qm}. We pursue this strategy here.

\medskip \noindent
Roughly speaking, the main result of the article states that the automorphism group of a graph product $\Gamma \mathcal{G}$ decomposes as a semidirect product
$$(\text{abelian}) \rtimes \left( Q \oplus \left[ \left( \bigoplus \left( \begin{array}{c} \text{acylindrically} \\ \text{hyperbolic} \end{array} \oplus \begin{array}{c} \text{virtually} \\ \text{cyclic} \end{array} \right) \right) \rtimes (\text{finite}) \right] \right),$$
where $Q$ is the automorphism group of a direct sum of vertex-groups. In other words, we break the automorphism group $\mathrm{Aut}(\Gamma \mathcal{G})$ into pieces with specific structures. The only flaw is that we cannot expect to understand $Q$ without additional information on vertex-groups, except when it is trivial (which happens when no vertex in $\Gamma$ is adjacent to all the other vertices). But, for right-angled Artin groups, $Q$ is a $\mathrm{GL}(n,\mathbb{Z})$; and for graph products of finite groups, $Q$ is finite.

\medskip \noindent
Formally, we prove the following statement:

\begin{thm}\label{BigThmIntro}
Let $\Gamma$ be a finite simplicial graph and $\mathcal{G}$ a collection of finitely generated groups indexed by $V(\Gamma)$. Assume that no group in $\mathcal{G}$ decomposes non-trivially as a graph product. Decompose $\Gamma$ as a join $\Gamma_0 \ast \Gamma_1 \ast \cdots \ast \Gamma_n$ where $\Gamma_0$ is complete and where each graph among $\Gamma_1, \ldots, \Gamma_n$ contains at least two vertices and is not a join. Then
$$\mathrm{Aut}(\Gamma \mathcal{G}) \simeq \mathrm{Hom} \left( \bigoplus\limits_{i=1}^n \langle \Gamma_i \rangle \to Z(\langle \Gamma_0 \rangle) \right) \rtimes \left( \mathrm{Aut}(\langle \Gamma_0 \rangle) \oplus \left[ \left( \bigoplus\limits_{i=1}^n \mathrm{Aut}(\langle \Gamma_i \rangle) \right) \rtimes S \right] \right)$$
where $S$ is a finite group permuting the isomorphic factors of $\langle \Gamma_1 \rangle \oplus \cdots \oplus \langle \Gamma_n \rangle$ and where, for every $1 \leq i \leq n$, $\mathrm{Aut}(\langle \Gamma_i \rangle)$ is acylindrically hyperbolic if $\Gamma_i$ is not a pair of isolated vertices both labelled by $\mathbb{Z}_2$.
\end{thm}

\noindent
Theorem \ref{BigThmIntro} has two important consequences. First, it reduces the study of automorphism groups of graph products to a case where powerful tools are available due to the acylindrical hyperbolicity of the graph product, namely when the simplicial graph is not a join. And second, it allows us to deduce non-trivial information on the automorphism group thanks to the properties satisfied by acylindrically hyperbolic groups. This aspect is detailed in the next paragraph.

\paragraph{Vastness properties.} Acylindrically hyperbolic groups define a broad family of groups satisfying properties of negative curvature. We refer to \cite{OsinAcyl} and Section \ref{section:AcylHyp} for more information. They are mostly known to satisfy \emph{vastness properties} (following the terminology of \cite{OutRAAGlarge}) such as being SQ-universal, virtually having many quasimorphisms and not being boundedly generated. Recall that a group $G$
\begin{itemize}
	\item is \emph{SQ-universal} if every countable group embeds into a quotient of $G$;
	\item \emph{virtually has many quasimorphisms} if it contains a finite-index subgroup whose space of homogeneous quasimorphisms is infinite-dimensional;
	\item is \emph{boundedly generated} if there exist $g_1, \ldots, g_r \in G$ such that every element can be written as $g_1^{n_1} \cdots g_r^{n_r}$ for some $n_1, \ldots, n_r \in \mathbb{Z}$.
\end{itemize}
Once combined with Theorem \ref{BigThmIntro}, one obtains the following statement:

\begin{thm}\label{thm:VastProp}
Let $\Gamma$ be a finite simplicial graph and $\mathcal{G}$ a collection of groups indexed by $V(\Gamma)$. Assume that no group in $\mathcal{G}$ decomposes non-trivially as a graph product. If $\Gamma$ is not a join of a complete graph with pairs of isolated vertices both labelled by $\mathbb{Z}_2$, then $\mathrm{Aut}(\Gamma \mathcal{G})$ is
\begin{itemize}
	\item SQ-universal;
	\item virtually has many quasimorphisms;
	\item is not boundedly generated.
\end{itemize}
\end{thm}

\noindent
If $\Gamma$ is the join of a complete graph with pairs of isolated vertices both labelled by $\mathbb{Z}_2$, then $\mathrm{Aut}(\Gamma \mathcal{G})$ contains a finite-index subgroup isomorphic to the direct sum of a free abelian group with the automorphism group of a sum of vertex-groups. So, in this case, we cannot conclude that $\mathrm{Aut}(\Gamma \mathcal{G})$ is vast without additional information on vertex-groups. This explains why this configuration is not covered by Theorem \ref{thm:VastProp}. 

\medskip \noindent
When restricted to right-angled Artin and Coxeter groups, one obtains the following statement:

\begin{cor}\label{cor:VastProp}
Let $\Gamma$ be a finite graph. If $A$ denotes the right-angled Artin group defined by $\Gamma$, the following assertions are equivalent:
\begin{itemize}
	\item $A$ is non-abelian, i.e., $\Gamma$ is not complete;
	\item $\mathrm{Aut}(A)$ is SQ-universal;
	\item $\mathrm{Aut}(A)$ involves all finite groups;
	\item $\mathrm{Aut}(A)$ virtually has many quasimorphisms;
	\item $\mathrm{Aut}(A)$ is not boundedly generated.
\end{itemize}
Also, if $G$ is a graph product of finite groups over $\Gamma$ (e.g., the right-angled Coxeter group defined by $\Gamma$), then the following assertions are equivalent:
\begin{itemize}
	\item $G$ is not virtually abelian, i.e., $\Gamma$ is not a join of isolated vertices and isolated pairs of vertices both labelled by $\mathbb{Z}_2$;
	\item $\mathrm{Aut}(G)$ is SQ-universal;
	\item $\mathrm{Aut}(G)$ involves all finite groups;
	\item $\mathrm{Aut}(G)$ virtually has many quasimorphisms;
	\item $\mathrm{Aut}(G)$ is not boundedly generated.
\end{itemize}
\end{cor}

\noindent
Following \cite{OutRAAGlarge}, we say that a group $G$ \emph{involves all finite groups} if every finite group is a quotient of a finite-index subgroup of $G$. This property does not appear in Theorem~\ref{thm:VastProp} because it depends on the vertex-groups we choose. (Indeed, notice that a graph product of finitely generated infinite simple groups (over a finite simplicial graph) does not contain proper finite-index subgroups.)

\paragraph{Acylindrical hyperbolicity of automorphism groups.} The initial motivation of our work was to determine when the automorphism group of a right-angled Artin group is acylindrically hyperbolic, with the following general question in mind \cite{AutHypAcyl}:

\begin{question}\label{QuestionHyp}
Is the automorphism group of a finitely generated acylindrically hyperbolic group necessarily acylindrically hyperbolic?
\end{question}

\noindent
Roughly speaking, we would like to know whether some negative curvature survives from the group to its automorphism group. Question \ref{QuestionHyp} has a positive answer for hyperbolic groups and most relatively hyperbolic groups \cite{AutHypAcyl, AutFreeProductsHyp}, and right-angled Artin groups were natural examples to test next. Theorem \ref{BigThmIntro} allows us to determine precisely when the automorphism group of a graph product is acylindrically hyperbolic:

\begin{thm}\label{thm:WhenAutHyp}
Let $\Gamma$ be a finite simplicial graph which is not a clique and $\mathcal{G}$ a collection of finitely generated groups indexed by $V(\Gamma)$. Assume that no group in $\mathcal{G}$ decomposes non-trivially as a graph product. Decompose $\Gamma$ as a join $\Gamma_0 \ast \Gamma_1 \ast \cdots \ast \Gamma_n$ where $\Gamma_0$ is a clique and where each graph among $\Gamma_1, \ldots, \Gamma_n$ contains at least two vertices and is not a join. The automorphism group $\mathrm{Aut}(\Gamma \mathcal{G})$ is acylindrically hyperbolic if and only if either the two following conditions hold:
\begin{itemize}
	\item $G_u$ is finite for every $u \in V(\Gamma_0)$,
	\item $n=1$ and $\Gamma_1$ is not a disjoint union of two isolated vertices both labelled by $\mathbb{Z}_2$;
\end{itemize}
or if the three following conditions hold:
\begin{itemize}
	\item $G_u$ has a finite abelianisation for every $u \in V(\Gamma_1 \cup \cdots \cup \Gamma_n)$,
	\item $\mathrm{Aut}(\langle \Gamma_0 \rangle)$ is finite,
	\item $n=1$ and $\Gamma_1$ is not a disjoint union of two isolated vertices both labelled by $\mathbb{Z}_2$.
\end{itemize}
\end{thm}

\noindent
Notice that the first case correspond to the situation where the graph product is acylindrically hyperbolic, so Question \ref{QuestionHyp} also has a positive answer for graph products of groups, including right-angled Artin groups and graph products of finite groups (such as right-angled Coxeter groups).

\begin{cor}\label{cor:WhenAutHyp}
Let $\Gamma$ be a finite simplicial graph. If $A$ denotes the right-angled Artin group defined by $\Gamma$, then $\mathrm{Aut}(A)$ is acylindrically hyperbolic if and only if $\Gamma$ is not a join and contains at least two vertices. If $G$ is a graph product of finite groups over $\Gamma$ (e.g., the right-angled Coxeter group defined by $\Gamma$), then $\mathrm{Aut}(G)$ is acylindrically hyperbolic if and only if $\Gamma$ is a join of a complete graph (possibly empty) with a graph which is not a join, not a single vertex, nor two isolated vertices both labelled by $\mathbb{Z}_2$.
\end{cor}

\noindent
Notice that a graph product which is not acylindrically hyperbolic may have its automorphism group which is acylindrically hyperbolic anyway. This phenomenon cannot occur for right-angled Artin groups, but, for instance, $G:=\mathbb{Z} \oplus (\mathbb{Z}_3 \ast \mathbb{Z}_2)$ is not acylindrically hyperbolic when its automorphism group $\mathrm{Aut}(G)$, which is isomorphic to $\mathbb{Z}_2 \oplus (\mathbb{Z}_3 \ast \mathbb{Z}_2)$, is (acylindrically) hyperbolic.

\paragraph{Applications to extensions.} As observed in \cite{AutHypAcyl}, once we know that the automorphism group $\mathrm{Aut}(G)$ of the group $G$ under consideration is acylindrically hyperbolic, non-trivial information can be deduced about extensions of $G$. As a first application of Theorem \ref{BigThmIntro} in this direction, we prove the following criterion:

\begin{thm}\label{thm:Extension}
Let $\Gamma$ be a finite simplicial graph which is not a clique and $\mathcal{G}$ a collection of finitely generated groups indexed by $V(\Gamma)$. Assume that no group in $\mathcal{G}$ decomposes non-trivially as a graph product. Fix a group $H$ and a morphism $\varphi : H \to \mathrm{Aut}(\Gamma \mathcal{G})$. Decompose $\Gamma$ as a join $\Gamma_0 \ast \Gamma_1 \ast \cdots \ast \Gamma_n$ where $\Gamma_0$ is a clique and where each graph among $\Gamma_1, \ldots, \Gamma_n$ contains at least two vertices and is not a join. The semidirect product $\Gamma \mathcal{G} \rtimes_\varphi H$ is acylindrically hyperbolic if and only if
\begin{itemize}
	\item $G_u$ is finite for every $u \in V(\Gamma_0)$;
	\item $n=1$ and $\Gamma_1$ is not a pair of isolated vertices both labelled by $\mathbb{Z}_2$;
	\item the kernel of $H \overset{\varphi}{\to} \mathrm{Aut}(\Gamma \mathcal{G}) \to \mathrm{Out}(\Gamma \mathcal{G})$ is finite.
\end{itemize}
\end{thm}

\noindent
Notice that, as a particular case, if $\Gamma$ and $\mathcal{G}$ satisfy the first two conditions given by Theorem \ref{thm:Extension}, then, for every automorphism $\phi \in \mathrm{Aut}(\Gamma \mathcal{G})$, the semidirect product $\Gamma \mathcal{G} \rtimes_\phi \mathbb{Z}$ is acylindrically hyperbolic if and only if $\phi$ has infinite order in $\mathrm{Out}(\Gamma \mathcal{G})$. 

\begin{cor}\label{cor:Extension}
Let $\Gamma$ be a finite simplicial graph and $H$ a group. If $A$ denotes the right-angled Artin group defined by $\Gamma$, then, for every morphism $\varphi : H \to \mathrm{Aut}(A)$, the semidirect product $A \rtimes_\varphi H$ is acylindrically hyperbolic if and only if
\begin{itemize}
	\item $A$ is non-abelian and irreducible, i.e., $\Gamma$ is not a join and contains at least two vertices;
	\item the kernel of $H \to \mathrm{Aut}(A) \overset{\varphi}{\to} \mathrm{Out}(A)$ is finite.
\end{itemize}
And, if $G$ is a graph products of finite groups over $\Gamma$ (e.g., the right-angled Coxeter group defined by $\Gamma$), then, for every morphism $\psi : H \to \mathrm{Aut}(G)$, the semidirect product $G \rtimes_\psi H$ is acylindrically hyperbolic if and only if
\begin{itemize}
	\item $\Gamma$ is a join of a complete graph (possibly empty) with a graph which is not a join, which contains at least two vertices, and which is not a pair of isolated vertices;
	\item the kernel of $H \overset{\psi}{\to} \mathrm{Aut}(G) \to \mathrm{Out}(G)$ is finite.
\end{itemize}
\end{cor}

\noindent
In the particular case of cyclic extensions, the extension turns out to satisfy the vastness properties satisfied by acylindrically hyperbolic groups even though its automorphism group may not be acylindrically hyperbolic. More precisely:

\begin{thm}\label{thm:ExtensionVast}
Let $\Gamma$ be a finite simplicial graph and $\mathcal{G}$ a collection of finitely generated groups indexed by $V(\Gamma)$. Assume that no group in $\mathcal{G}$ decomposes non-trivially as a graph product, and that $\Gamma$ is not a join of isolated vertices and isolated pairs of vertices both labelled by $\mathbb{Z}_2$. For every $\varphi \in \mathrm{Aut}(\Gamma \mathcal{G})$, the semidirect product $\Gamma \mathcal{G} \rtimes_\varphi \mathbb{Z}$ is
\begin{itemize}
	\item SQ-universal;
	\item virtually has many quasimorphisms;
	\item is not boundedly generated.
\end{itemize}
\end{thm}

\noindent
As a consequence, cyclic extensions of right-angled Artin and Coxeter groups which are not virtually abelian turn out to be automatically SQ-universal.

\begin{cor}\label{cor:ExtensionVast}
Let $\Gamma$ be a finite graph. If $A$ denotes the right-angled Artin group defined by $\Gamma$, then for every automorphism $\varphi \in \mathrm{Aut}(A)$, the following assertions are equivalent:
\begin{itemize}
	\item $A$ is not abelian, i.e., $\Gamma$ is not complete;
	\item $A \rtimes_\varphi \mathbb{Z}$ is SQ-universal;
	\item $A \rtimes_\varphi \mathbb{Z}$ involves all finite groups;
	\item $A \rtimes_\varphi \mathbb{Z}$ virtually has many quasimorphisms;
	\item $A \rtimes_\varphi \mathbb{Z}$ is not boundedly generated.
\end{itemize}
Also, if $G$ is a graph product of finite groups over $\Gamma$ (e.g., the right-angled Coxeter group defined by $\Gamma$), then, for every automorphism $\psi \in \mathrm{Aut}(G)$, the following assertions are equivalent:
\begin{itemize}
	\item $G$ is not virtually abelian, i.e., $\Gamma$ is not a join of isolated vertices and isolated pairs of vertices labelled by $\mathbb{Z}_2$;
	\item $G \rtimes_\psi \mathbb{Z}$ is SQ-universal;
	\item $G \rtimes_\psi \mathbb{Z}$ involves all finite groups;
	\item $G \rtimes_\psi \mathbb{Z}$ virtually has many quasimorphisms;
	\item $G \rtimes_\psi \mathbb{Z}$ is not boundedly generated.
\end{itemize}
\end{cor}

\paragraph{Other applications.} Theorems \ref{thm:VastProp}, \ref{thm:WhenAutHyp}, \ref{thm:Extension} and \ref{thm:ExtensionVast} mentioned above are all consequences of Theorem \ref{BigThmIntro}. Interestingly, the techniques used to prove the latter theorem have other consequences of independent interest. Some of them are given in Section~\ref{section:Other}, including a solution to the isomorphism problem between graph products in some cases, the proof that many automorphism groups of graph products do not satisfy Kazhdan's property (T), estimations of asymptotic invariants, and the stability of coarse median groups (as defined in \cite{MR3037559}) under graph products over finite graphs. We refer to this section for precise statements.

\paragraph{Description of the proof of Theorem \ref{BigThmIntro}.} In the rest of the introduction, we describe the strategy used in order to prove Theorem \ref{BigThmIntro}. Actually, Theorem \ref{BigThmIntro} will be a consequence of the proof of the following statement:

\begin{thm}\label{Intro:MainAcylHyp}
Let $\Gamma$ be a finite simplicial graph and $\mathcal{G}$ a collection of finitely generated irreducible groups indexed by $V(\Gamma)$. Assume that $\Gamma$ is not a join and that it contains at least two vertices. Either $\Gamma$ is a pair of two isolated vertices both labelled by $\mathbb{Z}_2$, and $\Gamma \mathcal{G}$ is infinite dihedral; or $\mathrm{Aut}(\Gamma \mathcal{G})$ is acylindrically hyperbolic.
\end{thm}

\noindent
So, from now on, we focus on this theorem. Its proof decomposes into four essentially independent steps, and, in each one of these four steps, the geometric framework introduced in \cite{Qm} is used in a fundamental way. Given a simplicial graph $\Gamma$ and a collection of groups $\mathcal{G}$ indexed by $V(\Gamma)$, the Cayley graph
$$\mathrm{QM}(\Gamma, \mathcal{G}):= \mathrm{Cayl} \left( \Gamma \mathcal{G}, \bigcup\limits_{u \in V(\Gamma)} G_u \backslash \{1\} \right)$$
turns out to have a very nice geometry: it is a \emph{quasi-median graph}. We refer to Section~\ref{section:QM} for a description of this geometry, but the picture to keep in mind is that $\QM$ looks like a CAT(0) cube complex in which hyperplanes separate the space into at least two connected components, possibly more and even infinitely many. 

\medskip \noindent
As a warm up, let us sketch our argument in the specific case of right-angled Artin groups. So let $\Gamma$ be a finite connected simplicial graph and let $A$ denote the right-angled Artin group it defines. We assume that $\Gamma$ is not a join and that it contains at least two vertices.
\begin{enumerate}
	\item By identifying vertices in $\Gamma$ with the same star or the same link, describe $A$ as a graph product $\Phi \mathcal{G}$ over a graph $\Phi$ in which no two vertices have the same link or the same star. It follows from the the description of $\mathrm{Aut}(A)$ as generated by inversions, transvections, partial conjugations and graph isomorphisms \cite{ServatiusCent, RAAGgenerators} that, for every $\prec$-maximal vertex $u \in V(\Phi)$, there exist an element $g \in A$ and a $\prec$-maximal vertex $v \in V(\Phi)$ such that $\varphi (G_u) = g G_v g^{-1}$. Here, $\prec$ refers to the relation defined on the vertices of $\Phi$ as follows: $u \prec v$ when $\mathrm{link}(u) \subset \mathrm{star}(v)$.
	\item As a consequence, $\mathrm{Aut}(A)$ acts on the graph $\mathrm{ST}$ whose vertices are the conjugates of vertex-groups labelled by $\prec$-maximal vertices of $\Phi$ and whose edges link two subgroups when they commute. Geometrically, it turns out that $\mathrm{ST}$ coincides with the graph whose vertices are some hyperplanes of $\mathrm{QM}(\Phi, \mathcal{G})$ and whose edges link two hyperplanes whenever they are transverse. By exploiting this geometric interpretation, one shows that $\mathrm{ST}$ is a quasi-tree and that the inner automorphism associated to an element $g \in A$ of full support is WPD (see Section \ref{section:AcylHyp}) if $\{ \varphi \in \mathrm{Aut}(A) \mid \varphi(g)=g \}$ is virtually cyclic.
	\item Assume that $g \in A$ is fixed by automorphisms $\varphi_1 , \varphi_2, \ldots \in \mathrm{Aut}(A)$ which have pairwise distinct image in $\mathrm{Out}(A)$. Following Paulin's construction \cite{Paulin} for hyperbolic groups, the sequence of twisted actions $$\left\{ \begin{array}{ccc} A & \to & \mathrm{Isom}(A) \\ g & \mapsto & (x \mapsto \varphi_n(g) \cdot x) \end{array} \right.$$ converges to a fixed-point free action of $A$ on one of its asymptotic cones for which $g$ is elliptic. The trick is to embed $A$ quasi-isometrically and equivariantly into a product of simplicial trees in order to produce a fixed-point free action of $A$ on a real tree with arc-stabilisers in proper parabolic subgroups (i.e. subgroups generated by a proper subset of generators) and for which $g$ is elliptic.
	\item Finally, we prove that $A$ satisfies the following relative fixed-point property: there exists an element $g \in A$ of full support such that, for every action of $A$ on real tree with arc-stabilisers in proper parabolic subgroups, if $g$ fixes a point then $A$ has a global fixed point.
\end{enumerate}
By putting the four steps together, we conclude that $\mathrm{Aut}(A)$ is acylindrically hyperbolic. We emphasize that, in our argument, graph products are necessary even though we are only interested in right-angled Artin groups. Now, let us go back to the study of general graph products.

\medskip \noindent
\textbf{STEP 1:} A rigidity theorem.

\medskip \noindent
In order to show that the automorphism group of our graph product is acylindrically hyperbolic, we want to construct a hyperbolic graph on which it acts with WPD elements (see Section \ref{section:AcylHyp}). The vertices of such a hyperbolic graph will be conjugates of vertex-groups, so first we need to identify a family of vertex-groups invariant under automorphisms modulo conjugacy. 

\begin{cor}\label{CorRigidity}
Let $\Gamma$ be a finite connected simplicial graph and $\mathcal{G}$ a collection of graphically irreducible groups indexed by $V(\Gamma)$. Assume that no two vertices in $\Gamma$ have the same star or the same link, and that $\Gamma$ contains at least two vertices and is not a join. For every $\varphi \in \mathrm{Aut}(\Gamma \mathcal{G})$ and every $\prec$-maximal vertex $u \in V(\Gamma)$, there exist an element $g \in \Gamma \mathcal{G}$ and a $\prec$-maximal vertex $v \in V(\Gamma)$ such that $\varphi(G_u)=gG_vg^{-1}$. 
\end{cor}

\noindent
Here, $\prec$ refers to the relation defined on the vertices of $\Gamma$ as follows: $u \prec v$ if $\mathrm{link}(u) \subset \mathrm{star}(v)$. Graphically irreducible groups are defined in Section~\ref{subsection:graphicallyirr}; roughtly speaking, they are groups that do not decompose non-trivially as graph products. In the previous statement, we asked that no two vertices in $\Gamma$ have the same star or the same link, which is not an assumption in Theorem \ref{Intro:MainAcylHyp}. However, as shown by Proposition \ref{prop:SameLinkStar}, this extra assumption is not restrictive.

\medskip \noindent
Corollary \ref{CorRigidity} is a particular case of the following rigidity theorem, which is of independent interest:

\begin{thm}\label{thm:IntroRigidity}
Let $\Phi, \Psi$ be two finite simplicial graphs and $\mathcal{G}, \mathcal{H}$ two collections of graphically irreducible groups respectively indexed by $V(\Phi),V(\Psi)$. Assume that no two distinct vertices in $\Phi$ and $\Psi$ have the same star. For every isomorphism $\varphi : \Phi \mathcal{G} \to \Psi \mathcal{H}$ and every $\prec$-maximal vertex $u \in V(\Phi)$, there exist an element $g \in \Psi \mathcal{H}$ and a $\prec$-maximal vertex $v \in V(\Psi)$ such that $\varphi(\langle [u] \rangle)=g \langle [v] \rangle g^{-1}$. 
\end{thm}

\noindent
Here, $[\cdot]$ denotes the equivalence class of vertices of $\Gamma$ with respect to the following relation: $v \sim w$ if $v \prec w$ and $w \prec v$. 

\medskip \noindent
The proof of Theorem \ref{thm:IntroRigidity} is purely algebraic, relying on general statements about isomorphisms between products of groups admitting specific generating sets.

\medskip \noindent
\textbf{STEP 2:} Constructing a hyperbolic model.

\medskip \noindent
Thanks to Corollary \ref{CorRigidity} from the previous step, we know that $\mathrm{Aut}(\Gamma \mathcal{G})$ naturally acts on the following graph by isometries:

\begin{definition}
The \emph{small crossing graph} $\ST$ is the graph whose vertices are the conjugates of vertex-groups indexed by $\prec$-maximal vertices of $\Gamma$ and whose edges link two subgroups if they commute.
\end{definition}

\noindent
The terminology is justified by the following geometric interpretation:

\begin{prop}\label{prop:GeomInterpretation}
The small crossing graph coincides with the graph whose vertices are the hyperplanes of $\QM$ labelled by $\prec$-maximal vertices of $\Gamma$ and whose edges link two hyperplanes whenever they are transverse.
\end{prop}

\noindent
Usually, one refers to the \emph{crossing graph} of a CAT(0) cube complex, or here of a quasi-median graph, as the the graph whose vertices are the hyperplanes and whose edges link two hyperplanes whenever they are transverse. So, here, we are only considering a subgraph of the crossing graph, but which turns out to be isometrically embedded and quasi-dense, so crossing graph and small crossing graph are not so different. 

\medskip \noindent
By standard techniques related to crossing graphs (or contact graphs) \cite{MR3217625, MR4057355, RankCentraliser}, it is not difficult to prove, thanks to the geometric interpretation provided by Proposition \ref{prop:GeomInterpretation}, that the small crossing graph is a quasi-tree. Determining when an element of the automorphism group is WPD (see Section \ref{section:AcylHyp}) is much more delicate. We prove:

\begin{thm}\label{thm:IntroWPD}
Let $\Gamma$ be a finite connected simplicial graph and $\mathcal{G}$ a collection of finitely generated groups indexed by $V(\Gamma)$. Assume that the groups in $\mathcal{G}$ are graphically irreducible and that $\Gamma$ contains at least two vertices and is not a join. The small crossing graph $\ST$ is a quasi-tree on which $\mathrm{Aut}(\Gamma \mathcal{G})$ acts by isometries, and, for every element $g \in \Gamma \mathcal{G}$ of full support, the following assertions are equivalent:
\begin{itemize}
	\item[(i)] the inner automorphism $\iota(g)$ is WPD with respect to $\mathrm{Aut}(\Gamma \mathcal{G}) \curvearrowright \ST$;
	\item[(ii)] $\{ \varphi \in \mathrm{Aut}(\Gamma \mathcal{G}) \mid \varphi(g)=g\}$ is virtually cyclic;
	\item[(iii)]  $\{ \varphi \in \mathrm{Aut}(\Gamma \mathcal{G}) \mid \varphi(g)=g\}$ has an infinite image in $\mathrm{Out}(\Gamma \mathcal{G})$;
	\item[(iv)] the centraliser of $\iota(g)$ in $\mathrm{Aut}(\Gamma \mathcal{G})$ is virtually cyclic.
\end{itemize}
\end{thm}

\noindent
We refer to Section~\ref{subsection:GPprel} for the definition of the support of an element in a graph product. As a consequence, the proof of Theorem \ref{Intro:MainAcylHyp} reduces to the construction of an element $g$ such that $\{\varphi \in \mathrm{Aut}(\Gamma \mathcal{G}) \mid \varphi(g)=g\}$ has finite image in $\mathrm{Out}(\Gamma \mathcal{G})$.


\medskip \noindent
\textbf{STEP 3:} Constructing actions on real trees.

\medskip \noindent
Fix an element $\Gamma \mathcal{G}$ which does not satisfy the conditions provided by Theorem \ref{thm:IntroWPD}, i.e., assume that there exist infinitely many automorphisms $\varphi_1, \varphi_2, \ldots \in \mathrm{Aut}(\Gamma \mathcal{G})$ fixing $g$. Inspired by a well-known strategy for hyperbolic groups \cite{Paulin}, we take the ultralimit of the sequence of twisted actions 
$$\left\{ \begin{array}{ccc} \Gamma \mathcal{G} & \to & \mathrm{Isom}(\Gamma \mathcal{G}) \\ g & \mapsto & (x \mapsto \varphi_n(g) \cdot x) \end{array} \right.$$
in order to construct a fixed-point free action of $\Gamma \mathcal{G}$ on one of its asymptotic cones $\mathrm{Cone}(\Gamma \mathcal{G})$; see Section~\ref{section:PaulinRips}. However, asymptotic cones of graph products are much more complicated than asymptotic cones of hyperbolic groups (which are real trees). The trick we use here is to embed $\Gamma \mathcal{G}$, equivariantly and quasi-isometrically, into a product of trees of spaces. As a consequence, $\mathrm{Cone}(\Gamma \mathcal{G})$ embeds into a product of \emph{tree-graded spaces} $T_1 \times \cdots \times T_r$ (see Section \ref{section:TreeGraded} for more information about tree-graded spaces) and the action of $\Gamma \mathcal{G}$ extends to an action on $T_1 \times \cdots \times T_r$. Inspired by \cite{DrutuSapirActions}, the idea is to deduce from the action of $\Gamma \mathcal{G}$ on the tree-graded spaces $T_1, \ldots, T_r$ a fixed-point free action on a real tree with specific arc-stabilisers.

\begin{thm}\label{Intro:ActionRealTree}
Let $\Gamma$ be a finite connected simplicial graph and $\mathcal{G}$ a collection of finitely generated groups indexed by $V(\Gamma)$. Fix an irreducible element $g \in \Gamma \mathcal{G}$ and suppose that $\{ \varphi \in \mathrm{Aut}(\Gamma \mathcal{G}) \mid \varphi(g)=g \}$ has an infinite image in $\mathrm{Out}(\Gamma \mathcal{G})$. Then $\Gamma \mathcal{G}$ admits an action on a real tree without a global fixed point, with arc-stabilisers in product subgroups and with $g$ as an elliptic isometry. 
\end{thm}

\noindent
A technical difficulty here is that it is not clear that the fact that $\Gamma \mathcal{G}$ acts on $\mathrm{Cone}(\Gamma \mathcal{G})$ fixed-point freely implies that it also acts on $T_1 \times \cdots \times T_r$ fixed-point freely. There could be a global fixed point in $T_1 \times \cdot \times T_r$ but not in $\mathrm{Cone}(\Gamma \mathcal{G})$. When $\Gamma \mathcal{G}$ is a right-angled Artin group, this difficulty can be overcome because its asymptotic cones are CAT(0). Therefore, it follows from the Cartan fixed-point theorem that $\Gamma \mathcal{G}$ acts on $\mathrm{Cone}(\Gamma \mathcal{G})$ with unbounded orbits, which implies that there is no global fixed point in $T_1 \times \cdots \times T_r$. For general graph products, we prove the following fixed-point theorem:

\begin{thm}\label{Intro:FixedPointCone}
Let $\Gamma$ be a finite connected simplicial graph and $\mathcal{G}$ a collection of groups indexed by $V(\Gamma)$. We endow $\Gamma \mathcal{G}$ with the word metric coming from the union of generating sets of vertex-groups. Let $G$ be a finitely generated group and $\varphi_1, \varphi_2, \ldots$ an infinite collection of morphisms $G \to \Gamma \mathcal{G}$. Assume that the limit of twisted actions
$$\left\{ \begin{array}{ccc} G & \to & \mathrm{Isom}(\Gamma \mathcal{G}) \\ g & \mapsto & (x \mapsto \varphi_n(g) \cdot x) \end{array} \right.$$
defines an action of $G$ on $\mathrm{Cone}(\Gamma \mathcal{G}):= \mathrm{Cone}(\Gamma \mathcal{G},(1/\lambda_n),o)$. If this action has bounded orbits, then either $G$ has a global fixed point or there exists some $i \geq 1$ such that $\varphi_i(G)$ lies in a star-subgroup.
\end{thm}

\noindent
Star-subgroups refer to conjugates of subgroups generated by stars of vertices; see Section~\ref{subsection:GPprel}. Theorem \ref{Intro:FixedPointCone} relies on a general fixed-point theorem for median subspaces in products of tree-graded spaces. We refer to Proposition \ref{prop:ProdTreeGraded} for an explicit statement.

\medskip \noindent
\textbf{STEP 4:} Relative fixed-point property.

\medskip \noindent
As a consequence of Theorems \ref{thm:IntroWPD} and \ref{Intro:ActionRealTree}, the proof of Theorem \ref{Intro:MainAcylHyp} follows from our next statement: 

\begin{thm}\label{Intro:FRH}
Let $\Gamma$ be a finite simplicial graph and $\mathcal{G}$ a collection of groups indexed by $V(\Gamma)$. There exists an element $g \in \Gamma \mathcal{G}$ of full support such that $\Gamma \mathcal{G}$ satisfies the following relative fixed-point property: for every action of $\Gamma \mathcal{G}$ on a real tree with arc-stabilisers in proper parabolic subgroups, if $g$ fixes a point then $\Gamma \mathcal{G}$ admits a global fixed point.
\end{thm}

\noindent
Parabolic subgroups refer to conjugates of subgroups generated by subgraphs of $\Gamma$; see Section~\ref{subsection:GPprel}. The proof of Theorem \ref{Intro:FRH} relies on the following general statement:

\begin{prop}\label{prop:GeneralRFH}
Let $G$ be a finitely generated group and $\mathcal{H}$ a collection of subgroups. If $G$ acts on a hyperbolic space with at least one WPD element and with all the subgroups in $\mathcal{H}$ elliptic, then there exists an element $g\in G$ which does not belong to any subgroup in $\mathcal{H}$ such that the following relative fixed-point property holds: for every action of $G$ on a real tree with arc-stabilisers lying in subgroups in $\mathcal{H}$, if $g$ fixes a point then $G$ admits a global fixed point.
\end{prop}

\noindent
The idea to prove Proposition \ref{prop:GeneralRFH} is to reduce the problem to finitely generated free groups thanks to the construction from \cite{DGO} of normal free subgroups by small cancellation methods. In order to apply Proposition \ref{prop:GeneralRFH} to graph products, one shows that the cone-off of a graph product over its proper parabolic subgroups is hyperbolic and that, under the assumptions of Theorem \ref{Intro:FRH}, the induced action contains WPD elements.

\paragraph{Acknowledgments.} I would like to thank Javier Aramayona for having asked me several years ago, at the very beginning of my PhD thesis, which right-angled Artin groups have their automorphism groups acylindrically hyperbolic. I am also grateful to Gilbert Levitt and Camille Horbez for interesting conversations, Ashot Minasyan for Remarks~\ref{MinasyanI} and~\ref{MinasyanII}, and Nima Hoda for his comments regarding Corollary~\ref{cor:ShortcutGP}. Finally, I would like to thank the University of Vienna, for its hospitality during the elaboration of a part of my work, where I was supported by the Ernst Mach Grant ICM-2017-06478, under the supervision of Goulnara Arzhantseva; and the University Paris-Sud, where I was supported by the Fondation Math\'ematique Jacques Hadamard.

\section{Preliminaries}

\subsection{Graph products of groups}\label{subsection:GPprel}

\noindent
In this subsection, we give basic definitions, fix the notation and record preliminary results about graph products of groups. Recall from the introduction that, given a simplicial graph $\Gamma$ and a collection of groups $\mathcal{G}= \{G_u \mid u \in V(\Gamma) \}$ indexed by the vertices of $\Gamma$, the \emph{graph product}\index{Graph products of groups} $\Gamma \mathcal{G}$ is defined as the quotient
$$\left( \underset{u \in V(\Gamma)}{\ast} G_u \right) / \langle \langle [g,h], \ g \in G_u, h\in G_v, \{u,v\} \in E(\Gamma) \rangle \rangle$$
where $E(\Gamma)$ denotes the edge-set of $\Gamma$. The groups in $\mathcal{G}$ are referred to as \emph{vertex-groups}\index{Vertex-groups}.

\medskip \noindent
\textbf{Convention:} In all the article, vertex-groups are always non-trivial. 

\medskip \noindent
Notice that this convention is not restrictive. Indeed, a graph product with some of its vertex-groups which are trivial can always be described as a graph product over a smaller graph obtained by removing the vertices correspond to these trivial vertex-groups.

\paragraph{Some vocabulary about graphs.} Below, we record the terminology used in the rest of the article about graphs. Given a simplicial graph $\Gamma$:
\begin{itemize}
	\item A subgraph $\Lambda \subset \Gamma$ is a \emph{join}\index{Join} if there exists a partition $V(\Lambda)= A \sqcup B$, where $A$ and $B$ are both non-empty, such that every vertex of $A$ is adjacent to every vertex of~$B$.
	\item Given a vertex $u \in V(\Gamma)$, its \emph{link}\index{Link}, denoted by $\mathrm{link}(u)$, is the subgraph generated by the neighbors of $u$.
	\item More generally, given a subgraph $\Lambda \subset \Gamma$, its \emph{link}, denoted by $\mathrm{link}(\Lambda)$, is the subgraph generated by the vertices of $\Gamma$ which are adjacent to all the vertices of~$\Lambda$.
	\item Given a vertex $u \in V(\Gamma)$, its \emph{star}\index{Star}, denoted by $\mathrm{star}(u)$, is the subgraph generated by $\mathrm{link}(u) \cup \{u\}$.
	\item The \emph{opposite graph}\index{Opposite graph} of $\Gamma$, denoted by $\Gamma^\mathrm{opp}$, is the graph whose vertices are the vertices of $\Gamma$ and whose edges link two vertices if they are not adjacent in $\Gamma$.
\end{itemize}
Notice that a graph decomposes as join if and only if its opposite graph is disconnected. As a consequence, a graph decomposes in a unique way as a join of subgraphs which are not joins themselves, the factors corresponding to the connected components of the opposite graph.

\paragraph{Normal form.}\index{Normal form of graph products} Fix a simplicial graph $\Gamma$ and a collection of groups $\mathcal{G}$ indexed by $V(\Gamma)$. A \emph{word} in $\Gamma \mathcal{G}$ is a product $g_1 \cdots g_n$ where $n \geq 0$ and where each $g_i$ belongs to a vertex-group; the $g_i$'s are the \emph{syllables}\index{Syllables} of the word, and $n$ is the \emph{length} of the word. Clearly, the following operations on a word do not modify the element of $\Gamma \mathcal{G}$ it represents:
\begin{itemize}
	\item[(O1)] delete the syllable $g_i=1$;
	\item[(O2)] if $g_i,g_{i+1} \in G$ for some $G \in \mathcal{G}$, replace the two syllables $g_i$ and $g_{i+1}$ by the single syllable $g_ig_{i+1} \in G$;
	\item[(O3)] if $g_i$ and $g_{i+1}$ belong to two adjacent vertex-groups, switch them.
\end{itemize}
A word is \emph{graphically reduced}\index{Graphically reduced word} if its length cannot be shortened by applying these elementary moves. Given a word $g_1 \cdots g_n$ and some $1\leq i<n$, if the vertex-group associated to $g_i$ is adjacent to each of the vertex-groups of $g_{i+1}, \ldots, g_n$, then the words $g_1 \cdots g_n$ and $g_1\cdots g_{i-1} \cdot g_{i+1} \cdots g_n \cdot g_i$ represent the same element of $\Gamma \mathcal{G}$; we say that $g_i$ \textit{shuffles to the right}\index{Shuffling syllables}. Analogously, one can define the notion of a syllable shuffling to the left. If $g=g_1 \cdots g_n$ is a graphically reduced word and $h$ a syllable, then a graphical reduction of the product $gh$ is given by
\begin{itemize}
	\item $g_1 \cdots g_n$ if $h=1$;
	\item $g_1 \cdots g_{i-1} \cdot g_{i+1} \cdots g_n$ if $h \neq 1$, $g_i$ shuffles to the right and $g_i=h^{-1}$;
	\item $g_1 \cdots g_{i-1} \cdot g_{i+1} \cdots g_n \cdot (g_ih)$ if $h \neq 1$, $g_i$ and $h$ belong to the same vertex-group, $g_i$ shuffles to the right, $g_i \neq h^{-1}$ and $(g_ih)$ is thought of as a single syllable.
\end{itemize}
As a consequence, every element $g$ of $\Gamma \mathcal{G}$ can be represented by a graphically reduced word, and this word is unique up to applying the operation (O3). Often, elements of $\Gamma \mathcal{G}$ and graphically reduced words are identified. 
We refer to \cite{GreenGP} for more details (see also \cite{GPvanKampen} for a more geometric approach). 

\medskip \noindent
The following definition will be also useful in the sequel:

\begin{definition}
Let $\Gamma$ be a simplicial graph, $\mathcal{G}$ a collection of groups indexed by $V(\Gamma)$ and $g \in \Gamma \mathcal{G}$ an element. 
The \emph{tail}\index{Tail of an element} of $g$, denoted by $\mathrm{tail}(g)$, is the collection of the last syllables appearing in the graphically reduced words representing $g$. 
\end{definition}

\noindent
In graph products, there is also a notion of cyclic reduction:

\begin{definition}
A word $g_1 \cdots g_n$ is \emph{graphically cyclically reduced}\index{Graphically cyclically reduced word} if it is graphically reduced and if there do not exist two indices $1 \leq i < j \leq n$ such that $g_i$ and $g_j$ belong to the same vertex-group, $g_i$ shuffles to the left and $g_j$ shuffles to the right.
\end{definition}

\noindent
Notice that, as proved in \cite{GreenGP}, every graphically reduced word $g \in \Gamma \mathcal{G}$ can be uniquely written as a graphically reduced product $xyx^{-1}$ where $y$ is graphically cyclically reduced. The subgraph of $\Gamma$ generated by the vertices corresponding to the vertex-groups used to write the syllables of $y$ is the \emph{support}\index{Support of an element} of $g$, denoted by $\mathrm{supp}(g)$. In other words, the support of an element $g$ is the smallest subgraph $\Lambda \subset \Gamma$ such that $g$ belongs to a conjugate of the subgroup generated by the vertex-groups indexing the vertices of $\Lambda$. An element has \emph{full support}\index{Full support} if its support is the whole graph $\Gamma$.

\paragraph{Parabolic subgroups.} Let $\Gamma$ be a simplicial graph and $\mathcal{G}$ a collection of groups indexed by $V(\Gamma)$. Given a subgraph $\Lambda \subset \Gamma$, we denote by $\langle \Lambda \rangle$ the subgroup generated by the vertex-groups indexing the vertices in $\Lambda$. Notice that, as a consequence of the previous paragraph, $\langle \Lambda \rangle$ is canonically isomorphic to the graph product $\Lambda \mathcal{G}_{|\Lambda}$ where $\mathcal{G}_{|\Lambda}$ denotes the subcollection of $\mathcal{G}$ containing only the vertex-groups indexing the vertices in $\Lambda$. A subgroup $H \leq \Gamma \mathcal{G}$ is \emph{parabolic}\index{Parabolic subgroup} if there exist an element $g\in \Gamma \mathcal{G}$ and a subgraph $\Lambda \subset \Gamma$ such that $H=g \langle \Lambda \rangle g^{-1}$. If $\Lambda$ is a join (resp. a star), then $H$ is referred to as a \emph{join-subgroup}\index{Join-subgroup} (resp. a \emph{star-subgroup}\index{Star-subgroup}).

\medskip \noindent
Notice that the family of parabolic subgroups is stable under several group theoretic operations:

\begin{prop}\label{prop:InterParabolic}\emph{\cite[Corollary 3.6]{MR3365774}}
In a graph product, the intersection of two parabolic subgroups is again a parabolic subgroup.
\end{prop}

\begin{lemma}\label{lem:normaliser}\emph{\cite[Proposition 3.13]{MR3365774}}
Let $\Gamma$ be a simplicial graph and $\mathcal{G}$ a collection of groups indexed by $V(\Gamma)$. For every subgraph $\Lambda \subset \Gamma$, the normaliser of $\langle \Lambda \rangle$ is $\langle \Lambda \cup \mathrm{link}(\Lambda) \rangle$.
\end{lemma}

\noindent
The following statement will be also useful:

\begin{lemma}\label{lem:Inclusion}\emph{\cite[Lemma 3.17]{ConjAut}}
Let $\Gamma$ be a simplicial graph and $\mathcal{G}$ a collection of groups indexed by $V(\Gamma)$. Fix two subgraphs $\Lambda_1, \Lambda_2 \subset \Gamma$ and an element $g\in \Gamma \mathcal{G}$. If $\langle \Lambda_1 \rangle \subset g \langle \Lambda_2 \rangle g^{-1}$, then $\Lambda_1 \subset \Lambda_2$ and $g \in \langle \Lambda_1 \rangle \cdot \langle \mathrm{link}(\Lambda_1) \rangle \cdot \langle \Lambda_2 \rangle$.
\end{lemma}

\paragraph{Centralisers.} Our next proposition describes centralisers in graph products.

\begin{prop}\label{prop:centraliser}\emph{\cite{GPcentralisers}}
Let $\Gamma$ be a simplicial graph, $\mathcal{G}$ a collection of groups indexed by $V(\Gamma)$, and $g\in \Gamma \mathcal{G}$ an element. Decompose $\mathrm{supp}(g)$ as a join $A_1 \ast \cdots \ast A_r \ast B_1 \ast \cdots \ast B_s$ where each $A_i$ is a single vertex and where each $B_i$ contains at least two vertices and is not a join; and write $g$ as a graphically reduced word $h \cdot a_1 \cdots a_r b_1 \cdots b_s \cdot h^{-1}$ where $h \in \Gamma \mathcal{G}$, $a_i \in \langle A_i \rangle$ for every $1 \leq i \leq r$, and $b_j \in \langle B_j \rangle$ for every $1 \leq j \leq s$. Then
$$C_{\Gamma \mathcal{G}}(g)= h \left( \langle \mathrm{link}(\mathrm{supp}(g)) \rangle \oplus \bigoplus\limits_{i=1}^r C_{A_i}(a_i) \oplus \bigoplus\limits_{j=1}^s \langle c_j \rangle \right) h^{-1}$$
where, for every $1 \leq j \leq s$, $c_j$ is a primitive element such that $b_j \in \langle c_j \rangle$. 
\end{prop}

\noindent
As a consequence, we obtain the following description of centers of graph products.

\begin{lemma}\label{lem:Center}\emph{\cite[Theorem 3.34]{GreenGP}}
Let $\Gamma$ be a simplicial graph and $\mathcal{G}$ a collection of groups indexed by $V(\Gamma)$. Decompose $\Gamma$ as a join $\Gamma_1 \ast \Gamma_2$ where $\Gamma_1$ is complete and where no vertex in $\Gamma_2$ is adjacent to all the other vertices of $\Gamma_2$. The center of $\Gamma \mathcal{G}$ coincides with the center of $\langle \Gamma_1 \rangle$.
\end{lemma}

\paragraph{Product subgroups.} We refer to \emph{product subgroups}\index{Product subgroup} as subgroups which decompose non-trivially as direct products; and to \emph{maximal product subgroups}\index{Maximal product subgroup} as the subgroups which are maximal with respect to the inclusion among the product subgroups. 

\begin{prop}\label{prop:MaxProducts}
Let $\Gamma$ be a simplicial graph and $\mathcal{G}$ a collection of groups indexed by $V(\Gamma)$. If $P$ is a maximal product subgroup in $\Gamma \mathcal{G}$, either $P$ is included in a conjugate of an isolated vertex-group or $P=g \langle \Lambda \rangle g^{-1}$ where $\Lambda \subset \Gamma$ is a maximal join (with respect to the inclusion).
\end{prop}

\noindent
The proof of our characterisation is based on the following statement:

\begin{lemma}\label{lem:MO}\emph{\cite[Corollary 6.15]{MinasyanOsin}}
Let $\Gamma$ be a simplicial graph, $\mathcal{G}$ a collection of groups indexed by $V(\Gamma)$, and $H \leq \Gamma \mathcal{G}$ a subgroup which is not contained in a conjugate of a vertex-group. Either $H$ is contained in a join-subgroup or it contains an element whose centraliser in $\Gamma \mathcal{G}$ is infinite cyclic. 
\end{lemma}

\noindent
Notice that, if a group decomposes non-trivially as direct sum, then it cannot contain an element whose centraliser is infinite cyclic.

\begin{proof}[Proof of Proposition \ref{prop:MaxProducts}.]
Suppose that $H \leq \Gamma \mathcal{G}$ is a maximal product subgroup which does not lie in a conjugate of a vertex-group. It follows from Lemma~\ref{lem:MO} that there exist an element $g \in \Gamma \mathcal{G}$ and a join $\Lambda \subset \Gamma$ such that $H \subset g \langle \Lambda \rangle g^{-1}$. If $\Lambda$ is not a maximal join, i.e., if there exists a join $\Xi \subset \Gamma$ satisfying $\Lambda \subsetneq \Xi$, then $H \subset g \langle \Lambda \rangle g^{-1} \subsetneq g \langle \Xi \rangle g^{-1}$, which is impossible since $H$ is maximal. Consequently, $H= g \langle \Lambda \rangle g^{-1}$ where $\Lambda \subset \Gamma$ is a maximal join.

\medskip \noindent
Now, we want to prove that, if $g \in \Gamma \mathcal{G}$ is an element and $\Lambda \subset \Gamma$ a maximal join, then $g \langle \Lambda \rangle g^{-1}$ is a maximal product subgroup. So let $P$ be a subgroup of $\Gamma \mathcal{G}$ splitting non-trivially as a direct product and containing $g \langle \Lambda \rangle g^{-1}$. It follows from Lemma \ref{lem:MO} that there exist an element $h \in \Gamma \mathcal{G}$ and a join $\Xi \subset \Gamma$ such that 
$$g \langle \Lambda \rangle g^{-1} \subset P \subset h \langle \Xi \rangle h^{-1}.$$
By applying Lemma \ref{lem:Inclusion}, we know that $\Lambda \subset \Xi$ and that $h \in g \langle \Lambda \rangle \cdot \langle \mathrm{link}(\Lambda) \rangle \cdot \langle \Xi \rangle$. As $\Lambda$ is a maximal join, necessarily $\Lambda= \Xi$ and $\mathrm{link}(\Lambda)= \emptyset$. As a consequence, $h \in g \langle \Lambda \rangle$, so that
$$g \langle \Lambda \rangle g^{-1} \subset P \subset h \langle \Xi \rangle h^{-1}= h \langle \Lambda \rangle h^{-1} = g \langle \Lambda \rangle g^{-1}.$$
Therefore, $g \langle \Lambda \rangle g^{-1} = P$, concluding the proof. 
\end{proof}

\noindent
In a graph product, we refer to an element whose support is not a single vertex and does lie in a join as an \emph{irreducible element}\index{Irreducible element}. Our next statement, which is a consequence of \cite[Theorem 6.16]{MinasyanOsin}, describes subgroups which contain such elements.

\begin{prop}\label{prop:IrrElementExist}
Let $\Gamma$ be a simplicial graph, $\mathcal{G}$ a collection of groups indexed by $V(\Gamma)$, and $H \leq \Gamma \mathcal{G}$ a subgroup. If $H$ is not contained in a product subgroup nor in a conjugate of a vertex-group, then it contains an irreducible element of $\Gamma \mathcal{G}$.
\end{prop}

\subsection{Quasi-median geometry}\label{section:QM}

\begin{figure}
\begin{center}
\includegraphics[scale=0.4]{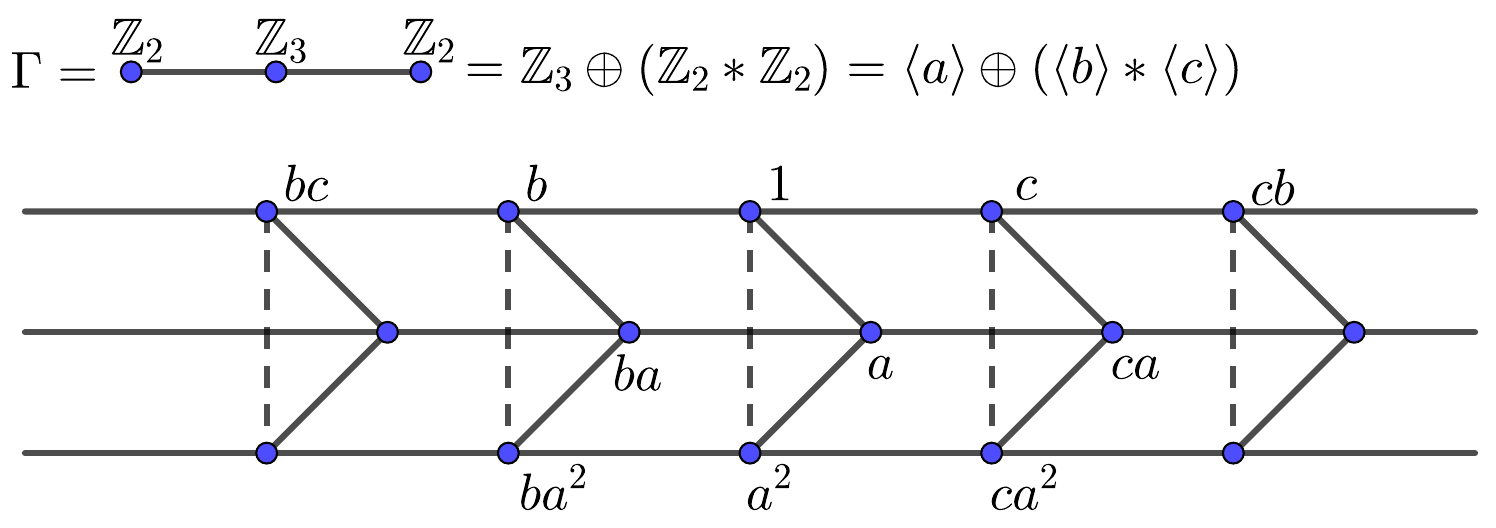}
\caption{Example of a graph $\QM$.}
\label{Cayl}
\end{center}
\end{figure}

\noindent
Let $\Gamma$ be a simplicial graph and $\mathcal{G}$ a collection of groups labelled by $V(\Gamma)$. This section is dedicated to the geometry of the following Cayley graph  of $\Gamma\mathcal{G}$:
$$\QM : = \mathrm{Cayl} \left( \Gamma \mathcal{G}, \bigcup\limits_{u \in V(\Gamma)} G_u \backslash \{1 \} \right),$$
i.e., the graph whose vertices are the elements of the group $\Gamma \mathcal{G}$ and whose edges link two distinct vertices $x,y \in \Gamma \mathcal{G}$ if $y^{-1}x$ is a non-trivial element of some vertex-group. Like in any Cayley graph, edges of $\QM$ are labelled by generators, namely by elements of vertex-groups. By extension, paths in $\QM$ are naturally labelled by words of generators. In particular, geodesics in $\QM$ correspond to words of minimal length. More precisely:

\begin{lemma}\label{lem:geodesicsQM}
Let $\Gamma$ be a simplicial graph and $\mathcal{G}$ a collection of groups indexed by $V(\Gamma)$. 
Fix two vertices $g,h \in \QM$. If $s_1 \cdots s_n$ is a graphically reduced word representing $g^{-1}h$, then 
$$g, \ gs_1, \ gs_1s_2, \ldots, gs_1 \cdots s_{n-1}, \ gs_1 \cdots s_{n-1}s_n = h$$
defines a geodesic in $\QM$ from $g$ to $h$. Conversely, if $s_1, \ldots, s_n$ is the sequence of elements of $\Gamma \mathcal{G}$ labelling the edges of a geodesic in $\QM$ from $g$ to $h$, then $s_1 \cdots s_n$ is a graphically reduced word representing $g^{-1}h$. 
\end{lemma}

\noindent
Notice that, as a consequence of Lemma \ref{lem:geodesicsQM}, for every $\Lambda \subset \Gamma$ the subgraph $\langle \Lambda \rangle \subset \QM$ is necessarily convex

\medskip \noindent
In \cite{Qm}, it is shown that $\QM$ is a \emph{quasi-median graph} and a general geometry of such graphs is developed in analogy with CAT(0) cube complexes. For simplicity, we describe the geometry of $\QM$ directly without referring to the geometry of quasi-median graphs.

\paragraph{The Cayley graph as a complex of prisms.} The first thing we want to highlight is that the Cayley graph $\QM$ has naturally the structure of a \emph{complex of prisms}. 

\begin{definition}
Let $X$ be a graph. A \emph{clique}\index{Clique} of $X$ is a maximal complete subgraph. A \emph{prism}\index{Prism} $P \subset X$ is an induced subgraph which decomposes as a product of finitely many~cliques.
\end{definition}

\noindent
In $\QM$, cliques and prisms correspond to cosets of parabolic subgroups, as described by our next lemmas:

\begin{lemma}\label{lem:CliqueStab}
Let $\Gamma$ be a simplicial graph and $\mathcal{G}$ a collection of groups indexed by $V(\Gamma)$. 
The cliques of $\QM$ coincide with the cosets $gG_u$, where $g \in \Gamma \mathcal{G}$ and $u \in V(\Gamma)$.
\end{lemma}

\begin{lemma}\label{lem:PrismGP}
Let $\Gamma$ be a simplicial graph and $\mathcal{G}$ a collection of groups indexed by $V(\Gamma)$. 
The prisms of $\QM$ coincide with the cosets $g \langle \Lambda \rangle$ where $g \in \Gamma \mathcal{G}$ and where $\Lambda \subset \Gamma$ is a complete subgraph.
\end{lemma}

\noindent
Therefore, $\QM$ can be thought of as the one-skeleton of a higher dimensional CW-complex by filling in the prisms described by Lemma \ref{lem:PrismGP} with products of simplices. Interestingly, such a prism complex can be naturally endowed with a CAT(0) metric, and in particular is contractible. We refer to \cite[Sections 2.11 and 2.12]{Qm} for more information.

\paragraph{Hyperplanes.} A fundamental tool in the geometric study of $\QM$ is the notion of \emph{hyperplanes}\index{Hyperplane}, defined as follows

\begin{definition}\label{def:hyp}
Let $\Gamma$ be a simplicial graph and $\mathcal{G}$ a collection of groups indexed by $V(\Gamma)$. A \emph{hyperplane} in $\QM$ is a class of edges with respect to the transitive closure of the relation claiming that two edges which are opposite in a square or which belong to a common triangle are equivalent. The \emph{carrier}\index{Carrier of a hyperplane} of a hyperplane $J$, denoted by $N(J)$, is the subgraph of $\QM$ generated by $J$; and a \emph{fiber}\index{Fiber of a hyperplane} of $J$ is a connected component of the graph $N(J) \backslash \backslash J$ obtained from $N(J)$ by removing the interiors of the edges in $J$. Two hyperplanes $J_1$ and $J_2$ are \emph{transverse}\index{Transverse hyperplanes} if they intersect a square along two distinct pairs of opposite edges, and they are \emph{tangent}\index{Tangent hyperplanes} if they are distinct, not transverse and if their carriers intersect.
\end{definition}
\begin{figure}
\begin{center}
\includegraphics[trim={0 16.5cm 10cm 0},clip,scale=0.4]{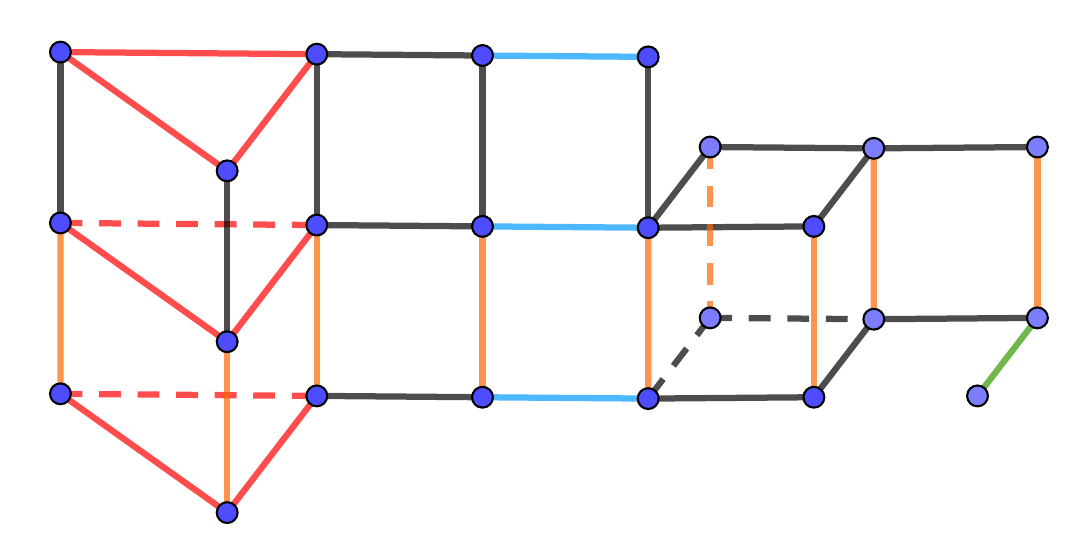}
\caption{Four hyperplanes in a quasi-median graph, colored in red, blue, green, and orange. The orange hyperplane is transverse to the blue and red hyperplanes. The green and orange hyperplanes are tangent.}
\label{figure3}
\end{center}
\end{figure}

\noindent
We refer to Figure \ref{figure3} for examples of hyperplanes. For convenience, for every vertex $u \in V(\Gamma)$ we denote by $J_u$ the hyperplane which contains all the edges of the clique $G_u$. Hyperplanes of $\QM$ can be described as follows:

\begin{thm}\label{thm:HypStab}\emph{\cite[Theorem 2.10]{ConjAut}}
Let $\Gamma$ be a simplicial graph and $\mathcal{G}$ a collection of groups indexed by $V(\Gamma)$. 
For every hyperplane $J$ of $\QM$, there exist some $g \in \Gamma \mathcal{G}$ and $u \in V(\Gamma)$ such that $J=gJ_u$. Moreover, $N(J)= g \langle \mathrm{star}(u) \rangle$ and $\mathrm{stab}(J)= g \langle \mathrm{star}(u) \rangle g^{-1}$. 
\end{thm}

\noindent
It is worth noticing that, as a consequence of Theorem \ref{thm:HypStab}, the hyperplanes of $\QM$ are naturally labelled by $V(\Gamma)$. More precisely, since any hyperplane $J$ of $\QM$ is a translate of some $J_u$, we say that the corresponding vertex $u \in V(\Gamma)$ \emph{labels}\index{Labels of hyperplanes} $J$. Equivalently, by noticing that the edges of $\QM$ are naturally labelled by vertices of $\Gamma$, the vertex of $\Gamma$ labelling a hyperplane coincides with the common label of all its edges (as justified by \cite[Facts 2.5 and 2.7]{ConjAut}). Let us record the following elementary but quite useful statement:

\begin{lemma}\label{lem:transverseimpliesadj}\emph{\cite[Lemma 2.13]{ConjAut}}
Let $\Gamma$ be a simplicial graph and $\mathcal{G}$ a collection of groups indexed by $V(\Gamma)$. 
Two transverse hyperplanes of $\QM$ are labelled by adjacent vertices of $\Gamma$, and two tangent hyperplanes of $\QM$ are labelled by distinct vertices of $\Gamma$. 
\end{lemma}

\noindent
The close connection between the geometry of $\QM$ and the combinatorics of its hyperplanes is justified by the following statement:

\begin{thm}\label{thm:BigThmQM}\emph{\cite[Theorem 2.14]{ConjAut}}
Let $\Gamma$ be a simplicial graph and $\mathcal{G}$ a collection of groups indexed by $V(\Gamma)$. 
The following statements hold:
\begin{itemize}
	\item For every hyperplane $J$, the graph $\QM \backslash \backslash J$ is disconnected. Its connected components are called \emph{sectors}\index{Sectors}.
	\item For any two vertices $x,y \in \QM$, $d(x,y)= \# \{ \text{hyperplanes separating $x$ and $y$} \}$.
	\item A path in $\QM$ is a geodesic if and only if it intersects each hyperplane at most once.
\end{itemize}
\end{thm}

\noindent
In this statement, we denoted by $\QM \backslash \backslash J$, where $J$ is a hyperplane, the graph obtained from $\QM$ by removing the interiors of the edges of $J$.

\paragraph{Rotative-stabilisers.} Another useful tool when working with the graph $\QM$ is the notion of \emph{rotative-stabiliser}.

\begin{definition}
Let $\Gamma$ be a simplicial graph and $\mathcal{G}$ a collection of groups indexed by $V(\Gamma)$. Given a hyperplane $J$ in $\QM$, its \emph{rotative-stabiliser}\index{Rotative-stabiliser} is the following subgroup of $\Gamma \mathcal{G}$:
$$\mathrm{stab}_{\circlearrowleft}(J) := \bigcap\limits_{\text{$C$ clique dual to $J$}} \mathrm{stab}(C).$$
\end{definition}

\noindent
Rotative-stabilisers of hyperplanes in $\QM$ are described as follows:

\begin{lemma}\label{lem:RotativeStab}\emph{\cite[Proposition 2.21]{ConjAut}}
Let $\Gamma$ be a simplicial graph and $\mathcal{G}$ a collection of groups indexed by $V(\Gamma)$. 
The rotative-stabiliser of a hyperplane $J$ of $\QM$ coincides with the stabiliser of any clique dual to $J$. Moreover, $\mathrm{stab}_{\circlearrowleft}(J)$ acts freely and transitively on the set of sectors delimited by $J$, and it stabilises each sector delimited by the hyperplanes transverse to $J$; in particular, it stabilises the hyperplanes transverse to $J$. 
\end{lemma}

\noindent
The picture to keep in mind is that the stabiliser of the hyperplane $J_u$ decomposes as $\langle \mathrm{star}(u) \rangle = \langle u \rangle \oplus \langle \mathrm{link}(u) \rangle$. One part of the stabiliser, namely $\langle \mathrm{link}(u) \rangle$, stabilises each fiber of $J_u$ and it does not stabilise any clique dual to $J_u$. And the other part of the stabiliser, namely $\langle u \rangle$, stabilises each clique dual to $J_u$ and it does not stabilise any fiber of $J_u$.

\paragraph{Projections on parabolic subgraphs.} A fruitful observation is that there exist natural projections onto some subgraphs of $\QM$. More precisely, if $\Lambda$ is a subgraph of $\Gamma$, then the vertices of $\QM$ can be projected onto the subgraph $\langle \Lambda \rangle \subset \QM$. This covers cliques and prisms, according to Lemmas \ref{lem:CliqueStab} and \ref{lem:PrismGP}, but also carriers of hyperplanes according to Theorem \ref{thm:HypStab}. Our projections are defined by the next statement:

\begin{prop}\label{prop:ProjHyp}\emph{\cite[Proposition 2.15]{ConjAut}}
Let $\Gamma$ be a simplicial graph and $\mathcal{G}$ a collection of groups indexed by $V(\Gamma)$. 
Fix a subgraph $\Lambda \subset \Gamma$ and a vertex $g \in \QM$. There exists a unique vertex $x$ of $\langle \Lambda \rangle$ minimising the distance to $g$. 
\end{prop}

\noindent
So, given a subgraph $\Lambda \subset \Gamma$ and an element $g \in \Gamma$, we can define the \emph{projection}\index{Projection onto subgraphs} $\mathrm{proj}_{g \langle \Lambda \rangle} : \QM \to g\langle \Lambda \rangle$ as the map which sends each vertex in $\QM$ to the unique vertex in $g \langle \Lambda \rangle$ which minimises the distance to it. Below, we record a few properties satisfied by this projection.

\begin{lemma}\label{lem:proj}\emph{\cite[Proposition 2.15]{ConjAut}}
Let $\Gamma$ be a simplicial graph and $\mathcal{G}$ a collection of groups indexed by $V(\Gamma)$. 
Fix a subgraph $\Lambda \subset \Gamma$ and a vertex $g \in \QM$. Any hyperplane separating $g$ from $x$ separates $g$ from $\langle \Lambda \rangle$ (i.e., $g$ and $\langle \Lambda \rangle$ lie in distinct sectors delimited by the hyperplane). 
\end{lemma}

\begin{lemma}\label{lem:minseppairhyp}\emph{\cite[Corollary 2.18]{ConjAut}}
Let $\Gamma$ be a simplicial graph, $\mathcal{G}$ a collection of groups indexed by $V(\Gamma)$, $\Lambda, \Xi \subset \Gamma$ two subgraphs and $g,h \in \Gamma \mathcal{G}$ two elements. Fix two vertices $x \in g \langle \Lambda \rangle$ and $y \in h \langle \Xi \rangle$ minimising the distance between $g\langle \Lambda \rangle$ and $h\langle \Xi \rangle$. The hyperplanes separating $x$ and $y$ are precisely those separating $g\langle \Lambda \rangle$ and $h\langle \Xi \rangle$. 
\end{lemma}

\begin{lemma}\label{lem:ProjInclusion}
Let $\Gamma$ be a simplicial graph, $\mathcal{G}$ a collection of groups indexed by $V(\Gamma)$, $\Lambda, \Xi \subset \Gamma$ two subgraphs and $g,h \in \Gamma \mathcal{G}$ two elements. If $g \langle \Lambda \rangle \subset h \langle \Xi \rangle$, then the equality $\mathrm{proj}_{g \langle \Lambda \rangle} \circ \mathrm{proj}_{h \langle \Xi \rangle} = \mathrm{proj}_{g \langle \Lambda \rangle}$ holds. 
\end{lemma}

\begin{proof}
Let $x \in \QM$ be a vertex. Let $x'$ denote its projection onto $h \langle \Xi \rangle$ and $x''$ the projection of $x'$ onto $g \langle \Lambda \rangle$. A hyperplane separating $x$ from $x''$ separates either $x$ and $x'$ or $x'$ and $x''$. But a hyperplane separating $x$ and $x'$ has to separate $x$ from $g \langle \Lambda \rangle$ according to Lemma \ref{lem:proj}. And a hyperplane separating $x'$ and $x''$ has to separate $x'$ from $g \langle \Lambda \rangle$ according to Lemma \ref{lem:proj}, but it cannot separate $x$ and $x'$ as a consequence of the previous observation, so it must separate $x$ from $g \langle \Lambda \rangle$. Thus, every hyperplane separating $x$ from $x''$ separates $x$ from $g \langle \Lambda \rangle$. Necessarily, $x''$ minimises the distance to $x$, i.e., it coincides with the projection of $x$ onto $g \langle \Lambda \rangle$, as desired.
\end{proof}

\paragraph{Bridges.} Below, we introduce bridges between subgraphs, which play an analogous role of the geodesics between the unique pair minimising the distance between two disjoint subtrees in a tree.

\begin{definition}
Let $\Gamma$ be a simplicial graph, $\mathcal{G}$ a collection of groups indexed by $V(\Gamma)$, $\Lambda, \Xi \subset \Gamma$ two subgraphs and $g,h \in \Gamma \mathcal{G}$ two elements. The \emph{bridge}\index{Bridge} between $g \langle \Lambda \rangle$ and $h \langle \Xi \rangle$ is the union of all the geodesics between pairs of vertices minimising the distance between $g \langle \Lambda \rangle$ and $h \langle \Xi \rangle$.
\end{definition}

\noindent
We record two properties satisfied by bridges for future use.

\begin{lemma}\label{lem:bridge}
Let $\Gamma$ be a simplicial graph, $\mathcal{G}$ a collection of groups indexed by $V(\Gamma)$, $\Lambda, \Xi \subset \Gamma$ two subgraphs and $g,h \in \Gamma \mathcal{G}$ two elements. The projection of $g \langle \Lambda \rangle$ onto $h \langle \Xi \rangle$ lies in the bridge between $g \langle \Lambda \rangle$ and $h \langle \Xi \rangle$.
\end{lemma}

\begin{proof}
Fix a vertex $p \in g \langle \Lambda \rangle$ and let $p' \in h \langle \Xi \rangle$ denote its projection onto $h \langle \Xi \rangle$. Also, let $p'' \in g \langle \Lambda \rangle$ denote the projection of $p'$ onto $g \langle \Lambda \rangle$. Let $J$ be a hyperplane separating $p'$ and $p''$. It follows from Lemma \ref{lem:proj} that $J$ separates $p'$ from $g \langle \Lambda \rangle$. As a consequence, it also separates $p$ and $p'$, so it follows from Lemma \ref{lem:proj} that $J$ separates $p$ from $h \langle \Xi \rangle$. We conclude that $J$ separates $g \langle \Lambda \rangle$ and $h \langle \Xi \rangle$. Therefore, the hyperplanes separating $p'$ and $p''$ turn out to coincide with the hyperplanes separating $g \langle \Lambda \rangle$ and $h \langle \Xi \rangle$. This shows that $p'$ and $p''$ minimise the distance between $g \langle \Lambda \rangle$ and $h \langle \Xi \rangle$, and we conclude that $p'$ belongs to the bridge between $g \langle \Lambda \rangle$ and $h \langle \Xi \rangle$ as desired.
\end{proof}

\begin{lemma}\label{lem:BridgeInProduct}
Let $\Gamma$ be a simplicial graph, $\mathcal{G}$ a collection of groups indexed by $V(\Gamma)$, $\Lambda, \Xi \subset \Gamma$ two subgraphs and $g,h \in \Gamma \mathcal{G}$ two elements. Assume that $g \langle \Lambda \rangle \cap h \langle \Xi \rangle = \emptyset$ and that there exists a hyperplane crossing both $g \langle \Lambda \rangle$ and $h \langle \Xi \rangle$. Then there exist an element $k \in \Gamma \mathcal{G}$ and a join $\Omega \subset \Gamma$ such that the bridge between $g \langle \Lambda \rangle$ and $h \langle \Xi \rangle$ lies in $k \langle \Omega \rangle$.
\end{lemma}

\begin{proof}
Let $J$ be a hyperplane separating two vertices $p$ and $q$ of the bridge between $g \langle \Lambda \rangle$ and $h \langle \Xi \rangle$. Fix two pairs of vertices $a \in g \langle \Lambda \rangle$, $b \in h \langle \Xi \rangle$ and $x \in g \langle \Lambda \rangle$, $y \in h \langle \Xi \rangle$ minimising the distance between $g \langle \Lambda \rangle$ and $h \langle \Xi \rangle$ such that $p$ belongs to a geodesic $[a,b]$ and $q$ to a geodesic $[x,y]$. As a consequence of Lemma \ref{lem:minseppairhyp}, only two cases may happen: either $J$ crosses $[a,b] \cup [x,y]$ and it separates $g \langle \Lambda \rangle$ and $h \langle \Xi \rangle$; or it separates $\{a,b\}$ and $\{x,y\}$, and it crosses both $g \langle \Lambda \rangle$ and $h \langle \Xi \rangle$. 

\medskip \noindent
Thus, we have proved that all the hyperplanes separating two vertices in the bridge between $g \langle \Lambda \rangle$ and $h \langle \Xi \rangle$ lie in the union $\mathcal{V} \cup \mathcal{H}$, where $\mathcal{V}$ denotes the set of all the hyperplanes separating $g \langle \Lambda \rangle$ and $h \langle \Xi \rangle$, and $\mathcal{H}$ the set of all the hyperplanes crossing both $g \langle \Lambda \rangle$ and $h \langle \Xi \rangle$. Notice that, by assumption, $\mathcal{H}$ and $\mathcal{V}$ are both non-empty. Let $\Phi$ (resp. $\Psi$) denote the subgraph of $\Gamma$ generated by the vertices labelling the hyperplanes in $\mathcal{V}$ (resp. $\mathcal{H}$). Because every hyperplane in $\mathcal{H}$ is transverse to every hyperplane in $\mathcal{V}$, it follows from Lemma \ref{lem:transverseimpliesadj} that every vertex in $\Phi$ is adjacent to every vertex in $\Psi$, i.e., $\Phi$ and $\Psi$ generate a join $\Omega \subset \Gamma$. 

\medskip \noindent
Fix an arbitrary vertex $k$ in our bridge. Because the hyperplanes separating $k$ from any other vertex of the bridge are labelled by vertices in $\Omega$, we conclude that the bridge lies in $k \langle \Omega \rangle$. 
\end{proof}

\paragraph{Median triangles.} Interestingly, in our graphs, every triple of vertices turns out to admit a \emph{median prism} in the same way that every triple of vertices in a tree admits a median point (namely, the center of the tripod the vertices define). This idea is made precise by Proposition \ref{prop:MedianTriangleQM} below, based on the following definition:

\begin{definition}
Let $\Gamma$ be a simplicial graph, $\mathcal{G}$ a collection of groups indexed by $V(\Gamma)$, and $x,y,z \in \QM$ three vertices. A triple of vertices $(x',y',z')$ is a \emph{median triple}\index{Median triple} if
$$\left\{ \begin{array}{l} d(x,y)=d(x,x')+d(x',y')+d(y',y) \\ d(x,z) = d(x,x')+ d(x',z')+d(z',z) \\ d(y,z)=d(y,y')+d(y',z')+d(z',z) \end{array} \right..$$
Its \emph{size} is $\max (d(x',y'),d(y',z'),d(x',z'))$. A \emph{median triangle}\index{Median triangle} is a median triple of minimal size.
\end{definition}

\noindent
The picture to keep in mind is that going from $x$ to $y$ by following geodesics through $x'$ and next $y'$ produces a geodesic from $x$ to $y$; and similarly for $x,z$ and $y,z$. See for instance Figure \ref{MedianTriangleProof} for an illustration. The existence and uniqueness of median triangles is provided by our next proposition.

\begin{prop}\label{prop:MedianTriangleQM}
Let $\Gamma$ be a simplicial graph, $\mathcal{G}$ a collection of groups indexed by $V(\Gamma)$, and $x,y,z \in \QM$ three vertices. The triple $(x,y,z)$ admits a unique median triangle $(x',y',z')$. Moreover, $x',y',z'$ belong to a common prism and the hyperplanes of the smallest prism which contains them coincide with the hyperplanes containing  $x',y',z'$ in pairwise distinct sectors.
\end{prop}

\noindent
We begin by proving a preliminary lemma. In its statement, we denote by $I(x,y)$ the \emph{interval} between two vertices $x$ and $y$, i.e., the union of all the geodesics between the two vertices.

\begin{lemma}\label{lem:MedianTriangle}
Let $\Gamma$ be a simplicial graph, $\mathcal{G}$ a collection of groups indexed by $V(\Gamma)$, and $x,y,z \in \QM$ three vertices. Let $w$ denote a vertex of $I(x,y) \cap I(x,z)$ which maximises the distance to $x$. Every hyperplane separating $w$ from $y$ and from $z$ has to separate $y$ and $z$.
\end{lemma}

\begin{proof}
Let $m$ be a vertex in $I(x,y) \cap I(x,z)$ and assume that there exists a hyperplane $J$ which separates $m$ from $y$ and $z$ but which does not separate $y$ and $z$. Let $F$ denote the fiber of $J$ corresponding to the sector which contains $y$ and $z$. Up to translating by an element of $\Gamma \mathcal{G}$, we may assume without loss of generality that $J=J_u$ for some $u \in V(\Gamma)$. As a consequence of \cite[Theorem 2.10 and Proposition 2.11]{ConjAut}, the fibers of $J_u$ are the $g \langle \mathrm{link}(u) \rangle$ where $g \in \langle u \rangle$. So the projection onto $F$ is well-defined: let $m'$ denote the projection of $m$ onto $F$. 

\medskip \noindent
Fix two geodesics $[m,m']$ and $[m',y]$. Notice that, as a consequence of Lemma \ref{lem:proj}, a hyperplane crossing $[m,m']$ cannot intersect $F$, so it cannot cross $[m',y]$. Consequently, $[m,m'] \cup [m',y]$ cannot cross a hyperplane twice: it must be a geodesic. In other words, $m'$ belongs to $I(m,y)$, which implies that it belongs to $I(x,y)$ since $m \in I(x,y)$. One shows similarly that $m'$ belongs to $I(x,z)$.

\medskip \noindent
Thus, we have proved that there exists a vertex $m' \in I(x,y) \cap I(x,z)$ which satisfies $d(x,m')>d(x,m)$, concluding the proof of our lemma.
\end{proof}

\begin{proof}[Proof of Proposition \ref{prop:MedianTriangleQM}.]
Let $x'$ (resp. $y'$, $z'$) be a vertex in $I(x,y) \cap I(x,z)$ (resp. $I(y,x) \cap I(y,z)$, $I(z,x) \cap I(z,y)$) which maximises the distance to $x$ (resp. $y$, $z$). We fix geodesics $[x,x']$, $[y,y']$, $[z,z']$, $[x',y']$, $[y',z']$, $[x',z']$, $[x,y']$, $[x,z']$, $[y,z']$, $[y,x']$, $[z, x']$ and $[z,y']$. First, we claim that that $(x',y',z')$ is a median triple. See Figure \ref{MedianTriangleProof}(a).
\begin{figure}
\begin{center}
\includegraphics[scale=0.4]{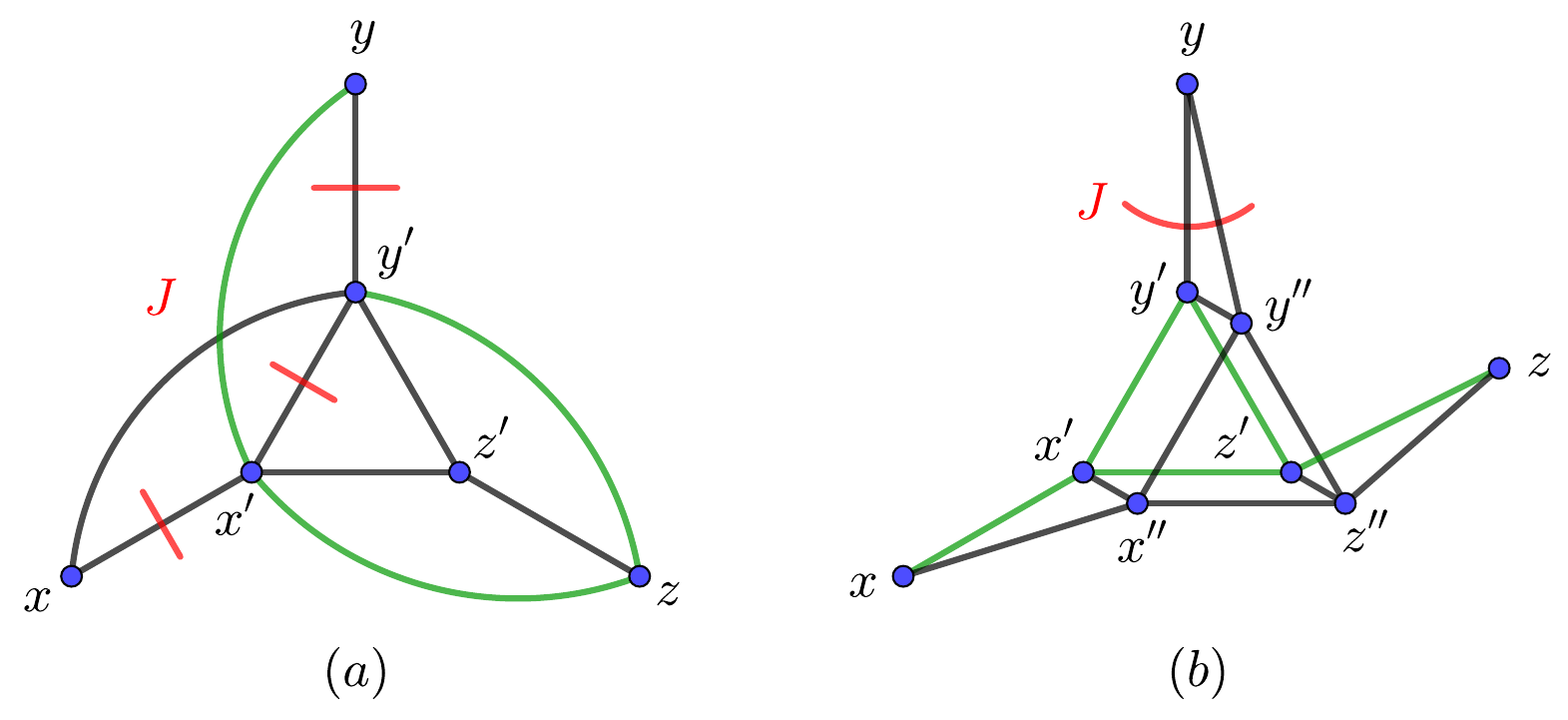}
\caption{In green, geodesics not crossed by the hyperplane $J$.}
\label{MedianTriangleProof}
\end{center}
\end{figure}

\medskip \noindent
Assume for contradiction that $[x',y'] \cup [y',y]$ is not a geodesic, which means that it crosses a hyperplane twice. So let $J$ be a hyperplane crossing both $[x',y']$ and $[y',y]$. Notice that:
\begin{itemize}
	\item Because $y' \in I(y,x) \cap I(y,z)$, $[y,y'] \cup [y',z]$ and $[y,y'] \cup [y',x]$ are geodesics and they cannot be crossed by $J$ twice. As $J$ already crosses $[y,y']$, it follows that $J$ does not cross $[y',z]$ nor $[y',x]$.
	\item Because $J$ crosses $[y',x']$, it has to cross $[y',x] \cup [x,x']$. But we know that $J$ does not cross $[y',x]$, so it has to cross $[x',x]$.
	\item Because $x' \in I(x,y) \cap I(x,z)$, $[x,x'] \cup [x',y]$ and $[x,x'] \cup [x',z]$ are geodesics and they cannot be crossed by $J$ twice. As $J$ already crosses $[x,x']$, it follows that $J$ does not cross $[x',y]$ nor $[x',z]$.
\end{itemize}
Thus, we conclude that $J$ does not cross the path $[y,x'] \cup [x',z] \cup [z,y']$ between $y$ and $y'$, contradicting the fact that $J$ separates $y$ and $y'$. 

\medskip \noindent
So $[x',y'] \cup [y',y]$ is a geodesic, and we deduce from the fact that $x' \in I(x,y)$ that $[x,x'] \cup [x',y'] \cup [y',y]$ must be a geodesic as well. Thanks to a symmetric argument, one shows that $[x,x'] \cup [x',z'] \cup [z',z]$ and $[y,y'] \cup [y',z'] \cup [z',z]$ are geodesics. In other words, $(x',y',z')$ is a median triple as claimed.

\begin{claim}\label{claim:TripleOne}
If a hyperplane $J$ separates two points among $x',y',z'$, then $x,y,z$ lie in three distinct sectors delimited by $J$.
\end{claim}

\noindent
Say that $J$ separates $x'$ and $y'$. So it has also to separate $x'$ and $z'$ or $y'$ and $z'$. Say that $J$ separates $x'$ and $z'$. So far, we know that $J$ separates $x'$ from $y'$ and from $z'$. We conclude from Lemma \ref{lem:MedianTriangle} that $y'$ and $z'$ are also separated by $J$. Because $[x,x'] \cup [x',y'] \cup [y',y]$ (resp. $[x,x'] \cup [x',z'] \cup [z',z]$, $[y,y'] \cup [y',z'] \cup [z',z]$) is a geodesic and that $J$ crosses $[x',y']$ (resp. $[x',z']$, $[y',z']$), it follows that $J$ separates $x$ and $y$ (resp. $x$ and $z$, $y$ and $z$). Thus, $x,y,z$ lie in three distinct sectors delimited by $J$. 

\begin{claim}\label{claim:TripleTwo}
Let $(x'',y'',z'')$ be a median triple of $(x,y,z)$. If $J$ is a hyperplane separating $x$ from $x''$, then it does not separate $y$ and $z$.
\end{claim}

\noindent
Fix geodesics $[x,x'']$, $[y,y'']$, $[z,z'']$, $[x'',y'']$, $[y'',z'']$, and $[x'',z'']$. Because $[x,x''] \cup [x'',y''] \cup [y'',y]$ is a geodesic, it cannot be crossed by $J$ twice, so $J$ cannot cross $[x'',y''] \cup [y'',y]$ since it already crosses $[x,x'']$. Similarly, $J$ cannot cross $[x'',z''] \cup [z'',z]$. Moreover, $J$ cannot cross $[y'',z'']$ since otherwise it would have to cross $[y'',x''] \cup [x'',z'']$, contradicting the previous observation. So $J$ does cross the geodesic $[y,y''] \cup [y'',z''] \cup [z'',z]$; or, in other words, it does not separate $y$ and $z$. 

\begin{claim}\label{claim:TripleThree}
Let $(x'',y'',z'')$ be a median triple of $(x,y,z)$. If $J$ is a hyperplane such that $x,y,z$ lie in three distinct sectors delimited by $J$, then so do $x'',y'',z''$.
\end{claim}

\noindent
Fix geodesics $[x,x'']$, $[y,y'']$, $[z,z'']$, $[x'',y'']$, $[y'',z'']$, and $[x'',z'']$. By applying the previous claim three times, we deduce that, if $J$ is a hyperplane such that $x,y,z$ lie in three distinct sectors delimited by $J$, then it has to cross the geodesics $[x,x''] \cup [x'',y''] \cup [y'',z'']$, $[x,x''] \cup [x'',z''] \cup [z'',z]$ and $[y,y''] \cup [y'',z''] \cup [z'',z]$ but it cannot cross $[x,x'']$, $[y,y'']$ and $[z,z'']$. So $J$ crosses $[x'',y'']$, $[y'',z'']$ and $[x'',z'']$. In other words, $x'', y'', z''$ lie in three distinct sectors delimited by $J$, concluding the proof of our claim.

\medskip \noindent
By combining Claims \ref{claim:TripleOne} and \ref{claim:TripleThree}, we deduce that the median triple $(x',y',z')$ has minimal size, i.e., it is a median triangle. We claim that it is the only one. See Figure~\ref{MedianTriangleProof}(b).

\medskip \noindent
So let $(x'',y'',z'')$ be another median triple. Notice that, if $J$ is a hyperplane separating $x$ from $x''$, then it separates $x$ from $\{y,z\}$ according to Claim \ref{claim:TripleTwo}. But, because $J$ separates $x$ and $z$ (resp. $x$ and $y$), it cannot separate $y$ from $y'$ (resp. $z$ from $z'$) according to Claim \ref{claim:TripleTwo}. Moreover, it cannot separate two vertices among $\{x',y',z'\}$ according to Claim \ref{claim:TripleOne}. Therefore, $J$ separates $x$ from $x'$. As a first consequence, we deduce that $d(x,x'') \leq d(x,x')$. By a symmetric argument, one also shows that $d(y,y'') \leq d(y,y')$. And, as a second consequence, it follows that, if $x' \neq x''$, there exists a hyperplane separating $x'$ and $x''$ and it must separate $x$ from $x'$ but it cannot separate $x$ from $x''$, hence $d(x,x'')< d(x,x')$. We deduce that
$$d(x'',y'') = d(x,y)-d(x,x'') -d(y,y'') > d(x,y)-d(x,x')-d(y,y')=d(x',y').$$
But, as a consequence of Claims \ref{claim:TripleOne} and \ref{claim:TripleThree}, $d(x',y')=d(y',z')=d(x',z')$ coincides with the size of the triple $(x',y',z')$. Thus, we have proved that $(x'',y'',z'')$ cannot be a median triangle, concluding the proof of the uniqueness.

\begin{claim}\label{claim:TripleFour}
Any two hyperplanes separating $x',y',z'$ are transverse.
\end{claim}

\noindent
Let $J_1$ and $J_2$ be two non-transverse hyperplanes such that $J_1$ separates $x',y',z'$. Let $J_1^+$ denote the sector delimited by $J_1$ which contains $J_2$ and $J_2^+$ the sector delimited by $J_2$ which contains $J_1$. Necessarily, at least two vertices among $x',y',z'$ have to belong to the complement of $J_1^+$, and these two vertices have to lie in $J_2^+$. So $J_2$ does not separate $x',y',z'$, concluding the proof of our claim.

\medskip \noindent
Let $\Lambda$ denote the subgraph of $\Gamma$ generated by the vertices which label the hyperplane separating $x',y',z'$. As a consequence of Claim \ref{claim:TripleFour} and Lemma \ref{lem:transverseimpliesadj}, $\Lambda$ is complete. We conclude that $x' \langle \Lambda \rangle$ is a prism which contains $x',y',z'$, that it is the smallest prism which satisfies this property, and that the hyperplanes which cross it coincide with the hyperplanes containing $x',y',z'$ in pairwise distinct sectors.
\end{proof}

\noindent
As an application of Proposition \ref{prop:MedianTriangleQM}, let us prove the following result about projections:

\begin{lemma}\label{lem:projInter}
Let $\Gamma$ be a simplicial graph, $\mathcal{G}$ a collection of groups indexed by $V(\Gamma)$, $\Lambda, \Xi \subset \Gamma$ two subgraphs and $g,h \in \Gamma \mathcal{G}$ two elements. If $g \langle \Lambda \rangle \cap h \langle \Xi \rangle \neq \emptyset$, then, for every vertex $x \in g \langle \Lambda \rangle$, the projection $\mathrm{proj}_{h \langle \Xi \rangle}(x)$ belongs to $g \langle \Lambda \rangle$. 
\end{lemma}

\begin{proof}
Fix a vertex $x \in g \langle \Lambda \rangle$ and let $y$ denote its projection onto $h \langle \Xi \rangle$. Also, fix a vertex $z \in g \langle \Lambda \rangle \cap h \langle \Xi \rangle$. Let $(x',y',z')$ be the median triangle of $(x,y,z)$ and let $[x,x']$, $[y,y']$, $[z,z']$, $[x',y']$, $[x',z']$ be geodesics. 

\medskip \noindent
Notice that $y'$ belongs to $h\langle \Xi \rangle$ since $[y,y'] \cup [y',z'] \cup [z',z] \subset h \langle \Xi \rangle$ by convexity. Because $y'$ is a also vertex of the geodesic $[x,x'] \cup [x',y'] \cup [y',y]$ and that $y$ minimises the distance to $x$ in $h \langle \Xi \rangle$, it follows that $y=y'$. Consequently, $y$ belongs to $x' \langle \Omega \rangle$ where $\Omega$ denotes the subgraph of $\Gamma$ generated by the vertices labelling the hyperplanes separating $x'$ and~$y'$.

\medskip \noindent
It follows from Proposition \ref{prop:MedianTriangleQM} that every hyperplane separating $x'$ and $y'$ also separates $x'$ and $z'$. Because $[x,x'] \cup [x',y'] \cup [y',y] \subset g \langle \Lambda \rangle$ by convexity, such a hyperplane must be labelled by a vertex in $\Lambda$. In other words, $\Omega \subset \Lambda$. We conclude that
$$y \in x' \langle \Omega \rangle \subset g \langle \Lambda \rangle \cdot \langle \Lambda \rangle = g \langle \Lambda \rangle$$
as desired.
\end{proof}

\paragraph{The metric $\delta$.} Let $\Gamma$ be a simplicial graph and $\mathcal{G}$ a collection of groups indexed by $V(\Gamma)$. It is clear that, if some vertex-groups are infinite, then the graph $\QM$ is not locally finite. Here, we show how to define a new metric on $\QM$ in order to recover some local finiteness (when the groups in $\mathcal{G}$ are finitely generated). 

\medskip \noindent
Fix a vertex $u \in V(\Gamma)$ and a generating set $S_u \subset G_u$. For every hyperplane $J$ labelled by $u$ and every clique $C \subset J$, define on $\QM$ the pseudo-metric
$$\delta_{J,C} : (x,y) \mapsto \left| \mathrm{proj}_C(x)^{-1} \mathrm{proj}_C(y) \right|_{S_u}$$
where $| \cdot |_{S_u}$ denotes the word length with respect to $S_u$. First, notice that $\delta_{J,C}$ is well-defined, i.e., $\mathrm{proj}_C(x)^{-1} \mathrm{proj}_C(y)$ belongs to $G_u$ for all $x,y \in \QM$. Indeed, it follows from \cite[Proposition 2.11]{ConjAut} that every edge in $C$ is labelled by an element of $G_u$. Also, notice that, for all vertices $x,y \in \QM$, we have $\delta_{J,C}(x,y)=0$ if and only if $J$ does not separate $x$ and $y$ (again according to \cite[Proposition 2.11]{ConjAut}). 

\medskip \noindent
The key observation is that $\delta_{J,C}$ does not depend on the clique $C$ we choose.

\begin{claim}\label{claim:NotDependClique}
For every hyperplane $J$ labelled by $u$ and for all cliques $C_1,C_2 \subset J$, the equality $\delta_{J,C_1} = \delta_{J,C_2}$ holds. 
\end{claim}

\begin{proof}
Let $x,y \in \QM$ be two vertices and let $x_1,y_1$ (resp. $x_2,y_2$) denote the projections of $x,y$ onto $C_1$ (resp. $C_2$). Clearly, the projection of $x$ onto $N(J)$ belongs to the fiber $A$ of $J$ corresponding to the sector which contains $x$, so we deduce from Lemmas~\ref{lem:ProjInclusion} and~\ref{lem:projInter} that $x_1$ and $x_2$ belong to $A$. Similarly, one shows that $y_1$ and $y_2$ belong to the fiber $B$ of $J$ corresponding to the sector which contains $y$. 

\medskip \noindent
If we write $J=gJ_u$ for some $g \in \Gamma \mathcal{G}$, we know from Theorem \ref{thm:HypStab} that $N(J)= g \langle \mathrm{star}(u) \rangle = g \left( G_u \times \langle \mathrm{link}(u) \rangle \right)$. As a consequence, there exist $a \in G_u$ and $b \in G_u$ such that $A=ga \langle \mathrm{link}(u) \rangle$ and $B= gb \langle \mathrm{link}(u) \rangle$; and there exist $\ell_1,\ell_2\in \langle \mathrm{link}(u) \rangle$ such that $C_1= g \ell_1 G_u$ and $C_2=g \ell_2 G_u$. By noticing that 
$$\left\{ \begin{array}{l} x_i \in C_i \cap A = g \ell_i G_u \cap ga \langle \mathrm{link}(u) \rangle = \{g \ell_i a \} \\ y_i \in C_i \cap B = g \ell_i G_u \cap gb \langle \mathrm{link}(u) \rangle = \{g \ell_i b\} \end{array} \right.$$
for $i=1,2$, we conclude that
$$\delta_{J,C_1}(x,y) = | x_1^{-1}y_1|_{S_u} = |a^{-1}b|_{S_u} = |x_2^{-1}y_2| = \delta_{J,C_2}(x,y)$$
as desired.
\end{proof}

\noindent
This claim allows us to define a pseudo-metric as follows:

\begin{definition}
For every hyperplane $J$ labelled by $u$, $\delta_J$ is the pseudo-metric
$$(x,y) \mapsto \left| \mathrm{proj}_C(x)^{-1} \mathrm{proj}_C(y) \right|_{S_u}$$
defined on $\QM$, where $C \subset J$ is an arbitrary clique.
\end{definition}

\noindent
Another consequence of Claim \ref{claim:NotDependClique} is that our family of new pseudo-metrics is $\Gamma \mathcal{G}$-equivariant:

\begin{fact}\label{fact:deltaequi}
For every hyperplane $J$ labelled by $u$ and every element $g \in \Gamma \mathcal{G}$, 
$$\delta_J(gx,gy)= \delta_{g^{-1}J}(x,y) \text{ for all vertices $x,y \in \QM$}.$$ 
As a consequence, $\displaystyle \delta_u:= \sum\limits_{\text{$J$ labelled by $u$}} \delta_J$ is $\Gamma \mathcal{G}$-equivariant.
\end{fact}

\begin{proof}
We have
$$\begin{array}{lcl} \delta_J(gx,gy) &= & \left| \mathrm{proj}_C(gx)^{-1} \mathrm{proj}_C(gy) \right|_{S_u} \\ \\ & = & \left| \mathrm{proj}_{g^{-1}C}(x)^{-1} \mathrm{proj}_{g^{-1}C}(y) \right|_{S_u} = \delta_{g^{-1}J}(x,y) \end{array}$$
for all  vertices $x,y \in \QM$. 
\end{proof}

\noindent
The philosophy behind this construction is that $\delta_J$ unfold the cliques in $J$. Moreover, the metric space we obtain by unfolding a clique turns out to coincides with $G_u$ endowed with the word metric associated to $S_u$:

\begin{fact}\label{fact:UnfoldClique}
For every hyperplane $J$ labelled by $u$ and every clique $C \subset J$, $(C,\delta_J)$ is isometric to $\mathrm{Cayl}(G_u,S_u)$.
\end{fact}

\begin{proof}
As a consequence of Fact \ref{fact:deltaequi}, $(C,\delta_J)$ is isometric to $(G_u,\delta_{J_u})= (G_u, |\cdot|_{S_u})$; 
\end{proof}

\noindent
By unfolding all the cliques in $\QM$ by this process, we obtain a new metric:

\begin{definition}
For every vertex $u \in V(\Gamma)$, fix a generating set $S_u \subset G_u$. We define a new metric on $\QM$ by the sum
$$\delta:= \sum\limits_{\text{$J$ hyperplane}} \delta_J.$$
\end{definition}

\noindent
Notice that $\delta$ takes only finite values because any two vertices in $\QM$ are separated by only finitely many hyperplanes. Moreover, $\delta$ is positive definite because any two distinct vertices in $\QM$ are separated by at least one hyperplane. 

\medskip \noindent
The key observation is that the new metric space we obtain, namely $(\QM,\delta)$, turns out to be isometric to a Cayley graph of $\Gamma \mathcal{G}$.

\begin{lemma}\label{lem:DeltaCayley}
Let $\Gamma$ be a simplicial graph and $\mathcal{G}$ a collection of groups indexed by $V(\Gamma)$. For every $u \in V(\Gamma)$, fix a generating set $S_u \subset G_u$; also, set $S= \bigcup\limits_{u \in V(\Gamma)} S_u$. The canonical map $\Gamma \mathcal{G} \to \QM$ induces an isometry
$$\mathrm{Cayl}(\Gamma \mathcal{G},S) \to (\QM,\delta).$$
\end{lemma}

\begin{proof}
First of all, notice that $\delta$ is a graph metric, i.e., for any two vertices $x,y \in \QM$, there exists $z_0, \ldots, z_n \in \QM$ such that $z_0=x$ and $z_n=y$, $\delta(z_i,z_{i+1})=1$ for every $0 \leq i \leq n-1$, and $n= \delta(x,y)$.

\medskip \noindent
Indeed, let $w_0, \ldots, w_k$ be a geodesic in $\QM$. For every $1 \leq i \leq k-1$, let $C_i$ denote the clique containing $w_i,w_{i+1}$ and let $J_i$ denote the hyperplane containing $C_i$. Finally, for every $1 \leq i \leq k-1$, let $\gamma_i$ be a geodesic in $(C_i,\delta_{J_i})$ from $w_i$ to $w_{i+1}$. (Recall from Fact \ref{fact:UnfoldClique} that $\delta_{J_i}$ defines a graph metric on $C_i$.) Because
$$\delta(x,y)= \sum\limits_{\text{$J$ separating $x$ and $y$}} \delta_J(x,y)= \sum\limits_{i=0}^k \delta_{J_i}(w_i,w_{i+1})= \sum\limits_{i=1}^k \mathrm{length}(\gamma_i),$$
it follows that $\gamma_0 \cup \cdots \cup \gamma_{k-1}$ is the sequence of vertices we are looking for. 

\medskip \noindent
Now, notice that two vertices $x,y \in \QM$ are at $\delta$-distance one if and only if $x^{-1}y$ belongs to $S_u$ for some $u \in V(\Gamma)$. Therefore, two vertices in $\mathrm{Cayl}(\Gamma \mathcal{G},S)$ are at $\delta$-distance one in $(\QM,\delta)$ if and only if they are adjacent in $\mathrm{Cayl}(\Gamma \mathcal{G},S)$. This concludes the proof of our lemma. 
\end{proof}

\noindent
\textbf{Convention:} When all the groups in $\mathcal{G}$ are finitely generated, we refer to $\delta$ as the metric defined above with respect to a family of finite generating sets of vertex-groups which is fixed once for all. However, because the choice of such a family do not matter for our purpose, it will not be fixed explicitly. Similarly, for every element $g \in \Gamma \mathcal{G}$, the notation $|g|$ will refer to the word length associated to the generating set of $\Gamma \mathcal{G}$ which is the union of the family of generating sets we fixed.

\subsection{Acylindrically hyperbolic groups}\label{section:AcylHyp}

\noindent
Recall from \cite{OsinAcyl} that:

\begin{definition}
A group $G$ is \emph{acylindrically hyperbolic}\index{Acylindrically hyperbolic group} if it admits an action by isometries on a hyperbolic space which is non-elementary and \emph{acylindrical}, i.e.,
$$\forall d \geq 0, \ \exists L,N \geq 0, \ \forall x,y \in X, \ d(x,y) \geq L \Rightarrow \# \{ g \in G \mid d(x,gx),d(y,gy) \leq d\} \leq N.$$
\end{definition}

\noindent
In general, it is difficult to show that a group is acylindrically hyperbolic by constructing an acylindrical action. According to \cite{OsinAcyl}, it suffices to show that the group is not virtually cyclic and that it acts on hyperbolic space with a \emph{WPD isometry}.

\begin{definition}
Let $G$ be a group acting on metric space $X$. An element $g \in G$ is \emph{WPD}\index{WPD element} (with respect to the action $G \curvearrowright X$) if
$$\forall d \geq 0, \ \forall x \in X, \ \exists n \geq 1, \  \# \{ h \in G \mid d(x,hx), d(g^nx,hg^nx) \leq d \}< \infty.$$
Given an element $g \in G$, if there exists an action of $G$ on a hyperbolic space for which $g$ is WPD then $g$ is a \emph{generalised loxodromic element}\index{Generalised loxodromic element}.
\end{definition}

\noindent
In the sequel, it will be useful to replace the definition of WPD elements with the following alternative characterisation:

\begin{lemma}\label{lem:WPDdef}
Let $G$ be a group acting on a metric space $X$. An element $g \in G$ is a WPD isometry if and only if, for every $d \geq 0$, there exist $x \in X$ and $n \geq 1$ such that 
$$\{ h \in G \mid d(x,hx), d(g^nx,hg^nx) \leq d\}$$
is a finite set.
\end{lemma}

\begin{proof}
The ``only if'' direction is clear. Conversely, suppose that the condition given by our lemma is satisfied. Fix some $d \geq 0$ and $x \in X$. We know that there exist $y \in X$ and $n \geq 1$ such that
$$E=\{ h \in G \mid d(y,hy), d(g^ny,hg^ny) \leq d+2d(x,y) \}$$
is finite. We claim that the set
$$F= \{ h \in G \mid d(x,hx), d(g^nx,hg^nx) \leq d\}$$
is finite as well. Notice that, for every $h \in F$, we have
$$d(y,hy) \leq d(x,hx)+2d(x,y) \leq d+2d(x,y)$$
and similarly $d(g^ny,hg^ny) \leq d+2d(x,y)$. Therefore, $F \subset E$, and we conclude that $F$ must be finite. Thus, we have proved that $g$ is a WPD isometry.
\end{proof}

\noindent
Recall from the previous section that, in a graph product, an element is \emph{irreducible} if its support does not lie in a join and is not a single vertex. 

\begin{prop}\label{prop:IrrAreGLoxo}
In a graph product, irreducible elements are generalised loxodromic elements.
\end{prop}

\noindent
We refer to \cite{MinasyanOsin} or Proposition \ref{prop:IrreducibleWPD} below for a proof. Notice that, however, there may exist generalised loxodromic elements which are not irreducible. For instance, given a simplicial graph $\Gamma$ and a collection of groups $\mathcal{G}$ indexed by $V(\Gamma)$, if $\Gamma$ contains a vertex $u$ such that $G_u$ is acylindrically hyperbolic and such that the link of $u$ is a (possibly empty) complete graph all of whose vertices are labelled by finite groups, then every generalised loxodromic element of $G_u$ also defines a generalised loxodromic element of $\Gamma \mathcal{G}$. (Because $\Gamma \mathcal{G}$ decomposes as an amalgamated product $\langle \Gamma \backslash \{u\} \rangle \underset{\langle \mathrm{link}(u) \rangle}{\ast} (\langle \mathrm{link}(u) \rangle \oplus G_u)$ over the finite subgroup $\langle \mathrm{link}(u) \rangle$.) However, none of these elements are irreducible. (And it can be proved that this is the only possible generalised loxodromic elements which are not irreducible.)

\medskip \noindent
By combining Proposition \ref{prop:IrrAreGLoxo} with our next statement, we will be able to show that irreducible elements admit only finitely many $n$th roots for every $n \geq 1$. 

\begin{prop}\label{prop:RootsHyp}
Let $G$ be an acylindrically hyperbolic group, $g \in G$ an element and $n \geq 1$ an integer. If $g$ is a generalised loxodromic element, then it admits only finitely many $n$th roots.
\end{prop}

\noindent
The proof of the proposition relies on the next lemma, proved in \cite[Corollary 6.6 and Lemma 6.5]{DGO}. In its statement, the following terminology is used. Given a group $G$, we refer to the \emph{elementary subgroup}\index{Elementary subgroup $E(\cdot)$} $E(g)$ as $\{h \in G \mid \exists n \in \mathbb{Z}\backslash \{0\}, \ hg^nh^{-1}=g^{\pm n} \}$.

\begin{lemma}\label{lem:subgroupE}
Let $G$ be an acylindrically hyperbolic group and $g \in G$ a generalised loxodromic element. The subgroup $E(g)$ is the unique maximal virtually cyclic group of $G$ which contains $\langle g \rangle$. Moreover, 
$$E^+(g):= \{ h \in G \mid \exists n \in \mathbb{Z} \backslash \{0\}, \ hg^nh^{-1}=g^n \}$$
coincides with the centraliser of a power of $g$. 
\end{lemma}

\begin{proof}[Proof of Proposition \ref{prop:RootsHyp}.]
Fix an $n \geq 1$ and assume that $g_1,g_2, \ldots$ are infinitely many $n$th roots of $g$. Notice that they all belong to $E(g)$. Because $\langle g \rangle$ has finite index in $E(g)$, we may suppose without loss of generality that $g_1,g_2, \ldots$ all lie in the same coset of $\langle g \rangle$, i.e., there exists some $h \in G$ such that, for every $i \geq 1$, there exists some $p(i) \in \mathbb{Z}$ such that $g_i= hg^{p(i)}$. Without loss of generality, assume that $p(i) \to + \infty$ as $i \to + \infty$. Let $X$ be a hyperbolic graph on which $G$ acts with $g$ loxodromic. Fix a vertex $x \in X$ on a quasi-axis of $g$. We have
$$\begin{array}{lcl} d(x,g_ix) & = & d(x,hg^{p(i)}x)= d(h^{-1}x, g^{p(i)}x) \geq d(x,g^{p(i)}x)-d(x,h^{-1}x) \\ \\ & \geq & p(i) \|g\| - d(x,hx) \underset{i \to + \infty}{\longrightarrow} + \infty \end{array}$$
where $\|g\|:= \lim\limits_{k \to + \infty} d(y,g^ky)/n$ for an arbitrary point $y \in X$ and where the last inequality is justified by \cite[Proposition 6.2]{CDP}. On the other hand, there exists some $C \geq 0$ such that, for every $i \geq 1$, we have
$$d(x,g_ix) \leq \|g_i\| + C = \frac{1}{n} \|g\|+C$$
where the equality is justified by $g_i^n=g$ and where the inequality is justified by the fact that $g$ and $g_i$ have axes at finite Hausdorff distance. We get a contradiction for $i$ large enough.
\end{proof}

\noindent
As a direct consequence of Propositions \ref{prop:IrrAreGLoxo} and \ref{prop:RootsHyp}:

\begin{cor}\label{cor:roots}
Let $\Gamma$ be a simplicial graph, $\mathcal{G}$ a collection of groups indexed by $V(\Gamma)$, and $g$ an irreducible element. For every $n \geq 1$, $g$ admits only finitely many $n$th roots.
\end{cor}

\noindent
Observe that Corollary \ref{cor:roots} can be probably strenghened by stating that, for every $n \geq 1$, an irreducible element admits at most one $n$th roots. However, Corollary \ref{cor:roots}, as it is stated, will be sufficient for our purpose.

\section{Step 1: A rigidity theorem}\label{section:Rigidity}

\noindent
In this section, our goal is essentially to show that, inside a graph product, some vertex-groups are preserved by automorphisms up to conjugacy. This observation will be fundamental in the construction of a hyperbolic graph on which the automorphism group acts, as described in Section \ref{section:StepTwo}. More precisely, our main statement is the following:

\begin{thm}\label{thm:IntroStepOne}
Let $\Phi, \Psi$ be two finite simplicial graphs and $\mathcal{G}, \mathcal{H}$ two collections of graphically irreducible groups respectively indexed by $V(\Phi),V(\Psi)$. Assume that no two distinct vertices in $\Phi$ (resp. $\Psi$) have the same star. For every isomorphism $\varphi : \Phi \mathcal{G} \to \Psi \mathcal{H}$ and every $\prec$-maximal vertex $u \in V(\Phi)$, there exist an element $g \in \Psi \mathcal{H}$ and a $\prec$-maximal vertex $v \in V(\Psi)$ such that $\varphi(\langle [u] \rangle)=g \langle [v] \rangle g^{-1}$. 
\end{thm}

\noindent
Recall that, given a graph $\Gamma$ and two vertices $u,v \in V(\Gamma)$, we write $u \prec v$ if $\mathrm{link}(u) \subset \mathrm{star}(v)$. The relation $\prec$ is reflexive and transitive, but it does not define an order in general as there may exist two distinct vertices $x$ and $y$ satisfying both $x \prec y$ and $y \prec x$. However, the relation $\sim$ defined by $u \sim v$ if $u \prec v$ and $v \prec u$ defines an equivalence relation, and we denote by $[u]$ the equivalence class of $u$. Notice that, either any two vertices in $[u]$ are adjacent and $[u]= \{ v \in V(\Gamma) \mid \mathrm{star}(u)= \mathrm{star}(v) \}$; or any two vertices in $[u]$ are non-adjacent and $[u]= \{ v \in V(\Gamma) \mid \mathrm{link}(u)= \mathrm{link}(v) \}$. 

\medskip \noindent
It is worth noticing that the automorphism group of a graph product $\Gamma \mathcal{G}$ may not preserve the subgroups generated by classes of $\prec$-maximal vertices up to conjugacy if groups in $\mathcal{G}$ decompose non-trivially as graph products. For instance, let $\Gamma$ be the graph illustrated by Figure \ref{CE}. Define $\mathcal{G} = \{G_u \mid u \in V(\Gamma)\}$ by fixing a (non-trivial) group $A$ and by setting $G_u=A\oplus (A \ast A)$ if $u$ is the orange vertex and $G_u=A$ otherwise. Clearly, there exists an automorphism $\varphi \in \mathrm{Aut}(\Gamma \mathcal{G})$ satisfying $\varphi(\langle o \rangle) = \langle g_1,g_2,g_3 \rangle$, where $o$ denotes the orange vertex and $g_1,g_2,g_3$ the three green vertices. But the central green vertex, say $g_2$, is $\prec$-maximal and $\varphi(\langle [g_2] \rangle) = \varphi( \langle g_2 \rangle) \subsetneq \langle o \rangle$. 
\begin{figure}
\begin{center}
\includegraphics[scale=0.4]{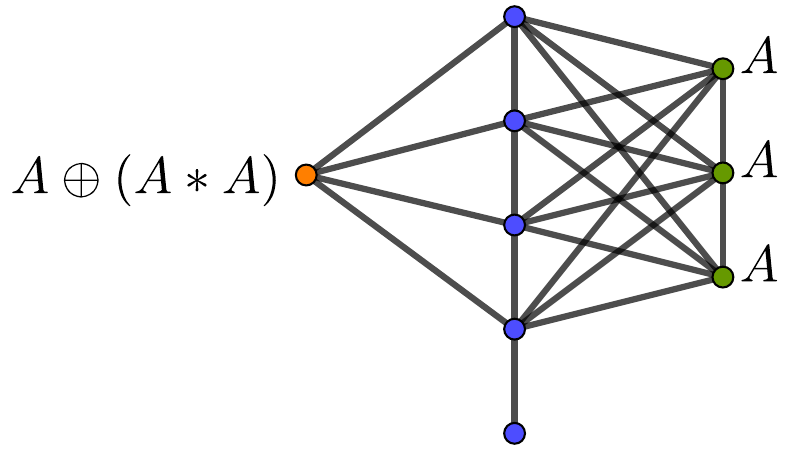}
\caption{A graph product for which the conclusion of Theorem \ref{thm:IntroStepOne} fails.}
\label{CE}
\end{center}
\end{figure}

\medskip \noindent
This observation motivates the introduction of \emph{graphically irreducible groups}, which are defined and studied in the first subsection.

\subsection{Graphically irreducible groups}\label{subsection:graphicallyirr}

\noindent
Roughly speaking, a group is \emph{graphically irreducible} if its only non-trivial decompositions as graph products are direct sums. More precisely:

\begin{definition}
A group $G$ is \emph{graphically irreducible}\index{Graphically irreducible group} if, for every simplicial graph $\Gamma$ and every collection $\mathcal{G}$ of groups indexed by $V(\Gamma)$ such that $G$ is isomorphic to $\Gamma \mathcal{G}$, $\Gamma$ is a complete graph.
\end{definition}

\noindent
For instance, $\mathbb{Z}^n$ is graphically irreducible for every $n \geq 1$; but $\mathbb{F}_2 \oplus \mathbb{Z}$ is not graphically irreducible because it decomposes as a graph product over a path of length two (corresponding to the decomposition $(\mathbb{Z} \ast \mathbb{Z}) \oplus \mathbb{Z}$). More generally, a right-angled Artin group is graphically irreducible if and only if it is free abelian. 

\medskip \noindent
The rest of the subsection is dedicated to the following statement:

\begin{prop}\label{prop:SumGraphReducible}
The direct sum of finitely many graphically irreducible groups is graphically irreducible.
\end{prop}

\noindent
We begin by proving the following general lemma, where we say that the centraliser of an element $g \in G$ is \emph{maximal} if, for every $h \in G$, the inclusion $C(g) \subset C(h)$ implies $C(g)=C(h)$.

\begin{lemma}\label{lem:MorphismCenterReducible}
Let $A,B, G_1, \ldots, G_n$ be a family of groups. Assume that $A$ is centerless and that it admits a generating set $S$ of pairwise non-commuting elements whose centralisers are maximal. For every isomorphism $\varphi : A \oplus B \to G_1 \oplus \cdots \oplus G_n$, there exists an index $1 \leq j \leq n$ such that 
$$\varphi \left( A \right) \subset Z(G_1) \oplus \cdots \oplus Z(G_{j-1}) \oplus G_j \oplus Z(G_{j+1}) \oplus \cdots \oplus Z(G_n).$$
\end{lemma}

\begin{proof}
Fix an element $g \in S$ and write $\varphi(g)=g_1 \cdots g_n$ where $g_1 \in G_1, \ldots, g_n \in G_n$. Because $g \notin Z(A \oplus B)$, necessarily $\varphi(g)$ does not belong to the center of $G_1 \oplus \cdots \oplus G_n$, i.e., there exists an index $j=j(g) \in \{1, \ldots, n\}$ such that $g_j \notin Z(G_j)$. Consequently, we have
$$\begin{array}{lcl} C(\varphi(g)) & = & C_{G_1}(g_1) \oplus \cdots \oplus C_{G_n}(g_n) \\ \\ & \subsetneq & C_{G_1}(g_1) \oplus \cdots \oplus C_{G_{j-1}}(g_{j-1}) \oplus G_j \oplus C_{G_{j+1}}(g_{j+1}) \oplus \cdots \oplus C_{G_n}(g_n) \\ \\ & \subsetneq & C(g_1 \cdots g_{j-1}g_{j+1} \cdots g_n) \end{array}$$
But, because $g$ has maximal centraliser in $A$, it follows that that $C_{A \oplus B}(g) \subsetneq C_{A \oplus B}(h)$ implies $h \in Z(A \oplus B)$ for every $h \in A \oplus B$. The same property must be satisfied by $\varphi(g)$ in $G_1 \oplus \cdots \oplus G_n$, so it follows that $C(g_1 \cdots g_{j-1}g_{j+1} \cdots g_n)=G_1 \oplus \cdots \oplus G_n$, hence $g_i \in Z(G_i)$ for every $i \neq j$. Notice that, if $j(h) \neq j(g)$ for some $h \in S$, then $\varphi(g)$ and $\varphi(h)$ commute, which is possible only if $g=h$. Thus, we have proved that
$$\varphi(g) \in Z(G_1) \oplus \cdots \oplus Z(G_{j-1}) \oplus G_j \oplus Z(G_{j+1}) \oplus \cdots \oplus Z(G_n)$$
for every $g \in S$ (where, now, the index $j$ does not depend on $g$). As $S$ generates $A$, the desired conclusion follows.
\end{proof}

\noindent
Lemma \ref{lem:MorphismCenterReducible} can be applied to graph products thanks to the following observation:

\begin{lemma}\label{lem:GeneratingSetNice}
Let $\Gamma$ be a finite simplicial graph and $\mathcal{G}$ a collection of groups indexed by $V(\Gamma)$. Assume that $\Gamma$ contains at least two vertices and is not a join. Then $\Gamma \mathcal{G}$ admits a generating set of pairwise non-commuting elements with maximal centralisers.
\end{lemma}

\begin{proof}
For all vertices $a,b \in V(\Gamma)$, let $\mathcal{P}(a,b)$ denote the set of all paths in $\Gamma^\mathrm{opp}$ which start from $a$, end at $b$ and which visit all the vertices of $\Gamma$. For every $\gamma \in \mathcal{P}(a,b)$, let $\Omega(\gamma)$ denote the collection of all the elements in $\Gamma \mathcal{G}$ (thought of as words) constructed as follows: if $u_1, \ldots, u_k$ are the vertices successively visited by $\gamma$ then $s_1 \cdots s_k$ belongs to $\Omega(\gamma)$ for all non-trivial elements $s_1 \in G_{u_1}, \ldots, s_k \in G_{u_k}$. Finally, for all $a,b \in V(\Gamma)$, let $\Omega(a,b)$ denote the union $\bigcup\limits_{\gamma \in \mathcal{P}(a,b)} \Omega(\gamma)$. 

\begin{claim}\label{claim:OmegaOne}
For all $a,b \in V(\Gamma)$ and $w \in \Omega(a,b)$, $w$ is graphically reduced; moreover, if $a \neq b$, it is also graphically cyclically reduced and irreducible.
\end{claim}

\noindent
No two successive syllables of $w$ belong to identical or adjacent vertex-groups, so clearly $w$ is graphically reduced. Also, the only possibility for $w$ not to be graphically cyclically reduced is that its first and last syllables belong to the same vertex-group, which happens precisely when $a=b$. Therefore, if $a \neq b$, then $w$ is graphically cyclically reduced. And if so, $w$ must have full support and consequently it must be irreducible.

\begin{claim}\label{claim:OmegaTwo}
For all $a,b \in V(\Gamma)$ and $w \in \Omega(a,b)$, if $a \notin \mathrm{star}_{\Gamma^\mathrm{opp}}(b)$ then $w$ is not a proper power. (See Question \ref{QuestionAcyl} and the related discussion.)
\end{claim}

\noindent
Assume that $a \neq b$ and that $w$ is a proper power, i.e., there exist $g \in \Gamma \mathcal{G}$ and $n \geq 2$ such that $w=g^n$. Write $g$ as a graphically reduced product $xyx^{-1}$ where $y$ is graphically cyclically reduced. Notice that, because $w$ is irreducible and $w=xy^nx^{-1}$, the \emph{floating syllables} of $y$ (i.e., the syllables which shuffle both to the end and to the beginning of $y$) must have exponent $n$ since otherwise $w$ would belong to a star-subgroup. Consequently, we may suppose without loss of generality that $y$ does not contain any floating syllable. This assumption implies that the product $y^n$ (and a fortiori $xy^nx^{-1}$) is graphically reduced. Because $w$ is graphically cyclically reduced, it follows from the equality $w=xy^nx^{-1}$ that $x$ must be trivial. Also, because no two consecutive syllables in $w$ commute, necessarily the word $w$ is a power of $y$. As a consequence, the first and last syllables of $w$ and $y$ coincide, hence $y \in \Omega(a,b)$. Because $y^n = w$ belongs to $\Omega(a,b)$ where $n \geq 2$, it follows that $a$ and $b$ are adjacent in $\Gamma^\mathrm{opp}$. 

\begin{claim}
Assume that $\Gamma$ is connected. There exist two maps $\alpha, \omega : V(\Gamma) \to V(\Gamma)$ such that $\alpha(u) \notin \mathrm{star}_{\Gamma^\mathrm{opp}}(u) \cup \mathrm{star}_{\Gamma^\mathrm{opp}}(\omega(u))$ and $\omega(u) \in \mathrm{link}_{\Gamma^\mathrm{opp}}(u)$ for every $u \in V(\Gamma)$.
\end{claim}

\noindent
Fix a vertex $u \in V(\Gamma^\mathrm{opp})$. If every vertex not in $\mathrm{link}_{\Gamma^\mathrm{opp}}(u)$ is adjacent to all the vertices in $\mathrm{link}_{\Gamma^\mathrm{opp}}(u)$, then $\Gamma^\mathrm{opp}$ would be a join and $\left( \Gamma^\mathrm{opp} \right)^\mathrm{opp} = \Gamma$ would not be connected. So there exists a vertex $x$ not in $\mathrm{link}_{\Gamma^\mathrm{opp}}(u)$ which is not adjacent to some vertex of $\mathrm{link}_{\Gamma^\mathrm{opp}}(u)$, say $y$. Then set $\alpha(u)=x$ and $\omega(u)=y$. This proves our claim.

\medskip \noindent
Now, we are ready to prove our lemma under the additional assumption that $\Gamma$ is connected. For every $u \in V(\Gamma)$, fix a path $\gamma_u \in \mathcal{P}(\alpha(u), \omega(u))$ and an element $g_u \in \Omega(\gamma_u)$. Without loss of generality, we assume that $|\mathrm{length}(\gamma_u)- \mathrm{length}(\gamma_v)| \geq 2$ for all $u,v \in V(\Gamma)$. Clearly, the set $S:= \{ g_us \mid u \in V(\Gamma), s \in G_u\}$ generates $\Gamma \mathcal{G}$ since it contains all the vertex-groups. 

\medskip \noindent
Fix two vertices $u,v \in V(\Gamma)$ and two elements $r \in G_u$, $s \in G_v$. Assume that $g_ur$ and $g_vs$ commute. Notice that, because $g_u \in \Omega(\alpha(u), \omega(u))$ and $\omega(u) \in \mathrm{link}_{\Gamma^\mathrm{opp}}(u)$, we have $g_ur \in \Omega(\alpha(u),u)$ if $r \neq 1$. Since $\alpha(u) \notin \mathrm{star}_{\Gamma^\mathrm{opp}}(u) \cup \mathrm{star}_{\Gamma^\mathrm{opp}}(\omega(u))$, it follows from Claims \ref{claim:OmegaOne} and \ref{claim:OmegaTwo} that $g_ur$ is an irreducible element which is not a proper power. Similarly, $g_vs$ is an irreducible element which is not a proper power. Because the centraliser of an irreducible element is cyclic, we deduce that, if $g_ur$ and $g_vs$ commute, then $g_ur= (g_vs)^{\pm 1}$. This equality implies that $|\mathrm{length}(\gamma_u) - \mathrm{length}(\gamma_v)| \leq 1$, hence $u=v$. From the equality $g_ur=(g_us)^{\pm 1}$, we deduce that either $r=s$ or $g_ur=s^{-1}g_u^{-1}$. Notice that the latter equality is impossible since $g_ur \in \Omega(\alpha(u),\omega(u))$ and $s^{-1}g_u^{-1} \in \Omega(\omega(u),\alpha(u))$ but $\alpha (u) \neq \omega(u)$.

\medskip \noindent
Thus, we have proved that no two distinct elements in $S$ commute. The fact that its elements have maximal centralisers follows from our next observation:

\begin{claim}
Every irreducible element of $\Gamma \mathcal{G}$ has maximal centraliser.
\end{claim}

\noindent
Let $g \in \Gamma \mathcal{G}$ be an irreducible element and let $h \in \Gamma \mathcal{G}$ be an element satisfying $C(g) \subset C(h)$. As a consequence of Proposition \ref{prop:centraliser}, $C(g)= \langle k \rangle$ for some $k$ which is not a proper power and whose support is the same as that of $g$. Because $g$ and $h$ commute, necessarily $h$ is a power of $k$, and it follows from Proposition \ref{prop:centraliser} that $C(h)= \langle k \rangle$. A fortiori, the equality $C(g)=C(h)$ holds, concluding the proof of our claim.

\medskip \noindent
Thus, we have proved our lemma under the additional assumption that $\Gamma$ is connected. The case where $\Gamma$ is disconnected follows from the next claim:

\begin{claim}
A free product admits a generating set of pairwise non-commuting elements whose centralisers are maximal.
\end{claim}

\noindent
Let $A$ and $B$ be two non-trivial groups. If $A=\langle a \rangle \simeq \mathbb{Z}_2$ and $B=\langle b \rangle \simeq \mathbb{Z}_2$, then $\{a,b\}$ defines a generating set of $A \ast B$ which satisfies the desired properties. From now on, we assume that $A$ has cardinality at least three. Fix distinct non-trivial elements $a_1,a_2 \in A$ and $b \in B$. Set $g:=a_1ba_2$ and $h:=a_1ba_1ba_2b$, and define $S:= \{gr, hs \mid r \in B, s \in A\}$. As $A,B \subset S$, it is clear that $S$ is a generating set of $A \ast B$. It is straightforward to verify that no two elements in $S$ commute and the elements in $S$ have maximal centralisers.
\end{proof}

\noindent
We are now ready to prove the main result of this subsection:

\begin{proof}[Proof of Proposition \ref{prop:SumGraphReducible}.]
Let $B$ and $C$ be two groups. Assume that $B \oplus C$ is isomorphic to a graph product $\Gamma \mathcal{G}$, where $\Gamma$ is a simplicial graph and $\mathcal{G}$ a collection of groups indexed by $V(\Gamma)$. Fix an isomorphism $\varphi : \Gamma \mathcal{G} \to B \oplus C$. Decompose $\Gamma$ as a join $\Gamma_0 \ast \Gamma_1 \ast \cdots \ast \Gamma_n$ where $\Gamma_0$ is a clique and where each graph among $\Gamma_1, \ldots, \Gamma_n$ contains at leat two vertices and is not a join. For convenience, we denote by $A_i$ the subgroup $\langle \Gamma_i \rangle$ for every $0 \leq i \leq n$. So $\varphi$ defines an isomorphism $A_0 \oplus A_1 \oplus \cdots \oplus A_n \to B \oplus C$. In order to prove our proposition, we suppose that $n \geq 1$ and we want to show that either $B$ or $C$ is graphically reducible. 

\medskip \noindent
According to Lemma \ref{lem:MorphismCenterReducible}, for every $1 \leq i \leq n$, $\varphi(A_i) \subset B \oplus Z(C)$ or $Z(B) \oplus C$. Notice that $\varphi(A_1)$ cannot lie in both $B \oplus Z(C)$ and $Z(B) \oplus C$, since otherwise $A_1$ would lie in $\varphi^{-1}(Z(B \oplus C))= Z(\Gamma \mathcal{G}) \subset A_0$. Without loss of generality, say that $\varphi(A_1) \subsetneq B \oplus Z(C)$. Notice that, because $\varphi^{-1}(Z(B)) \subset Z(\Gamma \mathcal{G}) \subset A_0$, $\varphi$ induces an isomorphism
$$\bar{\varphi} : \bar{A}_0 \oplus A_1 \oplus \cdots \oplus A_n \to \bar{B} \oplus C$$
where $\bar{A}_0:= A_0/ \varphi^{-1}(Z(B))$ and $\bar{B}:= B/Z(B)$. As a consequence of Lemma \ref{lem:MorphismCenterReducible}, for every $1 \leq i \leq n$, $\bar{\varphi}(A_i) \subset \bar{B} \oplus Z(C)$ or $C$. Notice that $\bar{\varphi}(A_1) \subset C$ because $\varphi(A_1) \subset Z(B) \oplus C$. Therefore, there exists an index $r \geq 1$ such that 
$$\bar{\varphi}(A_i) \subset \left\{ \begin{array}{cl} C & \text{if $1 \leq i \leq r$} \\ \bar{B} \oplus Z(C) & \text{if $i>r$} \end{array} \right.$$
Now, the key observation is that:
$$\begin{array}{lcl} \bar{\varphi}^{-1}(C) & = & \left\{ abc \mid a \in A_0, b \in A_1 \oplus \cdots \oplus A_r, c \in A_{r+1} \oplus \cdots \oplus A_n, \bar{\varphi}(ac) \in C \right\}\\ \\ & = & \{ac \mid a \in A_0, c \in A_{r+1} \oplus \cdots \oplus A_n, \bar{\varphi}(ac) \in C \} \oplus A_1 \oplus \cdots \oplus A_r \end{array}$$
We conclude that $C \simeq \bar{\varphi}^{-1}(C)$ is graphically reducible, as desired.
\end{proof}

\subsection{Isomorphisms between products}

\noindent
This subsection is dedicated to the proof of the following statement, which essentially reduces the isomorphism problem between graph products to the irreducible case:

\begin{prop}\label{prop:Product}
Let $\Phi, \Psi$ be two simplicial graphs and $\mathcal{G}, \mathcal{H}$ two collections of graphically irreducible groups respectively indexed by $V(\Phi), V(\Psi)$. Decompose $\Phi$ (resp. $\Psi$) as a join $\Phi_0 \ast \Phi_1 \ast \cdots \ast \Phi_n$ (resp. $\Psi_0 \ast \Psi_1 \ast \cdots \ast \Psi_m$) such that $\Phi_0$ (resp. $\Psi_0$) is complete and such that each graph among $\Phi_1, \ldots, \Phi_n$ (resp. $\Psi_1, \ldots, \Psi_n$) contains at least two vertices and is not a join. For every isomorphism $\varphi : \Phi \mathcal{G} \to \Psi \mathcal{H}$, there exist a bijection $\sigma : \{1, \ldots, n\} \to \{1, \ldots, m\}$, morphisms $\alpha_i : \langle \Phi_i \rangle \to Z(\langle \Psi_0 \rangle)$ and isomorphisms $\beta_i : \langle \Phi_i \rangle \to  \langle \Psi_{\sigma(i)} \rangle$ such that $\varphi(\langle \Phi_0 \rangle) = \langle \Psi_0 \rangle$ and $\varphi_{|\langle \Phi_i \rangle} : a \mapsto \alpha_i(a) \beta_i(a)$ for every $1 \leq i \leq n$.
\end{prop}

\noindent
We begin by proving a general lemma, which we think to be of independent interest.

\begin{lemma}\label{lem:Product}
Let $A_0, A_1,\ldots, A_n,B_0,B_1, \ldots, B_m$ be groups such that the following conditions are satisfied for every $i \geq 1$:
\begin{itemize}
	\item the centers of $A_i$ and $B_i$ are trivial;
	\item $A_i$ (resp. $B_i$) is neither isomorphic to nor a direct factor of $B_0$ (resp. $A_0$);
	\item $A_i$ (resp. $B_i$) admits a generating set $S_i$ (resp. $R_i$) of pairwise non-commuting elements whose centralisers are maximal.
\end{itemize}
For every isomorphism $\varphi : A_0 \oplus A_1 \oplus \cdots \oplus A_n \to B_0 \oplus B_1 \oplus \cdots \oplus B_m$, there exist a bijection $\sigma : \{1, \ldots, n\} \to \{1, \ldots, m\}$, morphisms $\alpha_i : A_i \to Z(B_0)$ and isomorphisms $\beta_i : A_i \to B_{\sigma(i)}$ such that $\varphi(A_0)=B_0$ and $\varphi_{|A_i} : a \mapsto \alpha_i(a) \beta_i(a)$ for every $1 \leq i \leq n$. 
\end{lemma}

\begin{proof}
As a consequence of Lemma \ref{lem:MorphismCenterReducible}, there exists a map $\sigma : \{ 1, \ldots, n\} \to \{0, \ldots, m\}$ such that, for every $1 \leq i \leq n$, either $\sigma(i)=0$ and $\varphi(A_i) \subset B_0$ or $\sigma(i) \neq 0$ and $\varphi(A_i) \subset Z(B_0) \oplus B_{\sigma(i)}$. Notice that $\varphi(A_1 \oplus \cdots \oplus A_n) \subset \langle B_i, \ i \in \mathrm{Im}(\sigma) \cup \{0\} \rangle$, hence
$$\begin{array}{lcl} A_0 & \simeq & A/ \left( A_1 \oplus \cdots \oplus A_n \right) \simeq B/ \varphi \left( A_1 \oplus \cdots \oplus A_n \right) \\ \\ & \simeq & \left[ \left( \bigoplus\limits_{i \in \mathrm{Im}(\sigma) \cup \{0\}} B_i \right) / \varphi(A_1 \oplus \cdots \oplus A_n) \right] \oplus  \bigoplus\limits_{i \notin \mathrm{Im}(\sigma) \cup \{0\}} B_i \end{array}$$
It follows that $\mathrm{Im}(\sigma) \cup \{0\} = \{0, \ldots, m\}$, because otherwise $A_0$ would be isomorphic to a $B_i$ or would have it as a direct factor, for some $i \geq 1$. 

\medskip \noindent
By a symmetric argument, there exists a map $\mu : \{1, \ldots, m\} \to \{0, \ldots, n\}$ such that, for every $1 \leq i \leq m$, either $\mu(i)=0$ and $\varphi^{-1}(B_i) \subset A_0$ or $\mu(i) \neq 0$ and $\varphi^{-1}(B_i) \subset Z(A_0) \oplus A_{\mu(i)}$. Moreover, $\mathrm{Im}(\mu) \cup \{0\} = \{0, \ldots, n\}$. Because
$$\# \{ 0,1, \ldots, m\} = | \mathrm{Im}(\sigma) \cup \{0\} | \leq n+1 = \# \{0,1, \ldots, n\} = |\mathrm{Im}(\mu) \cup \{0\}| \leq m+1,$$
it follows that $n=m$, $\mathrm{Im}(\sigma)= \{ 1, \ldots, m\}$ and $\mathrm{Im}(\mu) = \{1, \ldots, n\}$. In particular, $\sigma$ and $\mu$ are bijections. In fact, because
$$A_i = \varphi^{-1}( \varphi( A_i)) \subset \varphi^{-1}(Z(B_0) \oplus B_{\sigma(i)})) \subset Z(A_0) \oplus A_{\mu(\sigma(i))}$$
for every $1 \leq i \leq n$, we have $\mu= \sigma^{-1}$. 

\medskip \noindent
Now, let $\zeta : A_0 \to B_0$ and $\xi : A_0 \to B_1 \oplus \cdots \oplus B_m$ be two morphisms such that $\varphi_{|A_0} : A \mapsto \zeta(a) \xi(a)$. Also, for every $1 \leq i \leq n$, let $\alpha_i : A_i \to Z(B_0)$ and $\beta_i : A_i \to B_{\sigma(i)}$ be two morphisms such that $\varphi_{|A_i} : a \mapsto \alpha_i(a) \beta_i(a)$. We conclude the proof of our lemma by proving the next two claims.

\begin{claim}
The equality $\varphi(A_0) = B_0$ holds.
\end{claim}

\noindent
Notice that, for every $a \in A_0$ and $b \in A_1 \oplus \cdots \oplus A_n$, we have
$$\varphi(ab) = \varphi(a) \varphi(b) = \zeta(a) \xi(a) \varphi(b)$$
and 
$$\varphi(ba) = \varphi(b) \varphi(a) = \varphi(b) \zeta(a) \xi(a) = \zeta(a) \varphi(b) \xi(a)$$
where the last equality is justified by the fact that $\varphi(b)$ belongs to 
$$\varphi(A_1 \oplus \cdots \oplus A_n) \subset Z(B_0) \oplus B_1 \oplus \cdots \oplus B_m$$ 
and so commutes with $\zeta(a) \in B_0$. In other words, we know that $\xi(a) \varphi(b) = \varphi(b) \xi(a)$ for every $a \in A_0$ and $b \in A_1 \oplus \cdots \oplus A_n$, which amounts to saying that
$$\mathrm{Im}(\xi) \subset C_B \left( \varphi(A_1 \oplus \cdots \oplus A_n) \right) \cap (B_1 \oplus \cdots \oplus B_m)$$
where $C_B(\cdots)$ denotes a centraliser in $B:=B_0 \oplus B_1 \oplus \cdots \oplus B_m$. We know that, for every $1 \leq i \leq n$, $\varphi^{-1}(B_{\sigma(i)}) \subset Z(A_0) \oplus A_i$ or equivalently $B_{\sigma(i)} \subset Z(B_0) \oplus \varphi(A_i)$, so
$$\begin{array}{lcl} C_B \left( \varphi(A_1 \oplus \cdots \oplus A_n) \right) & = & C_B \left( \langle \varphi(A_1 \oplus \cdots \oplus A_n), Z(B) \rangle \right) \\ \\ & \subset & C_B\left( Z(B_0) \oplus B_1 \oplus \cdots \oplus B_m \right)=B_0 \end{array}$$
We deduce that $\mathrm{Im}(\xi) = \{1\}$, or equivalently that $\varphi(A_0) \subset B_0$. A symmetric argument shows that $\varphi^{-1}(B_0) \subset A_0$, hence $\varphi(A_0)=B_0$ as desired.

\begin{claim}
For every $1 \leq i \leq n$, $\beta_i$ is an isomorphism.
\end{claim}

\noindent
First, notice that, if $a \in \mathrm{ker}(\beta_i)$, then $\varphi(a) = \alpha_i(a) \beta_i(a)=\alpha_i(a) \in Z(B)$ hence $a \in A_i \cap Z(A) = \{1\}$. Therefore, $\beta_i$ is injective. Next, let $b \in B_{\sigma(i)}$. We know that there exist $a_0 \in A_0, \ldots, a_n \in A_n$ such that $\varphi(a_0 \cdots a_n) = b$. But
$$\begin{array}{lcl} \varphi(a_0 \cdots a_n) & = & \varphi(a_0) \alpha_1(a_1) \beta_1(a_1) \cdots \alpha_n(a_n) \beta_n(a_n) \\ \\ & = & \underset{\in B_0}{\underbrace{\varphi(a_0) \alpha_1(a_1) \cdots \alpha_n(a_n)}} \cdot \underset{\in B_{\sigma(1)}}{\underbrace{\beta_1(a_1)}} \cdots \underset{\in B_{\sigma(n)}}{\underbrace{\beta_n(a_n)}} \end{array}$$
It follows that $\beta_i(a_i) =b$. Therefore, $\beta_i$ is also surjective, concluding the proof of our claim.
\end{proof}

\begin{proof}[Proof of Proposition \ref{prop:Product}.]
Notice that $\Phi \mathcal{G}$ decomposes as $\langle \Phi_0 \rangle \oplus \langle \Phi_1 \rangle \oplus \cdots\oplus \langle \Phi_n \rangle$ and similarly $\Psi \mathcal{H}= \langle \Psi_0 \rangle \oplus \langle \Psi_1 \rangle \oplus \cdots \oplus \langle \Psi_m \rangle$. For every $i \geq 1$,
\begin{itemize}
	\item the centers of $\langle \Psi_i \rangle$ and $\langle \Phi_i \rangle$ are trivial according to Lemma \ref{lem:Center};
	\item it follows from Proposition \ref{prop:SumGraphReducible} that $\langle \Psi_0 \rangle$ (resp. $\langle \Phi_0 \rangle$) is graphically irreducible, so $\langle \Psi_i \rangle$ (resp. $\langle \Phi_i \rangle$) is neither isomorphic to nor a direct factor of $\langle \Phi_0 \rangle$ (resp. $\langle \Psi_0\rangle$);
	\item $\langle \Psi_i \rangle$ (resp. $\langle \Phi_i \rangle$) admits a generating set of pairwise non-commuting elements whose centralisers are maximal according to Lemma \ref{lem:GeneratingSetNice}.
\end{itemize}
Therefore, Lemma \ref{lem:Product} applies and yields the desired conclusion.
\end{proof}

\subsection{Proof of the rigidity theorem}

\noindent
We are now in position to prove our rigidity theorem. We recall its statement for the reader's convenience:

\begin{thm}\label{thm:BigRigidity}
Let $\Phi, \Psi$ be two finite simplicial graphs and $\mathcal{G}, \mathcal{H}$ two collections of graphically irreducible groups respectively indexed by $V(\Phi),V(\Psi)$. Assume that no two distinct vertices in $\Phi$ and $\Psi$ have the same star. For every isomorphism $\varphi : \Phi \mathcal{G} \to \Psi \mathcal{H}$ and every $\prec$-maximal vertex $u \in V(\Phi)$, there exist an element $g \in \Psi \mathcal{H}$ and a $\prec$-maximal vertex $v \in V(\Psi)$ such that $\varphi(\langle [u] \rangle)=g \langle [v] \rangle g^{-1}$. 
\end{thm}

\noindent
We begin by proving two preliminary lemmas.

\begin{lemma}\label{lem:CentraliserNormaliser}
Let $\Gamma$ be a simplicial graph, $\mathcal{G}$ a collection of groups indexed by $V(\Gamma)$, and $u \in V(\Gamma)$ a vertex. For every non-trivial $g \in \langle [u] \rangle$, the centraliser of $g$ lies in the normaliser of $\langle [u] \rangle$. 
\end{lemma}

\begin{proof}
Fix a non-trivial element $g \in  \langle [u] \rangle$ and write $g=hkh^{-1}$ for some $h,k \in \langle [u] \rangle$ where $k$ is graphically cyclically reduced. It follows from Proposition \ref{prop:centraliser} that $C(g) \subset h\langle \mathrm{supp}(g) \cup \mathrm{link}(\mathrm{supp}(g)) \rangle h^{-1}$. But a vertex which does not belong to $[u]$ and which is adjacent to some vertex of $[u]$ must be adjacent to all the vertices of $[u]$, so $\mathrm{link}(\mathrm{supp}(g)) \subset [u] \cup \mathrm{link}([u])$. Hence
$$C(g) \subset h \langle [u] \cup \mathrm{link}([u]) \rangle h^{-1} = \langle [u] \cup \mathrm{link}([u]) \rangle = N( \langle [u] \rangle),$$
where the last equality is justified by Lemma \ref{lem:normaliser}.
\end{proof}

\begin{lemma}\label{lem:Dichotomy}
Let $\Gamma$ be a simplicial graph and $\mathcal{G}$ a collection of groups indexed by $V(\Gamma)$. Fix a $\prec$-maximal vertex $u \in V(\Gamma)$ and a maximal join $\Lambda \subset \Gamma$ containing $\mathrm{star}(u)$. Decompose $\Lambda$ as a join $\Lambda_0 \ast \Lambda_1 \ast \cdots \ast \Lambda_n$ where $\Lambda_0$ is complete and where each graph among $\Lambda_1 ,\ldots, \Lambda_n$ contains at least two vertices and is not a join. If no two vertices in $[u]$ are adjacent, then 
\begin{itemize}
	\item either $\Lambda_0=[u]= \{u\}$;
	\item or $\Lambda_0= \emptyset$ and, if $1 \leq i \leq n$ denotes the index satisfying $u \in \Lambda_i$, then $[u] \subset \Lambda_i$ and $u$ is $\prec$-maximal in $\Lambda_i$.
\end{itemize}
\end{lemma}

\begin{proof}
First, assume that $\Lambda_0$ is non-empty and fix a vertex $x \in V(\Lambda_0)$. Notice that $\mathrm{star}(u) \subset \Lambda \subset \mathrm{star}(x)$, so $u \prec x$ and finally $x \in [u]$. By assumption, no two vertices of $[u]$ are adjacent, so we must have $u=x$. It follows that $\Lambda_0=\{u\}$. Also, because any two vertices in $[u]$ have the same link, notice that $[u] \ast \Lambda_1 \ast \cdots \ast \Lambda_n$ is a join containing $\Lambda$. By maximality, we conclude that $\Lambda_0= [u]$ as desired.

\medskip \noindent
Next, assume that $\Lambda_0$ is empty. Fix the index $1 \leq i \leq n$ such that $u$ belongs to $\Lambda_i$. Because the vertices in $[u]$ are pairwise non-adjacent, necessarily $[u] \subset \Lambda_i$. Fix a vertex $z \in V(\Lambda_i)$ such that $u \prec z$ in $\Lambda_i$, i.e., $\mathrm{link}_{\Lambda_i}(u) \subset \mathrm{star}_{\Lambda_i}(z)$. We have
$$\mathrm{link}(u) = \mathrm{link}_\Lambda(u) = \mathrm{link}_{\Lambda_i}(u) \cup \bigcup\limits_{j \neq i} \Lambda_j \subset \mathrm{star}(z),$$
hence $u \prec z$ in $\Gamma$. Because $u$ is $\prec$-maximal in $\Gamma$, we must have $z \prec u$ in $\Gamma$ and a fortiori in $\Lambda_i$. Thus, $u$ has to be $\prec$-maximal in $\Lambda_i$. 
\end{proof}

\noindent
In the proof which follows we use the following terminology. Given a graph $\Phi$, decompose it as a join $\Phi_0 \ast \Phi_1 \ast \cdots \ast \Phi_n$ where $\Phi_0$ is complete and where each graph among $\Phi_1, \ldots, \Phi_n$ contains at least two vertices and is not a join. Notice that this decomposition is unique since $\Phi_1, \ldots, \Phi_n$ correspond to the connected components of $\Phi^\mathrm{opp}$ which are not reduced to a single vertex and that $\Phi_0$ corresponds to the union of the components which are reduced to single vertices. Then $\Phi_0, \Phi_1, \ldots, \Phi_n$ will be referred to as the \emph{factors} of $\Phi$.

\begin{proof}[Proof of Theorem \ref{thm:BigRigidity}.]
Define the sequence $\Lambda_0 \supset \Delta_1 \supset \Lambda_1 \supset \Delta_2 \supset \Lambda_2 \supset \cdots \supset \Delta_r \supset \Lambda_r$ of subgraphs in $\Gamma$ as follows:
\begin{itemize}
	\item set $\Lambda_0= \Gamma$;
	\item for every $i \geq 1$, $\Delta_i$ is a maximal join in $\Lambda_{i-1}$ containing $\mathrm{star}(u) \cap \Lambda_{i-1}$;
	\item for every $i \geq 1$, $\Lambda_i$ is the factor of $\Delta_i$ which contains $u$.
\end{itemize}
We stop the sequence when $u$ does not lie in a join in $\Lambda_i$ which is not a clique. 

\begin{claim}\label{claim:Maximal}
For every $0 \leq i \leq r$, the vertex $u$ is $\prec$-maximal in $\Lambda_i$.
\end{claim}

\noindent
The assertion is clear for $i=0$. Now, assume that the assertion holds for some $0 \leq i < r$. By construction, $\Delta_{i+1}$ is a maximal join in $\Lambda_i$ which contains $\mathrm{star}_{\Lambda_i}(u)$. According to Lemma \ref{lem:Dichotomy}, either $\Lambda_{i+1}$ is the complete factor of $\Delta_{i+1}$, and then it is clear that $u$ is $\prec$-maximal in $\Lambda_{i+1}$; or the complete factor of $\Delta_{i+1}$ is empty and $u$ is $\prec$-maximal in $\Lambda_{i+1}$. Our claim follows by induction.

\begin{claim}
We have $\Lambda_r = [u]$.
\end{claim}

\noindent
Notice that, because no two vertices in $\Phi$ have the same star, necessarily the vertices in $[u]$ are pairwise non-adjacent. 

\medskip \noindent
We claim that $[u] \subset \Lambda_i$ for every $0 \leq i \leq r$. The assertion is clear for $i=0$. Now assume that it holds for some $0 \leq i < r$. By construction, $\Delta_{i+1}$ is a maximal join in $\Lambda_i$ which contains $\mathrm{star}_{\Lambda_i}(u)$. According to Lemma \ref{lem:Dichotomy}, either $\Lambda_{i+1}$ is the complete factor of $\Delta_{i+1}$ and $\Lambda_{i+1}= [u]$; or the complete factor of $\Delta_{i+1}$ is empty and $[u]$ lies in a factor of $\Delta_{i+1}$, which has to be $\Lambda_{i+1}$. The desired conclusion follows by induction.

\medskip \noindent
As a particular case, $[u] \subset \Lambda_r$. We also know that $u$ is $\prec$-maximal in $\Lambda_r$ (according to Claim \ref{claim:Maximal}) and that it does not lie in a join in it (by construction). It follows that $\mathrm{star}_{\Lambda_r}(u)= \{u\}$, i.e., $u$ is an isolated vertex in $\Lambda_r$. Moreover, by the maximality previously mentioned, $\Lambda_r$ cannot contain an edge, so $\Lambda_r$ is a disjoint union of isolated vertices. 

\medskip \noindent
Fix a vertex $v \in \Lambda_r$. We claim that $\mathrm{link}_{\Lambda_i}(u)= \mathrm{link}_{\Lambda_i}(v)$ for every $0 \leq i \leq r$. The assertion is clear for $i=r$, because the two links are empty. Now assume that the assertion is true for some $1 \leq i \leq r$. Because $\Lambda_i$ is a factor of $\Delta_i$ which contains both $u$ and $v$, necessarily $\mathrm{link}_{\Delta_i}(u)= \mathrm{link}_{\Delta_i}(v)$. By construction, $\mathrm{link}_{\Lambda_{i-1}}(u) \subset \Delta_i$, so 
$$\mathrm{link}_{\Lambda_{i-1}}(u)= \mathrm{link}_{\Delta_i}(u) = \mathrm{link}_{\Delta_i}(v) \subset \mathrm{link}_{\Lambda_{i-1}}(v),$$ 
hence $u \prec v$ in $\Lambda_{i-1}$. But $u$ is $\prec$-maximal in $\Lambda_{i-1}$ according to Claim \ref{claim:Maximal}, so we must have $\mathrm{link}_{\Lambda_{i-1}}(u)= \mathrm{link}_{\Lambda_{i-1}}(v)$. The desired conclusion follows by induction.

\medskip \noindent
As a particular case, $\mathrm{link}(u)= \mathrm{link}_{\Lambda_0}(u)= \mathrm{link}_{\Lambda_0}(v)= \mathrm{link}(v)$. So $v \in [u]$. Thus, we have proved that $\Lambda_r \subset [u]$, concluding the proof of our claim.

\begin{claim}\label{claim:NotClique}
For every $1 \leq i \leq r$, $\Delta_i$ cannot be a complete graph.
\end{claim}

\noindent
Fix an index $1 \leq i \leq r$ and assume that $\Delta_i$ is a complete graph. By construction, $\Lambda_i=\Delta_i$ must be complete, hence $i=r$. Fix a vertex $v \in \Delta_r$. 

\medskip \noindent
Notice that $\mathrm{star}_{\Delta_r}(u)= \mathrm{star}_{\Delta_r}(v)$ because $\Delta_r$ is complete. Now, assume that $\mathrm{star}_{\Delta_i}(u)= \mathrm{star}_{\Delta_i}(v)$ for some $2 \leq i \leq r$. By construction, $\mathrm{star}_{\Lambda_{i-1}}(u) \subset \Delta_i$, so 
$$\mathrm{star}_{\Lambda_{i-1}}(u)= \mathrm{star}_{\Delta_i}(u) = \mathrm{star}_{\Delta_i}(v) \subset \mathrm{star}_{\Lambda_{i-1}}(v),$$ 
hence $u \prec v$ in $\Lambda_{i-1}$. But $u$ is $\prec$-maximal in $\Lambda_{i-1}$ according to Claim \ref{claim:Maximal}, so we must have $\mathrm{star}_{\Lambda_{i-1}}(u)= \mathrm{star}_{\Lambda_{i-1}}(v)$. Next, assume that $\mathrm{star}_{\Lambda_{i}}(u)= \mathrm{star}_{\Lambda_{i}}(v)$ for some $1 \leq i < r$. Because $\Lambda_{i}$ is a factor of $\Delta_{i}$ which contains both $u$ and $v$, we deduce that $\mathrm{star}_{\Delta_{i}}(u)= \mathrm{star}_{\Delta_{i}}(v)$. As a consequence, we have $\mathrm{star}(u)= \mathrm{star}_{\Lambda_0}(u)= \mathrm{star}_{\Lambda_0}(v)= \mathrm{star}(v)$. But no two vertices in $\Psi$ have the same star, so $u=v$. 

\medskip \noindent
Thus, we have proved that our graph $\Delta_i$ is reduced to a single vertex, contradicting the fact that it is a join.

\begin{claim}\label{claim:NotIsolated}
For every $1 \leq i \leq r$, $\varphi(\langle \Delta_i \rangle)$ does not lie in a conjugate of an isolated vertex-group in $\Psi \mathcal{H}$.
\end{claim}

\noindent
If $\varphi( \langle \Delta_i \rangle)$ lies in a conjugate of an isolated vertex-group, say $h \langle v \rangle h^{-1}$, then $\varphi^{-1}(h \langle v \rangle h^{-1})$ would be a free factor of $\Phi \mathcal{G}$ containing $\langle \Delta_i \rangle$. Because $\Delta_i$ contains at least two vertices (otherwise it would not be a join), the subgroup $\langle \Delta_i \rangle$ cannot be infinite cyclic, so $\varphi^{-1}(h \langle v \rangle h^{-1})$ would coincides with $\langle \Gamma \rangle$ where $\Gamma$ denotes the connected component of $\Phi$ containing $\Delta_i$. But vertex-groups are graphically irreducible, so $\Gamma$ (and a fortiori $\Delta_i$) has to be a complete graph, contradicting Claim \ref{claim:NotClique}.

\begin{claim}\label{claim:Aux}
There exist subgraphs $\Theta_0 \supset \Xi_1 \supset \Theta_1 \supset \cdots \supset \Xi_r \supset \Theta_r$ in $\Psi$ and elements $g_0, \ldots, g_r \in \Psi \mathcal{H}$ such that 
\begin{itemize}
	\item $\Theta_0=\Psi$;
	\item $\varphi(\langle \Lambda_i \rangle)= g_i \langle \Theta_i \rangle g_i^{-1}$ for every $0 \leq i \leq r$;
	\item $\varphi(\langle \Delta_i \rangle)= g_i \langle \Xi_i \rangle g_i^{-1}$ for every $1 \leq i \leq r$;
	\item for every $1 \leq i \leq r$, $\Xi_i$ is a maximal join in $\Theta_{i-1}$;
	\item for every $1 \leq i \leq r$, $\Theta_i$ is a factor of $\Xi_i$.
\end{itemize}
Moreover, for every $1 \leq i \leq r$, if $\Xi_i$ has a complete factor, then it coincides with $\Theta_i$.
\end{claim}

\noindent
Set $g_0=1$ and $\Theta_0=\Psi$. Now, suppose that $\Theta_i$ and $g_i$ are defined for some $0 \leq i<r$. Because $\langle \Delta_{i+1} \rangle$ is a maximal product subgroup in $\langle \Lambda_{i} \rangle$, necessarily $\varphi(\langle \Delta_{i+1} \rangle)$ is a maximal product subgroup in $\varphi(\langle \Lambda_i \rangle)= g_i \langle \Theta_i \rangle g_i^{-1}$. As a consequence of Claim \ref{claim:NotIsolated} and Proposition \ref{prop:MaxProducts}, there exist maximal join $\Xi_{i+1}$ in $\Theta_i$ and an element $g_{i+1} \in \Psi \mathcal{H}$ such that $\varphi( \langle \Delta_{i+1} \rangle) =g_{i+1} \langle \Xi_{i+1} \rangle g_{i+1}^{-1}$. In order to define $\Theta_{i+1}$, we distinguish two cases. Indeed, according to Lemma \ref{lem:Dichotomy}, two cases may happen depending on whether $\Delta_{i+1}$ has a complete factor.

\medskip \noindent
First, assume that $\Delta_{i+1}$ has a complete factor. According to Lemma \ref{lem:Dichotomy}, this factor has to coincide with $\{u\}$. By construction, we must have $\Lambda_{i+1}=\{ u \}$. As a consequence of Proposition~\ref{prop:Product}, $\varphi(\langle \Lambda_{i+1} \rangle) = g_{i+1}\langle \Theta_{i+1} \rangle g_{i+1}^{-1}$ where $\Theta_{i+1}$ denotes the complete factor of $\Xi_{i+1}$. 

\medskip \noindent
Next, assume that $\Delta_{i+1}$ does not have a complete factor. It follows from Proposition~\ref{prop:Product} that $\Xi_{i+1}$ does not have a complete factor and that $\varphi(\langle \Lambda_{i+1} \rangle) = g_{i+1}\langle \Theta_{i+1} \rangle g_{i+1}^{-1}$ for some factor $\Theta_{i+1}$ of $\Xi_{i+1}$. So Claim \ref{claim:Aux} is proved.

\medskip \noindent
Notice that $\varphi (\langle [u] \rangle)= \varphi( \langle \Lambda_r \rangle) = g_r \langle \Theta_r \rangle g_r^{-1}$. Therefore, in order to conclude the proof of our theorem, it remains to show that $\Theta_r$ is the class of a $\prec$-maximal vertex in $\Psi$. 

\begin{claim}\label{claim:linklink}
Let $v$ be a neighbor of one of the vertices in $\Theta_r$. Either $v$ belongs to $\Theta_r$ or it is adjacent to all the vertices in $\Theta_r$.
\end{claim}

\noindent
Fix a non-trivial element $h \in \langle v \rangle$. Because $v$ is adjacent to a vertex in $\Theta_r$, $h$ commutes with an element of $\langle \Theta_r \rangle$. So $\varphi^{-1}(g_rhg_r^{-1})$ commutes with an element of $\varphi^{-1} \left( g_r \langle \Theta_r \rangle g_r^{-1} \right) = \langle \Lambda_r \rangle = \langle [u] \rangle$. It follows from Lemma \ref{lem:CentraliserNormaliser} that this element belongs to the normaliser of $\langle [u] \rangle$. So $g_rhg_r^{-1}$ belongs to the normaliser of $\varphi ( \langle [u] \rangle) = g_r \langle \Theta_r \rangle g_r^{-1}$, which is $g_r \langle \Theta_r \cup \mathrm{link}(\Theta_r) \rangle g_r^{-1}$ according to Lemma \ref{lem:normaliser}. Therefore, $h$ is a non-trivial element in $\langle v \rangle \cap \langle \Theta_r \cup \mathrm{link}(\Theta_r) \rangle$. We conclude that $v$ belongs to $\Theta_r \cup \mathrm{link}(\Theta_r)$ as desired.

\begin{claim}\label{claim:IsolatedCliques}
$\Theta_r$ is a disjoint union of isolated cliques.
\end{claim}

\noindent
The restriction of $\varphi$ to $\langle \Lambda_r \rangle = \langle [u] \rangle$ induces an isomorphism $\langle [u] \rangle \to g_r \langle \Theta_r \rangle g_r^{-1}$. If $\Gamma$ is a connected component of $\Theta_r$ which contains at least two vertices, then $\varphi^{-1}(g_r \langle \Gamma \rangle g_r^{-1} )$ is a free factor of $\langle [u] \rangle$ which is not cyclic, so there must exists some $z \in [u]$ such that $\langle z \rangle$ is isomorphic to $\langle \Gamma \rangle$. Because vertex-groups are graphically irreducible, this implies that $\Gamma$ must be a complete graph. Our claim is proved.

\medskip \noindent
As a consequence of Claims \ref{claim:linklink} and \ref{claim:IsolatedCliques}, any two adjacent vertices in $\Theta_r$ have the same star in $\Psi$. But, no two distinct vertices in $\Psi$ have the same star, so it follows from Claim~\ref{claim:IsolatedCliques} that $\Theta_r$ is a disjoint union of isolated vertices. It also follows from Claim~\ref{claim:linklink} that any two vertices in $\Theta_r$ have the same link in $\Psi$, so $\Theta_r = [v]$ where $v \in \Theta_r$ is an arbitrary vertex. We conclude the proof of our theorem by showing the next claim.

\begin{claim}
The vertex $v$ is $\prec$-maximal in $\Psi$.
\end{claim}

\noindent
Let $w$ be a vertex in $\Psi$ satisfying $v \prec w$. We claim that $w$ belongs to $\Theta_i$ for every $0 \leq i \leq r$. The assertion is clear for $i=0$. Now, assume that $w$ belongs to $\Theta_i$ for some $0 \leq i < r$. According to Claim \ref{claim:Aux}, $\Xi_{i+1}$ is a maximal join in $\Theta_i$ and $\Theta_{i+1}$ is a factor of $\Xi_{i+1}$. Let $\Theta_{i+1}, \Omega_1, \ldots, \Omega_s$ denote the factors of $\Xi_{i+1}$. Notice that, because $v \in \Theta_r \subset \Theta_{i+1}$, we have
$$\Omega_1 \cup \cdots \cup \Omega_s \subset \mathrm{link}(v) \subset \mathrm{star}(w),$$
so $(\Theta_{i+1} \cup \{w\}) \ast \Omega_1 \ast \cdots \ast \Omega_s$ would be a join in $\Theta_i$ containing $\Xi_{i+1}$ properly if $w \notin \Xi_{i+1}$. By maximality of $\Xi_{i+1}$, $w$ has to belong to $\Xi_{i+1}$. If $w \in \Omega_j$ for some $1 \leq j \leq s$, then it follows from the previous inclusion that $w$ is adjacent to all the vertices of $\Xi_{i+1}$, i.e., $w$ belongs to the complete factor of $\Xi_{i+1}$. But we know from Claim \ref{claim:Aux} that, if $\Xi_{i+1}$ has a complete factor, then it must be $\Theta_{i+1}$. We conclude that $w$ belongs to $\Theta_{i+1}$. The desired conclusion follows by induction.

\medskip \noindent
As a particular case, we deduce that $w \in \Theta_r=[v]$. Thus, we have proved that $v$ is $\prec$-maximal. 
\end{proof}

\section{Step 2: Generalised loxodromic elements}\label{section:StepTwo}

\noindent
In this section, our goal is to determine precisely when the inner automorphism associated to an element of our graph product (which has full support) defines a generalised loxodromic element in the automorphism group. More precisely, the section is dedicated to the proof of the following criterion:

\begin{thm}\label{thm:WhenGenLox}
Let $\Gamma$ be a finite connected simplicial graph, $\mathcal{G}$ a collection of finitely generated graphically irreducible groups indexed by $V(\Gamma)$, and $g \in \Gamma \mathcal{G}$ an element of full support. Assume that $\Gamma$ contains at least two vertices and is not a join. The following assertions are equivalent:
\begin{itemize}
	\item the inner automorphism $\iota(g)$ is a generalised loxodromic element of $\mathrm{Aut}(\Gamma \mathcal{G})$;
	\item $\{ \varphi \in \mathrm{Aut}(\Gamma \mathcal{G}) \mid \varphi(g)=g\}$ is virtually cyclic;
	\item  $\{ \varphi \in \mathrm{Aut}(\Gamma \mathcal{G}) \mid \varphi(g)=g\}$ has a finite image in $\mathrm{Out}(\Gamma \mathcal{G})$;
	\item the centraliser of $\iota(g)$ in $\mathrm{Aut}(\Gamma \mathcal{G})$ is virtually cyclic.
\end{itemize}
\end{thm}

\noindent
In Subsection \ref{section:Simplification}, we show that we may suppose without loss of generality that, in the graph $\Gamma$ defining our graph product $\Gamma \mathcal{G}$, no two vertices have the same link or the same star. Moreover, under this assumption, we apply the ridigity proved in the previous section in order to show that $\mathrm{Aut}(\Gamma \mathcal{G})$ conjugates vertex-groups labelled by $\prec$-maximal vertices. This is the key observation which allows us to construct, in Subsection \ref{section:Crossing}, an action of $\mathrm{Aut}(\Gamma \mathcal{G})$ on the \emph{small crossing graph} $\ST$. This graph is defined in terms of hyperplanes in $\QM$, and we exploit this geometric point of view to show that $\ST$ is quasi-isometric to a tree, and so hyperbolic. Subsections \ref{section:Normal} and \ref{section:Irreducible} contain preliminary results needed in Subsection \ref{section:WPD}, where we determine when the inner automorphism $\iota(g)$ associated to an element $g \in \Gamma \mathcal{G}$ of full support defines a WPD element of $\mathrm{Aut}(\Gamma \mathcal{G})$ with respect to the action on the small crossing graph. Theorem~\ref{thm:WhenGenLox} follows from this criterion.

\subsection{A simplification}\label{section:Simplification}

\noindent
In this subsection, our goal is to show that we may suppose without loss of generality that, in the graph defining our graph product, no two vertices have the link or the same star. This observation will be fundamental in the next subsections. Our main result is:

\begin{prop}\label{prop:SameLinkStar}
Let $\Gamma$ be a a finite connected graph and $\mathcal{G}$ a collection of finitely generated graphically irreducible groups indexed by $V(\Gamma)$. Assume that $\Gamma$ does not decompose as a join and contains at least two vertices. There exist a simplicial graph $\Phi$ and a collection of finitely generated groups $\mathcal{H}$ such that:
\begin{itemize}
	\item $\Phi$ is finite, connected, it does not decompose as a join and it contains at least two vertices;
	\item no two vertices of $\Phi$ have the same star or the same link;
	\item there is an isomorphism $\Gamma \mathcal{G} \to \Phi \mathcal{H}$ which sends an element of full support of $\Gamma \mathcal{G}$ to an element of full support of $\Phi \mathcal{H}$.
	\item for every $\varphi \in \mathrm{Aut}(\Phi \mathcal{H})$ and every $\prec$-maximal vertex $u \in V(\Phi)$, there exist an element $g \in \Phi \mathcal{H}$ and a $\prec$-maximal vertex $v \in V(\Phi)$ such that $\varphi(\langle u \rangle)=g \langle v \rangle g^{-1}$. 
\end{itemize}
\end{prop}

\noindent
We begin by proving the following easy observation:

\begin{lemma}\label{lem:GraphSimplication}
Let $\Gamma$ be a simplicial graph. Let $\Psi$ denote the graph obtained from $\Gamma$ by identifying any two vertices having the same star, and $\Phi$ the graph obtained from $\Psi$ by identifying any two vertices having the same link. Then 
\begin{itemize}
	\item no two vertices in $\Psi$ have the same star;
	\item no two vertices in $\Phi$ have the same link or the same star;
	\item a $\prec$-maximal vertex in $\Phi$ is the image of a $\prec$-maximal vertex in $\Psi$;
	\item if $\Gamma$ is connected, contains at least two vertices and is not a join, then so does $\Phi$.
\end{itemize}
\end{lemma}

\begin{proof}
Let $u,v \in V(\Phi)$ be two vertices having the same link. Fix two vertices $\bar{u}, \bar{v} \in V(\Psi)$ such that their images in $\Phi$ via the canonical map $\Psi \to \Phi$ are respectively $u,v$. If $\bar{w}$ belongs to $\mathrm{link}(\bar{u})$ then its image in $\Psi$ has to belong to $\mathrm{link}(u)= \mathrm{link}(v)$, so we deduce that $\bar{w}$ is adjacent to $\bar{v}$ as well. Similarly, one shows that any vertex adjacent to $\bar{v}$ must be adjacent to $\bar{u}$. It follows that $\bar{u}$ and $\bar{v}$ have the same link, hence $u=v$. Thus, we have proved that no two distinct vertices in $\Phi$ have the same link.

\medskip \noindent
Now, assume that $u,v \in V(\Phi)$ are two vertices which have the same star. The same argument as above shows that, if $\bar{u}$ and $\bar{v}$ are two lifts in $\Psi$ of $u$ and $v$ respectively, then $\bar{u}$ and $\bar{v}$ have the same star. And similarly, if $\tilde{u}$ and $\tilde{v}$ are two lifts in $\Gamma$ of $\bar{u}$ and $\bar{v}$ respectively, then $\tilde{u}$ and $\tilde{v}$ have the same star. It follows that $\bar{u}=\bar{v}$, and so $u=v$. Thus, we have proved that no two distinct vertices in $\Phi$ (or in $\Psi$) have the same star.

\medskip \noindent
Next, let $u \in V(\Phi)$ be a $\prec$-maximal vertex and let $\bar{u}$ be one of its lifts in $\Psi$. We claim that $\bar{u}$ is $\prec$-maximal in $\Psi$. So let $\bar{v} \in V(\Psi)$ be a vertex satisfying $\bar{u} \prec \bar{v}$. Notice that $u \prec v$. Indeed, if $w \in \mathrm{link}(u)$ then any lift $\bar{w}$ in $\Psi$ of $w$ must be adjacent to $\bar{u}$. Therefore, $\bar{w}$ belongs to $\mathrm{link}(\bar{u}) \subset \mathrm{star}(\bar{v})$, which implies that $w$ belongs to the star of the image $v$ of $\bar{v}$ in $\Phi$. Thus, we have proved that $\mathrm{link}(u) \subset \mathrm{star}(v)$, i.e., $u \prec v$ as desired. But $u$ is $\prec$-maximal in $\Phi$, so we also have $v \prec u$. This implies that $u$ and $v$ have the same star or the same link, hence $u=v$ by the previous paragraphs. Therefore, $\bar{u}$ and $\bar{v}$ must have the same link, hence $\bar{v} \prec \bar{u}$. Thus, we have proved that $\bar{u}$ is $\prec$-maximal as claimed.

\medskip \noindent
Finally, assume that $\Gamma$ is connected, contains at least two vertices and is not a join. Because a graph and its opposite graph remain connected when we identify vertices, necessarily $\Phi$ and $\Psi$ are connected and are not joins. If $\Phi$ were reduced to a single vertex, then any two vertices in $\Psi$ would have the same link, i.e., $\Psi$ would be a disjoint union of isolated vertices. In fact, as we know that $\Psi$ is connected, it would be reduced to a single vertex as well. Consequently, any two vertices in $\Gamma$ would have the same star, i.e., $\Gamma$ would be a complete graph, which is impossible as it is not a join and it contains at least two vertices. Thus, $\Phi$ has to contain at least two vertices.
\end{proof}

\begin{proof}[Proof of Proposition \ref{prop:SameLinkStar}.]
Let $\Psi$ denote the graph obtained from $\Gamma$ by identifying any two vertices having the same star, and let $\mathcal{K}$ denote the collection of groups $\{ K_u \mid u \in V(\Psi) \}$ where
$$K_u:= \bigoplus\limits_{\text{$w$ lift of $u$ in $\Gamma$}} G_w$$
for every $u \in V(\Psi)$. Notice that $\Gamma \mathcal{G}$ is naturally isomorphic to $\Psi \mathcal{K}$ and that an element of full support in $\Gamma \mathcal{G}$ is sent to an element of full support in $\Psi \mathcal{K}$. Also, as a consequence of Proposition \ref{prop:SumGraphReducible}, the groups in $\mathcal{K}$ are graphically irreducible. Similarly, let $\Phi$ denote the graph obtained from $\Psi$ by identifying any two vertices having the same star, and let $\mathcal{H}$ denote the collection of groups $\{ H_u \mid u \in V(\Phi) \}$ where
$$H_u:= \underset{\text{$w$ lift of $u$ in $\Psi$}}{\ast} K_w$$
for every $u \in V(\Phi)$. Again, $\Psi \mathcal{K}$ is naturally isomorphic to $\Phi \mathcal{H}$ and an element of full support in $\Psi \mathcal{K}$ is sent to an element of full support in $\Phi \mathcal{H}$. Also, notice that the groups in $\mathcal{H}$ are finitely generated; that $\Phi$ is finite; that, according to Lemma \ref{lem:GraphSimplication}, $\Gamma$ is connected, it contains at least two vertices, it is not a join, and no two of its vertices have the same link or the same star.

\medskip \noindent
According to Lemma \ref{lem:GraphSimplication}, we also know that no two vertices in $\Psi$ have the same star, so Theorem \ref{thm:IntroStepOne} applies, i.e., for every automorphism $\varphi \in \mathrm{Aut}(\Psi \mathcal{K})$ and every $\prec$-maximal vertex $u \in V(\Psi)$, there exist an element $g \in \Psi \mathcal{K}$ and a $\prec$-maximal vertex $v \in V(\Psi)$ such that $\varphi(\langle [u] \rangle)=g \langle [v] \rangle g^{-1}$. Notice that, because no two vertices in $\Psi$ have the same star, the equivalence classes of $\prec$-maximal vertices are collections of vertices having the same link, so they are sent to vertices in $\Phi$. We deduce that there exists a collection of vertices $\mathscr{S} \subset V(\Phi)$ such that, for every automorphism $\varphi \in \mathrm{Aut}(\Phi \mathcal{H})$ and every vertex $u \in \mathscr{S}$, there exist an element $g \in \Phi \mathcal{H}$ and a vertex $v \in \mathscr{S}$ such that $\varphi( \langle u \rangle)= g \langle v \rangle g^{-1}$. As a consequence of Lemma \ref{lem:GraphSimplication}, $\mathscr{S}$ contains the $\prec$-maximal vertices of $\Phi$, but it may be larger. 

\medskip \noindent
Let $\varphi \in \mathrm{Aut}(\Phi \mathcal{H})$ be an automorphism and $u \in \mathscr{S}$ a vertex. We claim that, if $\varphi(\langle u \rangle)$ is not conjugate to a maximal vertex-group (i.e., a vertex-group indexed by a $\prec$-maximal vertex), then $u$ is not $\prec$-maximal in $\Phi$. This assertion is sufficient in order to conclude the proof of our proposition.

\medskip \noindent
Let $v \in \mathscr{S}$ be a vertex and $g \in \Phi \mathcal{H}$ an element such that $\varphi(\langle u \rangle)= g \langle v \rangle g^{-1}$. Assume that $v$ is not $\prec$-maximal. So there exists a $\prec$-maximal vertex $w$ such that $v \prec w$ but not $w \prec v$. Let $x \in \mathscr{S}$ be a vertex and $h \in \Phi \mathcal{H}$ an element such that $\varphi^{-1}(g \langle w \rangle g^{-1})= h \langle x \rangle h^{-1}$. We distinguish two cases.

\medskip \noindent
First, assume that $v$ and $w$ are adjacent. So we have $\mathrm{star}(v) \subsetneq \mathrm{star}(w)$. We have
$$\varphi( N(\langle u \rangle)) =  N(g \langle v \rangle g^{-1})= g \langle \mathrm{star}(v) \rangle g^{-1} \subsetneq g \langle \mathrm{star}(w) \rangle g^{-1}= N(g \langle w \rangle g^{-1})$$
hence
$$\langle \mathrm{star}(u) \rangle = N(\langle u \rangle) \subsetneq N(h \langle x \rangle h^{-1}) = h \langle \mathrm{star}(x) \rangle h^{-1},$$
and finally $\mathrm{star}(u) \subsetneq \mathrm{star}(x)$ as a consequence of Lemma \ref{lem:Inclusion}. This implies that $u \prec x$ but not $x \prec u$, so $u$ is not $\prec$-maximal.

\medskip \noindent
Next, assume that $v$ and $w$ are not adjacent. So $\mathrm{link}(v) \subsetneq \mathrm{link}(w)$ and $v \notin \mathrm{link}(w)$. As a consequence,
$$\begin{array}{lcl} N(\langle v \rangle) & = & \langle \mathrm{star}(v) \rangle = \langle v \rangle \oplus \langle \mathrm{link}(v) \rangle = \langle v \rangle \oplus \langle \mathrm{star}(v) \cap \mathrm{star}(w) \rangle \\ \\ & = & \langle v \rangle \oplus \left( N(\langle v \rangle) \cap N(\langle w \rangle) \right), \end{array}$$
so
$$\begin{array}{lcl} \varphi( N(\langle u \rangle)) & = & N(g \langle v \rangle g^{-1}) = g \langle v \rangle g^{-1} \oplus g \left( N(\langle v \rangle) \cap N(\langle w \rangle) \right) g^{-1} \\ \\ & = & \varphi (\langle u \rangle) \oplus \left( \varphi( N(\langle u \rangle)) \cap N(g \langle w \rangle g^{-1}) \right). \end{array}$$
By applying $\varphi^{-1}$, we get
$$\langle \mathrm{star}(u) \rangle = N(\langle u \rangle) = \langle u \rangle \oplus \left( N(\langle u \rangle) \cap N(h \langle x \rangle h^{-1}) \right) = \langle u \rangle \oplus \left( \langle \mathrm{star}(u) \rangle \cap h \langle \mathrm{star}(x) \rangle h^{-1} \right).$$
According to Proposition \ref{prop:InterParabolic}, $\langle \mathrm{star}(u) \rangle \cap h \langle \mathrm{star}(x) \rangle h^{-1}$ is parabolic subgroup, say $k \langle \Lambda \rangle k^{-1}$ for some $\Lambda \subset \Phi$ and some $k \in \Phi \mathcal{H}$. Without loss of generality, we assume that the tail of $k$ does not contain syllables in $\langle \Lambda \cup \mathrm{link}(\Lambda) \rangle$. As consequence, it follows from Lemma~\ref{lem:Inclusion} and from the inclusion $k \langle \Lambda \rangle k^{-1} \subset \langle \mathrm{star}(u) \rangle$ that $k \in \langle \mathrm{star}(u) \rangle$. So
$$\langle \mathrm{star}(u) \rangle = \langle u \rangle \oplus k \langle \Lambda \rangle k^{-1}= k \left( \langle u \rangle \oplus \langle \Lambda \rangle \right) k^{-1}= k \langle \Lambda \cup \{u \} \rangle k^{-1}.$$
Again as a consequence of Lemma \ref{lem:Inclusion}, we have $\mathrm{star}(u) \subset \Lambda \cup \{u \}$; we also deduce from the inclusion $k \langle \Lambda \rangle k^{-1} \subset h \langle \mathrm{star}(x) \rangle h^{-1}$ that $\Lambda \subset \mathrm{star}(x)$. Therefore, 
$$\mathrm{link}(u) = \mathrm{star}(u) \backslash \{u\} \subset \Lambda  \subset \mathrm{star}(x).$$
Thus, we have proved that $u \prec x$. On the other hand, because the inclusion $\mathrm{link}(v) \subsetneq \mathrm{link}(w)$ is strict, we have
$$\left( N(\langle v \rangle) \cap N(\langle w \rangle) \right) \oplus \langle w \rangle = \langle \mathrm{star}(v) \cap \mathrm{star}(w) \rangle \oplus \langle w \rangle \subsetneq \langle \mathrm{star}(w) \rangle = N(\langle w \rangle).$$
By conjugating by $g$ and next by applying $\varphi^{-1}$, we get
$$\underset{=\langle \mathrm{star}(u) \rangle \cap h \langle \mathrm{star}(x) \rangle h^{-1} = k \langle \Lambda \rangle k^{-1}}{\underbrace{\left( N(\langle u \rangle) \cap h N(\langle x \rangle) h^{-1} \right)}} \oplus h \langle x \rangle h^{-1} \subsetneq h N(\langle x \rangle) h^{-1}= h \langle \mathrm{star}(x) \rangle h^{-1},$$
hence, as a consequence of Lemma \ref{lem:Inclusion}, $\mathrm{link}(u) \subset \Lambda \subsetneq \mathrm{star}(x),$ which proves that the relation $x \prec u$ does not hold. Thus, we have proved that $u$ is not $\prec$-maximal.
\end{proof}

\subsection{The small crossing graph}\label{section:Crossing}

\noindent
In this subsection, our goal is to associate to $\QM$ a quasi-tree on which $\mathrm{Aut}(\Gamma \mathcal{G})$ acts by isometries. For convenience, we introduce the following terminology:

\begin{definition}
Let $\Gamma$ be a simplicial graph and $\mathcal{G}$ a collection of groups indexed by $V(\Gamma)$. A vertex-group is \emph{maximal}\index{Maximal vertex-group} if it is indexed by a $\prec$-maximal vertex of $\Gamma$. A hyperplane in $\QM$ is \emph{maximal}\index{Maximal hyperplane} if it is labelled by a $\prec$-maximal vertex of $V(\Gamma)$.
\end{definition}

\noindent
Our geometric model is the following:

\begin{definition}
Let $\Gamma$ be a simplicial graph and $\mathcal{G}$ a collection of groups indexed by $V(\Gamma)$. The \emph{crossing graph}\index{Crossint graph $\crossing$} $T(\Gamma ,\mathcal{G})$ of $\mathrm{QM}(\Gamma, \mathcal{G})$ is the graph whose vertices are the hyperplanes of $\mathrm{QM}(\Gamma, \mathcal{G})$ and whose edges link two hyperplanes whenever they are transverse. The \emph{small crossing graph}\index{Small crossing graph $\ST$} $ST(\Gamma ,\mathcal{G})$ of $\mathrm{QM}(\Gamma, \mathcal{G})$ is the subgraph of $\crossing$ spanned by the maximal hyperplanes.
\end{definition}

\noindent
Our first observation is that the automorphism group of a graph product as given by Proposition \ref{prop:SameLinkStar} naturally acts on the small crossing graph.

\begin{lemma}\label{lem:DefAction}
Let $\Gamma$ be a simplicial connected graph and $\mathcal{G}$ a collection of groups indexed by $V(\Gamma)$. Assume that, for every $\prec$-maximal vertex $u \in V(\Gamma)$, there exist an element $g\in \Gamma \mathcal{G}$ and a $\prec$-maximal vertex $v\in V(\Gamma)$ such that $\varphi(\langle u \rangle)= g \langle v \rangle g^{-1}$. Then $\mathrm{Aut}(\Gamma \mathcal{G})$ acts by isometries on $S\crossing$ via
$$\left\{ \begin{array}{ccc} \mathrm{Aut}(\Gamma \mathcal{G}) & \to & \mathrm{Isom}(S\crossing) \\ \varphi & \mapsto & \left( J \mapsto \text{hyperplane whose rotative-stabiliser is $\varphi ( \mathrm{stab}_\circlearrowleft(J))$} \right) \end{array} \right..$$
Moreover, by identifying $\Gamma \mathcal{G}$ with the group of inner automorphisms $\mathrm{Inn}(\Gamma \mathcal{G})$ through the map $g \mapsto \iota(g)$, the action $\mathrm{Inn}(\Gamma \mathcal{G}) \curvearrowright S\crossing$ defined as above extends the action $\Gamma \mathcal{G} \curvearrowright S\crossing$ induced by $\Gamma \mathcal{G} \curvearrowright \mathrm{QM}(\Gamma, \mathcal{G})$. 
\end{lemma}

\begin{proof}
Because the collection of conjugates of maximal vertex-groups coincides with the collection of rotative-stabilisers of maximal hyperplanes, as a consequence of Lemma~\ref{lem:RotativeStab}, it follows that the map
$$\varphi \mapsto \left( J \mapsto \text{hyperplane whose rotative-stabiliser is $\varphi ( \mathrm{stab}_\circlearrowleft(J))$} \right) $$
defines an action of $\mathrm{Aut}(\Gamma \mathcal{G})$ on the vertices of $S\crossing$. It remains to show that the adjacency relation is preserved. This follows from the following algebraic characterisation of transversality between hyperplanes:

\begin{claim}
Let $J_1,J_2$ be two distinct hyperplanes of $\mathrm{QM}(\Gamma, \mathcal{G})$. Then 
$$\langle \mathrm{stab}_\circlearrowleft(J_1) , \mathrm{stab}_\circlearrowleft(J_2) \rangle = \left\{ \begin{array}{cl} \mathrm{stab}_\circlearrowleft(J_1) \oplus \mathrm{stab}_\circlearrowleft(J_2) & \text{if $J_1$ and $J_2$ are transverse} \\ \mathrm{stab}_\circlearrowleft(J_1) \ast \mathrm{stab}_\circlearrowleft(J_2) & \text{otherwise} \end{array}. \right.$$
\end{claim}

\noindent
If $J_1$ and $J_2$ are transverse, then there exist an element $g \in \Gamma \mathcal{G}$ and two adjacent vertices $u,v \in V(\Gamma)$ such that $J_1=gJ_u$ and $J_2=gJ_v$. Consequently,
$$\langle \mathrm{stab}_\circlearrowleft(J_1) , \mathrm{stab}_\circlearrowleft(J_2) \rangle = g \langle G_u,G_v \rangle g^{-1}= gG_ug^{-1} \oplus gG_vg^{-1}= \mathrm{stab}_\circlearrowleft(J_1) \oplus \mathrm{stab}_\circlearrowleft(J_2),$$
as desired. Now, assume that $J_1$ and $J_2$ are not transverse. Let $S_1$ (resp. $S_2$) denote the union of all the sectors delimited by $J_1$ (resp. $J_2$) which are disjoint from $J_2$ (resp. $J_1$), and fix a vertex $x_0 \in \QM$ lying between $J_1$ and $J_2$. As a consequence of Lemma~\ref{lem:RotativeStab}, we know that
\begin{itemize}
	\item $g S_1 \subset S_2$ and $gx_0 \in S_2$ for every $g \in \mathrm{stab}_\circlearrowleft(J_2) \backslash \{1\}$;
	\item $gS_2 \subset S_1$ and $gx_0 \in S_1$ for every $g \in \mathrm{stab}_\circlearrowleft(J_1) \backslash \{1\}$. 
\end{itemize}
We conclude that $\langle \mathrm{stab}_\circlearrowleft(J_1) , \mathrm{stab}_\circlearrowleft(J_2) \rangle= \mathrm{stab}_\circlearrowleft(J_1) \ast \mathrm{stab}_\circlearrowleft(J_2)$ by a ping-pong argument. (For instance, follow the proof of \cite[Proposition III.12.2]{MR1812024} or \cite[Lemma II.B.24]{MR1786869}. Alternatively, \cite[Proposition 3.26]{EmbeddingV} applies directly.)
\end{proof}

\noindent
Now, our goal is to show that the small crossing graph is quasi-isometric to a tree. We begin by observing that the crossing graph is connected as soon as our simplicial graph is also connected. 

\begin{lemma}
Let $\Gamma$ be a connected simplicial graph and $\mathcal{G}$ a collection of groups indexed by $V(\Gamma)$. The crossing graph $\crossing$ is a connected.
\end{lemma}

\begin{proof}
Let $J_1,J_2$ be two hyperplanes of $\QM$. If they are equal or if they are transverse, it is clear that there exists a path between $J_1$ and $J_2$ in the crossing graph. From now on, suppose that $N(J_1)$ and $N(J_2)$ are disjoint. Fix a maximal collection $H_1, \ldots, H_n$ of pairwise non-transverse hyperplanes separating $N(J_1)$ and $N(J_2)$. We index this collection so that $H_{i}$ separates $H_{i-1}$ and $H_{i+1}$ for every $2 \leq i \leq n-1$, and $H_1$ separates $J_1$ and $H_2$. Also, for convenience, we set $H_0=J_1$ and $H_{n+1}=J_2$. Notice that, for every $0 \leq i \leq n$, the carriers $N(H_i)$ and $N(H_{i+1})$ intersect. Indeed, if $N(H_i) \cap N(H_{i+1})= \emptyset$, then according to Lemma \ref{lem:minseppairhyp} a hyperplane separating two vertices minimising the distance between $N(H_i)$ and $N(H_{i+1})$ would separate $N(H_i)$ and $N(H_{i+1})$, so that $J,H_1, \ldots, H_n$ would provide a larger family of pairwise non-transverse hyperplanes separating $J_1$ and $J_2$. Consequently, in order to show that the crossing graph is connected, it is sufficient to show the following observation:

\begin{claim}\label{claim:disttangenthyp}
Let $J_1,J_2$ be two hyperplanes of $\QM$ whose carriers intersect. In the crossing graph of $\QM$, there exists a path of length at most $\mathrm{diam}(\Gamma)$ between $J_1$ and~$J_2$.
\end{claim} 

\noindent
Up to translating by an element of $\Gamma \mathcal{G}$, we may suppose without loss of generality that the vertex $1$ belongs to $N(J_1) \cap N(J_2)$. As a consequence, there exist two vertices $u,v \in V(\Gamma)$ such that $J_1=J_u$ and $J_2=J_v$. Let $a_1, \ldots, a_m$ be a geodesic from $u$ to $v$ in $\Gamma$. (Notice that $m \leq \mathrm{diam}(\Gamma)$.) Then the sequence $J_{a_1},\ldots, J_{a_m}$ defines a path from $J_1$ to $J_2$ in the crossing graph, concluding the proof of our claim.
\end{proof}

\noindent
Our next preliminary lemma shows that the small crossing graph defines a geodesic subgraph in the crossing graph. As a consequence, results about the geometry of the crossing graph will often extend automatically to the small crossing graph.

\begin{lemma}\label{lem:SmallGeodesicIn}
Let $\Gamma$ be a connected simplicial graph and $\mathcal{G}$ a collection of groups indexed by $V(\Gamma)$. The small crossing graph $S\crossing$ is a geodesic subgraph of $\crossing$, i.e., between any two vertices of $\ST$ there exists a geodesic in $\crossing$ lying in $\ST$. In particular, $\ST$ is connected.
\end{lemma}

\noindent
We emphasize that we are not claiming that $\ST$ is a \emph{convex subgraph} of $\crossing$: there may be a geodesic in $\crossing$ between two vertices of $\ST$ that does not lie entirely in $\ST$. 

\begin{proof}[Proof of Lemma \ref{lem:SmallGeodesicIn}.]
Let $A,B$ be two maximal hyperplanes. Fix a geodesic $J_1, \ldots, J_n$ from $A$ to $B$ in the crossing graph. For every $2 \leq i \leq n-1$, write $J_i=g_iJ_{u_i}$ where $g_i \in \Gamma \mathcal{G}$ and $u_i \in V(\Gamma)$. Fix a $\prec$-maximal vertex $u_i^+$ such that $u_i \prec u_i^+$, and set $J_i^+:= g_iJ_{u_i^+}$. By construction, a fiber of $J_i$ (namely $g_i \langle \mathrm{link}(u_i) \rangle$) is contained in $N(J_i^+)= g_i \langle \mathrm{star}(u_i^+) \rangle$, so a hyperplane which is transverse to $J_i$ is either $J_i^+$ or transverse to it. Consider the sequence
$$A=J_1, J_2, \ldots, J_{i-1},J_i^+,J_{i+1}, \ldots, J_{n-1}, J_n=B.$$
Because this path cannot be shortened, necessarily $J_{i-1}$ and $J_{i+1}$ are distinct from $J_{i}^+$. In other words, the previous path is a geodesic. By iterating the argument, $J_1, J_2^+, \ldots, J_{n-1}^+, J_n$ defines a path in the small crossing graph from $A$ to $B$ which is also a geodesic in the crossing graph. The desired conclusion follows. 
\end{proof}

\noindent
The rest of the subsection is dedicated to the proof of the following statement:

\begin{thm}\label{thm:Phyp}
Let $\Gamma$ be a finite and connected simplicial graph, and let $\mathcal{G}$ be a collection of groups indexed by $V(\Gamma)$. The graphs $\crossing$ and $\ST$ are quasi-trees.
\end{thm}

\noindent
Notice that, as a consequence of Lemma \ref{lem:SmallGeodesicIn}, it suffices to prove the theorem for the crossing graph. The first step towards the proof of Theorem \ref{thm:Phyp} is to estimate the metric in the crossing graph. In order to state our estimate, the following terminology is needed: 

\begin{definition}
Let $\Gamma$ be a simplicial graph and $\mathcal{G}$ a collection of groups indexed by $V(\Gamma)$. Two hyperplanes $J_1,J_2$ of $\QM$ are \emph{strongly separated}\index{Strongly separated hyperplanes} if no hyperplane is transverse to both $J_1$ and $J_2$. We denote by $\Delta(J_1,J_2)$ the maximal number of pairwise strongly separated hyperplanes separating $J_1$ and $J_2$.
\end{definition}

\noindent
Our estimate is the following:

\begin{prop}\label{prop:deltaestimate}
Let $\Gamma$ be a finite and connected simplicial graph and $\mathcal{G}$ a collection of groups indexed by $V(\Gamma)$. The inequalities
$$\Delta(A,B) \leq d_{T} (A,B) \leq (4+ \mathrm{diam}(\Gamma)) \cdot (\Delta(A,B)+1)$$
hold for all hyperplanes $A,B \in \crossing$.
\end{prop}

\noindent
Before turning to the proof of Proposition \ref{prop:deltaestimate}, we begin by proving the following elementary but useful observation.

\begin{lemma}\label{lem:Thypseparate}
Let $\Gamma$ be a connected simplicial graph, $\mathcal{G}$ a collection of groups indexed by $V(\Gamma)$ and $A,B,C \in \crossing$ three hyperplanes. If $B$ separates $A$ and $C$, then every path between $A$ and $C$ in $\crossing$ contains a vertex corresponding to a hyperplane transverse to $B$.
\end{lemma}

\begin{proof}
Let $J_1, \ldots, J_n$ be a path from $A$ to $B$. Let $1 \leq i \leq n-1$ be the largest index such that $J_i$ is included in the sector delimited by $B$ which contains $A$. So $J_{i+1}$ intersects the sector delimited by $B$ containing $C$. Moreover, $J_{i+1}$ is transverse to $J_i$, so it also intersects the sector delimited by $B$ containing $A$. Consequently, $J_{i+1}$ must be transverse to $B$, concluding the proof of our lemma. 
\end{proof}

\begin{proof}[Proof of Proposition \ref{prop:deltaestimate}.]
Let $J_1, \ldots, J_r$ be a maximal collection of pairwise strongly separated hyperplanes separating $A$ and $B$. We index this collection such that $J_i$ separates $J_{i-1}$ and $J_{i+1}$ for every $2 \leq i \leq r-1$, and such that $J_1$ separates $A$ and $J_2$. 

\medskip \noindent
If $H_1, \ldots, H_s$ is geodesic between $A$ and $B$ in $\crossing$, we know from Lemma \ref{lem:Thypseparate} that, for every $1 \leq i \leq r$, there exists some $1 \leq n(i) \leq s$ such that $H_{n(i)}$ is transverse to $J_i$. Notice that, for all distinct $1 \leq i,j \leq r$, the hyperplanes $H_{n(i)}$ and $H_{n(j)}$ are necessarily distinct since $J_i$ and $J_j$ are strongly separated. Consequently,
$$d_T(A,B)= s \geq r = \Delta(A,B),$$
proving the first inequality of our lemma. Next, setting $J_0=A$ and $J_{r+1}=B$, notice that
$$d_T(A,B) \leq \sum\limits_{i=0}^r d_T(J_i,J_{i+1}),$$
so it is sufficient to show that $d_T(J_i,J_{i+1}) \leq 4+ \mathrm{diam}(\Gamma)$ for every $0 \leq i \leq r$ in order to deduce the second inequality of our lemma. Fix some $0 \leq i \leq r$. Let $K_1, \ldots, K_q$ be a maximal collection of pairwise non-transverse hyperplanes separating $J_i$ and $J_{i+1}$. We index this collection such that $K_j$ separates $K_{j-1}$ and $K_{j+1}$ for every $2 \leq j \leq q-1$, and such that $K_1$ separates $J_i$ and $K_2$. Also, we set $K_0=J_i$ and $K_{q+1}=J_{i+1}$. Notice that, for every $0 \leq j \leq q$, the maximality of our collection implies that no hyperplane separates $K_j$ and $K_{j+1}$, so these two hyperplanes must be tangent. Let $1 \leq a \leq q$ be the largest index such that $J_i$ and $K_a$ are not strongly separated. Notice that, as a consequence of the maximality of the collection $J_1, \ldots, J_r$, the hyperplanes $K_{a+1}$ and $J_{i+1}$ cannot be strongly separated. So we have
$$d_T(J_i,J_{i+1}) \leq d_T(J_i,K_a)+ d_T(K_a,K_{a+1}) + d_T(K_{a+1},J_{i+1}) \leq d_T(K_a,K_{a+1})+4,$$
and since $K_a$ and $K_{a+1}$ are tangent hyperplanes, we deduce from Claim \ref{claim:disttangenthyp} that
$$d_T(J_i,J_{i+1}) \leq 4+ \mathrm{diam}(\Gamma),$$
concluding the proof of our proposition. 
\end{proof}

\noindent
We need a last preliminary lemma before turning to the proof of Theorem \ref{thm:Phyp}.

\begin{lemma}\label{lem:distfromaseparatinghyp}
Let $\Gamma$ be a finite and connected simplicial graph, $\mathcal{G}$ a collection of groups indexed by $V(\Gamma)$ and $A,B \in \crossing$ two distinct hyperplanes which are neither transverse nor tangent. Every vertex of every geodesic between $A$ and $B$ in $\crossing$ is at distance at most $3(5+ \mathrm{diam}(\Gamma))$ from a hyperplane separating $A$ and $B$. 
\end{lemma}

\begin{proof}
Let $J_1, \ldots, J_n$ be a geodesic in $\crossing$ from $A$ to $B$ (notice that $n \geq 3$), and let $K_1, \ldots, K_m$ be a maximal collection of pairwise strongly separated hyperplanes separating $J_1$ and $J_n$. Suppose that $K_i$ separates $K_{i-1}$ and $K_{i+1}$ for every $2 \leq i \leq m-1$, and that $K_1$ separates $A$ and $K_2$. Fix some $1 \leq k \leq n$. Notice that the maximality of the collection $K_1, \ldots, K_m$ implies that $\Delta(A,K_1) \leq 1$, $\Delta(B,K_m) \leq 1$ and $\Delta(K_i,K_{i+1}) \leq 3$ for every $1 \leq i \leq m-1$. We claim that $J_k$ is at distance at most $3(5+ \mathrm{diam}(\Gamma))$ from some $K_i$.

\medskip \noindent
First of all, notice that $J_k$ cannot be included in a sector delimited by $A$ which does not contain $B$. Otherwise, Lemma \ref{lem:Thypseparate} would imply that there exists some $k+1 \leq i \leq n$ such that $J_i$ is transverse to $A$, and replacing $J_1, \ldots, J_k$ with $J_1,J_k$ would shorten our geodesic $J_1, \ldots, J_n$, which is impossible. Similarly, $J_k$ cannot be included in a sector delimited by $B$ which does not contain $A$. Next, if $J_k$ is equal or transverse to $A$, then once again we deduce from Lemma \ref{lem:Thypseparate} that there exists some $k \leq i \leq n$ such that $J_i$ is transverse to $K_1$, so that
$$\begin{array}{lcl} d_T(J_k,K_1) & \leq & d_T(J_k,J_i) + d_T(J_i,K_1) \leq d_T(A,J_i) +1 \leq d_T(A,K_1)+ d_T(K_1,J_i)+1 \\ \\ & \leq & (4+ \mathrm{diam}(\Gamma)) \cdot \Delta(A,K_1) +2 \leq 6+ \mathrm{diam}(\Gamma) \end{array}$$
where the penultimate inequality is justified by Proposition \ref{prop:deltaestimate}. Similarly, we show that $d_T(J_k,K_m) \leq 6+ \mathrm{diam}(\Gamma)$ if $J_k$ is equal or transverse to $B$. Therefore, from now on, we can suppose that $J_k$ lies in the subspace delimited by $A$ and $B$. If $J_k$ is transverse to some $K_i$, we are done, so we suppose that $J_k$ is disjoint from the $K_i$'s. Setting $K_0=A$ and $K_{m+1}=B$, there exists some $0 \leq i \leq m$ such that $J_k$ is included in the subspace delimited by $K_i$ and $K_{i+1}$. By applying Lemma \ref{lem:Thypseparate} twice, we know that there exists some $0 \leq a \leq k-1$ such that $J_a$ is transverse to $K_i$, and some $k+1 \leq b \leq m+1$ such that $J_b$ is transverse to $K_{i+1}$. Then
$$\begin{array}{lcl} d_T(J_k,K_i) & \leq & d_T(J_k,J_a)+d_T(J_a,K_i) \leq d_T(J_a,J_b) +1 \\ \\ & \leq & d_T(K_i,K_{i+1}) + d_T(J_a,K_i)+d_T(J_b,K_{i+1})+1 \\ \\ & \leq & (4+ \mathrm{diam}(\Gamma)) \cdot \Delta(K_i,K_{i+1}) +3 \leq 3(5+ \mathrm{diam}(\Gamma) ) \end{array}$$
where the penultimate inequality is justified by Proposition \ref{prop:deltaestimate}. Similarly, one shows that $d_T(J_k,K_{i+1}) \leq 3(5+ \mathrm{diam}(\Gamma))$. Thus, we have proved that $J_k$ is at distance at most $3(5+ \mathrm{diam}(\Gamma))$ from $K_j$ for some $1 \leq j \leq m$, concluding the proof of our lemma. 
\end{proof}

\noindent
Now we are ready to prove Theorem \ref{thm:Phyp}. The goal will be to apply the following characterisation of quasi-trees, known as the \emph{bottleneck criterion} \cite{bottleneck} (which we adapt below for graphs):

\begin{prop}
A graph $Y$ is a quasi-tree if and only if there exists a constant $\delta \geq 0$ such that the following holds. For all vertices $x,y \in Y$, there exists a \emph{midpoint} $m$ of $x$ and $y$, i.e., a vertex satisfying $|d(x,m)-d(x,y)/2| \leq 1/2$, such that every path between $x$ and $y$ contains a vertex at distance at most $\delta$ from $m$.
\end{prop}

\begin{proof}[Proof of Theorem \ref{thm:Phyp}.]
Let $A,B \in \crossing$ be two hyperplanes. Fix a geodesic $J_1, \ldots, J_n$ from $A$ to $B$ in $\crossing$, some $1 \leq k \leq n$ such that $J_k$ is a midpoint of $A$ and $B$, and some path $H_1, \ldots, H_m$ from $A$ to $B$ in $\crossing$. Set $\delta = 1+3 (5+ \mathrm{diam}(\Gamma)) $. If $d_T(A,B) \leq \mathrm{diam}(\Gamma)$ then 
$$d_T(H_1,m) = d_T(A,m) \leq d_T(A,B) \leq \mathrm{diam}(\Gamma) \leq \delta$$ 
and we are done. From now on, suppose that $d_T(A,B) > \mathrm{diam}(\Gamma)$. As a consequence of Claim \ref{claim:disttangenthyp}, $A$ and $B$ are neither transverse nor tangent, so that Lemma \ref{lem:distfromaseparatinghyp} applies, i.e., there exists a hyperplane $J$ separating $A$ and $B$ such that $d_T(m,J) \leq 3(5+ \mathrm{diam}(\Gamma))$. Next, we know from Lemma \ref{lem:Thypseparate} that there exists some $1 \leq j \leq m$ such that $H_j$ is transverse to $J$. Therefore,
$$d_T(H_j,m) \leq d_T(H_j,J)+ d_T(J,m) \leq 1+3 (5+ \mathrm{diam}(\Gamma)) = \delta.$$
Consequently, the bottleneck criterion applies, proving that the crossing graph $\crossing$ is a quasi-tree. The same conclusion for $\ST$ follows from Lemma \ref{lem:SmallGeodesicIn}.
\end{proof}

\subsection{Straight geodesics in the crossing graph}\label{section:Normal}

\noindent
In this subsection, we define \emph{straight geodesics} in the crossing graph and show a few preliminary results about them. Such geodesics will play a fundamental role in Subsection~\ref{section:WPD}.
\begin{figure}
\begin{center}
\includegraphics[scale=0.4]{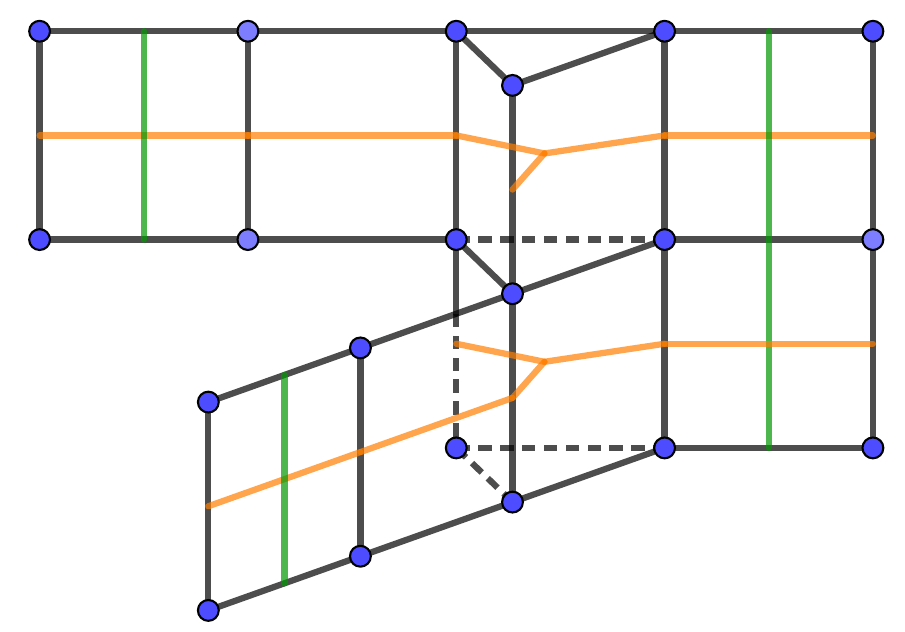}
\caption{A geodesic in the crossing graph which is not straight.}
\label{NotStraight}
\end{center}
\end{figure}

\begin{definition}
Let $\Gamma$ be a simplicial graph and $\mathcal{G}$ a collection of groups indexed by $V(\Gamma)$. A geodesic $J_1, \ldots, J_n$ of length $\geq 3$ in $\crossing$ is \emph{straight}\index{Straight geodesics} if, for every $2 \leq i \leq n-1$, there does not exist a hyperplane $J$ transverse to both $J_{i-1}$ and $J_{i+1}$ which separates $N(J_i)$ from $\mathrm{proj}_{N(J_1)}N(J_n)$ and from $\mathrm{proj}_{N(J_n)} N(J_1)$.
\end{definition}

\noindent
For instance, the geodesic illustrated by Figure \ref{NotStraight} is not straight. We will see later that straight geodesics always exist (see Proposition \ref{prop:NormalExist} below). But first, we show how to associate a geodesic in $\QM$ to each straight geodesic in $\crossing$. 

\begin{definition}
Let $\Gamma$ be a simplicial graph, $\mathcal{G}$ a collection of groups indexed by $V(\Gamma)$, and $J_1, \ldots, J_n$ a geodesic of length $\geq 3$ in $\crossing$. Define by induction a sequence of vertices $x_1, \ldots, x_n \in \QM$ by:
$$\left\{ \begin{array}{l} \text{$x_1$ is the projection of $N(J_n)$ onto $N(J_1)$} \\ \text{for all $i \leq n-1$, $x_{i+1}= \mathrm{proj}_{N(J_{i+1})} (x_i)$} \end{array} \right..$$
A \emph{straight path}\index{Straight path} associated to $J_1, \ldots, J_n$ is a path of the form $[x_1,x_2] \cup \cdots \cup [x_{n-1},x_n]$, where, for every $1 \leq i \leq n-1$, $[x_i,x_{i+1}]$ is a geodesic in $\QM$ between $x_i$ and $x_{i+1}$.
\end{definition}
\begin{figure}
\begin{center}
\includegraphics[scale=0.4]{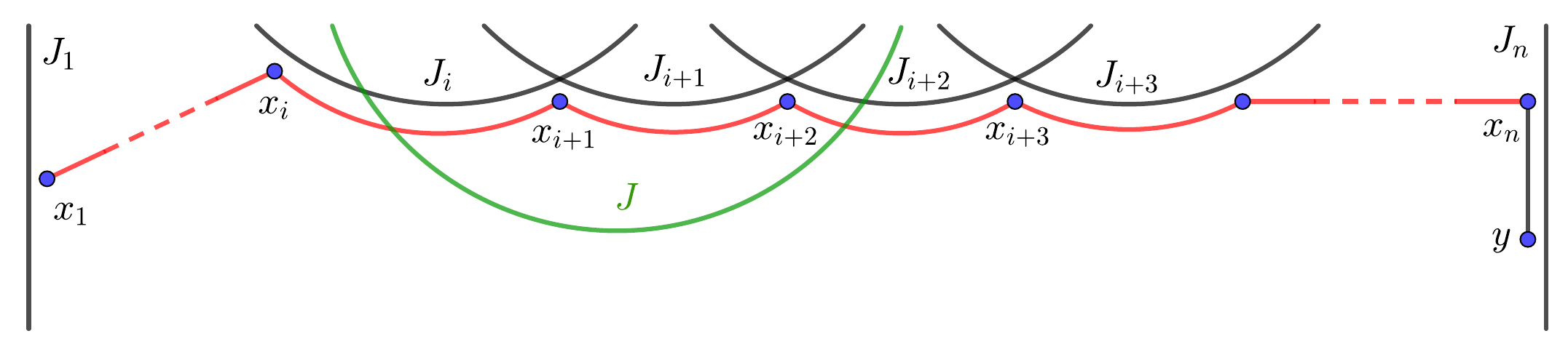}
\caption{Configuration from the proof of Lemma \ref{lem:straightgeo}.}
\label{straightgeo}
\end{center}
\end{figure}

\begin{lemma}\label{lem:straightgeo}
Let $\Gamma$ be a simplicial graph and $\mathcal{G}$ a collection of groups indexed by $V(\Gamma)$. A straight path associated to a straight geodesic $J_1, \ldots, J_n$ in $\crossing$ defines a geodesic in $\QM$.
\end{lemma}

\begin{proof}
Let $x_1, \ldots, x_n$ denote the sequence of vertices of $\QM$ associated to $J_1, \ldots, J_n$. Fix a straight path $[x_1,x_2] \cup \cdots \cup [x_{n-1},x_n]$ and let $J$ be a hyperplane which separates $x_i$ and $x_{i+1}$ for some $1 \leq n-1$. Assume that it does not cross $[x_1,x_2] \cup \cdots \cup [x_{i-1},x_i]$. See Figure \ref{straightgeo}. Notice that $J$ is not transverse to $J_{i+1}$ according to Lemma \ref{lem:proj}, so $J$ does not separate $x_{i+1}$ and $x_{i+2}$. Also, if $J$ separates $x_j$ and $x_{j+1}$ for some $j \geq i+3$, then
$$J_1, \ldots, J_i, J, J_j, \ldots, J_n$$
would define a path of length $n-(j-1-i) \leq n-2< n-1=d_T(J_1,J_n)$ between $J_1$ and $J_n$, contradicting the fact that $J_1, \ldots, J_n$ is a geodesic. Finally, assume that $J$ separates $x_{i+2}$ and $x_{i+3}$ (which imposes that $i \leq n-3$). 
Notice that, because $J$ crosses $[x_1,x_2] \cup \cdots \cup [x_i,x_{i+1}]$ only once and because we know that $J$ is not transverse to $J_{i+1}$, necessarily $J$ separates $N(J_{i+1})$ from $x_1$. Since our geodesic is straight, it follows that $J$ cannot separate $N(J_{i+1})$ from $y:= \mathrm{proj}_{N(J_n)}(J_1)$. Consequently, $J$ cannot cross the path $[x_{i+2},x_{i+3}] \cup \cdots \cup [x_{n-1},x_n] \cup [x_n,y]$ from $x_{i+2} \in N(J_{i+1})$ to $y$ only once, where $[x_n,y]$ is geodesic we fix. We saw that $J$ does not cross $[x_{i+3},x_{i+4}] \cup \cdots \cup [x_{n-1},x_n]$, so $J$ separates $x_n$ and $y$. In particular, $J$ is transverse to $J_n$, so the path $J_1, \ldots, J_i, J, J_n$ shortens our geodesic since $i \leq n-3$. 
\end{proof}

\begin{remark}
Notice that, if $x_1, \ldots, x_n$ is the sequence of vertices of $\QM$ associated to a straight geodesic $J_1, \ldots, J_n$ in $\crossing$, then necessarily $x_1=x_2$. Indeed, by definition $x_1$ is the unique vertex of $N(J_1)$ minimising the distance to $N(J_n)$, and we know from Lemma \ref{lem:straightgeo} that any straight path $[x_1,x_2] \cup \cdots \cup [x_{n-1},x_n]$ is a geodesic, hence $d(x_1,N(J_n)) = d(x_1,x_2)+d(x_2,N(J_n))$. As $x_2$ also belongs to $N(J_1)$ as a consequence of Lemma \ref{lem:projInter}, necessarily $x_1=x_2$.
\end{remark}

\noindent
Now we are ready to state the main result of this subsection, which shows that every geodesic in the crossing graph can be straightened in a specific way. 

\begin{prop}\label{prop:NormalExist}
Let $\Gamma$ be a simplicial graph, $\mathcal{G}$ a collection of groups indexed by $V(\Gamma)$, and $J_1, \ldots, J_n$ a geodesic of length $ \geq 3$ in $\crossing$. There exists a straight geodesic $H_1=J_1, H_2, \dots, H_{n-1}, H_n=J_n$ such that, for every $2 \leq i \leq n-1$, there exist an element $g_i \in \Gamma \mathcal{G}$ and a maximal join $\Lambda_i \subset \Gamma$ such that the subgraph $g_i \langle \Lambda_i \rangle \subset \QM$ intersects every straight path associated to $H_1, \ldots, H_n$ and is crossed by $J_i$ and $H_i$; in particular, $g_i \langle \Lambda_i \rangle g_i^{-1}$ is a maximal product subgroup containing the rotative-stabilisers of both $J_i$ and $H_i$.
\end{prop}

\noindent
We begin by proving a preliminary result which will be useful in the proof of the proposition.

\begin{lemma}\label{lem:straight}
Let $\Gamma$ be a simplicial graph, $\mathcal{G}$ a collection of groups indexed by $V(\Gamma)$, and $J_1, \ldots, J_n$ a geodesic in $\crossing$. Fix an index $3 \leq i \leq n-1$ and assume that there exists a hyperplane $K$ which separates $N(J_i)$ from $\mathrm{proj}_{N(J_1)} N(J_n)$ and which is transverse to both $J_{i-1}$ and $J_{i+1}$. Then $K$ separates $N(J_i)$ from $N(J_{i-2})$. \\ 
Similarly, fix an index $2 \leq i \leq n-2$ and assume that there exists a hyperplane $K$ which separates $N(J_i)$ from $\mathrm{proj}_{N(J_n)} N(J_1)$ and which is transverse to both $J_{i-1}$ and $J_{i+1}$. Then $K$ separates $N(J_i)$ from $N(J_{i+2})$. 
\end{lemma}

\begin{proof}
For convenience, let $x$ denote the unique vertex of $\mathrm{proj}_{N(J_1)} N(J_n)$. Let $y$ denote the projection of $x$ onto $N(J_{i+1})$, let $z$ be a vertex of $N(J_i) \cap N(J_{i+1})$ such that $J$ separates $y$ and $z$, let $a \in N(J_{i-2})$ and $b \in N(J_i)$ be two vertices minimising the distance between $N(J_{i-2})$ and $N(J_i)$, and let $\gamma \subset N(J_1) \cup  \cdots \cup N(J_{i-2})$ be a path from $x$ to $a$. Finally, fix four geodesics $[a,b]$, $[b,z]$, $[y,z]$ and $[x,y]$. See Figure \ref{lemstraight}.
\begin{figure}
\begin{center}
\includegraphics[scale=0.4]{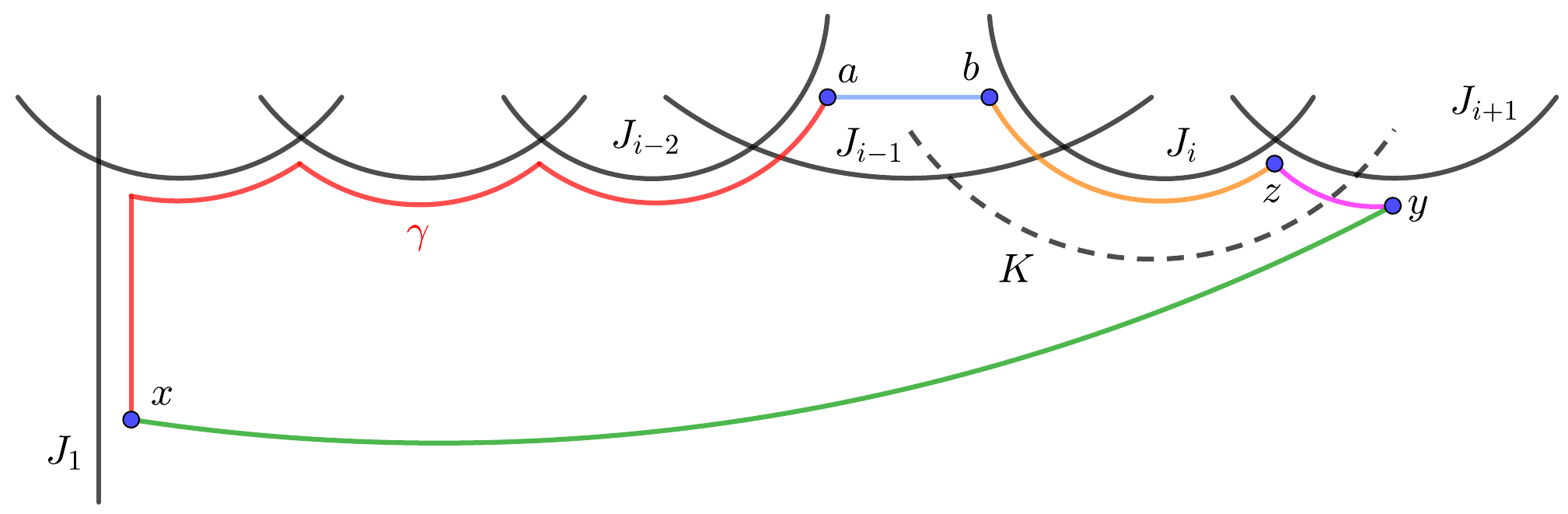}
\caption{Configuration from the proof of Lemma \ref{lem:straight}.}
\label{lemstraight}
\end{center}
\end{figure}

\medskip \noindent
By assumption, $K$ is not transverse to $J_i$ so it cannot cross $[b,z]$; and it is transverse to $J_{i+1}$ so it follows from Lemma \ref{lem:proj} that it cannot cross $[x,y]$ either. Next, notice that, if $K$ crosses $\gamma$, then it has to be transverse to $J_k$ for some $1 \leq i \leq i-2$. But then
$$J_1, \ldots, J_k, J, J_{i+1}, \ldots, J_n$$
would shorten our geodesic $J_1, \ldots, J_n$, which is impossible. So far, we know that, along the loop $\gamma \cup [a,b] \cup [b,z] \cup [z,y] \cup [y,x]$, $K$ crosses $[z,y]$ but neither $[y,x] \cup \gamma$ nor $[b,z]$. Necessarily, $K$ has to cross $[a,b]$. We conclude from Lemma \ref{lem:minseppairhyp} that $K$ separates $N(J_i)$ and $N(J_{i-2})$ as desired. 

\medskip \noindent
The second assertion in our lemma follows from a symmetric argument.
\end{proof}

\begin{proof}[Proof of Proposition \ref{prop:NormalExist}.]
Let $x_1$ denote the projection of $N(J_n)$ onto $N(J_1)$. Fix an index $2 \leq i \leq n-1$. Let $\mathcal{H}_i$ denote the collection of all the hyperplanes $H$ transverse to both $J_{i-1}$ and $J_{i+1}$ which separate $N(J_i)$ from $\mathrm{proj}_{N(J_1)}N(J_n)$ and from $\mathrm{proj}_{N(J_n)}N(J_1)$ such that the distance from $x_1$ to $N(H)$ is minimal. If $\mathcal{H}_i= \emptyset$ we set $H_i:=J_i$, and otherwise we fix an arbitrary $H_i \in \mathcal{H}_i$. Also, we set $H_1:=J_1$ and $H_n:=J_n$. 
\begin{figure}
\begin{center}
\includegraphics[scale=0.35]{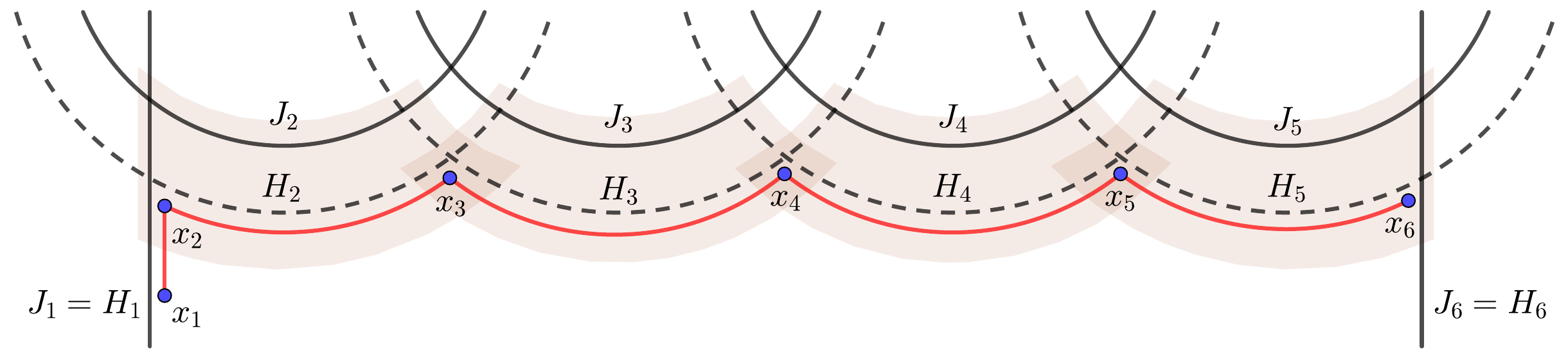}
\caption{Straightening a geodesic in the crossing graph.}
\label{straightening}
\end{center}
\end{figure}

\medskip \noindent
First of all, we claim that $H_1, \ldots, H_n$ defines a path in $\crossing$, i.e., $H_i$ is transverse to $H_{i+1}$ for every $1 \leq i \leq n-1$. Clearly, $H_2$ is transverse to $J_1=H_1$ and $H_{n-1}$ is transverse to $J_n=H_n$. Now fix an index $2 \leq i \leq n-2$. If $H_{i+1}=J_{i+1}$ then $H_i$ is transverse to $H_{i+1}$ by construction. Otherwise, $H_{i+1}$ separates $N(J_{i+1})$ and $N(J_{i-1})$ according to Lemma \ref{lem:straight}. Because $H_i$ is transverse to both $J_{i-1}$ and $J_{i+1}$, it follows that $H_i$ and $H_{i+1}$ are transverse. Therefore, $H_1, \ldots, H_n$ is indeed a path in $\crossing$, and, because it has the same length as $J_1, \ldots, J_n$, it must be a geodesic.

\medskip \noindent
Next, we claim that $H_1, \ldots, H_n$ is straight. Assume that there exist an index $2 \leq i \leq n-1$ and a hyperplane $K$ transverse to both $H_{i-1}$ and $H_{i+1}$ which separates $N(H_i)$ from $\mathrm{proj}_{N(H_1)}N(H_n)$ and from $\mathrm{proj}_{N(H_n)}N(H_1)$. 

\medskip \noindent
If $3 \leq i \leq n-2$, then it follows from Lemma \ref{lem:straight} that $K$ separates $N(H_i)$ from $N(H_{i-2})$ and $N(H_{i+2})$. But $H_j$ is transverse to both $J_{j-1}$ and $J_{j+1}$ for every $2 \leq j \leq n-1$, $H_1=J_1$ is transverse to $J_2$ and $H_n=J_n$ is transverse to $J_{n-1}$, so $J_j$ is transverse to both $H_{j-1}$ and $H_{j+1}$ for every $2 \leq j \leq n-1$. This implies that $J_{i-1}$ is transverse to both $H_{i-2}$ and $H_{i}$, and that $J_{i+1}$ is transverse to both $H_i$ and $H_{i+2}$. Consequently, $K$ has to be transverse to both $J_{i-1}$ and $J_{i+1}$. Thus, we have proved that $K$ separates $N(H_i)$, and so $N(J_i)$, from $\mathrm{proj}_{N(H_1)}N(H_n) = \mathrm{proj}_{N(J_1)}N(J_n)$ and from $\mathrm{proj}_{N(H_n)}N(H_1)= \mathrm{proj}_{N(J_n)}N(J_1)$, contradicting the choice of $H_i$. 

\medskip \noindent
Similarly, we find contradictions if $i=2$ and $i=n-1$. More precisely, if $i=2$, then $K$ separates $H_2$ and $H_4$ according to Lemma \ref{lem:straight}, and, as $J_3$ is transverse to both $H_2$ and $H_4$, it follows that $K$ is transverse to $J_3$. Moreover, $K$ is transverse to $H_1=J_1$. Therefore, $K$ is transverse to both $J_1$ and $J_3$ and it separates $N(H_2)$, and so $N(J_2)$, from $\mathrm{proj}_{N(J_1)}N(J_n)$ and from $\mathrm{proj}_{N(J_n)}N(J_1)$, contradicting the choice of $H_2$. If $i=n-1$, then $K$ separates $H_{n-1}$ and $H_{n-3}$, so it has to be transverse to $J_{n-2}$ because $J_{n-2}$ is transverse to both $H_{n-3}$ and $H_{n-1}$. Moreover, $K$ is transverse to $H_n=J_n$. Therefore, $K$ is transverse to both $J_{n-2}$ and $J_n$ and it separates $N(H_{n-1})$, and so $N(J_{n-1})$, from $\mathrm{proj}_{N(J_1)}N(J_n)$ and from $\mathrm{proj}_{N(J_n)}N(J_1)$, contradicting the choice of $H_{n-1}$.

\medskip \noindent
Thus, we have proved that $H_1, \ldots, H_n$ is a straight geodesic in $\crossing$.

\medskip \noindent
Now, let $x_1, \ldots, x_n$ be the sequence of vertices in $\QM$ associated to $H_1, \ldots, H_n$. We fix a straight geodesic $[x_1,x_2] \cup \cdots \cup [x_{n-1},x_n]$. For every $2 \leq i \leq n-1$, let $B_i$ denote the carrier $N(J_i)$ if $H_i=J_i$ and the bridge between $N(H_{i-1})$ and $N(H_{i+1})$ otherwise.

\begin{claim}\label{claim:straightproduct}
For every $2 \leq i \leq n-1$, $x_{i+1}$ belongs to $B_i$.
\end{claim}

\noindent
Fix an index $2 \leq i \leq n-1$. Assume first that $H_i=J_i$. So we know that $x_i \in N(H_i)=N(J_i)$. Because $x_{i+1}$ is the projection of $x_i$ onto $N(H_{i+1})$ and that $H_{i+1}$ is transverse to $H_i=J_i$, it follows from Lemma \ref{lem:projInter} that $x_{i+1} \in N(H_i) \cap N(H_{i+1}) \subset B_i$. Next, assume that $H_i \neq J_i$. We know that $x_{i-1} \in N(H_{i-1})$; and that $x_i \in N(H_{i-1}) \cap N(H_i)$, as implied by the fact that $x_i$ is the projection $x_{i-1}$ onto $N(H_i)$ and by Lemma \ref{lem:projInter}. Because $x_{i+1}$ is the projection of $x_i$ onto $N(H_{i+1})$, it follows from Lemma \ref{lem:bridge} that $x_{i+1} \in \mathrm{proj}_{N(H_{i+1})} N(H_{i-1}) \subset B_i$ as desired. Thus, our claim is proved.

\medskip \noindent
Fix an index $2 \leq i \leq n-1$. Notice that there exist an element $g_i \in \Gamma \mathcal{G}$ and a join $\Lambda_i \subset \Gamma_i$ such that $B_i \subset g_i \langle \Lambda_i \rangle$. Indeed, either $B_i$ is the carrier of $J_i$, and the conclusion follows from Theorem \ref{thm:HypStab}; or $B_i \neq N(J_i)$ and $N(H_{i-1}) \cap N(H_{i+1}) = \emptyset$, and the conclusion follows from Lemma \ref{lem:BridgeInProduct}; or $B_i \neq N(J_i)$ and $N(H_{i-1}) \cap N( H_{i+1}) \neq \emptyset$, and we know that $B_i= N(H_{i-1}) \cap N(H_{i+1})$ lies in the carrier of a hyperplane so the conclusion follows from Theorem \ref{thm:HypStab}. Without loss of generality, we suppose that $\Lambda_i$ is a maximal join in $\Gamma_i$. According to Claim \ref{claim:straightproduct}, $x_i$ belongs to $g_i \langle \Lambda_i \rangle$ so $g_i \langle \Lambda_i \rangle$ must intersect all the straight paths associated to $H_1, \ldots, H_n$. Moreover, $J_i$ and $H_i$ both cross $B_i$ by construction.
\end{proof}

\subsection{Irreducible elements}\label{section:Irreducible}

\noindent
In this subsection, our goal is to exhibit hyperbolic behaviors of irreducible elements in graph products. These preliminary results will be fundamental in Subsection \ref{section:WPD}. Our first proposition essentially shows that an irreducible element $g \in \Gamma \mathcal{G}$ admits an axis in $\QM$ which is \emph{contracting} with respect to the $\delta$-metric. 

\begin{prop}\label{prop:ContractingAxis}
Let $\Gamma$ be a simplicial graph, $\mathcal{G}$ a collection of finitely generated groups indexed by $V(\Gamma)$, and $g$ an irreducible element which is graphically cyclically reduced. Fix a geodesic $[1,g]$ between the vertices $1,g \in \QM$ and a hyperplane $J$ which separates $1$ and $g$. Then 
$$\bigcup\limits_{n \in \mathbb{Z}} g^n [1,g] = \cdots \cup g^{-1} [1,g] \cup [1,g] \cup g [1,g] \cup \cdots$$ 
is a geodesic and $\{g^{2Dk}J \mid k \in \mathbb{Z}\}$ is a collection of pairwise strongly separated hyperplanes, where $D$ denotes the diameter of $\mathrm{supp}(g)^{\mathrm{opp}}$. \\
Moreover, for every indices $i,j \in \mathbb{Z}$ satisfying $j \geq i + 3$ and for every geodesic $[x,y]$ between two vertices $x,y$ separated by $g^{2Di}J$ and $g^{2Dj}J$, the subsegment of $[x,y]$ delimited by $g^{2D(i+1)}J$ and $g^{2D(j-1)}J$ lies in the $6D|g|$-neighborhood of $\bigcup\limits_{2D(i+1) \leq k \leq 2D(j-1)} g^k[1,g]$ with respect to the $\delta$-metric.
\end{prop}

\begin{proof}
First of all, we claim that the product $g \cdot g$ is graphically reduced. Indeed, write $g$ as a graphically reduced product of syllables $g_1 \cdots g_n$; notice that $n \geq 2$ because $g$ does not belong to the conjugate of a vertex-group. If the product $g \cdot g$ is not graphically reduced, then there exist a syllable $g_i$ in the first copy of $g$ which shuffles to the end in $g_1 \cdots g_n$ and a syllable $g_j$ in the second copy of $g$ which shuffles to the beginning in $g_1 \cdots g_n$ such that $g_i$ and $g_j$ belong to the same vertex-group, say $G_u$. As $g$ is graphically cyclically reduced, necessarily $i=j$. This implies that the syllable $g_i=g_j$ commutes with all the other syllables of $g_1 \cdots g_n$, or equivalently that $g_k \in \langle \mathrm{link}(u) \rangle$ for every $k \neq i$. Therefore, $g$ must belong to the join-subgroup $\langle \mathrm{star}(u) \rangle$, which is impossible, concluding the proof of our claim.

\medskip \noindent
As a consequence of the previous observation and of Lemma \ref{lem:geodesicsQM}, it follows that the concatenation 
$$\ell := \bigcup\limits_{n \in \mathbb{Z}} g^n \cdot [1,g]$$
is a bi-infinite geodesic on which $g$ acts as a translation.

\begin{claim}
Every hyperplane crossing $[1,g]$ separates $g^{-D}J$ and $g^DJ$. 
\end{claim}

\noindent
Let $H$ be a hyperplane crossing $[1,g]$, and let $a,b \in V(\Gamma)$ denote the vertices labelling $J$ and $H$ respectively. Fix a path $u_0, \ldots, u_D$ in $\mathrm{supp}(g)^{\mathrm{opp}}$ from $b$ to $a$. (Here, the equality $u_i=u_{i+1}$ is allowed for some indices $i$.) For every $0 \leq i \leq D$, fix a hyperplane $J_i$ labelled by $u_i$ which crosses $g^i[1,g]$; for $i=0$ and $i=D$, we set $J_0=H$ and $J_D=g^DJ$. Because two transverse hyperplanes must be labelled by adjacent (and so distinct) vertices of $\Gamma$ according to Lemma \ref{lem:transverseimpliesadj}, it follows that $J_i$ separates $J_{i-1}$ and $J_{i+1}$ for every $1 \leq i \leq D-1$. We conclude that $J_0=H$ is not transverse to $J_D=g^DJ$. Similarly, one shows that $H$ cannot be transverse to $g^{-D}J$. Therefore, $H$ separates $g^{-D}J$ and $g^DJ$, proving our claim.

\begin{claim}
The hyperplanes $g^{-D}J$ and $g^DJ$ are strongly separated.
\end{claim}

\noindent
Suppose by contradiction that there exists a hyperplane $H$ transverse to both $g^{-D}J$ and $g^DJ$. A fortiori, $H$ must be transverse to the hyperplanes separating $g^{-D}J$ and $g^DJ$, including all the hyperplanes crossing $[1,g]$ according to the previous claim. Therefore, as a consequence of Lemma \ref{lem:transverseimpliesadj}, the vertex labelling $H$, say $w$, is adjacent to all the vertices of $\mathrm{supp}(g)$. This implies that $\mathrm{supp}(g)$ is included in a join, namely the star of $w$, which is impossible since $g$ does not belong to a join-subgroup. This concludes the proof of our claim.

\medskip \noindent
The conclusion is that $\{g^{2Dk}J \mid k \in \mathbb{Z}\}$ is a collection of pairwise strongly separated hyperplanes such that $g^{2Dk}J$ separates $g^{2D(k-1)}J$ and $g^{2D(k+1)}J$ for every $k \in\mathbb{Z}$.
\begin{figure}
\begin{center}
\includegraphics[scale=0.35]{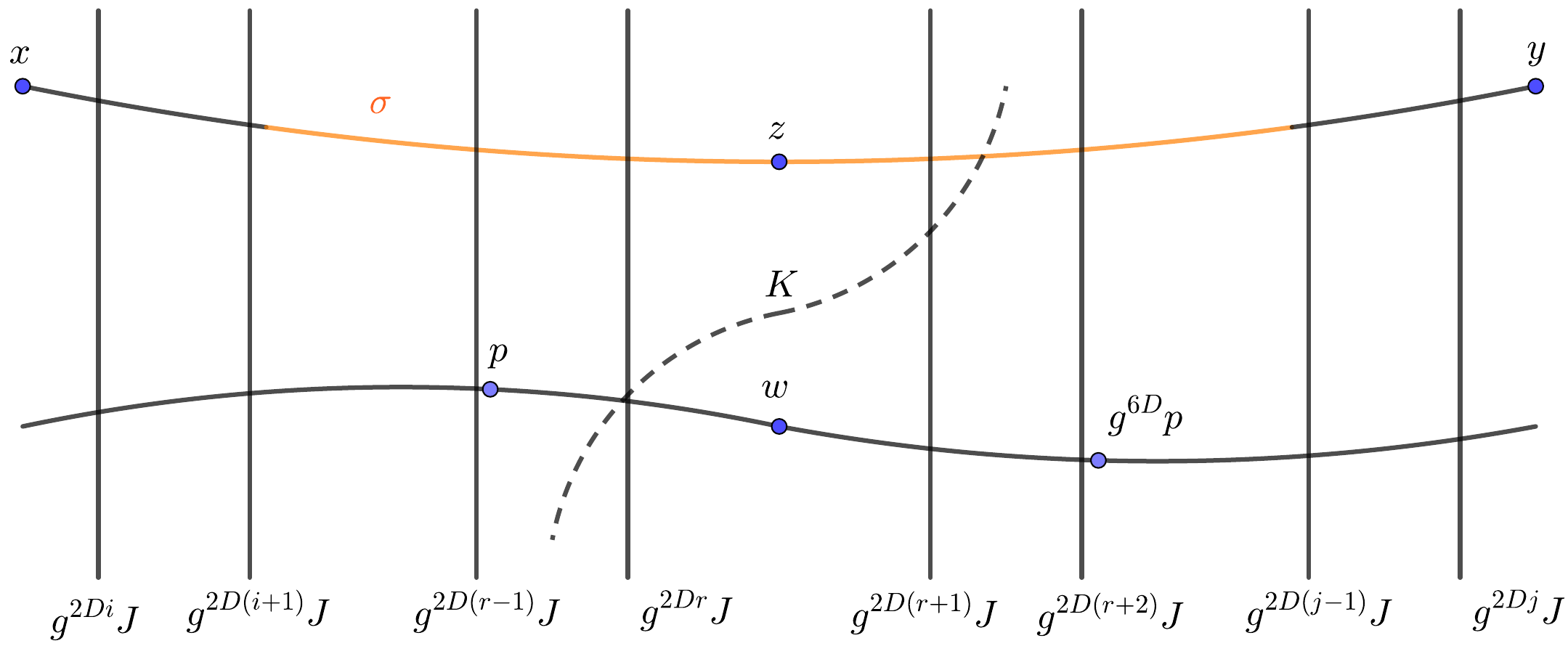}
\caption{Configuration from the proof of Proposition \ref{prop:ContractingAxis}.}
\label{contractingaxis}
\end{center}
\end{figure} 

\medskip \noindent
Now, fix two indices $i,j \in \mathbb{Z}$ satisfying $j \geq i+3$. Let $x,y$ be two vertices separated by $g^{2Di}J$ and $g^{2Dj}J$, and $[x,y]$ a geodesic connecting them. Let $\sigma$ denote the subsegment of $[x,y]$ delimited by $g^{2D(i+1)}J$ and $g^{2D(j-1)}J$. We fix a vertex $z \in \sigma$, and we denote by $r$ the index such that $z$ lies between $g^{2Dr}J$ and $g^{2D(r+1)}J$. See Figure \ref{contractingaxis}. Fixing an arbitrary vertex $w \in \ell$ which lies between $g^{2Dr}J$ and $g^{2D(r+1)}J$, we claim that $\delta(z,w) \leq 6D|g|$. So let $K$ be a hyperplane separating $z$ and $w$. Notice that $K$ must lie between $g^{2D(r-1)}J$ and $g^{2D(r+2)}J$, since otherwise it would be transverse to both $g^{2Dr}J$ and $g^{2D(r-1)}J$ or to both $g^{2D(r+1)}J$ and $g^{2D(r+2)}J$. As a consequence, $K$ crosses $\ell$ between $g^{2D(r-1)}J$ and $g^{2D(r+2)}J$. Therefore, if we fix a vertex $p \in \ell \cap N(g^{2D(r-1)}J)$, then we have
$$\delta(z,w) = \sum\limits_{\text{$H$ hyperplane}} \delta_H(z,w) \leq \sum\limits_{\text{$H$ hyperplane}} \delta_H(p, g^{6D}p)= \delta(p, g^{6D}p) = 6D |g|$$
as desired. Therefore, $\sigma$ is contained in the $6D|g|$-neighborhood of the subsegment $\bigcup\limits_{2D(i+1) \leq k \leq 2D(j-1)} g^{k}[1,g]$ of $\ell$. 
\end{proof}

\noindent
We record the following consequence of Proposition \ref{prop:ContractingAxis} for future use:

\begin{cor}\label{cor:SkewerContracting}
Let $\Gamma$ be a simplicial graph, $\mathcal{G}$ a collection of finitely generated groups indexed by $V(\Gamma)$, and $g$ an irreducible element. For every hyperplane $H$ in $\QM$, $H$ and $g^{6D}H$ are strongly separated, where $D$ denotes the diameter of $\mathrm{supp}(g)$ in $\Gamma^\mathrm{opp}$. 
\end{cor}

\begin{proof}
As a consequence of Proposition \ref{prop:ContractingAxis}, there exists a hyperplane $J$ such that $\{g^{2Dk} J \mid k \in \mathbb{Z}\}$ is a collection of pairwise strongly separated hyperplanes and such that $g^{2Di}J$ separates $g^{2D(i-1)}J$ and $g^{2D(i+1)}J$ for every $i \in \mathbb{Z}$. Fix an index $i \in \mathbb{Z}$ such that $H$ lies between $g^{2Di}J$ and $g^{2D(i+2)}J$. Then $g^{2D(i+2)}J$ and $g^{2D(i+3)}J$ separate $H$ and $g^{6D}H$. The fact that $g^{2D(i+2)}J$ and $g^{2D(i+3)}J$ are strongly separated then implies that $H$ and $g^{6D}H$ must be strongly separated as well.
\end{proof}

\noindent
Our second proposition motivating that irreducible elements are ``hyperbolic'' shows that, with respect to the action of a graph product on its crossing graph, irreducible elements are WPD.

\begin{prop}\label{prop:IrreducibleWPD}
Let $\Gamma$ be a finite and connected simplicial graph and $\mathcal{G}$ a collection of groups indexed by $V(\Gamma)$. Every irreducible element in $\Gamma \mathcal{G}$ is WPD with respect to the action of $\Gamma \mathcal{G}$ on its crossing graph $\crossing$. 
\end{prop}

\begin{proof}
According to Proposition \ref{prop:ContractingAxis}, there exists a hyperplane $J$ such that $\{g^{2Dk}J \mid k \in \mathbb{Z} \}$ is a collection of pairwise strongly separated hyperplanes and such that $g^{2Dk}J$ separates $g^{2D(k-1)}J$ and $g^{2D(k+1)}J$ for every $k \in \mathbb{Z}$, where $D$ denotes the diameter of $\mathrm{supp}(g)^{\mathrm{opp}}$. Fix a hyperplane $H$ and a constant $\epsilon \geq 0$, and set
$$N = 8(\epsilon +4\delta +3)$$
where $\delta$ is the hyperbolicity constant of $\crossing$. Let $i \in \mathbb{Z}$ be an index such that $H$ lies between $g^{2Di}J$ and $g^{2D(i+2)}J$. For convenience, we set $J_0=g^{2Di}J$ and $J_k = g^{4Dk}J_0$ for every $k \in \mathbb{Z}$. Fix four indices $0 < a < b < c < d < N$ such that
\begin{itemize}
	\item $a \geq \epsilon+1$ and $N-d \geq \epsilon$;
	\item $b-a,d-c \geq \epsilon+8 \delta+4$;
	\item and $c-b \geq 3$.
\end{itemize}
Notice that such indices exist since we chose $N$ sufficiently large. Our goal is to show that
$$F:= \{ h \in \Gamma \mathcal{G} \mid d_T(H,hH), d_T(g^{4DN}H,hg^{4DN}H) \leq \epsilon \}$$
is a finite set. For convenience, for every $k \in \mathbb{Z}$, we denote by $J_k^-$ (resp. $J_k^+$) the sector delimited by $J_k$ which contains $J_{k-1}$ (resp. $J_{k+1}$). 

\begin{claim}\label{claim:one}
For every $h \in F$ and every hyperplane $K$ separating $J_b$ and $J_c$, the inequality $d_T(K,hK) \leq \epsilon+8\delta+2$ holds.
\end{claim}

\noindent
Let $H_1, \ldots, H_p$ be a geodesic in $\crossing$ between $H$ and $g^{4DN}H$. As a consequence of Lemma \ref{lem:Thypseparate}, there exists an index $1 \leq r \leq p$ such that $H_r$ is transverse to $K$. But, because the metric in a $\delta$-hyperbolic space is $8\delta$-convex \cite[Corollaire 5.3]{CDP}, we know that 
$$d_T(H_r, hH_r) \leq \max(d_T(H_1,hH_1), d_T(H_p,hH_p))+8\delta \leq \epsilon+8\delta,$$
hence $d_T(K,hK) \leq 2d_T(K,H_r)+ d_T(H_r,hH_r) \leq \epsilon+8\delta+2,$ as desired.

\begin{claim}\label{claim:two}
For every $h \in F$ and every hyperplane $K$ separating $J_b$ and $J_c$, $hK$ separates $J_a$ and $J_d$. 
\end{claim}

\noindent
If $hK$ does not lie in $J_{a}^+$, then it intersects $(J_a^+)^c \subset J_{a+1}^-$. As $J_{a+1}$ and $J_{a+2}$ are strongly separated, it follows that $hK$ lies in $J_{a+2}^-$. So $J_{a+2}, \ldots, J_b$ separate $hK$ and $K$, and we deduce from Proposition \ref{prop:deltaestimate} that
$$d_T(K,hK) \geq \Delta(K,hK) \geq b-a-1> \epsilon+8\delta+2,$$
contradicting Claim \ref{claim:one}. Therefore, $hK$ must lie in $J_a^+$. One shows by a symmetric argument that $hK$ lies in $J_d^-$. So $hK$ lies between $J_a$ and $J_d$. Because $K$ separates $H$ and $g^{4DN}H$, it follows that if $hK$ does not separate $J_a$ and $J_d$ then either $H$ or $g^{4DN}H$ is sent in a sector delimited by $hK$ which does not contain $J_a$ and $J_d$, so either
$$d_T(H,hH) \geq \Delta (H,hH) \geq \# \{ J_1, \ldots, J_a\} = a > \epsilon$$
or 
$$d_T \left( g^{4DN}H,hg^{4DN}H \right) \geq \Delta \left( g^{4DN}H,hg^{4DN}H \right) \geq \# \{ J_d, \ldots, J_N\} = N-d+1>\epsilon,$$
contradicting the fact that $h \in F$. Consequently, $hK$ has to separate $J_a$ and $J_d$, concluding the proof of our claim.

\medskip \noindent
As a consequence of Claim \ref{claim:two}, every element of $F$ defines a map from the collection $\mathcal{W}(J_b,J_c)$ of the hyperplanes separating $J_b$ and $J_c$ to the collection $\mathcal{W}(J_a,J_d)$ of the hyperplanes separating $J_a$ and $J_d$. Because these two collections are finite, if $F$ is infinite then there exist infinitely many elements $h_0,h_1, \ldots \in F$ inducing the same map $\mathcal{W}(J_b,J_c) \to \mathcal{W}(J_a,J_d)$, hence infinitely many elements 
$$h_0^{-1}h_1, h_0^{-1}h_2, \ldots \in \bigcap\limits_{K \in \mathcal{W}(J_b,J_c)} \mathrm{stab}(K) \subset \mathrm{stab}(J_{b+1}) \cap \mathrm{stab}(J_{b+2}).$$
But $J_{b+1}$ and $J_{b+2}$ are strongly separated, so $\mathrm{stab}(J_{b+1}) \cap \mathrm{stab}(J_{b+2})$ stabilises the vertex $\mathrm{proj}_{N(J_{b+1})} N(J_{b+2})$. Because vertex-stabilisers are trivial, we get a contradiction. Therefore, $F$ must be finite, which means that $g$ is WPD with respect to the action of $\Gamma \mathcal{G}$ on $\crossing$.
\end{proof}

\noindent
We saw in Proposition \ref{prop:ContractingAxis} that an irreducible element $g \in \Gamma \mathcal{G}$ admits an axis in $\QM$. However, we do not know whether $g$ also admits an axis in $\crossing$ (or in $\ST$), i.e., a bi-infinite geodesic in $\crossing$ (or in $\ST$) on which it acts as a translation. We end this subsection by proving a weak solution to this problem.

\begin{prop}\label{prop:AlmostAxis}
Let $\Gamma$ be a finite and connected simplicial graph, $\mathcal{G}$ a collection of finitely generated groups indexed by $V(\Gamma)$, and $g \in \Gamma \mathcal{G}$ an irreducible element of full support. For every $n \geq 1$, there exist a power $s \geq 1$ and a maximal hyperplane $J$ such that $J,g^sJ, \ldots g^{ns}J$ lie on a geodesic in the small crossing graph and intersect the $(6D|g|+2)$-neighborhood of every axis of $g$.
\end{prop}

\begin{proof}
Up to conjugating $g$, we assume without loss of generality that $g$ is graphically cyclically reduced. According to Proposition \ref{prop:ContractingAxis}, there exist a hyperplane $J$ and a power $r \geq 1$ such that $J$ crosses an axis $\ell \subset \QM$ of $g$ and such that the hyperplanes in the collection $\{g^{kr}J \mid k \in \mathbb{Z} \}$ are pairwise strongly separated. For every $k \in \mathbb{Z}$, fix a straight geodesic $\Sigma_k$ between $g^{-kr}J$ and $g^{kr}J$ in $\crossing$ and let $\sigma_k$ denote an associated straight path; notice that $\sigma_k \subset \bigcup\limits_{H \in \Sigma_k} N(H)$, that $N(H) \cap \sigma_k \neq \emptyset$ for every $H \in \Sigma_k$ and that, according to Proposition \ref{prop:ContractingAxis}, the subsegment of $\sigma_k$ delimited by $g^{-(k+1)r}J$ and $g^{(k-1)r}J$ lies in the $6D|g|$-neighborhood $V$ of the axis $\ell$ with respect to the $\delta$-metric, where $D$ denotes the diameter of $\mathrm{supp}(g)^{\mathrm{opp}}$. Because $V$ is locally finite and that only finitely many hyperplanes contain a given vertex in their carriers, it follows that the sequences $(\Sigma_k)$ and $(\sigma_k)$ subconverge to bi-infinite geodesics $\Sigma$ (in $\crossing$) and $\sigma$ (in $\QM$) such that $\sigma \subset \left( \bigcup\limits_{H \in \Sigma} N(H) \right) \cap V$ and such that $N(H) \cap \sigma \neq \emptyset$ for every $H \in \Sigma$. 

\medskip \noindent
For every hyperplane $H$ in $\QM$, let $\mathcal{M}(H)$ denote the collection of all the hyperplanes $K$ constructed as follows: write $H=gJ_u$ for some $g \in \Gamma \mathcal{G}$ and $u \in V(\Gamma)$, fix a $\prec$-maximal vertex $u^+ \in V(\Gamma)$ satisfying $u \prec u^+$, and set $K= gJ_{u^+}$. Notice that $\mathcal{M}(H)$ has cardinality at most $|V(\Gamma)|$. For every $H \in \Sigma$, we fix a hyperplane $H^+ \in \mathcal{M}(H)$.

\begin{claim}
The collection $\{H^+ \mid H \in \Sigma\}$ defines a bi-infinite geodesic $\Sigma^+$ in $\ST$.
\end{claim}

\noindent
Fix an $H \in \Sigma$. By construction, the carrier of $H^+$ contains a fiber of $H$. Consequently, a hyperplane which is transverse to $H$ is either equal to $H^+$ or transverse to it. So, if we replace $H$ with $H^+$ in $\Sigma$, and identify the possible neighbors of $H$ equal to $H^+$ to our new vertex, we get a new path in $\crossing$. Actually, because the second operation shortens $\Sigma$, we know that it cannot happen, so replacing $H$ with $H^+$ in $\Sigma$ yields a new bi-infinite geodesics in $\crossing$. By iterating the process step by step, it follows that $\Sigma^+:= \{H^+ \mid H \in \Sigma\}$ defines a bi-infinite geodesic in $\crossing$. As $\Sigma^+$ clearly lies in $\ST$, our claim is proved.

\medskip \noindent
Now, fix a point $x \in \ell$ and let $B$ denote the ball centered at $x$ of radius $6D|g|$. For every $i \in \mathbb{Z}$, fix a hyperplane $H_i \in \Sigma$ which crosses $g^iB$. Define the color set
$$\mathcal{C}= \bigcup\limits \{ \mathcal{M}(H) \mid \text{$H$ hyperplane satisfying $N(H) \cap V \cap B \neq \emptyset$} \}$$ 
and the color map $\mathbb{Z} \to \mathcal{C}$ as follows: the integer $i\in \mathbb{Z}$ is colored by $g^{-i}H_i^+ \in \mathcal{C}$. Because $\mathcal{C}$ is finite, it follows from van der Waerden's theorem that $\mathbb{Z}$ contains a monochromatic arithmetic progression of length $n+1$, i.e., there exist $p,q \in \mathbb{N}$ such that all the $p+kq$ for $0 \leq k \leq n$ have the same color. In other words, 
$$g^{-p}H_p^+= g^{-(p+q)}H_{p+q}^+= g^{-(p+2q)} H_{p+2q}^+ = \cdots = g^{-(p+nq)} H_{p+nq}^+.$$ 
Consequently, if $K$ denotes this common maximal hyperplane, then all the hyperplanes $g^pK,g^{p+q}K, g^{p+2q}K, \ldots, g^{p+nq}K$ lie on the geodesic $\Sigma^+$ in the small crossing graph. Moreover, for every $0 \leq i \leq n$, the hyperplane $g^{p+iq}K=H_{p+iq}^+$ contains a fiber of $H_{p+iq}$ in its carrier, and we know that $N(H_{p+iq}) \cap \sigma \neq \emptyset$ and that $\sigma \subset V$, so $g^{p+iq}K$ crosses the $(6D|g|+2)$-neighborhood of $\ell$ (not with respect to the $\delta$-metric!).
\end{proof}

\subsection{WPD isometries}\label{section:WPD}

\noindent
In this subsection, our goal is to determine when the inner automorphism associated to an element $g\in \Gamma \mathcal{G}$ of full support defines a WPD element with respect to the action of $\mathrm{Aut}(\Gamma \mathcal{G})$ on the small crossing graph $\ST$ as defined by Lemma \ref{lem:DefAction}.

\begin{prop}\label{prop:InnerWPD}
Let $\Gamma$ be a finite and connected simplicial graph, and $\mathcal{G}$ a collection of groups indexed by $V(\Gamma)$. Assume that $\Gamma$ is not a join and that it contains at least two vertices. Also, assume that, for every automorphism $\varphi \in \mathrm{Aut}(\Gamma\mathcal{G})$ and every $\prec$-maximal vertex $u \in V(\Gamma)$, there exist an element $g\in \Gamma \mathcal{G}$ and a $\prec$-maximal vertex $v\in V(\Gamma)$ such that $\varphi(\langle u \rangle)= g \langle v \rangle g^{-1}$. Let $g \in \Gamma \mathcal{G}$ be an element of full support which is graphically cyclically reduced. If the inner automorphism $\iota(g)$ is not a WPD element of $\mathrm{Aut}(\Gamma \mathcal{G})$ with respect to its action on the small crossing graph $\ST$, then $\{ \varphi \in \mathrm{Aut}(\Gamma \mathcal{G}) \mid \varphi(g)=g\}$ has infinite image in $\mathrm{Out}(\Gamma \mathcal{G})$.
\end{prop}

\noindent
We begin by proving two short preliminary lemmas.

\begin{lemma}\label{lem:oneWPD}
Let $\Gamma$ be a finite simplicial graph and $\mathcal{G}$ a collection of groups indexed by $V(\Gamma)$. The inequality 
$$d_T(A,B) \leq \mathrm{diam}(\Gamma) \cdot (d_{QM}(N(A),N(B)) +2) $$
holds for all hyperplanes $A,B$ in $\QM$.
\end{lemma}

\begin{proof}
Fix a geodesic $x_0, \ldots, x_n$ between two vertices minimising the distance between $N(A)$ and $N(B)$. For every $0 \leq i \leq n-1$, let $J_i$ denote the hyperplane dual to the edge $[x_i,x_{i+1}]$. For every $0 \leq i \leq n-1$, the hyperplanes $J_i$ and $J_{i+1}$ are either tangent or transverse, so it follows from Claim \ref{claim:disttangenthyp} that $d_T(J_i,J_{i+1}) \leq \mathrm{diam}(\Gamma)$. The same is true for $A$ and $J_0$, and for $J_n$ and $B$. Therefore
$$d_T(A,B) \leq d_T(A,J_0)+ \sum\limits_{i=0}^{n-1} d_T(J_i,J_{i+1}) + d_T(J_n,B) \leq (n+2) \mathrm{diam}(\Gamma),$$
concluding the proof of our lemma.
\end{proof}

\begin{lemma}\label{lem:twoWPD}
Let $\Gamma$ be a simplicial graph and $\mathcal{G}$ a collection of groups indexed by $V(\Gamma)$. If $P$ and $Q$ are two conjugate maximal product subgroups both containing the rotative-stabiliser of a hyperplane $J$, then $P=hQh^{-1}$ for some $h \in \mathrm{stab}(J)$.
\end{lemma}

\begin{proof}
Up to translating $J$ and conjugating $P,Q$ (by the same element), assume without loss of generality that $J=J_u$ for some $u \in V(\Gamma)$. Because $P$ and $Q$ are conjugate, there exist a maximal join $\Lambda \subset \Gamma$ and two elements $g,h \in \Gamma \mathcal{G}$ such that $P=g \langle \Lambda \rangle g^{-1}$ and $Q= h \langle \Lambda \rangle h^{-1}$. Because $P$ and $Q$ both contain the rotative-stabiliser $\langle u \rangle$ of $J$, it follows from Lemma \ref{lem:Inclusion} that $g,h \in \langle \mathrm{star}(u) \rangle \langle \Lambda \rangle$. So there exist $a,p \in \langle \mathrm{star}(u) \rangle$ and $b,q \in \langle \Lambda \rangle$ such that $g=ab$ and $h=pq$. We have
$$\begin{array}{lcl} ap^{-1} Qpa^{-1} & = & ap^{-1} h \langle \Lambda \rangle h^{-1} pa^{-1} = a q \langle \Lambda \rangle q^{-1} a^{-1} \\ \\ & = & a \langle \Lambda \rangle a^{-1} = ab \langle \Lambda \rangle b^{-1}a^{-1} =P \end{array}$$
where $ap^{-1} \in \langle \mathrm{star}(u) \rangle = \mathrm{stab}(J)$ as desired.
\end{proof}

\noindent
We are now ready to prove Proposition \ref{prop:InnerWPD}. The proof is long and technical, and we refer the reader to Figure \ref{ForWPD} for a global picture of the configuration under consideration.

\begin{proof}[Proof of Proposition \ref{prop:InnerWPD}.]
The automorphism group $\mathrm{Aut}(\Gamma \mathcal{G})$ permutes the $\prec$-maximal vertices and the maximal joins in $\Gamma$ via
$$\left\{ \begin{array}{ccc} \mathrm{Aut}(\Gamma \mathcal{G}) & \to & \mathrm{Bij}( \text{maximal vertices in $\Gamma$}) \\ \varphi & \mapsto & \left( u \mapsto \text{$v$ such that $\langle v \rangle$ and $\varphi(\langle u \rangle)$ are conjugate} \right) \end{array} \right.$$
and
$$\left\{ \begin{array}{ccc} \mathrm{Aut}(\Gamma \mathcal{G}) & \to & \mathrm{Bij}( \text{maximal joins  $\Gamma$}) \\ \varphi & \mapsto & \left( \Lambda \mapsto \text{$\Xi$ such that $\langle \Xi \rangle$ and $\varphi(\langle \Lambda \rangle)$ are conjugate} \right) \end{array} \right.$$
Let $\mathrm{Aut}_0(\Gamma \mathcal{G})$ denote the intersection of the kernels of these two actions. In other words, for every $\prec$-maximal vertex $u \in V(\Gamma)$ (resp. for every maximal join $\Lambda \subset \Gamma$) and every $\varphi \in \mathrm{Aut}_0(\Gamma \mathcal{G})$, the subgroups $\langle u \rangle$ and $\varphi(\langle u \rangle)$ (resp. $\langle \Lambda \rangle$ and $\varphi( \langle \Lambda \rangle)$) are conjugate. Notice that $\mathrm{Aut}_0(\Gamma \mathcal{G})$ is a finite-index subgroup of $\mathrm{Aut}(\Gamma \mathcal{G})$ which contains $\mathrm{Inn}(\Gamma \mathcal{G})$.

\medskip \noindent
If $\iota(g)$ is not WPD with respect to the action of $\mathrm{Aut}(\Gamma \mathcal{G})$ on $\ST$, then it is not WPD with respect to the action of $\mathrm{Aut}_0(\Gamma \mathcal{G})$ on $\ST$ either. As a consequence of Lemma~\ref{lem:WPDdef}, it follows that there exists an $\epsilon>0$ such that, for every maximal hyperplane $J$ and for every power $n \geq 1$, the set
$$F(J,n):= \{ \varphi\in \mathrm{Aut}_0(\Gamma \mathcal{G}) \mid d_T(\varphi J, J), d_T(\varphi g^nJ, g^nJ) \leq \epsilon \}$$
is infinite. Notice that, as a consequence of Proposition \ref{prop:IrreducibleWPD}, there exists some $n_0 \geq 1$ such that 
$$\{ h \in \Gamma \mathcal{G} \mid d_T(J,hJ), d_T(g^nJ, hg^nJ) \leq \max \left( 4, 2(\epsilon+1) \right) \}$$
is finite for every $n \geq n_0$ and every hyperplane $J$. In particular, $F(J,n) \cap \mathrm{Inn}(\Gamma \mathcal{G})$ is finite for every $n \geq n_0$ and every hyperplane $J$. Set
$$p:= n_0+2 \mathrm{diam}(\Gamma) \left[ 6D|g|^2 + (2\epsilon+8\delta +5)(D|g|+1) + (6D+5)|g|+4 \right]$$
where $\delta$ denotes the hyperbolicity constant of $\crossing$ and $D$ the diameter of $\mathrm{supp}(g)^{\mathrm{opp}}$. According to Proposition \ref{prop:AlmostAxis}, there exist a power $s \geq 1$ and a maximal hyperplane $J$ such that $J, g^{s}J, \ldots, g^{3ps}J$ lie on a geodesic in the small crossing graph, say $H_1, \ldots, H_n$. 

\medskip \noindent
According to Proposition \ref{prop:ContractingAxis}, if $J_0$ is a hyperplane crossing an axis $\ell$ of $g$ then $\{g^{2Dk}J_0 \mid k \in \mathbb{Z}\}$ is a collection of pairwise strongly separated hyperplanes. For convenience, we set $J_i= g^{2Di}J_0$ for every $i \in \mathbb{Z}$. Up to shifting the indices, we suppose that $J_1, \ldots, J_k$ is the maximal subcollection of $\{ \ldots, J_{-1},J_0,J_1, \ldots\}$ which separates $J$ and $g^{3ps}J$; and, up to reading the indices from right to left, we suppose that $J_1$ separates $J$ and $J_k$. Let $b$ be the last index such that $g^{ps}J$ lies between $J_b$ and $J_k$; and let $c$ be the first index such that $g^{2ps}J$ lies between $J$ and $J_c$.
\begin{figure}
\begin{center}
\includegraphics[scale=0.27]{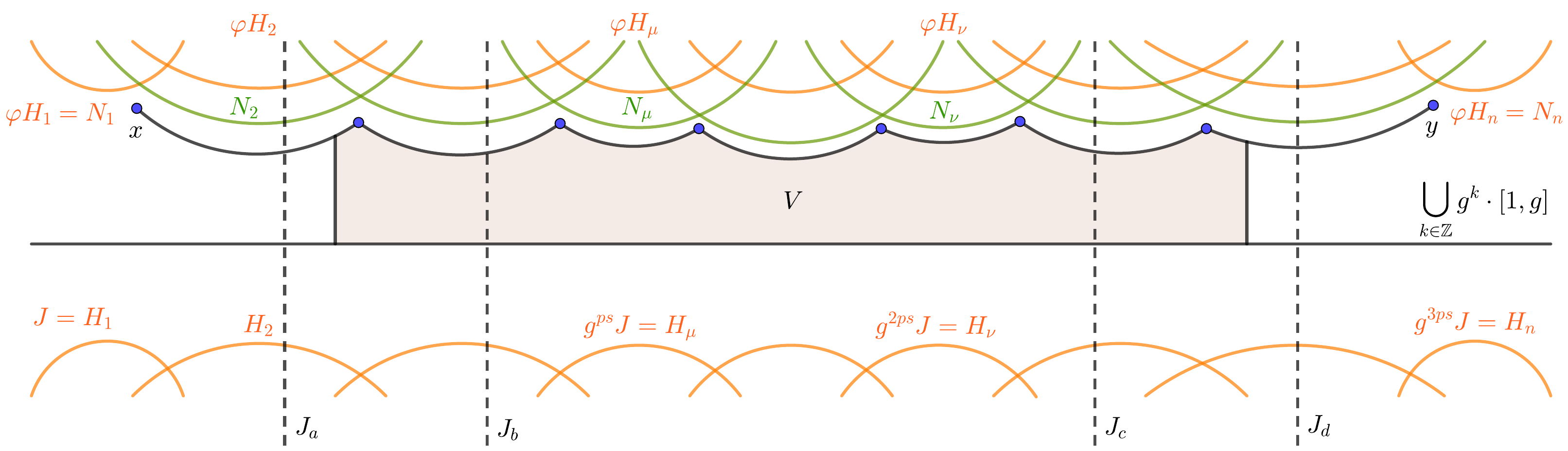}
\caption{Configuration from the proof of Proposition \ref{prop:InnerWPD}.}
\label{ForWPD}
\end{center}
\end{figure}

\begin{claim}\label{claim:WPDone}
We have $b, k-c \geq 2(\epsilon+2(\delta+1))$.
\end{claim}

\noindent
First of all, notice that
$$ps \leq d_T(J,g^{ps}J) \leq d_T(J,J_1) + \sum\limits_{i=1}^{b-1}d_T(J_i,J_{i+1}) + d_T(J_b,g^{ps}J).$$
As a consequence of Lemma \ref{lem:oneWPD}, we know that
$$d_T(J_i,J_{i+1}) \leq (d_{QM}(N(J_i),N(J_{i+1})+2) \cdot \mathrm{diam}(\Gamma) \leq 2(D|g|+1) \cdot \mathrm{diam}(\Gamma)$$
for every $1 \leq i \leq b-1$.
By definition of our collection $\{J_1, \ldots, J_k\}$, $J_0$ does not separate $J$ and $J_1$, so either $J$ is transverse to $J_0$ or it lies between $J_{-1}$ and $J_0$. In any case, $J$ lies between $J_{-2}$ and $J_0$. We also know from Proposition \ref{prop:AlmostAxis} that $J$ intersects the $(6D|g|+2)$-neighborhood of $\ell$. So let $w \in \ell$ be a vertex at distance at most $6D|g|+2$ from $N(J)$. Necessarily, $w$ lies between $J_{-2-(6D|g|+2)}$ and $J_{6D|g|+2}$, which implies that there is a subsegment $\xi \subset \ell$ of length at most $(6D|g|+5)|g|$ between $w$ and $N(J_1)$. Hence
$$\begin{array}{lcl} d_T(J,J_1) & \leq & (d_{QM}(N(J),N(J_1))+2) \mathrm{diam}(\Gamma) \\ \\ & \leq & (d_{QM}(N(J),w)+ \mathrm{length}(\xi)+2) \mathrm{diam}(\Gamma) \\ \\ & \leq & (6D|g|^2+(6D+5)|g|+4) \mathrm{diam}(\Gamma) \end{array}$$
A symmetric argument shows that $d_T(J_b,g^{ps}J) \leq (6D|g|^2+(6D+5)|g|+4) \mathrm{diam}(\Gamma)$. We conclude that
$$ps \leq 2(6D|g|^2+(6D+5)|g|+4) \mathrm{diam}(\Gamma) + 2 (b-1) (D|g|+1) \mathrm{diam}(\Gamma),$$
hence
$$b \geq 1+ \frac{ps-2(6D|g|^2+ (6D+5)|g|+4) \mathrm{diam}(\Gamma)}{ 2(D|g|+1) \mathrm{diam}(\Gamma)} \geq 2(\epsilon+2(\delta+1)).$$
A symmetric argument proves the same inequality for $k-c$, concluding the proof of our claim.

\medskip \noindent
Fix indices $1 \leq a \leq b$ and $c \leq d \leq k$ such that $a,k-d \geq \epsilon+2$ and such that $b-a,d-c \geq \epsilon+8\delta+4$. Notice that such indices exist according to Claim \ref{claim:WPDone}. 

\medskip \noindent
For convenience, for every $1 \leq i \leq k$, we denote by $J_i^-$ (resp. $J_i^+$) the sector delimited by $J_i$ which contains $J$ (resp. $g^{3ps}J$). Fix an automorphism $\varphi \in F(J,3ps)$.

\begin{claim}
$\varphi J$ lies in $J_a^-$ and $\varphi g^{3ps}J$ in $J_d^+$.
\end{claim}

\noindent
If $\varphi J$ does not lie in $J_a^-$, then it has to lie in $J_{a-1}^+$ because $J_a$ and $J_{a-1}$ are strongly separated, hence
$$d_T(J, \varphi J) \geq \Delta (J,\varphi J) \geq \# \{ J_1, \ldots, J_{a-1} \}= a-1 > \epsilon$$
contradicting the fact that $\varphi$ belongs to $F(J,3ps)$. A symmetric argument shows that $\varphi g^{3ps}J$ lies in $J_d^+$, concluding the proof of our claim.

\medskip \noindent
Let $N_1, \ldots, N_n$ denote the straight geodesic associated to $\varphi H_1, \ldots, \varphi H_n$ by Proposition~\ref{prop:NormalExist}. As a consequence of Lemma \ref{lem:straightgeo} (which applies because $n \geq 3ps \geq 4$), a straight path associated to $N_1, \ldots, N_n$ defines a geodesic $[x,y]$ from a vertex $x$ in $J_a^-$ (namely, the projection of $N(\varphi g^{3ps}J)$ onto $N(\varphi J)$) to a vertex $y$ in $J_d^+$ (namely, the projection of $N(\varphi J)$ onto $N(\varphi g^{3ps}J)$). As a consequence of Proposition \ref{prop:ContractingAxis}, the subsegment $\sigma \subset [x,y]$ lying between $J_{a+1}$ and $J_{d-1}$ is contained in the $6D|g|$-neighborhood $V$ of $\ell \cap J_{a+1}^+ \cap J_{d-1}^-$ with respect to the $\delta$-metric. 

\medskip \noindent
For convenience, let $\mu,\nu$ be the two indices such that $g^{ps}J=H_\mu$ and $g^{2ps}J= H_\nu$.

\begin{claim}\label{claim:WPDtwo}
The hyperplanes $N_\mu$ and $N_\nu$ lie between $J_{a+1}$ and $J_{d-1}$. 
\end{claim}

\noindent
Because the metric in a $\delta$-hyperbolic space is $8\delta$-convex \cite[Corollaire 5.3]{CDP}, we have
$$d_T(H_\mu,\varphi H_\mu) \leq \max \left( d_T(H_1, \varphi H_1), d_T( H_n, \varphi H_n) \right) +8\delta \leq \epsilon+8\delta.$$
If $\varphi H_\mu$ does not lie in $J_{a+3}^+$, then it must intersect $(J_{a+3}^+)^c$ and so it must lie in $J_{a+4}^-$. This implies that
$$d_T(H_\mu, \varphi H_\mu) \geq \Delta( H_\mu, \varphi H_\mu) \geq \# \{ J_{a+4}, \ldots, J_b \} = b-a-3 > \epsilon+8 \delta,$$
contradicting the previous observation. Thus, $\varphi H_\mu$ lies in $J_{a+3}^+$, and we show similarly that it also lies in $J_{d-3}^-$, so $\varphi H_{\mu}$ lies between $J_{a+3}$ and $J_{d-3}$. We know from Proposition~\ref{prop:NormalExist} that $\varphi H_\mu$ and $N_\mu$ cross a common product $P$. Because no two hyperplanes crossing $P$ are strongly separated, necessarily $P$ lies between $J_{a+2}$ and $J_{d-2}$. So $N_\mu$ intersects $J_{a+2}^+ \cap J_{d-2}^-$, which implies that it must lie between $J_{a+1}$ and $J_{d-1}$ as desired. A symmetric argument shows that $N_\nu$ also lies between $J_{a-1}$ and $J_{d-1}$. Our claim is proved.

\medskip \noindent
According to Proposition \ref{prop:NormalExist}, the rotative-stabilisers of $\varphi H_\mu= \varphi g^{ps}J$ and $N_\mu$ (resp. $\varphi H_\nu = \varphi g^{2ps}J$ and $N_\nu$) lie in a common maximal product subgroup $\alpha_\varphi \langle \Phi_\varphi \rangle \alpha_\varphi^{-1}$ (resp. $\beta_\varphi \langle \Psi_\varphi \rangle \beta_\varphi^{-1}$). Moreover, the subgraph $\alpha_\varphi \langle \Phi_\varphi \rangle$ (resp. $\beta_\varphi \langle \Psi_\varphi \rangle$) intersects the straight path $[x,y]$ and the carrier of $N_\mu$ (resp. the carrier of $N_\nu$); as a consequence of Claim~\ref{claim:WPDtwo}, it also intersects the subsegment $\sigma \subset [x,y]$, and a fortiori the neighborhood $V$.

\medskip \noindent
Thus, we have proved that every $\varphi \in F(J,3ps)$ sends a maximal product subgroup $P_\varphi:= \varphi^{-1} \left( \alpha_\varphi \langle \Phi_\varphi \rangle \alpha_\varphi^{-1} \right)$ containing the rotative-stabiliser of $H_\mu = g^{ps}J$ to a maximal product subgroup $\alpha_\varphi \langle \Phi_\varphi \rangle \alpha_\varphi^{-1}$ such that $\alpha_\varphi \langle \Phi_\varphi \rangle$ intersects $V$. Similarly, every $\varphi \in F(J,3ps)$ sends a maximal product subgroup $Q_\varphi:= \varphi^{-1} \left( \beta_\varphi \langle \Psi_\varphi \rangle \beta_\varphi^{-1} \right)$ containing the rotative-stabiliser of $H_\nu = g^{2ps}J$ to a maximal product subgroup $\beta_\varphi \langle \Psi_\varphi \rangle \beta_\varphi^{-1}$ such that $\beta_\varphi \langle \Psi_\varphi \rangle$ intersects $V$. Because $V$ is finite and because there exist only finitely many cosets of maximal product subgroups containing a given element, it follows that there exists an infinite family $\varphi_1,\varphi_2, \ldots \in F(J,3ps)$ such that $\Phi_{\varphi_i}= \Phi_{\varphi_j}$, $\Psi_{\varphi_i}= \Psi_{\varphi_i}$, $\varphi_i(P_{\varphi_i})=\varphi_j(P_{\varphi_j})$ and $\varphi_i(Q_{\varphi_i})= \varphi_j(Q_{\varphi_j})$ for every $i, j \geq 1$. Moreover, because $F(J,3ps) \cap \mathrm{Inn}(\Gamma \mathcal{G})$ is finite, we may suppose without loss of generality that $\varphi_1, \varphi_2, \ldots$ have distinct images in $\mathrm{Out}(\Gamma \mathcal{G})$. Notice that, for every $i,j \geq 1$, $P_{\varphi_i}$ and $P_{\varphi_j}$ (resp. $Q_{\varphi_i}$ and $Q_{\varphi_j}$) are conjugate; and that 
$$d_T(J, \varphi_1^{-1} \varphi_i \cdot J) \leq d_T(J, \varphi_1^{-1} J)+d_T(J, \varphi_i J) \leq 2 \epsilon$$
for every $i \geq 1$.

\medskip \noindent
Therefore, by setting $\psi_i : = \varphi_1^{-1}\varphi_i$, $r=ps$, $P:=P_{\varphi_1}$, $Q:= Q_{\varphi_1}$, $P_i:= P_{\varphi_i}$ and $Q_i:=Q_{\varphi_i}$ for every $i \geq 1$, we have proved that:

\begin{fact}
There exist 
\begin{itemize}
	\item a maximal hyperplane $J$ and a power $r \geq n_0$;
	\item an infinite collection $\psi_1, \psi_2, \ldots \in \mathrm{Aut}_0(\Gamma \mathcal{G})$ with distinct images in $\mathrm{Out}(\Gamma \mathcal{G})$;
	\item pairwise conjugate maximal product subgroups $P,P_1,P_2, \ldots$ containing the rotative-stabiliser of $g^{r}J$;
	\item pairwise conjugate maximal product subgroups $Q,Q_1,Q_2, \ldots$ containing the rotative-stabiliser of $g^{2r}J$,
\end{itemize}
such that $d_T(J,\psi_i \cdot J) \leq 2\epsilon$ and $\psi_i(P_i)=P, \psi_i(Q_i)=Q$ for every $i \geq 1$.
\end{fact}

\noindent
We will be able to conclude from this observation.

\medskip \noindent
Fix an index $i \geq 1$ and observe that
$$g^rP^{g^{-r}}g^{-r} = P = \psi_i(P_i) = \psi_i \left( g^r P_i^{g^{-r}} g^{-r} \right) = \psi_i(g)^r \psi_i \left( P_i^{g^{-r}} \right) \psi_i(g)^{-r}.$$
As $\psi_i$ belongs to $\mathrm{Aut}_0(\Gamma \mathcal{G})$, the hyperplanes $J$ and $\psi_i J$ have the same label, so there exists some $h_i \in \Gamma \mathcal{G}$ such that $\psi_i J = h_i J$. Because $\psi_i \left( P_i^{g^{-r}} \right)$ contains the rotative-stabiliser of $h_i J$, it follows that $h_i^{-1} \psi_i \left( P_i^{g^{-r}} \right) h_i$ and $g^{-r}Pg^r$ are two maximal product subgroups containing the rotative-stabiliser of $J$. Moreover, they are conjugate because $\psi_i \in \mathrm{Aut}_0(\Gamma \mathcal{G})$ and because $P$ and $P_i$ are conjugate. It follows from Lemma \ref{lem:twoWPD} that there exists an element $k_i \in \mathrm{stab}(J)$ which conjugates the first one to the second one, hence
$$g^r P^{g^{-r}} g^{-r} =P=\psi_i(P_i)= \psi_i(g)^r h_i k_i \cdot P^{g^{-r}} \cdot k_i^{-1}h_i^{-1} \psi_i(g)^{-r}.$$
Because $P^{g^{-r}}$ is self-normalising according to Lemma \ref{lem:normaliser}, it follows that $g^{-r}\psi_i(g^r)$ belongs to $P^{g^{-r}} k_i^{-1}h_i^{-1}$. So $\psi_i(g^r)= g^r a_i k_i^{-1}h_i^{-1}$ for some $a_i \in P^{g^{-r}}$.

\medskip \noindent
The same argument for $g^{2r}$, $Q$ and $Q_i$ shows that $g^{-2r} \psi_i(g^{2r}) \in Q^{g^{-2r}} \ell_i^{-1}h_i^{-1}$ for some $\ell_i \in \mathrm{stab}(J)$. But we deduce from the previous observation that 
$$g^{-2r} \psi_i(g^{2r}) = g^{-2r} \psi_i(g^r) \psi_i(g^r) = g^{-r} a_ik_i^{-1}h_i^{-1}g^r a_ik_i^{-1}h_i^{-1}.$$
Therefore,
$$a_ik_i^{-1}h_i^{-1} \in g^r \left( Q^{g^{-2r}} \mathrm{stab}(J) P^{g^{-r}} \right) g^{-r} \cap \left( P^{g^{-r}} \mathrm{stab}(J) h_i^{-1} \right) \subset g^r S_1 g^{-r} \cap S_2$$
where $S_1:= Q^{g^{-2r}} \mathrm{stab}(J) P^{g^{-r}}$ and $S_2:= P^{g^{-r}} \mathrm{stab}(J) S$, if $S$ denotes the set $\{h \in \Gamma \mathcal{G} \mid d(J,hJ) \leq 2 \epsilon \}$.

\begin{claim}
The set $g^r S_1g^{-r} \cap S_2$ is finite.
\end{claim}

\noindent
Fix an $x \in S_1$. So there exist $a \in Q^{g^{-2r}}$, $b \in \mathrm{stab}(J)$ and $c \in P^{g^{-r}}$ such that $x=abc$. Notice that $d_T(J,aJ),d_T(J,cJ) \leq 2$ because $P^{g^{-r}}$ and $Q^{g^{-2r}}$ are two maximal product subgroups containing the rotative-stabiliser of $J$, so
$$d_T(x,xJ) \leq d_T(J,aJ)+ d_T(J,bJ)+d_T(J,cJ) \leq 4.$$
Now fix an $x \in S_2$. So there exist $a \in P^{g^{-r}}$, $b \in \mathrm{stab}(J)$ and $c \in S$ such that $x=abc$. Here we have
$$d_T(x,xJ) \leq d_T(J,aJ)+ d_T(J,bJ)+d_T(J,cJ) \leq 2 \epsilon+2.$$
Consequently, $d(J,xJ) \leq 2(\epsilon+1)$ and $d(g^rJ,x g^rJ) \leq 4$ for every $x \in g^r S_1g^{-r} \cap S_2$. The desired conclusion follows from the fact that $r \geq n_0$.

\medskip \noindent
So far, we have proved that $\psi_i(g^r) \in g^r \left( g^r S_1g^{-r} \cap S_2 \right)$ for every $i \geq 1$, where $g^r S_1g^{-r} \cap S_2$ is a finite set which does not depend on $i$. Therefore, there exists an infinite set $I \subset \mathbb{N}$ such that $\psi_i(g^r)=\psi_j(g^r)$ for every $i,j \in I$. By fixing some $i_0 \in I$ and setting $\xi_i:= \psi_{i_0}^{-1} \psi_i$, we find an infinite collection $\{\xi_i \mid i \in I\}$ of automorphisms which are pairwise distinct in $\mathrm{Out}(\Gamma \mathcal{G})$ and which fix $g^r$. Notice that $\xi_i(g)$ is an $r$th root of $\xi_i(g^r)=g^r$ for every $i \in I$, so it follows from Corollary \ref{cor:roots} that there exists $J \subset I$ infinite such that $\xi_i(g)=\xi_j(g)$ for every $i,j \in J$. Fixing some $j_0 \in J$, we find that $\{ \xi_{j_0}^{-1}\xi_i \mid i \in J\}$ is an infinite collection of automorphisms which are pairwise distinct in $\mathrm{Out}(\Gamma \mathcal{G})$ and which fix $g$. This concludes the proof of our proposition. 
\end{proof}

\subsection{Generalised loxodromic inner automorphisms}

\noindent
We are now ready to conclude the second step of our argument by proving the following criterion:

\begin{thm}
Let $\Gamma$ be a finite connected simplicial graph, $\mathcal{G}$ a collection of graphically irreducible groups indexed by $V(\Gamma)$, and $g \in \Gamma \mathcal{G}$ an element of full support. Assume that $\Gamma$ contains at least two vertices and is not a join. The following assertions are equivalent:
\begin{itemize}
	\item[(i)] the inner automorphism $\iota(g)$ is a generalised loxodromic element of $\mathrm{Aut}(\Gamma \mathcal{G})$;
	\item[(ii)] $\{ \varphi \in \mathrm{Aut}(\Gamma \mathcal{G}) \mid \varphi(g)=g\}$ is virtually cyclic;
	\item[(iii)]  $\{ \varphi \in \mathrm{Aut}(\Gamma \mathcal{G}) \mid \varphi(g)=g\}$ has an infinite image in $\mathrm{Out}(\Gamma \mathcal{G})$;
	\item[(iv)] the centraliser of $\iota(g)$ in $\mathrm{Aut}(\Gamma \mathcal{G})$ is virtually cyclic.
\end{itemize}
\end{thm}

\begin{proof}
As $\Gamma$ is not a join, the center of $\Gamma \mathcal{G}$ must be trivial, so $\{\varphi \in \mathrm{Aut}(\Gamma \mathcal{G}) \mid \varphi(g)=g\}$ coincides with the centraliser of $\iota(g)$ in $\mathrm{Aut}(\Gamma \mathcal{G})$. Therefore, the equivalence $(ii) \Leftrightarrow (iv)$ is clear. Also, because $\langle \iota(g) \rangle$ is an infinite cyclic subgroup in $\{ \varphi \in \mathrm{Aut}(\Gamma \mathcal{G}) \mid \varphi(g) =g\}$, the equivalence $(ii) \Leftrightarrow (iii)$ is straightforward. The implication $(i) \Rightarrow (iv)$ follows from Lemma \ref{lem:subgroupE}. Finally, the implication $(iii) \Rightarrow (i)$ follows from Propositions \ref{prop:SameLinkStar} and~\ref{prop:InnerWPD}.
\end{proof}

\section{Step 3: Fixators in the outer automorphism group}\label{section:step2}

\noindent
The main result of this section is the following statement:

\begin{thm}\label{thm:ActionRealTree}
Let $\Gamma$ be a finite connected simplicial graph and $\mathcal{G}$ a collection of finitely generated groups indexed by $V(\Gamma)$. Fix an irreducible element $g \in \Gamma \mathcal{G}$ and suppose that $\{ \varphi \in \mathrm{Aut}(\Gamma \mathcal{G}) \mid \varphi(g)=g \}$ has an infinite image in $\mathrm{Out}(\Gamma \mathcal{G})$. Then $\Gamma \mathcal{G}$ admits an action on a real tree without a global fixed point, with arc-stabilisers in product subgroups and with $g$ as an elliptic isometry. 
\end{thm}

\noindent
Starting with a family of automorphisms $\varphi_1, \varphi_2, \ldots \in \mathrm{Aut}(\Gamma \mathcal{G})$ fixing $g$ and having pairwise distinct images in $\mathrm{Out}(\Gamma \mathcal{G})$, the idea is to take the limit of the twisted actions
$$\left\{ \begin{array}{ccc} \Gamma \mathcal{G} & \to & \mathrm{Isom}(\Gamma \mathcal{G}) \\ g & \mapsto & (x \mapsto \varphi_n(g) \cdot x) \end{array} \right.$$
in order to define a fixed-point free action of $\Gamma \mathcal{G}$ on one of its asymptotic cones $\mathrm{Cone}(\Gamma \mathcal{G})$. This construction, introduced in \cite{Paulin} for hyperbolic groups, is detailed in full generality in Subsection \ref{section:PaulinRips}. The next step is to associate to the previous action an action on a real tree, and the key observation is that a graph product can be embeded equivariantly and quasi-isometrically into a product of trees of spaces. Such an embedding is described in Subsection \ref{section:TreeSpaces}. As a consequence, $\mathrm{Cone}(\Gamma \mathcal{G})$ embeds equivariantly into a product of \emph{tree-graded spaces} $T_1 \times \cdots \times T_n$. A technical difficulty is to verify that $\Gamma \mathcal{G}$ does not fix a point in $T_1 \times \cdots \times T_n$. This problem is solved by Theorem \ref{thm:FixedPointCone}, a fixed-point theorem proved in Subsection \ref{section:FixedPointCone} and whose proof relies on Subsections \ref{section:TreeGraded} and \ref{section:ProductsTreeGraded} dedicated to general statements about tree-graded spaces. Finally, Theorem \ref{thm:ActionRealTree} is proved in Subsection \ref{section:relativesplittings}.

\subsection{Tree-graded spaces}\label{section:TreeGraded}

\noindent
In the following, we use the definition of \emph{tree graded-spaces} given in \cite{MR2507252}. As notice in \cite[Remark 3.2]{MR2507252}, results proved in \cite{DrutuSapirTreeGraded} still apply.

\begin{definition}
Let $T$ be a complete geodesic metric space and let $\mathcal{P}$ be a collection of closed geodesic subspaces, called \emph{pieces}. If the following conditions are satisfied:
\begin{itemize}
	\item every two distinct pieces have at most one point in common;
	\item every simple non-trivial geodesic triangle in $T$ is contained in a piece,
\end{itemize}
then $T$ is \emph{tree-graded with respect to $\mathcal{P}$}\index{Tree-graded spaces}.
\end{definition}

\noindent
Recall that a subspace $Y$ in a geodesic metric space $X$ is \emph{geodesic} if, for all points $x,y \in Y$, there exists a geodesic in $X$ between $x$ and $y$ which lies in $Y$. By convention, all our isometries of tree-graded spaces permute the pieces. In particular, when we refer to an action (by isometries) of a group on a tree-graded space, we always assume that the action permutes the pieces. 

\medskip \noindent
The idea to keep in mind is that tree-graded spaces behave like real trees relative to their pieces, as motivated by the following statements, proved in \cite[Lemmas 2.6 and 2.15]{DrutuSapirTreeGraded}.

\begin{lemma}\label{lem:Loop}
In a tree-graded space, any non-trivial simple loop lies in a piece.
\end{lemma}

\begin{lemma}\label{lem:PieceProj}
Let $T$ be a tree-graded space and $P \subset T$ a piece. For every $x \in T$, there exists a unique point $y \in P$ such that $d(x,y)=d(x,P)$; moreover, $y$ lies on every geodesic from $x$ to a point in $P$. 
\end{lemma}

\paragraph{Full subspaces and projections.} Lemma \ref{lem:PieceProj} shows that, in tree-graded spaces, projections onto pieces are naturally defined. Our goal now is to extend this observation to larger family of subspaces.

\begin{definition}
Let $T$ be a tree-graded space. A subspace $R \subset T$ is \emph{full}\index{Full subspace} if it is closed, connected, and if $P \subset R$ for every piece $P$ satisfying $|P \cap R| \geq 2$. 
\end{definition}

\noindent
Of course, pieces are examples of full subspaces. The idea is that full subspaces will play the role of subtrees in real trees. For instance, notice that:

\begin{lemma}\label{lem:FullConvex}
Let $T$ be a tree-graded space. A full subspace $R \subset T$ is \emph{strongly convex} (i.e., every topological arc between two points of $R$ lies in $R$).
\end{lemma}

\begin{proof}
Fix two points $x,y \in R$, a topological arc $\alpha \subset T$ between them, and a point $z \in \alpha$. We want to prove that $z$ belongs to $R$. Because $R$ is connected, there exists a topological arc $\beta \subset T$ between $x$ and $y$. If $z$ belongs to $\beta$, there is nothing to prove, so assume that $z \notin \beta$. As a consequence, there exist two subsegments $\mu \subset \alpha$ and $\nu \subset \beta$ such that $\mu$ contains $z$ in its interior and such that $\mu \cup \nu$ is a simple loop. It follows from Lemma \ref{lem:Loop} that $\mu \cup \nu$ lies in a piece $P$. Because $P \cap R$ contains at least two points (namely, the endpoints of $\nu$), necessarily $P \subset R$. Hence $z \in \mu \subset P \subset R$ as desired.
\end{proof}

\noindent
What we want to do is to show that projections onto full subspaces are naturally defined. But, before stating our main result in this direction, we need to introduce some terminology.

\begin{definition}
Let $T$ be a tree-graded space and $x,y \in T$ two points. Given a topological arc $[x,y]$ between $x$ and $y$, $\mathrm{Cutp}([x,y])$ is the complement of the interiors of the subsegments of positive lengths which lie in a piece. According to \cite[Corollary 2.8]{DrutuSapirActions}, this set does not depend on the choice of the arc, so we also denote it by $\mathrm{Cutp}(x,y)$.
\end{definition}

\noindent
Notice that, if $P,Q \subset T$ are two distinct points or pieces and if $p \in P$ (resp. $q \in Q$) denotes the projection of $Q$ onto $P$ (resp. of $P$ onto $Q$), then $\mathrm{Cutp}(x,y)= \mathrm{Cutp}(p,q)$ for all $x \in P$ and $y \in Q$. Consequently, we naturally define $\mathrm{Cutp}(P,Q)$ as $\mathrm{Cutp}(p,q)$. In the case $P=Q$, we set $\mathrm{Cutp}(P,Q)= \emptyset$. 

\medskip \noindent
Projections onto full subspaces are well-defined as a consequence of the next statement:

\begin{prop}\label{prop:ProjectionFull}
Let $T$ be a tree-graded space, $R \subset T$ a full subspace, and $z \in T$ a point. There exists a unique point $w \in R$ such that $w$ belongs to $\mathrm{Cutp}(z,p)$ for every $p \in R$. In particular, $w$ is the unique point of $R$ which minimises the distance to $z$.
\end{prop}

\begin{proof}
Fix an arbitrary point $x \in R$ and a geodesic $[z,x]$ from $z$ to $x$. Because $R$ is closed, the first point $w$ along $[z,x]$ which belongs to $R$ is well-defined. Fix a point $p \in R$ and a geodesic $[w,p]$ from $w$ to $p$. 

\medskip \noindent
Assume for contradiction that $w \notin \mathrm{Cutp}(z,p)$. Two cases may happen. First, $[z,w] \cup [w,p]$ could be a geodesic. If so, it contains a subsegment having $w$ in its interior which is included in a piece. Second, $[z,w] \cup [w,p]$ could not be a geodesic. If so, there exist two points $a \in [z,w]$ and $b \in [w,p]$ such that a geodesic $[a,b]$ together with the subsegments $[a,w] \subset [z,w]$ and $[w,b] \subset [w,p]$ define a simple loop. It follows from Lemma~\ref{lem:Loop} that this loop lies in a piece. Thus, we have proved that, if $w \notin \mathrm{Cutp}(z,p)$, then there exists a piece $P$ which contains a subsegment of $[z,w] \cup [w,p]$ having $w$ in its interior. Because $[w,p] \subset R$ as a consequence of Lemma \ref{lem:FullConvex}, we know that the intersection between $P$ and $R$ is infinite, hence $P \subset R$. Therefore, $[z,w]$ contains a point in $P \subset R$ before $w$, contradicting our definition of $w$. 

\medskip \noindent
Thus, we have proved that $w \in \mathrm{Cutp}(z,p)$ for every $p \in R$. Now, let $w'$ be another point which satisfies this property. Because $w' \in \mathrm{Cutp}(z,w)$, we have $d(z,w)=d(z,w')+d(w',w)$; and because $w \in \mathrm{Cutp}(z,w')$, we have $d(z,w')=d(z,w)+d(w,w')$. It follows that $d(w,w')=0$, i.e., $w=w'$ as desired.
\end{proof}

\noindent
We record the definition of projection provided by Proposition \ref{prop:ProjectionFull}.

\begin{definition}
Let $T$ be a tree-graded space and $R \subset T$ a full subspace. The \emph{projection onto $R$} is the map which sends each point of $T$ to the unique point of $R$ which minimises the distance to it.
\end{definition}

\noindent
In the spirit of \cite[Corollary 2.11]{DrutuSapirTreeGraded}, observe that:

\begin{lemma}\label{lem:ProjJustPoint}
Let $T$ be a tree-graded space, $R \subset T$ a full subspace and $Y \subset T$ a connected subspace. If $Y \cap R$ has cardinality at most one, then the projection of $Y$ onto $R$ is a single point.
\end{lemma}

\begin{proof}
Let $a,b \in Y$ be two points and let $a',b'$ denote their respective projections onto $R$. Fix a topological arc $[a,b] \subset Y$ and three geodesics $[a,a'], [b,b'] \subset T$ and $[a',b'] \subset R$. Let $\alpha \subset [a,b]$, $\beta \subset [a,a']$ and $\gamma \subset [b,b']$ be three maximal subsegments such that $\alpha \cap \beta$ and $\alpha \cap \gamma$ are reduced to single points. Notice that $\beta \cap [a',b']= \{a'\}$ and $\gamma \cap [a',b'] = \{b'\}$. 

\medskip \noindent
If $\beta \cap \gamma= \emptyset$, then $\alpha \cup \beta \cup [a',b'] \cup \gamma$ is a simple loop, and it follows from Lemma \ref{lem:Loop} that there eixsts a piece $P$ which contains it. Because $[a',b'] \subset P$, necessarily $P \subset R$ because $R$ is full (notice that the condition $\beta \cap \gamma= \emptyset$ implies that $a'\neq b'$). Hence $\alpha \subset P \cap Y \subset R \cap Y$, contradicting the fact that $R \cap Y$ has cardinality at most one.

\medskip \noindent
So $\beta \cap \gamma \neq \emptyset$. Let $p \in \beta \cap \gamma$ denote the closest point to $a'$ and $b'$. Then $[a',b']$ and the subsegments of $\beta,\gamma$ delimited by $p,a',b'$ define a simple loop. If this loop is non-trivial, it follows from Lemma \ref{lem:Loop} that there exists a piece $P$ which contains it. Because $R$ is full, necessarily $P \subset R$. But then $d(a,p)<d(a,a')$ and $d(b,p)<d(b,b')$, contradicting the fact that $a',b'$ are the projections of $a,b$. Therefore, our loop has to be trivial, hence $a'=b'$ as desired.
\end{proof}

\paragraph{The median operation.} In the same way that three points in a real tree have a natural center, referred to as the \emph{median point} (namely, the center of the tripod delimited by the three points), we want to associate to every triple of points in a tree-graded space a \emph{median set} which will be either a single point or a single piece. But first of all, we need to introduce \emph{intervals} in tree-graded spaces (referred to as \emph{strict saturations} in \cite{DrutuSapirActions}).

\begin{definition}
Let $T$ be a tree-graded space. The \emph{interval}\index{Intervals in tree-graded spaces} between two points $x,y \in T$, denoted by $I(x,y)$, is the union of the geodesics between $x$ and $y$ together with the pieces they intersect along subsegments of positive lengths. The \emph{interval} between two subsets $A,B \subset T$ is $I(A,B):= \bigcup\limits_{a \in A, b \in B} I(a,b)$. 
\end{definition}

\noindent
Notice that the interval between two points $x,y$ in a tree-graded space coincides with the union of $\mathrm{Cutp}(x,y)$ with the pieces which are crossed by topological arcs joining $x$ and $y$ along subsegments of positive lengths. Moreover, every point $z$ of $\mathrm{Cutp}(x,y)$ separates $I(x,y)$ into exactly two connected components; more precisely, $I(x,y)= I(x,z) \cup I(z,y)$ with $I(x,z) \cap I(z,y)=\{z\}$. Next, we observe that intervals are examples of full subspaces:

\begin{lemma}\label{lem:IntervalClosed}
Let $T$ be a tree-graded space and $x,y \in T$ two points. The interval $I(x,y)$ is full.
\end{lemma}

\begin{proof}
Let $(z_n)$ be a sequence of points in $I(x,y)$ which converges in $T$ to a point $z$. First of all, notice that, if $z_n$ belongs to $\mathrm{Cutp}(x,y)$ for infinitely many $n$, then the following observation implies that $z$ belongs to $\mathrm{Cutp}(x,y) \subset I(x,y)$. 

\begin{claim}\label{claim:CutClosed}
The set $\mathrm{Cutp}(x,y)$ is closed in $T$.
\end{claim}

\noindent
Fix a geodesic $[x,y]$ between $x$ and $y$. Because $[x,y]$ is compact in $T$, it suffices to show that $\mathrm{Cutp}(x,y)$ is closed in $[x,y]$. Let $\{\sigma_i \mid i \in I\}$ denote all the maximal open subsegments of $[x,y]$ which have positive lengths and which lie in a piece. Because the intersection of any two distinct pieces contains at most a point, $\{\sigma_i \mid i \in I\}$ is a collection of pairwise disjoint open subsegments of $[x,y]$. As a consequence, the sequence of real numbers $(\mathrm{length}(\sigma_i))$ is summable (and its sum is bounded above by $\mathrm{length}([x,y])=d(x,y)$), which implies that $I$ is countable. We deduce from
$$\mathrm{Cutp}(x,y)= [x,y] \backslash \bigcup\limits_{i \in I} \sigma_i$$
that $\mathrm{Cutp}(x,y)$ is closed in $[x,y]$, concluding the proof of our claim.

\medskip \noindent
From now on, assume that $z_n$ never belongs to $\mathrm{Cutp}(x,y)$. So, for every $n \geq 1$, there exist a piece $P_n$ and two distinct points $a_n,b_n$ of a geodesic $[x,y]$ we fix such that $z_n \in P_n$ and such that $P_n \cap [x,y]$ coincides with the subsegment $[a_n,b_n] \subset [x,y]$ between $a_n$ and $b_n$. If $P_n$ takes the same value $P$ for infinitely many $n$, then $z \in P \subset I(x,y)$ because pieces are closed. From now on, assume that $P_n \neq P_m$ for all $n \neq m$. As a consequence, the subsegments $[a_n,b_n] \subset [x,y]$ are pairwise disjoint. So, up to extracting a subsequence, we may suppose without loss of generality that either $a_1<b_1<a_2<b_2< \cdots$ or $b_1>a_1>b_2>a_2> \cdots$ along $[x,y]$; the two cases being symmetric, assume that the former case happens. As a consequence of \cite[Corollary 2.9]{DrutuSapirTreeGraded}, $b_n$ belongs to a geodesic between $z_n$ and $z_{n+1}$, hence 
$$d(b_n,z) \leq d(b_n,z_n) +d (z_n,z) \leq d(z_n,z_{n+1}) +d(z_n,z) \underset{n \to + \infty}{\longrightarrow} 0.$$
So $z$ is also the limit of $(b_n)$. As $b_n$ belongs to $\mathrm{Cutp}(x,y)$ for every $n \geq 1$, we conclude from Claim \ref{claim:CutClosed} that $z$ belongs to $\mathrm{Cutp}(x,y) \subset I(x,y)$. 

\medskip \noindent
Thus, we have proved that $I(x,y)$ is closed. It is also connected because, since pieces are geodesic by definition, every point of $I(x,y)$ can be joined by an arc to a geodesic between $x$ and $y$. Finally, let $P$ be a piece satisfying $|P \cap I(x,y)| \geq 2$. Fix two distinct points $a,b \in P \cap I(x,y)$ and a geodesic $[a,b]$ between $a$ and $b$. Because pieces are convex \cite[Lemma 2.6]{DrutuSapirTreeGraded}, we know that $[a,b] \subset P$. Notice that two distinct points in $\mathrm{Cutp}(x,y)$ (and a geodesic between them) cannot belong to a common piece, so we cannot have $[a,b] \subset \mathrm{Cutp}(x,y)$. So $P$ must have an infinite intersection with a piece in $I(x,y)$, hence $P \subset I(x,y)$. Thus, we have proved that $I(x,y)$ is full, as desired.
\end{proof}

\noindent
Median sets will be defined thanks to the next proposition.

\begin{prop}\label{prop:Median}
Let $T$ be a tree-graded space and $x,y,z \in T$ three points. Either $I(x,y) \cap I(y,z) \cap I(x,z)$ is a single point $w$ and 
$$I(x,y) \cup I(y,z) \cup I(x,z)= \{w\} \sqcup \left( I(x,w) \backslash w \right) \sqcup \left( I(y,w) \backslash w \right) \sqcup \left( I(z,w) \backslash w \right);$$
or $I(x,y) \cap I(y,z) \cap I(x,z)$ is a single piece $P$ and
$$I(x,y) \cup I(y,z) \cup I(x,z)= P \sqcup \left( I(x,x') \backslash \{x'\} \right) \sqcup \left( I(y,y') \backslash \{y'\}\right) \sqcup \left( I(z,z') \backslash \{z'\} \right)$$
where $x',y',z'$ denote the projections of $x,y,z$ onto $P$.
\end{prop}

\begin{proof}
As a consequence of Lemma \ref{lem:IntervalClosed}, the projection $x'$ of $x$ onto $I(y,z)$ is well-defined. We know from Lemma \ref{prop:ProjectionFull} that $x'$ belongs to both $\mathrm{Cutp}(x,y)$ and $\mathrm{Cutp}(x,z)$, so
$$\left\{ \begin{array}{l} I(x,y)= \left( I(x,x') \backslash \{x'\} \right) \sqcup \{x'\} \sqcup \left( I(x',y) \backslash \{x'\} \right) \\ I(x,z)= \left( I(x,x') \backslash \{x'\} \right) \sqcup \{x'\} \sqcup \left( I(x',z) \backslash \{x'\} \right) \end{array} \right..$$
If $x' \in \mathrm{Cutp}(y,z)$, then we also have
$$I(y,z)= \left( I(y,x') \backslash x' \right) \sqcup \{x'\} \sqcup \left( I(x',y) \backslash x' \right).$$
Therefore, $I(x,y) \cap I(y,z) \cap I(x,z) = \{x'\}$ and
$$I(x,y) \cup I(y,z) \cup I(x,z)= \{w\} \sqcup \left( I(x,w) \backslash w \right) \sqcup \left( I(y,w) \backslash w \right) \sqcup \left( I(z,w) \backslash w \right)$$
if we set $w=x'$. Next, assume that $x' \notin \mathrm{Cutp}(y,z)$. So there exists a piece $P$ which contains $x'$ and which intersects some geodesic $[y,z]$ between $y$ and $z$ along a subsegment of positive length. Let $y',z'$ denote the endpoints of this subsegment, $y'$ being the point closest to $y$. Notice that
$$I(x',y) =  P \sqcup \left( I(y',y) \backslash \{y'\} \right) \text{ and } I(x',z)= P \sqcup \left( I(z',z) \backslash \{z'\} \right).$$ 
Because $[y,z]$ passes through $P$, it has to pass through the projection of $y$ onto $P$. More precisely, this projection has to be the first point along $[x,y]$ which belongs to $P$. And similarly for the projection of $z$. It follows that $y'$ and $z'$ coincide with the projections of $y$ and $z$ onto $P$. Consequently,
$$I(y,z)= \left( I(y,y') \backslash \{y'\} \right) \sqcup P \sqcup \left( I(z',z) \backslash \{z'\} \right).$$
We conclude that $I(x,y) \cap I(y,z) \cap I(x,z)=P$ and that 
$$I(x,y) \cup I(y,z) \cup I(x,z)= P \sqcup \left( I(x,x') \backslash \{x'\} \right) \sqcup \left( I(y,y') \backslash \{y'\}\right) \sqcup \left( I(z,z') \backslash \{z'\} \right),$$
concluding the proof of our proposition.
\end{proof}

\noindent
We record the definition of median sets provided by Proposition \ref{prop:Median}. See Figure \ref{Median}.
\begin{figure}
\begin{center}
\includegraphics[scale=0.4]{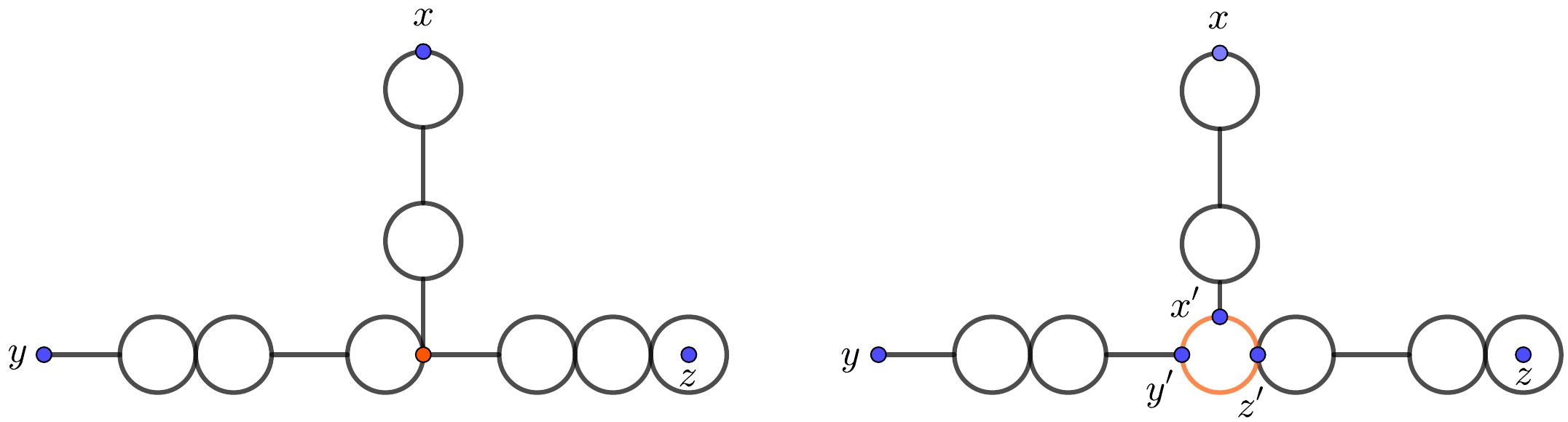}
\caption{The union $I(x,y) \cup I(y,z) \cup I(x,z)$ and the median set $\mu(x,y,z)$.}
\label{Median}
\end{center}
\end{figure}

\begin{definition}
Let $T$ be a tree-graded space and $x,y,z \in T$ three points. The intersection
$$\mu(x,y,z)= I(x,y) \cap I(y,z) \cap I(x,z)$$
is the \emph{median set}\index{Median sets} of the triple $x,y,z$. More generally, for arbitrary subsets $A,B,C \subset T$, we denote
$$\mu(A,B,C)= I(A,B) \cap I(B,C) \cap I(A,C).$$
\end{definition}

\noindent
The following observation will be useful in the sequel.

\begin{lemma}\label{lem:MedianPieces}
Let $T$ be a tree-graded space. If $P,Q,R$ are points or pieces, then $\mu(P,Q,R)$ is either a single point or a single piece.
\end{lemma}

\begin{proof}
If $P=Q$, then $\mu(P,Q,R)= I(P,P) \cap I(P,R) = P$ and we are done. Similarly, if $P=R$ then $\mu(P,Q,R)=P$, and if $Q=R$ then $\mu(P,Q,R)=Q$. From now on, we assume that $P,Q,R$ are pairwise distinct. 

\medskip \noindent
Assume that $Q$ and $R$ have distinct projections onto $P$. Let $q$ (resp. $r$) denote the projection of $R$ onto $Q$ (resp. of $Q$ onto $R$). Also, let  $q'$ (resp. $r'$) denote the projection of $Q$ onto $P$ (resp. of $R$ onto $P$). Then
$$I(Q,R) = \left( Q \backslash \{q\} \right) \sqcup I(q,q') \sqcup \left( P \backslash \{q',r'\} \right) \sqcup I(r',r) \sqcup \left( R \backslash \{r\} \right).$$
We conclude that the median set $\mu(P,Q,R)=P$ is a single point or a single piece. We argue similarly if the projections of $P,Q$ onto $R$ are distinct or if the projections of $P,R$ onto $Q$ are distinct.  From now on, we assume that the projections of $P,Q$ onto $R$ (resp. $P,R$ onto $Q$; $Q,R$ onto $P$) coincide.

\medskip \noindent
Let $p$ (resp. $q$, $r$) denote the projection of $Q,R$ onto $P$ (resp. $P,R$ onto $Q$; $P,Q$ onto $R$). We have
$$\left\{ \begin{array}{l} I(P,Q) = \left( P \backslash \{p\} \right) \sqcup I(p,q) \sqcup \left( Q \backslash \{q\} \right) \\ I(P,R)= \left( P \backslash \{p\} \right) \sqcup I(p,r) \sqcup \left( R \backslash \{r\} \right) \\ I(Q,R) = \left( Q \backslash \{q\} \right) \sqcup I(q,r) \sqcup \left( R \backslash \{r\} \right) \end{array} \right..$$
The fact that the projections of $Q,R$ onto $P$ coincide implies that $P \cap I(q,r) \subset \{p\}$. Similarly, $Q \cap I(p,r) \subset \{q\}$ and $R \cap I(p,q)  \subset \{ r\}$. Therefore,
$$\mu(P,Q,R) = I(p,q) \cap I(q,r) \cap I(p,r) = \mu(p,q,r).$$
We conclude that $\mu(P,Q,R)$ is either a single point or a single piece, as desired.
\end{proof}

\paragraph{Elliptic isometries.} In the sequel, we refer to \emph{elliptic isometries} as isometries with bounded orbits. In \cite{MR2783135}, it is proved that elliptic isometries stabilise points or pieces (setwise). More generally, \cite[Lemma 3.12]{MR2783135} states that:

\begin{prop}\label{prop:FixedGraded}
If a finitely generated group acts on a tree-graded space with bounded orbits, then it has to fix a point or to stabilise a piece.
\end{prop}

\noindent
This statement essentially follows from the cyclic case, i.e., when we consider a single elliptic isometry. The construction of the point or the piece which is stabilised by such an isometry will be useful in the next subsection, so we record it below.

\begin{definition}
Let $T$ be a tree-graded space and $x,y \in T$ two points. The \emph{middle-set}\index{Middle-sets in tree-graded spaces} of $x$ and $y$, denoted by $M(x,y)$, is defined as follows. 
\begin{itemize}
	\item If $x=y$, then $M(x,y)= \{x\}$. 
	\item If $x \neq y$ and if there exists a point $p \in \mathrm{Cutp}(x,y)$ at equal distance from $x$ and $y$, then $M(x,y)=\{p\}$.
	\item If $x \neq y$ and if no point in $\mathrm{Cutp}(x,y)$ is at equal distance from $x$ and $y$, then $M(x,y)$ is the unique piece containing $\{a,b\}$, where $a$ (resp. $b$) is the point in $\mathrm{Cutp}(x,y)$ satisfying $d(x,a) \leq d(x,y)/2$ (resp. $d(y,b) \leq d(x,y)/2$) which is the farthest from $x$ (resp. from $y$).
\end{itemize}
(In the latter case, the existence of the piece is justified by \cite[Lemma 3.9]{MR2783135}.)
\end{definition}

\noindent
Now, \cite[Lemmas 3.10 and 3.11]{MR2783135} shows that:

\begin{prop}\label{prop:MiddleSet}
Let $T$ be a tree-graded space and $g \in \mathrm{Isom}(T)$ an elliptic isometry. For every point $x \in T$, either $gx=x$ or $g$ stabilises the middle-set $M(x,gx)$. Moreover, if $x \neq gx$ and if $M(x,gx)$ is a piece, then it contains all the points equidistant to $x$ and~$gx$.
\end{prop}

\noindent
Notice that, if we fix a geodesic $[x,y]$ between two points $x$ and $y$ in a tree-graded space, then the middle-set $M(x,y)$ is the middle-point of $[x,y]$ if it belongs to $\mathrm{Cutp}(x,y)$, or the unique piece in $I(x,y)$ which contains the middle-point of $[x,y]$.

\medskip \noindent
In order to understand the dynamics of an elliptic isometry, we need to introduce the following definition:

\begin{definition}
Let $T$ be a tree-graded space and $g \in \mathrm{Isom}(T)$ an elliptic isometry. The \emph{elliptic set}\index{Elliptic set} of $g$, denoted by $\mathrm{Ell}(g)$, is the union of all the points and pieces stabilised by $g$.
\end{definition}

\noindent
As motivated by our next proposition, elliptic sets of elliptic isometries in tree-graded spaces play the same role as sets of fixed points of elliptic isometries in real trees. See Figure \ref{Elliptic}.
\begin{figure}
\begin{center}
\includegraphics[scale=0.4]{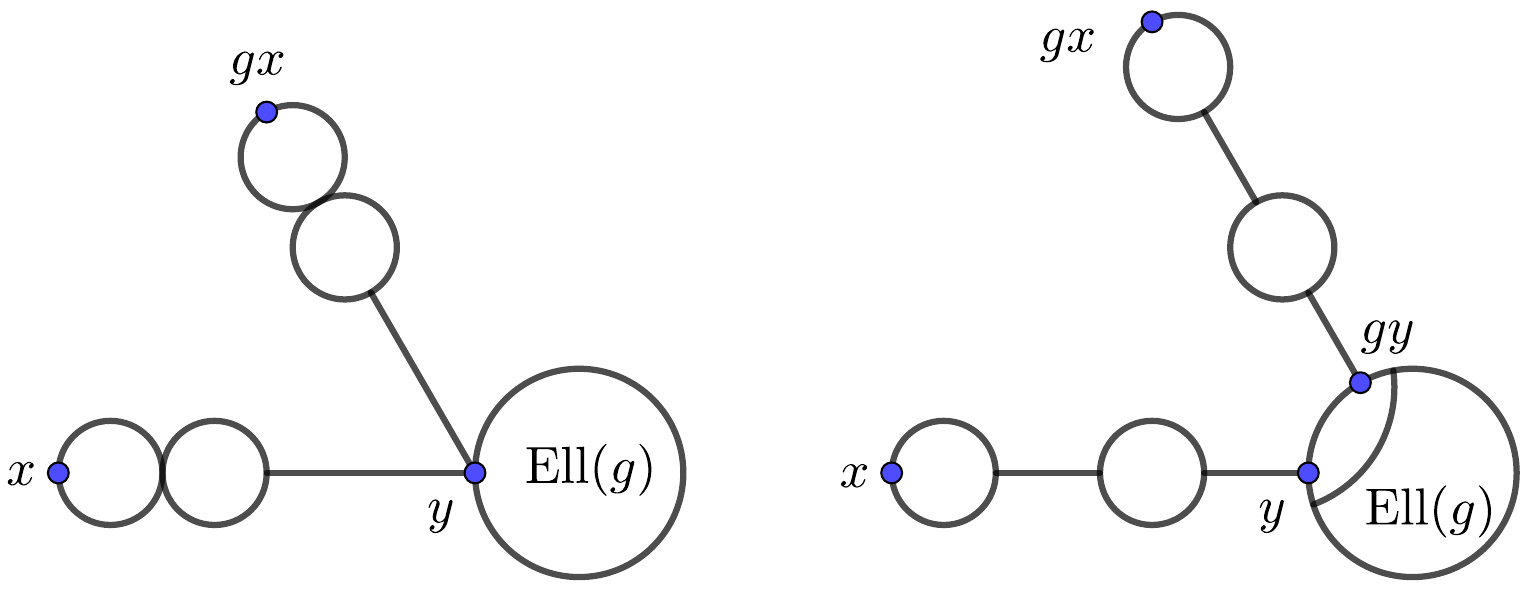}
\caption{}
\label{Elliptic}
\end{center}
\end{figure}

\begin{prop}\label{prop:TreeGradedEllipticSet}
Let $T$ be a tree-graded space, $g \in \mathrm{Isom}(T)$ an elliptic isometry, and $x \in T$ a point. If $y$ denotes the projection of $x$ onto $\mathrm{Ell}(g)$, then  
$$I(x,gx)= \left( I(x,y) \backslash \{y\} \right) \sqcup I(y,gy) \sqcup \left( I(gy,gx) \backslash \{gy\} \right).$$
\end{prop}

\noindent
We begin by observing that elliptic sets are full.

\begin{lemma}
Let $T$ be a tree-graded space and $g \in \mathrm{Isom}(T)$ an elliptic isometry. Then $\mathrm{Ell}(g)$ is non-empty and full.
\end{lemma}

\begin{proof}
It follows from Proposition \ref{prop:FixedGraded} that $\mathrm{Ell}(g)$ is non-empty. Also, notice that, if $x$ and $y$ are two distinct points or pieces stabilised by $g$, then $g$ stabilises the projection $x'$ of $y$ onto $x$ and the projection $y'$ of $x$ onto $y$, hence $x \cup I(x',y') \cup y \subset \mathrm{Ell}(g)$. Therefore, $\mathrm{Ell}(g)$ must be connected. Now, let $P$ be a piece containing two distinct points $x,y \in \mathrm{Ell}(g)$. Four cases may happen.
\begin{itemize}
	\item Assume that $x$ and $y$ are both fixed by $g$. Then $gP$ is the unique piece containing $gx=x$ and $gy=y$, hence $gP=P$.
	\item Assume that $g$ fixes $x$ but not $y$. There must exist a piece $Q$ stabilised by $g$ which contains $y$. If $Q \neq P$, then $P \cap Q = \{y\}$, which implies that $y$ is the projection of $x$ onto $Q$. As $gy$ belongs to $Q$, necessarily $y$ belongs to a geodesic between $x$ and $gy$, hence $d(x,gy)>d(x,y) = d(gx,gy)=d(x,gy)$, a contradiction. Therefore, we must have $P=Q$, and $P$ is stabilised by $g$.
	\item If $g$ fixes $y$ but not $x$, the configuration is symmetric to the previous case.
	\item Assume that $g$ fixes neither $x$ nor $y$. There must exist two pieces $Q,R$ stabilised by $g$ which contain $x,y$ respectively. If $Q=R$, then this piece contains $x,y \in P$, hence $P=Q=R$. If $Q \neq R$, then $g$ fixes the projection $x'$ of $x$ onto $R$ and the projection $y'$ of $y$ onto $Q$. Notice that, because $x'$ and $y'$ belong to a geodesic between $x$ and $y$, necessarily $x',y' \in P$ by convexity of $P$. It follows from the first case we have considered that $P$ is stabilised by $g$.
\end{itemize}
Thus, we have proved that $P$ is stabilised by $g$ in any case, hence $P \subset \mathrm{Ell}(g)$. It remains to show that $\mathrm{Ell}(g)$ is closed. So let $(x_n)$ be a sequence of points in $\mathrm{Ell}(g)$ which converges to a point $x$ in $T$. If $x_n$ is fixed by $g$ for infinitely many $n$, then $g$ has to fix $x$, hence $x \in \mathrm{Ell}(g)$. From now on, we assume that, for every $n \geq 1$, $g$ does not fix $x_n$; as a consequence, there exists a piece $P_n$ stabilised by $g$ which contains $x_n$. If $P_n$ takes the same value $P$ for infinitely many $n$, then $x \in P \subset \mathrm{Ell}(g)$ because pieces are closed. From now on, we assume that $P_n \neq P_m$ for all $n\neq m$. For every $n \geq 1$, let $y_n$ denote the projection of $P_n$ onto $P_{n+1}$. Because $P_n$ and $P_{n+1}$ are both stabilised by $g$, it follows that $g$ fixes $y_n$. Moreover, $y_n$ is also the projection of $x_n$ onto $P_{n+1}$, hence 
$$d(x_n,y_n) \leq d(x_n,x_{n+1}) \underset{n \to + \infty}{\longrightarrow} 0.$$
Consequently, $x$ is a limit of points fixed by $g$, namely $(y_n)$, which implies that $x$ is fixed by $g$ and so belongs to $\mathrm{Ell}(g)$. Thus, we have proved that $\mathrm{Ell}(g)$ is closed.
\end{proof}

\begin{proof}[Proof of Proposition \ref{prop:TreeGradedEllipticSet}.]
First, assume that $y$ is not fixed by $g$. So there exists a piece $P \subset \mathrm{Ell}(g)$ stabilised by $g$ which contains $y$. Notice that $y$ is also the projection of $x$ onto $P$. Moreover, $gy$ is the projection of $gx$ onto $gP=P$. As $y \neq gy$, it follows that
$$I(x,gx)= \left( I(x,y) \backslash \{y\} \right) \sqcup \underset{=I(y,gy)}{\underbrace{P}} \sqcup \left( I(gy,gx) \backslash \{gy\} \right).$$
Next, assume that $y$ is fixed by $g$. Fix two geodesics $[x,y]$ and $[y,gx]$. If $y \notin \mathrm{Cutp}(x,gx)$ then two cases may happen: either $[x,y] \cup [y,gx]$ is a geodesic, and there exists a piece $P$ which intersects $[x,y] \cup [y,gx]$ along a subsegment of positive length having $y$ in its interior; or $[x,y] \cup [y,gx]$ is not a geodesic. In the latter case, there exist two points $a \in [x,y]$ and $b \in [y,gx]$ such that $[x,a] \cup [a,b] \cup [b,gx]$ is a geodesic and such that $[a,y] \cup [y,b] \cup [b,a]$ is a simple loop, where $[a,b]$ is a geodesic we fix and where $[x,a],[b,gx]$ are subsegments of $[x,y],[y,gx]$ between the corresponding points. It follows from Lemma \ref{lem:Loop} that there exists a piece $P$ containing $[x,a] \cup [a,y] \cup [b,a]$. Therefore, if $y$ does not belong to $\mathrm{Cutp}(x,gx)$, then there exists a piece $P$ which intersects $[x,y] \cup [y,gx]$ along a subsegment of positive length having $y$ in its interior. Fix a point $z \in [x,y]$ which is sufficiently close to $y$ so that $z \in P$ and $gz \in P$. Because $P \cap gP$ contains both $gz$ and $gy=y$, necessarily $P=gP$, hence $P \subset \mathrm{Ell}(g)$. But then $[x,y]$ contains a point of $\mathrm{Ell}(g)$ different from $y$, contradicting the fact that $y$ is the projection of $x$ onto $\mathrm{Ell}(g)$. Consequently, we must have $y \in \mathrm{Cutp}(x,gx)$, hence
$$I(x,gx)= \left( I(x,y) \backslash \{y\} \right) \sqcup  \underset{=I(y,gy)}{\underbrace{\{y\}}} \sqcup \underset{=I(gy,gx) \backslash \{gy\}}{\underbrace{\left( I(y,gx) \backslash \{y\} \right)}}.$$
This completes the proof of our lemma.
\end{proof}

\subsection{Products of tree-graded spaces}\label{section:ProductsTreeGraded}

\noindent
In this subsection, our goal is to prove a fixed-point theorem in products of tree-graded spaces. Before stating our result, we need to introduce some terminology.

\begin{definition}
Let $T_1, \ldots, T_n$ be a collection of tree-graded spaces. A \emph{piece} in the product $T_1 \times \cdots \times T_n$ is a subset $P_1 \times \cdots \times P_n$ where each $P_i \subset T_i$ is either a piece or a single point. The \emph{median set} of three pieces $P= P_1 \times \cdots \times P_n$, $Q= Q_1 \times \cdots \times Q_n$ and $R=R_1 \times \cdots \times R_n$ is 
$$\mu(P,Q,R):= \mu(P_1,Q_1,R_1) \times \cdots \times \mu(P_n,Q_n,R_n).$$ 
A subspace $M \subset T_1 \times \cdots \times T_n$ is \emph{median} if, for all points $x,y,z \in M$, the inclusion $\mu(x,y,z) \subset M$ holds. 
\end{definition}

\noindent
Notice that, as a consequence of Lemma \ref{lem:MedianPieces}, the median set of three pieces is again a piece. We record the following observation for future use:

\begin{lemma}\label{lem:Median}
Let $T_1, \ldots, T_n$ be a collection of tree-graded spaces, and $M \subset T_1 \times \cdots \times T_n$ a median subspace. For all pieces $P,Q,R \subset M$, we have $\mu(P,Q,R) \subset M$.
\end{lemma}

\begin{proof}
First, we write our pieces
$$P=P_1 \times \cdots \times P_n, \ Q=Q_1 \times \cdots \times Q_n, \ R= R_1 \times \cdots \times R_n$$
as products of pieces. As a consequence of the claim which follows, for every $1 \leq i \leq n$, there exist $p_i \in P_i$, $q_i \in Q_i$ and $r_i \in R_i$ such that $\mu(P_i,Q_i,R_i)= \mu(p_i,q_i,r_i)$. Therefore,
$$\begin{array}{lcl} \mu(P,Q,R) & = & \mu(P_1,Q_1,R_1) \times \cdots \times \mu(P_n,Q_n,R_n)= \mu(p_1,q_1,r_1) \times \cdots \times \mu(p_n,q_n,r_n) \\ \\ & = & \mu ( (p_1, \ldots, p_n), (q_1, \ldots, q_n), (r_1,\ldots, r_n)) \end{array}$$
Notice that $(p_1, \ldots, p_n) \in P$, $(q_1, \ldots, q_n) \in Q$ and $(r_1,\ldots, r_n) \in R$. Because $M$ is median, we conclude that $\mu(P,Q,R) \subset M$ as desired.

\begin{claim}
Let $T$ be a tree-graded space and $P,Q,R \subset T$ three points or pieces. There exist three points $p \in P$, $q \in Q$ and $r \in R$ such that $\mu(P,Q,R)= \mu(p,q,r)$.
\end{claim}

\noindent
First, assume that two elements among $P,Q,R$ coincide, say $P$ and $Q$. If $P=Q$ is a point, then $\mu(P,Q,R)=P=\mu(P,Q,r)$ where $r \in R$ is an arbtrary point, and we are done. Otherwise, if $P=Q$ is a piece which is not a point, then $\mu(P,Q,R)=P=\mu(p,q,r)$ where $r \in R$ is an arbitrary point and where $p,q \in P$ are two distinct points. From now on, we suppose that $P,Q,R$ are pairwise distinct.

\medskip \noindent
Next, assume that one element among $P,Q,R$ lies in the interval of the other two, say $P \subset I(Q,R)$. Notice that $I(Q,R)=Q \cup I(q,r) \cup R$ where $q$ denotes the projection of $R$ onto $Q$ and $r$ the projection of $Q$ onto $R$. If $P \subset Q$, then $P$ is a point lying in the piece $Q$ which is not a single point. Either $P=\{q\}$, and $\mu(P,Q,R)= P= \mu(P,q,r)$; or $P \neq \{q\}$, and $\mu(P,Q,R)=Q = \mu(P,q,r)$. We argue similarly if $P \subset R$. From now on, we assume that $P \subset I(q,r)$. Three cases may happen: 
\begin{itemize}
	\item $P$ is a point in $\mathrm{Cutp}(q,r)$, and we have $\mu(P,Q,R)= P=\mu(P,q,r)$; 
	\item or $P$ is a point not in $\mathrm{Cutp}(q,r)$, i.e., there exists a piece $S$ containing $P$ and intersecting a geodesic between $q$ and $r$ along a subsegment of positive length, and we have $\mu(P,Q,R)=S= \mu(P,q,r)$; 
	\item or $P$ is not a point, i.e., it is a piece intersecting a geodesic between $q$ and $r$ along a subsegment $\sigma$ of positive length, and we have $\mu(P,Q,R)=P= \mu(p,q,r)$ where $p$ is an arbitrary point in the interior of $\sigma$. 
\end{itemize}
From now on, we suppose that no element among $P,Q,R$ lies in the interval of the other two.

\medskip \noindent
As a consequence of our last assumption, the projections of two elements among $P,Q,R$ onto the third one coincide. Let $p \in P$ denote the projection of $Q$ and $R$ onto $P$, $q \in Q$ the projection of $P$ and $R$ onto $Q$, and $r \in R$ the projection of $P$ and $Q$ onto $R$. Notice that
$$\left\{ \begin{array}{l} I(P,Q)= \left( P \backslash \{p\} \right) \sqcup I(p,q) \sqcup \left( Q \backslash \{q\} \right) \\ I(P,R)= \left( P \backslash \{p\} \right) \sqcup I(p,r) \sqcup \left( R \backslash \{r\} \right) \\ I(Q,R)= \left( Q \backslash \{q\} \right) \sqcup I(q,r) \sqcup \left( R \backslash \{r\} \right) \end{array} \right..$$
Because $(P \backslash \{p\} ) \cap I(Q,R)= \emptyset$, $(Q \backslash \{q\} ) \cap I(P,R)= \emptyset$ and $(R \backslash \{r\} ) \cap I(P,Q)= \emptyset$, we deduce that
$$\mu(P,Q,R)= I(P,Q) \cap I(P,R) \cap I(Q,R) = I(p,q) \cap I(p,r) \cap I(q,r)= \mu(p,q,r),$$
concluding the proof of our claim.
\end{proof}

\noindent
The main result of this subsection is the following statement:

\begin{prop}\label{prop:ProdTreeGraded}
Let $G$ be a finitely generated group and let $T_1, \ldots, T_n$ be a family of tree-graded spaces on which $G$ acts. Assume that $T_1 \times \cdots \times T_n$ contains a $G$-invariant subspace $M$ which is geodesic and median. If $G$ has bounded orbits in $T_1 \times \cdots \times T_n$, then it stabilises a piece which lies in $M$.
\end{prop}

%
%

\noindent
The key step towards the proof of Proposition \ref{prop:ProdTreeGraded} is to understand how to find,  in a tree-graded space, a fixed point or a stabilised piece starting from an arbitrary point and using only the median operation. This is done in the next three preliminary lemmas.

\begin{lemma}\label{lem:FixOnePoint}
Let $g$ be an elliptic isometry of a tree-graded space $T$. For every point $x \in T$, $g^2$ stabilises $\mu(x,gx,g^2x)$.
\end{lemma}

\begin{proof}
If $x=gx$ or $x=g^2x$, then $g^2$ stabilises $\mu(x,gx,g^2x)=\{x\}$. From now on, we assume that $x$, $gx$ and $g^2x$ are pairwise distinct. Fix a geodesic $[x,gx]$ from $x$ to $gx$. Of course, $g[x,gx]$ is a geodesic from $gx$ to $g^2x$. Let $p$ denote the last point along this geodesic which belongs to $[x,gx]$. By construction, $p$ belongs to $I(x,gx) \cap I(gx,g^2x)$. Notice that
$$d(x,p)+d(p,gx)= d(x,gx)= d(gx,g^2x)=d(gx,p)+d(p,g^2x)$$
hence $d(x,p)= d(p,g^2x)$. So, according to Proposition \ref{prop:MiddleSet}, $p$ belongs to the middle-set $M(x,g^2x)$. We distinguish three cases.

\medskip \noindent
First, assume that $p \in \mathrm{Cutp}(x,g^2x)$. Then $\mu(x,gx,g^2x)= \{p\} = M(x,g^2x)$. The desired conclusion follows from Proposition \ref{prop:MiddleSet}.

\medskip \noindent
Second, assume that $p \notin \mathrm{Cutp}(x,g^2x)$ but $[x,p] \cup [p,g^2x]$ is a geodesic, where $[x,p]$ and $[p,g^2x]$ denote the subsegments of $[x,gx]$ and $g[x,gx]$ between the corresponding points. So there exists a piece $P$ which contains a subsegment of $[x,p] \cup [p,g^2x]$ having $p$ in its interior. Because $P$ intersects geodesics between $x$, $gx$ and $g^2x$ along subsegments of positive lengths (namely, $[x,p]$, $[p,g^2x]$ and $[x,p] \cup [p,g^2x]$), necessarily $P \subset I(x,gx) \cap I(gx,g^2x) \cap I(x,g^2x)$, hence $\mu(x,gx,g^2x)=P$. Moreover, as $p$ is the middle point of $[x,p] \cup [p,g^2x]$, necessarily $P$ is the middle-piece of $x$ and $g^2x$. The desired conclusion follows from Proposition \ref{prop:MiddleSet}.

\medskip \noindent
Third, assume that $[x,p] \cup [p,g^2x]$ is not a geodesic. So there exist two distinct points $a \in [x,p]$ and $b \in [p,g^2x]$ such that $[x,a]\cup[a,b] \cup[b,g^2x]$ is a geodesic and such that $[a,p] \cup [p,b] \cup [a,b]$ defines a simple loop, where $[a,b]$ is a geodesic between $a,b$ we fix and $[a,b],[p,b]$ are the subsegments of $[x,p], [p,g^2x]$ between the corresponding points. According to Lemma \ref{lem:Loop}, there exists a piece $P$ which contains the loop $[a,p] \cup [p,b] \cup [a,b]$. Because $P$ intersects geodesics between $x$, $gx$ and $g^2x$ along subsegments of positive lengths (namely, $[x,p]$, $[p,g^2x]$ and $[x,a] \cup [a,b] \cup [b,g^2x]$), necessarily $P \subset I(x,gx) \cap I(gx,g^2x) \cap I(x,g^2x)$, hence $\mu(x,gx,g^2x)=P$. Next, notice that
$$d(x,a)+d(a,p)+d(p,gx)=d(x,gx)=d(gx,g^2x)=d(gx,p)+d(p,b)+d(b,g^2x),$$
hence
$$\left\{ \begin{array}{l} d(x,a)= d(g^2x,b) + d(p,b)-d(a,p) \leq d(g^2x,b)+d(a,b) \\ d(g^2x,b)=d(x,a) + d(a,p)-d(p,b) \leq d(x,a)+d(a,b) \end{array} \right..$$
As a consequence, the midpoint of the geodesic $[x,a] \cup [a,b] \cup [b,g^2x]$ belongs to $[a,b] \subset P$. As $P$ contains two points which are equidistant to $x$ and $g^2x$, we deduce from Proposition~\ref{prop:MiddleSet} that $P=M(x,g^2x)$. The desired conclusion follows from Proposition~\ref{prop:MiddleSet}.
\end{proof}

\begin{lemma}\label{lem:ForFixedPoint}
Let $T$ be a tree-graded space, $g\in \mathrm{Isom}(T)$ an isometry and $a \in T$ a point fixed by $g^2$. For every geodesic $[a,ga]$ and every point $c \in [a,ga]$, $g^2$ stabilises $\mu(a,c,gc)$. Moreover, if $c$ is the middle-point of $[a,ga]$, then $g$ stabilises $\mu(a,c,gc)$.
\end{lemma}

\begin{proof}
First, assume that $c \in \mathrm{Cutp}(a,ga)$. Because $c$ and $g^2c$ are two points in $\mathrm{Cutp}(a,ga)$ at equal distance from $a$, necessarily $c=g^2c$. Moreover, $\mu(a,c,gc)=c$ or $gc$ depending on whether $d(a,c) \leq d(a,gc)$ or $d(a,c)>d(a,gc)$. So $\mu(a,c,gc)$ is stabilised by $g^2$, but also by $g$ if $c$ is the midpoint of $[a,ga]$.

\medskip \noindent
Next, assume that $c \notin \mathrm{Cutp}(a,ga)$. So there exists a piece $P$ which contains $c$ and which intersects $[a,ga]$ along a subsegment of positive length. If $gc=c$ then $\mu(a,c,gc)=c$ is stabilised by $g$ and $g^2$, so from now on we suppose that $gc \neq c$. If $P=gP$, then $\mu(a,c,gc)=P$ is stabilised by $g$ and $g^2$, so from now on we also suppose that $gP \neq P$. Because $P$ and $gP$ are two distinct pieces in $I(a,ga)$, we must have 
$$I(a,ga)= I(a, r) \cup P \cup I(s,t) \cup gP \cup I(u,ga)$$ 
where $r$ denotes the projection of $a$ onto $P$, $s$ the projection of $gP$ onto $P$, $t$ the projection of $P$ onto $gP$, and $u$ the projection of $ga$ onto $gP$; or
$$I(a,ga)= I(a, r) \cup gP \cup I(s,t) \cup P \cup I(u,ga)$$ 
where $r$ denotes the projection of $a$ onto $gP$, $s$ the projection of $P$ onto $gP$, $t$ the projection of $gP$ onto $P$, and $u$ the projection of $ga$ onto $P$. Notice that, in both cases, $c$ cannot be the midpoint of $[a,ga]$. In the former case, $P$ is the unique piece in $I(a,ga)$ containing all the points in $I(a,ga)$ at distance $d(a,c)$ from $a$ and $\mu(a,c,gc)=P$, so $g^2$ stabilises $\mu(a,c,gc)$. And in the latter case, $gP$ is the unique piece in $I(a,ga)$ containing all the points in $I(a,ga)$ at distance $d(a,c)$ from $a$ and $\mu(a,c,gc)=gP$, so $g^2$ stabilises $\mu(a,c,gc)$. 
\end{proof}

\begin{lemma}\label{lem:GlobalFixedPoint}
Let $T$ be a tree-graded space and $g,h \in \mathrm{Isom}(T)$ two isometries. Assume that $g$ (resp. $h$) stabilises a point or a piece $x$ (resp. $y$). If $gh$ is elliptic, then $g$ and $h$ both stabilise $\mu(x,hx,y)$. 
\end{lemma}

\begin{proof}
We distinguish three cases.

\medskip \noindent
\textbf{Case 1:} $x \subset \mathrm{Ell}(h)$.

\medskip \noindent
If $x$ is stabilised by $h$, then $\mu(x,hx,y)=\mu(x,x,y)$ coincides with $x$, and so is stabilised by $h$. From now on, assume that $x$ is a point which belongs to a piece $P$ stabilised by $h$ but that $h x \neq x$. Notice that $\mu(x,hx,y)=P$ unless $x$ or $hx$ coincides with the projection of $y$ onto $P$. But this is impossible since this projection must be fixed by $h$. So we need to show that $P$ is stabilised by $g$. According to Proposition \ref{prop:FixedGraded}, there exists a point or a piece $Q$ which is stabilised by both $g$ and $h$. If $Q=P$, then we are done. Otherwise, $g$ and $h$ have to fix the projection $p$ of $Q$ onto $P$. Because $h$ does not fix $x$, necessarily $p \neq x$. Therefore, $g$ fixes two distinct points of $P$, namely $p$ and $x$. We conclude that $g$ stabilises $P$.

\medskip \noindent
\textbf{Case 2:} $x \nsubseteq \mathrm{Ell}(h)$ and the projection of $x$ onto $\mathrm{Ell}(h)$ is a point $p$ fixed by $h$.

\medskip \noindent
Because $y \subset \mathrm{Ell}(h)$, it follows from Proposition \ref{prop:TreeGradedEllipticSet} that
$$\begin{array}{lcl} \mu(x,hx,y) & = & I(x,y) \cap I(x,hx) \cap I(y,hx) \\ \\ & = & (I(x,p) \cup I(p,y)) \cap (I(x,p) \cup I(p,hx)) \cap (I(y,p) \cup I(p,hx)) \end{array}$$
See Figure \ref{EllTwoCases}(a). But we know from Proposition \ref{prop:TreeGradedEllipticSet} that $I(x,p) \cap I(p,hx)= \{p\}$, and of course $I(x,p) \cap I(p,y)= \{p\}$ and $I(hx,p) \cap I(p,y)= \{p\}$; so $\mu(x,hx,y)=p$. We already know that $p$ is fixed by $h$. As a consequence of Proposition \ref{prop:FixedGraded}, there exists a point or a piece $z$ stabilised by $g$ which lies in $\mathrm{Ell}(h)$. According to Proposition~\ref{prop:ProjectionFull}, $p \in \mathrm{Cutp}(x,z)$. Because $x$ and $z$ are both stabilised by $g$, it follows that $g$ also stabilises~$p$. 
\begin{figure}
\begin{center}
\includegraphics[scale=0.4]{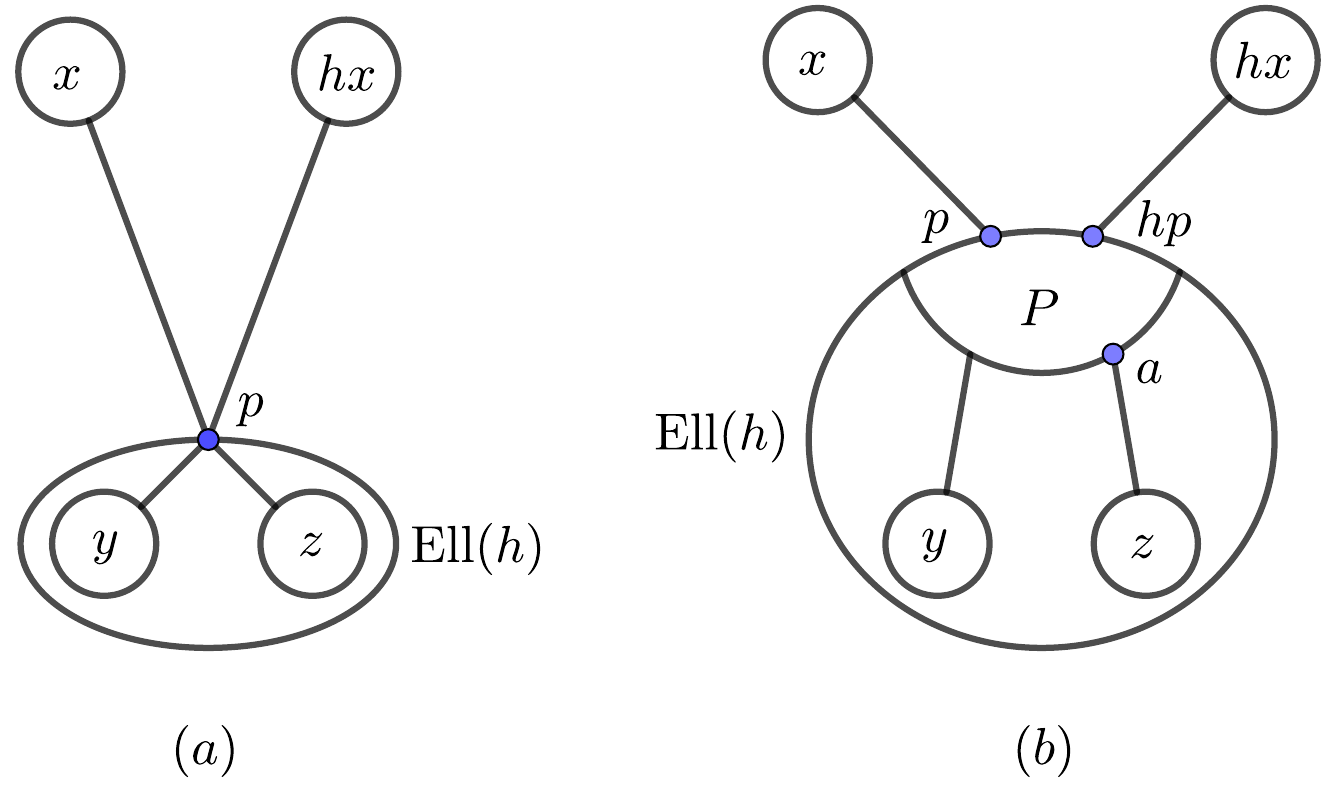}
\caption{Cases 2 and 3 from the proof of Lemma \ref{lem:GlobalFixedPoint}.}
\label{EllTwoCases}
\end{center}
\end{figure}

\medskip \noindent
\textbf{Case 3:} $x \nsubseteq \mathrm{Ell}(h)$ and the projection of $x$ onto $\mathrm{Ell}(h)$ is a point $p$ not fixed by $h$.

\medskip \noindent
Let $P$ denote the piece in $\mathrm{Ell}(h)$ which contains $p$ and $hp$. As before, it follows from Proposition \ref{prop:TreeGradedEllipticSet} that
$$\begin{array}{lcl} \mu(x,hx,y) & = & I(x,y) \cap I(x,hx) \cap I(y,hx) \\ \\ & = & (I(x,p) \cup I(p,y)) \cap (I(x,p) \cup P \cup I(hp,hx)) \cap (I(y,hp) \cup I(hp,hx)) \end{array}$$
hence $\mu(x,hx,y)=P$. See Figure \ref{EllTwoCases}(b). We already know that $P$ is stabilised by $h$. According to Proposition \ref{prop:FixedGraded}, there exists a point or a piece $Q \subset \mathrm{Ell}(h)$ which is stabilised by both $g$ and $h$. If $Q=P$, then we are done, so assume that $Q \neq P$. Because $h$ stabilises both $Q$ and $P$, it has to fix the projection $a$ of $Q$ onto $P$. Notice that, since $h$ does not fix $p$, necessarily $a \neq p$. As a consequence, $P \subset I(x,Q) \subset \mathrm{Ell}(g)$. So $g$ stabilises $P$ as desired.
\end{proof}

\begin{proof}[Proof of Proposition \ref{prop:ProdTreeGraded}.]
Assume that $G$ acts on $T$ with bounded orbits, and fix an arbitrary element $g \in G$. Notice that $g^2$ stabilises a piece in $M$. Indeed, if we fix an arbitrary point $x \in M$, then it follows from Lemma~\ref{lem:FixOnePoint} that $g^2$ stabilises the piece $\mu(x,gx,g^2x)$, which lies in $M$ because $M$ is median. Our fist goal is to show that $g$ also stabilises a piece in $M$. This assertion will be a direct consequence of the next claim where, given a piece $P=P_1 \times \cdots \times P_n$ stabilised by $g^2$, we define its \emph{complexity} as 
$$\chi(P):= \left( \#\{ 1 \leq i \leq n \mid gP_i \neq P_i\}, \ \# \{ 1 \leq i \leq n \mid \text{$P_i$ is not a point} \} \right),$$
ordered with respect to the lexicographic order.

\begin{claim}\label{claim:PieceComplexity}
If there exists a piece $Q \subset M$ stabilised by $g^2$ with complexity $\chi(Q) \geq (1,0)$, then there exists a piece in $M$ stabilised by $g^2$ with complexity $<\chi(Q)$. 
\end{claim}

\noindent
Write $Q$ as $P_1 \times \cdots \times P_r \times \{q_1\} \times \cdots \times \{q_s\} \times A_1 \times \cdots A_t$ where each $P_i$ is a piece stabilised by $g$, each $q_i$ is a point stabilised by $g$, each $A_i$ is a point or a piece stabilised by $g^2$ but not by $g$. Notice that, if there exists some $1 \leq i \leq t$ such that $A_i$ is not a single point, then 
$$R:= P_1 \times \cdots \times P_r \times \{q_1\} \times \cdots \times \{q_s\} \times A_1 \times \cdots \times A_{i-1} \times \{a_i \} \times A_{i+1} \times \cdots \times A_t$$
is a piece in $M$ of complexity $<\chi(Q)$, where $a_i$ denotes the projection of $gA_i$ onto $A_i$. Moreover, $g^2$ fixes $a_i$ because it stabilises both $A_i$ and $gA_i$, so $R$ must be stabilised by $g^2$. From now on, we assume that each $A_i$ is a single point, say $a_i$. As a consequence, $\chi(Q)= (t,r)$, so we must have $t \geq 1$.

\medskip \noindent
Fix points $p_1 \in P_1, \ldots, p_r \in P_r$ and a geodesic $\gamma \subset M$ between 
$$u:=(p_1, \ldots, p_r, q_1, \ldots, q_s, a_1, \ldots, a_t)$$ 
and $gu$. The projection of $\gamma$ onto $T_{r+s+1}$ defines an unparametrized geodesic from $a_1$ to $ga_1$, so $\gamma$ contains a point $v=(p_1', \ldots, p_r', q_1, \ldots, q_s, c_1, \ldots, c_t)$ where $p_i'$ belongs to a geodesic between $p_i$ and $gp_i$ (and so belongs to $P_i$) for every $1 \leq i \leq r$, where $c_i$ belongs to a geodesic between $a_i$ and $ga_i$ for every $1 \leq i \leq t$, and where $c_1$ is at equal distance from $a_1$ and $ga_1$. Set $R:= \mu(Q, v, gv)$. Notice that $R \subset M$ according to Lemma \ref{lem:Median}, and that
$$R=  P_1 \times \cdots \times P_r \times q_1 \times \cdots \times q_s \times \mu(a_1, c_1,gc_1) \times \cdots \times \mu(a_t,c_t,gc_t)$$
because the fact that $p_i',gp_i' \in P_i$ implies that $\mu(P_i,p_i',gp_i')= P_i$ for every $1 \leq i \leq r$ and because $\mu(q_i, q_i, gq_i)=q_i$ for every $1 \leq i \leq s$. We deduce from Lemma \ref{lem:ForFixedPoint} that $R$ is stabilised by $g^2$, and also that $g$ fixes $c_1$ so $\mu(a_1,c_1,gc_1)=\mu(a_1,c_1,c_1)=c_1$ is fixed by $g$. Therefore, $R$ has complexity $< \chi(Q)$, as desired. The proof of our claim is complete.

\medskip \noindent
We deduce from Claim \ref{claim:PieceComplexity} that $g^2$ stabilises a piece of complexity $<(1,0)$, which precisely means that this piece is stabilised by $g$. Thus, we have proved that every element of $G$ stabilises a piece in $M$.

\medskip \noindent
Now, fix a finite generating set $\{s_1, \ldots, s_r\}$ of $G$. From our previous observation, we know that each $s_i$ stabilises a piece in $M$, say $y_i$. Define the sequence $x_1, \ldots, x_r$ as follows:
$$\left\{ \begin{array}{l} x_1=y_1 \\ x_{i+1}= \mu(x_i,s_{i+1}x_i,y_{i+1}) \text{ for every $1 \leq i \leq r-1$} \end{array} \right..$$
Notice that, as a consequence of Lemma \ref{lem:Median}, $x_i$ belongs to $M$ for every $1 \leq i \leq r$. Moreover, it follows from Lemma \ref{lem:GlobalFixedPoint} that, for every $1 \leq i \leq n$ and every $1 \leq j \leq r$, the isometries $s_1,\ldots, s_j$ stabilise the $i$th coordinate of $x_j$. We conclude that $G$ stabilises the piece $x_r \subset M$. 
\end{proof}


\subsection{Asymptotic cones and Paulin's construction}\label{section:PaulinRips}

\noindent
In this subsection, we give general definitions and constructions related to asymptotic cones. Our arguments in the next sections will be based on this formalism.

\paragraph{Ultralimits.} An \emph{ultrafilter}\index{Ultrafilters} $\omega$ over a set $S$ is a collection of subsets of $S$ satisfying the following conditions:
\begin{itemize}
	\item $\emptyset \notin \omega$ and $S \in \omega$; 
	\item for all $A,B \in \omega$, $A \cap B \in \omega$;
	\item for every $A \subset S$, either $A\in \omega$ or $A^c \in \omega$.
\end{itemize}
Basically, an ultrafilter may be thought of as a labelling of the subsets of $S$ as ``small'' (if they do not belong to $\omega$) or ``big'' (if they belong to $\omega$). More formally, notice that the map
$$\left\{ \begin{array}{ccc} \mathfrak{P}(S) & \to & \{ 0,1\} \\ A & \mapsto & \left\{ \begin{array}{cl} 0 & \text{if $A \notin \omega$} \\ 1 & \text{if $A \in \omega$} \end{array} \right. \end{array} \right.$$
defines a finitely additive measure on $S$. 

\medskip \noindent
The easiest example of an ultrafilter is the following. Fixing some $s \in S$, set $\omega= \{ A \subset S \mid s \in A \}$. Such an ultrafilter is called \emph{principal}\index{Principal ultrafilters}. The existence of non-principal ultrafilters is assured by Zorn's lemma; see \cite[Section 3.1]{KapovichLeebCones} for a brief explanation.

\medskip \noindent
Now, fix a non-principal ultrafilter $\omega$ over $\mathbb{N}$, a sequence of metric spaces $X= (X_n,d_n)$ and a sequence of basepoints $o=(o_n) \in \prod\limits_{n \geq 1} X_n$. A sequence $(r_n) \in \mathbb{R}^{\mathbb{N}}$ is \emph{$\omega$-bounded} if there exists some $M \geq 0$ such that $\{ n \in \mathbb{N} \mid |r_n| \leq M \} \in \omega$ (i.e., if $|r_n| \leq M$ for ``$\omega$-almost all $n$''). Set
$$B(X,o) = \left\{ (x_n) \in \prod\limits_{n \geq 1} X_n \mid \text{$(d(x_n,o_n))$ is $\omega$-bounded} \right\}.$$
We can define a pseudo-distance on $B(X,o)$ as follows. First, we say that a sequence $(r_n) \in \mathbb{R}^{\mathbb{N}}$ \emph{$\omega$-converges} to a real $r \in \mathbb{R}$ if, for every $\epsilon>0$, $\{ n \in \mathbb{N} \mid |r_n-r| \leq \epsilon \} \in \omega$. If so, we write $r= \lim\limits_\omega r_n$. Then, our pseudo-distance is
$$\left\{ \begin{array}{ccc} B(X,o)^2 & \to & [0,+ \infty) \\ ((x_n),(y_n)) & \mapsto & \lim\limits_\omega d(x_n,y_n) \end{array} \right..$$
Notice that the $\omega$-limit always exists since the sequence under consideration is $\omega$-bounded. 

\begin{definition}
The \emph{ultralimit}\index{Ultralimit} $\lim\limits_\omega (X_n,d_n,o_n)$ is the metric space obtained by quotienting $B(X,o)$ by the relation: $(x_n) \sim (y_n)$ if $d\left( (x_n),(y_n) \right)=0$.  
\end{definition}

\noindent
Important examples of ultralimits are \emph{asymptotic cones}:

\begin{definition}
Let $X$ be a metric space, $\epsilon=(\epsilon_n)$ a \emph{scaling sequence} satisfying $\epsilon_n \to 0$, and $o=(o_n) \in X^\mathbb{N}$ a sequence of basepoints. The \emph{asymptotic cone}\index{Asymptotic cone} $\mathrm{Cone}_\omega(X,\epsilon,o)$ is the ultralimit $\lim\limits_{\omega} (X, \epsilon_nd,o_n)$.
\end{definition}

\noindent
The picture to keep in mind is that $(X, \epsilon_n d)$ is a sequence of spaces we get from $X$ by ``zooming out'', and the asymptotic cone if the ``limit'' of this sequence. Roughly speaking, the asymptotic cones of a metric space are asymptotic pictures of the space. For instance, every asymptotic cone of $\mathbb{Z}^2$, thought of as the infinite grid in the plane, is isometric to $\mathbb{R}^2$ endowed with the $\ell^1$-metric; and the asymptotic cones of a simplicial tree (and more generally of every Gromov-hyperbolic space) are real trees. 

\medskip \noindent
In the next sections, we will use several times the element observation that the asymptotic cones of a product of finitely many metric spaces naturally coincide with the products of asymptotic cones of these spaces. See for instance \cite[Lemma 10.55]{MR3753580} for more details.

\paragraph{Ultralimits of actions.} Let $G$ be a group, $(X_n,d_n)$ a sequence of metric spaces and $o=(o_n) \in \prod\limits_{n \geq 1} X_n$ a sequence of basepoints. For every $n \geq 1$, fix an action $\rho_n : G \to \mathrm{Isom}(X_n,d_n)$ of $G$ on $(X_n,d_n)$. We claim that, if $(d(o_n, \rho_n(s)o_n))$ is $\omega$-bounded for every element $s \in S$ of a generating set $S \subset G$, then
$$\left\{ \begin{array}{ccc} G & \to & \mathrm{Isom} \left( \lim\limits_\omega (X_n,d_n,o_n) \right) \\ g & \mapsto & \left( (x_n) \mapsto (\rho_n(g)x_n) \right) \end{array} \right.$$
defines an action of $G$ on the ultralimit $\lim\limits_\omega (X_n,d_n,o_n)$. The only fact to verify is that, if $g \in G$ and if $(x_n)$ defines a point of the ultralimit, then so does $(\rho_n(g) \cdot x_n)$. Write $g$ as a product of generators $s_1 \cdots s_r$. We have
$$d(o_n, \rho_n(g) \cdot x_n) \leq 2 \sum\limits_{i=1}^r d(o_n, \rho_n(s_i) \cdot x_n) \leq 2 \sum\limits_{i=1}^r \left( d(o_n, \rho_n(s_i) \cdot o_n) + d(o_n,x_n) \right).$$
Consequently, if the sequence $(d(o_n,x_n))$ is $\omega$-bounded, i.e., if $(x_n)$ defines a point of ultralimit, then the sequence $(d(o_n,\rho_n(g) \cdot x_n))$ is $\omega$-bounded as well, i.e., $(\rho_n(g) \cdot x_n)$ also defines a point of the ultralimit.

\paragraph{Paulin's construction.}\index{Paulin's construction} This paragraph is dedicated to the main construction of \cite{Paulin}, allowing us to construct, from a group action and a sequence of automorphisms which are pairwise distinct in the outer automorphism group, an action of the group on one of the asymptotic cones of the space under consideration. As our language is slightly different, we give a self-contained exposition of the construction below. 

\medskip \noindent
Let $G$ be a finitely generated group acting by isometries on a graph $X$ and let $\varphi_1, \varphi_2, \ldots \in \mathrm{Aut}(G) \backslash \{ \mathrm{Id} \}$ be a collection of automorphisms. We fix a finite generating set $S \subset G$. The goal is to construct a non-trivial action of $G$ on one of the asymptotic cones of $X$ from the sequence of twisted actions
$$\left\{ \begin{array}{ccc} G & \to & \mathrm{Isom}(X) \\ g & \to & \left( x \mapsto \varphi_n(g) \cdot x \right) \end{array} \right., \ n \geq 1.$$
For every $n \geq 1$, set $\lambda_n := \min\limits_{x \in X^{(0)}} \max\limits_{s \in S} d(x, \varphi_n(s) \cdot x)$. Notice that:

\begin{fact}
If $G$ does not fix a vertex in $X$, then $\lambda_n \geq 1$.
\end{fact}

\noindent
Fix a non-principal ultrafilter $\omega$ over $\mathbb{N}$ and a sequence of vertices $o=(o_n) \in X^{\mathbb{N}}$ satisfying, for every $n \geq 1$, the equality $\max\limits_{s \in S} d(o_n, \varphi_n(s) \cdot o_n) =\lambda_n$. From now on, we suppose that $\lambda_n \underset{n \to + \infty}{\longrightarrow}+ \infty$, so that the ultralimit of $(X, d/\lambda_n,o_n)$ is an asymptotic cone of $X$. By definition of $(\lambda_n)$ and $(o_n)$, we know that $(d(o_n, \varphi_n(s) \cdot o_n)/\lambda_n)$ is $\omega$-bounded for every $s \in S$. Consequently, the map
$$\left\{ \begin{array}{ccc} G & \to & \mathrm{Isom}(\mathrm{Cone}(X)) \\ g & \mapsto & \left( (x_n) \mapsto (\varphi_n(g) \cdot x_n) \right) \end{array} \right.$$
defines an action by isometries on $\mathrm{Cone}(X):= \mathrm{Cone}_{\omega}(X, (1/\lambda_n),o)$. 

\begin{fact}\label{fact:fixedpointfree}
The action $G \curvearrowright \mathrm{Cone}(X)$ does not fix a point.
\end{fact}

\noindent
Suppose that $G$ fixes a point $(x_n)$ of $\mathrm{Cone}(X)$. Then, for $\omega$-almost all $n$ and all $s \in S$, the inequality
$$\frac{1}{\lambda_n} d(x_n, \varphi_n(s) \cdot x_n) \leq \frac{1}{2}$$
holds, hence
$$\lambda_n \leq \max\limits_{s \in S} d(x_n, \varphi_n(s) \cdot x_n) \leq \lambda_n/2,$$
which is impossible.

\medskip \noindent
The conclusion is that we can associate a fixed-point free action of $G$ on one of the asymptotic cones of $X$ from an infinite collection of automorphisms, provided that our sequence $(\lambda_n)$ tends to infinity. So the natural question is now: when does it happen? One case of interest is the following:

\begin{fact}\label{fact:lambdainfinity}
Assume that $G$ acts properly and cocompactly on $X$. If the automorphisms $\varphi_1, \varphi_2, \ldots$ of $G$ are pairwise distinct in $\mathrm{Out}(G)$, then the equality $\lim\limits_{\omega} \lambda_n = + \infty$ holds. 
\end{fact}

\noindent
Because $G$ acts cocompactly on $X$, up to extracting a subsequence we may suppose that the vertices $o_n$ belong to the same $G$-orbit, i.e., there exist a vertex $o \in X$ and elements $g_1,g_2, \ldots \in G$ such that $o_n=g_n o$ for every $n \geq 1$. Suppose that the $\omega$-limit of $(\lambda_n)$ is not infinite. So, up to extracting a subsequence, we may suppose that $(\lambda_{n})$ is bounded above by some constant $R$. So, for every $n \geq 1$ and every $s \in S$, one has
$$d \left( o, g_n^{-1} \varphi_n(s)g_n \cdot o \right) = d \left( g_no, \varphi_n(s) g_no \right) = d \left( o_{n}, \varphi_{n}(s) o_{n} \right) \leq \lambda_{n} \leq R,$$
i.e., $g_n^{-1} \varphi_n(s)g_n \cdot o \in B(o,R)$. Since the ball $B(1,R)$ is finite, up to extracting a subsequence, we may suppose that
$$g_n^{-1} \varphi_n(s)g_n \cdot o = g_m^{-1} \varphi_m(s)g_m \cdot o$$
for all $m,n \geq 1$ and every $s \in S$. In fact, because $\mathrm{stab}(o)$ is finite, up to extracting a subsequence, we may suppose that
$$g_n^{-1} \varphi_n(s)g_n = g_m^{-1} \varphi_m(s)g_m$$
for all $m,n \geq 1$ and every $s \in S$. Thus, we have proved that $\varphi_{n}$ and $\varphi_{m}$ have the same image in $\mathrm{Out}(G)$ for infinitely many indices $m,n$. This concludes the proof of our fact.

\subsection{Products of trees of spaces}\label{section:TreeSpaces}

\noindent
Understanding asymptotic cones of a graph product $\Gamma \mathcal{G}$ and its graph $\QM$ will be fundamental in Subsection \ref{section:relativesplittings}. The key observation, described in this subsection, is that $\QM$ embeds quasi-isometrically and equivariantly into a product of trees, so that its asymptotic cones will be subspaces in products of real trees. Similarly, $\Gamma \mathcal{G}$ embeds quasi-isometrically and equivariantly into a product of trees of spaces, so that its asymptotic cones will be subspaces in products of tree-graded spaces. We begin by studying the graph $\QM$.

\begin{definition}
Let $\Gamma$ be a simplicial graph, $\mathcal{G}$ a collection of groups indexed by $V(\Gamma)$, and $u \in V(\Gamma)$ a vertex. The graph $T_u$\index{Tree $T_u$} is the graph whose vertices are the hyperplanes of $\QM$ labelled by $u$ and the connected components of $\bigcap\limits_{\text{$J$ labelled by $u$}} \QM \backslash \backslash J$; and whose edges link a connected component to the hyperplanes tangent to it.
\end{definition}

\noindent
By construction, $T_u$ is a bipartite graph. For convenience, we refer to vertices associated to hyperplanes (resp. to connected components) as \emph{hyperplane-vertices} (resp. as \emph{component-vertices}). Let us record a few elementary observations:
\begin{itemize}
	\item Because no two hyperplanes labelled by $u$ are transverse, the graph $T_u$ turns out to be a tree. As an alternative description, $T_u$ coincides with the Bass-Serre tree of the splitting $\Gamma \mathcal{G} = \langle \Gamma \backslash \{u\} \rangle \underset{\langle \mathrm{link}(u) \rangle}{\ast} \langle \mathrm{star}(u) \rangle$. 
	\item The action of $\Gamma \mathcal{G}$ on $\QM$ induces a natural action on $T_u$. 
	\item There is a natural projection $\eta_u : \QM \to T_u$ which sends each vertex to the component it belongs to.
\end{itemize}
Our main result about the trees we just defined is the following:

\begin{prop}\label{prop:Eta}
Let $\Gamma$ be a finite simplicial graph and $\mathcal{G}$ a collection of groups indexed by $V(\Gamma)$. The map
$$\eta:= \prod\limits_{u \in V(\Gamma)} \eta_u : \QM \to \prod\limits_{u \in V(\Gamma)} T_u$$
defines a $\Gamma \mathcal{G}$-equivariant quasi-isometric embedding whose image is \emph{almost median}, i.e., there exists a constant $C \geq 0$ such that, for all $x,y,z \in \QM$, the median point $\mu(\eta(x),\eta(y),\eta(z))$ is at distance at most $C$ from the image of $\eta$.
\end{prop}

\noindent
Recall from Theorem \ref{thm:BigThmQM} that the distance between two vertices in $\QM$ coincides with the number of hyperplanes separating them. For convenience, we denote by $d_u$ the pseudo-metric defined on $\QM$ which counts the number of hyperplanes labelled by $u$ which separate two given vertices. So the graph metric of $\QM$ coincides with $\sum\limits_{u \in V(\Gamma)} d_u$. We begin by proving the following estimate:

\begin{lemma}\label{lem:EstimateEta}
Let $\Gamma$ be a finite simplicial graph, $\mathcal{G}$ a collection of groups indexed by $V(\Gamma)$ and $u \in V(\Gamma)$ a vertex. The equality
$$d_{T_u}(\eta_u(x), \eta_u(y)) = 2 d_u(x,y)$$ 
holds for all vertices $x,y \in \QM$.
\end{lemma}

\begin{proof}
Let $J_1, \ldots, J_n$ denote the hyperplanes labelled by $u$ which separate $x$ and $y$. Up to reindexing the collection, assume that $J_i$ separates $J_{i-1}$ and $J_{i+1}$ for every $2 \leq i \leq n-1$, and that $J_1$ separates $x$ and $J_2$. For every $1 \leq i \leq n-1$, let $C_i$ denote the connected component of $\bigcap\limits_{\text{$J$ labelled by $u$}} \QM \backslash \backslash J$ delimited by $J_i$ and $J_{i+1}$; also, let $C_0$ denote the component which contains $x$ and $C_n$ the component which contains $y$. Then 
$$C_0= \eta_u(x), \ J_1, \ C_1, \ldots, \ J_{n-1}, \ C_{n-1}, \ J_n, \ C_{n}= \eta_u(y)$$
defines a path in $T_u$, hence $d_{T_u}(\eta_u(x), \eta_u(y)) \leq 2n= 2d_u(x,y)$. Now, assume that
$$Q_0= \eta_u(x), \ H_1, \ Q_1, \ldots, \ H_{n-1}, \ Q_{n-1}, \ H_n, \ Q_n= \eta_u(y)$$
is a geodesic in $T_u$, where $Q_0, \ldots, Q_n$ are components and $H_1, \ldots, H_n$ hyperplanes. Fix a vertex $a_i \in Q_i$ for every $1 \leq i \leq n-1$ and set $a_0:=x$, $a_n:=y$. Also, for every $0 \leq i \leq n-1$, fix a geodesic $[a_i,a_{i+1}]$. So the concatenation $[a_0,a_1] \cup \cdots \cup [a_{n-1},a_n]$ defines a path from $x$ to $y$, and the hyperplanes labelled by $u$ it crosses are precisely $H_1, \ldots, H_n$. Of course, this collection has to contain all the hyperplanes labelled by $u$ which separate $x$ and $y$, hence $d_u(x,y) \leq n = d_{T_u}(\eta_u(x),\eta_u(y))/2$. 

\medskip \noindent
This concludes the proof of our lemma. As a consequence, the path in $T_u$ defined in the first part of our argument turns out to be a geodesic. We record this fact for future use:

\begin{fact}\label{fact:GeodesicsEta}
Let $x,y \in \QM$ be two vertices. Fix a geodesic $[x,y]$ between $x$ and $y$ in $\QM$. The vertices of the geodesic in $T_u$ between $\eta_u(x)$ and $\eta_u(y)$ are the vertices in $\eta_u([x,y])$ and the hyperplanes labelled by $u$ which separate $x$ and $y$ in $\QM$.
\end{fact}
\end{proof}

\begin{proof}[Proof of Proposition \ref{prop:Eta}.]
The map $\eta$ is clearly $\Gamma \mathcal{G}$-equivariant, and it induces a quasi-isometric embedding as a consequence of Lemma \ref{lem:EstimateEta}. It remains to show that its image is almost median. 

\medskip \noindent
So let $x,y,z \in \QM$ be three vertices. Let $(x',y',z')$ denote the corresponding median triangle. It follows from Fact \ref{fact:GeodesicsEta} that, for every $u \in V(\Gamma)$, either there exists a hyperplane $J(u)$ labelled by $u$ which separates $x',y',z'$ and $\mu(\eta_u(x),\eta_u(y),\eta(z))=J(u)$; or no hyperplane labelled by $u$ separates $x',y',z'$ and $\mu(\eta_u(x),\eta_u(y),\eta(z))= \eta_u(x')=\eta_u(y')=\eta_u(z')$. Notice that, in the former case, $\eta_u(x')$, $\eta_u(y')$ and $\eta_u(z')$ are all adjacent to $J(u)$ because $x',y',z'$ belong to a prism according to Proposition \ref{prop:MedianTriangleQM} and that at most one hyperplane labelled by $u$ may cross a prism. We conclude that $\eta(x')$, $\eta(y')$ and $\eta(z')$ are at distance at most $\#V(\Gamma)$ from the median point $\mu(\eta(x),\eta(y),\eta(z))$. 
\end{proof}

\noindent
Now, let us turn to the graph product itself.

\begin{definition}
Let $\Gamma$ be a simplicial graph, $\mathcal{G}$ a collection of groups indexed by $V(\Gamma)$, and $u \in V(\Gamma)$ a vertex. Fix a generating set $S_u$ of the vertex-group $G_u$. The graph $TS_u$\index{Tree of spaces $TS_u$} is the graph whose vertices are the fibers of the hyperplanes in $\QM$ labelled by $u$ and the connected components of $\bigcap\limits_{\text{$J$ labelled by $u$}} \QM \backslash \backslash J$; and whose edges link a fiber to the component it belongs to and two fibers of a hyperplane $J$ if they are at $\delta_u$-distance one.
\end{definition}

\noindent
Recall from Fact \ref{fact:deltaequi} that the pseudo-metric $\displaystyle \delta_u:= \sum\limits_{\text{$J$ labelled by $u$}} \delta_J$ is $\Gamma \mathcal{G}$-invariant, so $\Gamma \mathcal{G}$ naturally acts on the graph $TS_u$ by isometries.

\medskip \noindent
By construction, $TS_u$ has two types of vertices. For convenience, we refer to vertices associated to fibers (resp. to connected components) as \emph{fiber-vertices} (resp. as \emph{component-vertices}). Each component-vertex separates $TS_u$ and a connected component of the graph obtained from $TS_u$ by removing all its component-vertices (in other words, a subgraph generated by the fibers of a given hyperplane) coincides with $\mathrm{Cayl}(G_u,S_u)$. We refer to such a subgraph as a \emph{piece} of $TS_u$. Therefore, $TS_u$ can be thought of as a tree of spaces whose vertex-spaces are copies of $\mathrm{Cayl}(G_u,S_u)$. Moreover, the underlying tree coincides with $T_u$ through the canonical map $\rho_u : TS_u \to T_u$ which sends each fiber to the hyperplane it belongs. Figure \ref{Four} illustrates the four graphs $\QM$, $(\QM,\delta)$, $TS_u$ and $T_u$ when $\Gamma$ is a disjoint union of two isolated vertices $u,v$ and when $\mathcal{G}=\{G_u= \mathbb{Z}_4, G_v= \mathbb{Z}_3 \}$. 
\begin{figure}
\begin{center}
\includegraphics[scale=0.36]{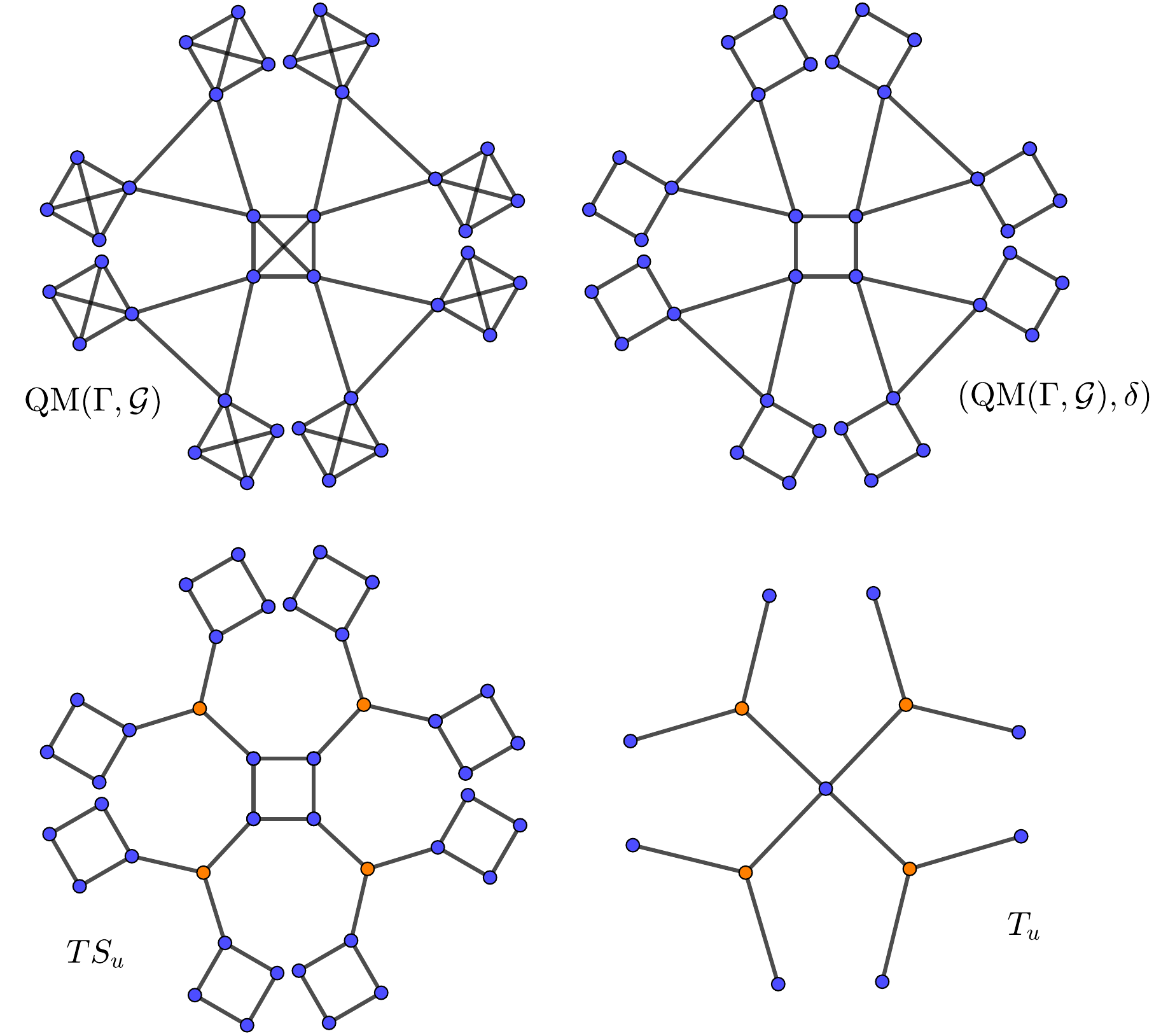}
\caption{}
\label{Four}
\end{center}
\end{figure}

\medskip \noindent
From this description of $TS_u$ as a tree of spaces, the following observation is straightforward:

\begin{fact}\label{fact:TStreegraded}
Let $\Gamma$ be a simplicial graph, $\mathcal{G}$ a collection of groups indexed by $V(\Gamma)$ and $u \in V(\Gamma)$ a vertex. The space $TS_u$ is tree-graded with respect to its pieces.
\end{fact}

\noindent
As for $T_u$, notice that the action of $\Gamma \mathcal{G}$ on $\QM$ induces a natural action on $TS_u$. Also, there is a natural projection $\pi_u : \QM \to TS_u$ which sends a vertex to the component it belongs. Our main result about the trees of spaces just defined is the following:

\begin{prop}\label{prop:Pi}
Let $\Gamma$ be a finite simplicial graph and $\mathcal{G}$ a collection of groups indexed by $V(\Gamma)$. For every $u \in V(\Gamma)$, fix a generating set $S_u$ of $G_u$, and set $S:= \bigcup\limits_{u \in V(\Gamma)} S_u$. The map
$$\pi:= \prod\limits_{u \in V(\Gamma)} \pi_u : (\QM,\delta) \to \prod\limits_{u \in V(\Gamma)} TS_u$$
defines a $\Gamma \mathcal{G}$-equivariant quasi-isometric embedding whose image is \emph{almost median}, i.e., there exists a constant $C \geq 0$ such that, for all $x,y,z \in \QM$, the median-set $\mu(\pi(x),\pi(y),\pi(z))$ lies in the $C$-neighborhood of the image of $\pi$.
\end{prop}

\noindent
Recall from Lemma \ref{lem:DeltaCayley} that $(\QM,\delta)$ is naturally isometric to $\mathrm{Cayl}(\Gamma \mathcal{G},S)$, so Proposition \ref{prop:Pi} provides a quasi-isometric embedding of $\Gamma \mathcal{G}$ (endowed with the word metric associated to $S$) into a product of trees of spaces whose vertex-spaces are copies of vertex-groups. We begin by proving the following estimate:

\begin{lemma}\label{lem:EstimationTS}
Let $\Gamma$ be a finite simplicial graph, $\mathcal{G}$ a collection of groups indexed by $V(\Gamma)$ and $u \in V(\Gamma)$ a vertex. Fix a generating set $S_u \subset G_u$. The equality 
$$d_{TS_u}(\pi_u(x),\pi_u(y)) =2 d_u(x,y)+ \delta_u(x,y)$$
holds for all $x,y \in \QM$. As a consequence, $$\delta_u(x,y) \leq d_{TS_u}(\pi_u(x), \pi_u(y)) \leq 3 \delta_u(x,y)$$ for all vertices $x,y \in \QM$.
\end{lemma}


\begin{proof}
Let $J_1, \ldots, J_n$ denote the hyperplanes labelled by $u$ which separate $x$ and $y$. Up to reindexing our collection, assume that $J_i$ separates $J_{i-1}$ and $J_{i+1}$ for every $2 \leq i \leq n-1$ and that $J_1$ separates $x$ and $J_2$. For every $1 \leq i \leq n$, fix a clique $C_i$ dual to $J_i$ and let $\gamma_i$ be the sequence of fibers of $J_i$ crossed by a $\delta_{J_i}$-geodesic from $\mathrm{proj}_{C_i}(x)$ to $\mathrm{proj}_{C_i}(y)$. For every $1 \leq i \leq n-1$, let $Q_i$ denote the component of $\bigcap\limits_{\text{$J$ labelled by $u$}}\QM \backslash \backslash J$ delimited by $J_i$ and $J_{i+1}$; also, let $Q_0$ denote the component which contains $x$ and $Q_n$ the component which contains $y$. Then
$$Q_0, \ \gamma_1, \ Q_1, \ \gamma_2, \ Q_2, \ldots, \gamma_{n}, \ Q_n$$
defines a path in $TS_u$ from $\pi_u(x)$ to $\pi_u(y)$, and we deduce that 
$$\begin{array}{lcl} d_{TS_u}(\pi_u(x),\pi_u(y)) & \leq & \displaystyle 2n+ \sum\limits_{i=1}^n \mathrm{length}(\gamma_i) = 2n+  \sum\limits_{\text{$J$ labelled by $u$}} \delta_J(x,y) \\ \\ &  \leq & 2d_u(x,y)+ \delta_u(x,y). \end{array}$$
Now, consider a geodesic $\ell$ in $TS_u$ from $\pi_u(x)$ to $\pi_u(y)$. We can write it as
$$Q_0, \ \gamma_1, \ Q_1, \ \gamma_2, \ Q_2, \ldots, \gamma_{n}, \ Q_n$$
where $Q_0 , \ldots, Q_n$ are component-vertices and where $\gamma_1, \ldots, \gamma_n$ are paths of fiber-vertices. For every $1 \leq i \leq n$, let $J_i$ denote the hyperplanes such that the vertices in $\gamma_i$ are fibers of $J_i$. Also, for every $1 \leq i \leq n-1$, fix a vertex $a_i \in Q_i$; and set $a_0:=x$, $a_n:=y$. Because $\ell$ is a geodesic, necessarily each $\gamma_i$ has to correspond to a $\delta_{J_i}$-geodesic, hence $\mathrm{length}(\gamma_i)= \delta_{J_i}(x_{i-1},x_i)= \delta_u(x_{i-1},x_i)$. Moreover, if we fix a geodesic $[x_i,x_{i+1}]$ in $\QM$ between $x_i$ and $x_{i+1}$ for every $0 \leq i \leq n-1$, then $[x_0,x_1] \cup \cdots \cup [x_{n-1},x_n]$ defines a path from $x$ to $y$ such that $J_1, \ldots, J_n$ are the hyperplanes labelled by $u$ it crosses. Of course, the hyperplanes labelled by $u$ which separate $x$ and $y$ must belong to this list. Therefore,
$$\begin{array}{lcl} d_{TS_u}(\pi_u(x),\pi_u(y)) & = & \displaystyle 2n + \sum\limits_{i=1}^n \mathrm{length}(\gamma_i) = 2n + \sum\limits_{i=0}^n \delta_u(x_i,x_{i+1}) \\ \\ & \geq & 2 d_u(x,y) + \delta_u(x,y). \end{array}$$
This concludes the proof of our lemma. As a consequence, the path in $TS_u$ described in the first part of the argument turns out to be a geodesic. We record this fact for future use:

\begin{fact}\label{fact:GeodesicsPi}
Let $x,y \in \QM$ be two vertices. Fix a geodesic $[x,y]$ between $x$ and $y$ in $\QM$. For every $u \in V(\Gamma)$, the interval $I(\pi_u(x),\pi_u(y))$ in $TS_u$ is the union of $\pi_u([x,y])$ with the pieces in $TS_u$ corresponding to the hyperplanes labelled by $u$ which separate $x$ and $y$.
\end{fact}
\end{proof}

\begin{proof}[Proof of Proposition \ref{prop:Pi}.]
The map $\pi$ is clearly $\Gamma \mathcal{G}$-equivariant, and it induces a quasi-isometric embedding as a consequence of Lemma \ref{lem:EstimationTS}. It remains to show that its image is almost median. 

\medskip \noindent
So let $x,y,z \in \QM$ be three vertices. Let $(x',y',z')$ denote the corresponding median triangle and let $P$ denote the smallest prism containing $x',y',z'$ (which exists according to Proposition \ref{prop:MedianTriangleQM}). It follows from Fact \ref{fact:GeodesicsPi} that, for every $u \in V(\Gamma)$, either there exists a hyperplane $J(u)$ labelled by $u$ which separates $x',y',z'$ and $\mu(\pi_u(x),\pi_u(y),\pi(z))$ is the piece $P_u$ corresponding to $J(u)$; or no hyperplane labelled by $u$ separates $x',y',z'$ and $\mu(\pi_u(x),\pi_u(y),\pi(z))= \{ \pi_u(x') \}=\{ \pi_u(y') \} = \{ \pi_u(z') \}$. Notice that, in the former case, $J(u)$ crosses $P$ so every (fiber-)vertex in $P_u$ is adjacent to a vertex of $\pi_u(P)$ (for instance, $\pi_u(p)$ where $p \in P$ belongs to the fiber under consideration). We conclude that the median-set $\mu(\pi(x),\pi(y),\pi(z))$ lies in the $|V(\Gamma)|$-neighborhood of $\pi(P)$.
\end{proof}

\subsection{Fixed-point theorems in asymptotic cones}\label{section:FixedPointCone}

\noindent
This subsection is dedicated to the proof of two fixed-point theorems: Theorem \ref{thm:FixedPointCone}, in asymptotic cones of graph products; and Proposition \ref{prop:FixedPointConeQM}, in asymptotic cones on graphs associated to graph products. 

\begin{thm}\label{thm:FixedPointCone}
Let $\Gamma$ be a finite connected simplicial graph and $\mathcal{G}$ a collection of groups indexed by $V(\Gamma)$. We endow $\Gamma \mathcal{G}$ with the word metric coming from the union of generating sets of vertex-groups. Let $G$ be a finitely generated group and $\varphi_1, \varphi_2, \ldots$ an infinite collection of morphisms $G \to \Gamma \mathcal{G}$. Assume that the limit of twisted actions
$$\left\{ \begin{array}{ccc} G & \to & \mathrm{Isom}(\Gamma \mathcal{G}) \\ g & \mapsto & (x \mapsto \varphi_n(g) \cdot x) \end{array} \right.$$
defines an action of $G$ on $\mathrm{Cone}(\Gamma \mathcal{G}):= \mathrm{Cone}(\Gamma \mathcal{G},(1/\lambda_n),o)$. If this action has bounded orbits, then either $G$ has a global fixed point or there exists some $i \geq 1$ such that $\varphi_i(G)$ lies in a star-subgroup.
\end{thm}

\noindent
Recall that, as defined in Section \ref{subsection:GPprel}, a star-subgroup is a conjugate of a subgroup generated by the star of a vertex. Before proving the theorem, we want to show that an ultralimit of tree-graded spaces (and, in particular, an asymptotic cone of a tree-graded space) is naturally tree-graded with respect to the limits of its pieces, and we want to understand how limits of median-sets are related to median-sets in the limit. In the rest of the article, an ultralimit of tree-graded spaces will be always endowed with the tree-graded structure given by the next lemma. 

\begin{lemma}\label{lem:limitTreeGraded}
Let $(X_n,d_n)$ be a sequence of metric spaces, $(o_n) \in X^\mathbb{N}$ a sequence of basepoints and $\omega$ an ultrafilter over $\mathbb{N}$. Assume that, for every $n \geq 1$, $X_n$ is tree-graded with respect to a collection of subspaces $\mathcal{P}_n$. Then the ultralimit $\lim\limits_\omega (X_n,d_n,o_n)$ is tree-graded with respect to $\mathcal{P}:= \{ \lim\limits_\omega P_n \mid \forall n \in \mathbb{N}, P_n \in \mathcal{P}_n \}$. Moreover, the inclusion  
$$\mu(x,y,z) \subset \lim\limits_\omega \mu(x_n,y_n,z_n)$$
holds for all points $x=(x_n)$, $y=(y_n)$ and $z=(z_n)$ in the ultralimit.
\end{lemma}

\begin{proof}
The first assertion is proved in \cite[Theorem 3.30]{DrutuSapirTreeGraded}. In order to prove the second assertion, let $x=(x_n)$, $y=(y_n)$ and $z=(z_n)$ be three points in the ultralimit $X:= \lim\limits_\omega (X_n,d_n,o_n)$. For all $n \geq 1$, let $x_n',y_n',z_n'$ denote the projections of $x_n,y_n,z_n$ onto $\mu(x_n,y_n,z_n)$. Notice that
$$\left\{ \begin{array}{l} d_n(x_n,y_n)=d_n(x_n,x_n')+d_n(x_n',y_n') + d_n(y_n',y_n) \\ d_n(x_n,z_n)=d_n(x_n,x_n')+d_n(x_n',z_n')+ d_n(z_n',z_n) \\ d_n(y_n,z_n) = d_n(y_n,y_n')+d_n(y'_n,z_n')+d_n(z_n',z_n) \end{array} \right.$$
for every $n \geq 1$. As a first consequence, $(x_n')$ defines a point $x'$ in $X$. Indeed,
$$\begin{array}{lcl} d_n(o_n,x_n') & \leq & d_n(o_n,x_n)+d_n(x_n,x_n') \leq d_n(o_n,x_n)+d_n(x_n,y_n) \\ \\ & \leq & 2d_n(o_n,x_n)+d(o_n,y_n) \end{array}$$
for every $n \geq 1$, so the sequence $(d_n(o_n,x_n'))$ is $\omega$-bounded. Similarly, $(y_n')$ and $(z_n')$ define points $y'$ and $z'$ in $X$. And, as a second consequence, we have
$$\left\{ \begin{array}{l} d(x,y)= d(x,x')+d(x',y')+ d(y',y) \\ d(x,z)=d(x,x')+d(x',z')+ d(z',z) \\ d(y,z)=d(y,y')+d(y',z')+d(z',z) \end{array} \right..$$
Fix geodesics $[x,x']$, $[y,y']$, $[z,z']$, $[x',y']$, $[y',z']$, $[x',z']$ which are limits of geodesics $[x_n,x_n']$, $[y_n,y_n']$, $[z_n,z_n']$, $[x_n',y_n']$, $[y_n',z_n']$, $[x_n',z_n']$. We distinguish three cases.

\medskip \noindent
\textbf{Case 1:} the points $x',y',z'$ are pairwise distinct.

\medskip \noindent 
For $\omega$-almost all $n$, the points $x_n',y_n',z_n'$ are also pairwise distinct which implies that $\mu(x_n,y_n,z_n)$ is a piece $P_n$. Let $P$ denote the limit of $(P_n)$, which is well-defined since we saw that $(x_n')$ defines a point in the ultralimit. Notice that, because the piece $P$ contains $[x',y']$, $[y',z']$ and $[x',z']$, clearly it intersects the geodesic $[x,x'] \cup [x',y'] \cup [y',y]$, $[x,x'] \cup [x',z'] \cup [z',z]$ and $[y,y'] \cup [y',z'] \cup [z',z]$ along subsegments of positive lengths.  hence
$$\mu(x,y,z)= P = \lim\limits_\omega P_n = \lim\limits_\omega \mu(x_n,y_n,z_n)$$
as desired.

\medskip \noindent
\textbf{Case 2:} The median-set $\mu(x,y,z)$ is not a single point. 

\medskip \noindent
So $\mu(x,y,z)$ is a piece $P$ which intersects the geodesics $[x,x'] \cup [x',y'] \cup [y',y]$, $[x,x'] \cup [x',z'] \cup [z',z]$ and $[y,y'] \cup [y',z'] \cup [z',z]$ along subsegments of positive lengths. Write $P$ as a limit of pieces $(P_n)$. Necessarily, for $\omega$-almost all $n$, $P_n$ intersects the geodesics $[x_n,x_n'] \cup [x_n',y_n'] \cup [y_n',y_n]$, $[x_n,x_n'] \cup [x_n',z_n'] \cup [z_n',z_n]$ and $[y_n,y_n'] \cup [y_n', z_n'] \cup [z_n',z_n]$ along subsegments of positive lengths. We deduce that
$$\mu(x,y,z)= P = \lim\limits_\omega P_n = \lim\limits_\omega \mu(x_n,y_n,z_n)$$
as desired.

\medskip \noindent
\textbf{Case 3:} $\mu(x,y,z)$ is a single point and $x',y',z'$ are not pairwise distinct.

\medskip \noindent
Say that $x'=y'$. Because the point $x'=y'$ is common to three geodesics between $x,y,z$, it follows that it belongs to $\mu(x,y,z)$. Hence
$$\mu(x,y,z) = \{x'\} = \lim\limits_\omega \{ x_n'\} \subset \lim\limits_\omega \mu(x_n,y_n,z_n)$$
as desired.
\end{proof}

\begin{proof}[Proof of Theorem \ref{thm:FixedPointCone}.]
As a consequence of Proposition \ref{prop:Pi}, $\pi$ induces a $\Gamma \mathcal{G}$-equivariant Lipschitz embedding 
$$\mathrm{Cone}(\Gamma \mathcal{G}) \overset{\pi_\infty}{\longrightarrow} \mathrm{Cone} \left( \prod\limits_{u \in V(\Gamma)} TS_u, (1/\lambda_n), \pi(o) \right) = \prod\limits_{u \in V(\Gamma)} \mathrm{Cone}(TS_u),$$
where $\mathrm{Cone}(TS_u):= \mathrm{Cone}(TS_u, (1/\lambda_n),\pi_u(o))$ for every $u \in V(\Gamma)$. Therefore, $G$ acts on $\displaystyle \prod\limits_{u \in V(\Gamma)} \mathrm{Cone}(TS_u)$ through the limit of the twisted actions
$$\left\{ \begin{array}{ccc} G & \to & \displaystyle \mathrm{Isom} \left( \prod\limits_{u \in V(\Gamma)} TS_u \right) \\ g & \mapsto & (x \mapsto \varphi_n(g) \cdot x ) \end{array} \right.,$$
and, if $G$ has bounded orbits in $\mathrm{Cone}(\Gamma \mathcal{G})$, then it has bounded orbits in $\displaystyle \prod\limits_{u \in V(\Gamma)} \mathrm{Cone}(TS_u)$. Therefore, as the image of $\pi_\infty$ is median as a consequence of Proposition \ref{prop:Pi} and Lemma~\ref{lem:limitTreeGraded}, it follows from Proposition \ref{prop:ProdTreeGraded} that $G$ stabilises a piece in $\displaystyle \prod\limits_{u \in V(\Gamma)} \mathrm{Cone}(TS_u)$. If this piece is a point, we are done. Otherwise, there exists some $u \in V(\Gamma)$ such that $G$ stabilises a piece which is not a point in $\mathrm{Cone}(TS_u)$ with respect to the limit of the twisted actions
$$\left\{ \begin{array}{ccc} G & \to & \mathrm{Isom} \left(  TS_u \right) \\ g & \mapsto & (x \mapsto \varphi_n(g) \cdot x ) \end{array} \right..$$
Let $P \subset \mathrm{Cone}(TS_u)$ be such a piece, which we write as a limit of pieces $(P_n)$ in $TS_u$. Fix a finite generating set $S \subset G$. If $s \in S$, the fact that $S$ stabilises $P$ implies that $P$ is also the limit of the pieces $(\varphi_n(s)P_n)$. But, in $TS_u$, the Hausdorff distance between any two unbounded pieces is infinite, so we must have $\varphi_n(s)P_n= P_n$ for $\omega$-almost all $n$. Because $S$ is finite, it follows that $\varphi_n(G) \subset \mathrm{stab}(P_n)$ for $\omega$-almost all $n$. But piece-stabilisers in $TS_u$ are hyperplane-stabilisers in $\QM$ so star-subgroups, which concludes the proof.
\end{proof}

\noindent
The same proof as Theorem \ref{thm:FixedPointCone} also leads to the following statement:

\begin{prop}\label{prop:FixedPointConeQM}
Let $\Gamma$ be a finite connected simplicial graph and $\mathcal{G}$ a collection of groups indexed by $V(\Gamma)$. Let $G$ be a finitely generated group and $\varphi_1, \varphi_2, \ldots$ an infinite collection of morphisms $G \to \Gamma \mathcal{G}$. Assume that the limit of twisted actions
$$\left\{ \begin{array}{ccc} G & \to & \mathrm{Isom}(\QM) \\ g & \mapsto & (x \mapsto \varphi_n(g) \cdot x) \end{array} \right.$$
defines an action of $G$ on $\mathrm{Cone}(\QM):= \mathrm{Cone}(\QM,(1/\lambda_n),o)$. If this action has bounded orbits, then it has a global fixed point.
\end{prop}

\begin{proof}
As a consequence of Proposition \ref{prop:Eta}, $\eta$ induces a $\Gamma \mathcal{G}$-equivariant Lipschitz embedding 
$$\mathrm{Cone}(\QM) \overset{\eta_\infty}{\longrightarrow} \mathrm{Cone} \left( \prod\limits_{u \in V(\Gamma)} T_u, (1/\lambda_n), \pi(o) \right) = \prod\limits_{u \in V(\Gamma)} \mathrm{Cone}(T_u),$$
where $\mathrm{Cone}(T_u):= \mathrm{Cone}(T_u, (1/\lambda_n),\pi_u(o))$ for every $u \in V(\Gamma)$. Therefore, $G$ acts on $\displaystyle \prod\limits_{u \in V(\Gamma)} \mathrm{Cone}(T_u)$ through the limit of the twisted actions
$$\left\{ \begin{array}{ccc} G & \to & \displaystyle \mathrm{Isom} \left( \prod\limits_{u \in V(\Gamma)} T_u \right) \\ g & \mapsto & (x \mapsto \varphi_n(g) \cdot x ) \end{array} \right.,$$
and, if $G$ has bounded orbits in $\mathrm{Cone}(\QM)$, then it has also bounded orbits in $\displaystyle \prod\limits_{u \in V(\Gamma)} \mathrm{Cone}(T_u)$. Therefore, as the image of $\eta_\infty$ is median as a consequence of Proposition~\ref{prop:Eta} and Lemma~\ref{lem:limitTreeGraded}, it follows from Proposition \ref{prop:ProdTreeGraded} that $G$ stabilises a piece in the product $\displaystyle \prod\limits_{u \in V(\Gamma)} \mathrm{Cone}(T_u)$. Because all the pieces are points, the desired conclusion follows.
\end{proof}

\begin{remark}
Notice that an alternative proof of Proposition \ref{prop:FixedPointConeQM} is possible. Indeed, it follows from \cite[Theorem 2.120]{Qm} that filling in prisms in $\QM$ with products of simplices produces a CAT(0) complex. As a consequence, if $\mathrm{clique}(\Gamma)$ is finite, the asymptotic cones of $\QM$ can be endowed with a CAT(0) metric which is biLipschitz equivalent to the metric coming from the limit of graph metrics and which is invariant under the action of the ultrapower of $\Gamma \mathcal{G}$. Thus, Proposition \ref{prop:FixedPointConeQM} follows from the Cartan fixed-point theorem in CAT(0) spaces. 
\end{remark}

\subsection{Construction of an action on a real tree}\label{section:relativesplittings}

\noindent
We are finally ready to prove the main theorem of this section, namely:

\begin{thm}\label{thm:relativesplittings}
Let $\Gamma$ be a finite connected simplicial graph and $\mathcal{G}$ a collection of finitely generated groups indexed by $V(\Gamma)$. Fix an irreducible element $g \in \Gamma \mathcal{G}$ and suppose that $\{ \varphi \in \mathrm{Aut}(\Gamma \mathcal{G}) \mid \varphi(g)=g \}$ has an infinite image in $\mathrm{Out}(\Gamma \mathcal{G})$. Then $\Gamma \mathcal{G}$ admits an action on a real tree without a global fixed point, with arc-stabilisers in product subgroups and with $g$ as an elliptic isometry. 
\end{thm}

\noindent
We begin by proving a particular case:

\begin{lemma}\label{lem:RealTree}
Let $\Gamma$ be a finite connected simplicial graph and $\mathcal{G}$ a collection of finitely generated groups indexed by $V(\Gamma)$. Fix an irreducible element $g \in \Gamma \mathcal{G}$ and a family of automorphisms $\varphi_1, \varphi_2, \ldots$ fixing $g$. Also, fix a vertex $u \in V(\Gamma)$, a sequence $(\nu_n)$ satisfying $\nu_n \underset{n \to + \infty}{\longrightarrow} + \infty$, and a sequence $o=(o_n)$ of vertices in $T_u$. If the limit of the twisted actions
$$\left\{ \begin{array}{ccc} \Gamma \mathcal{G} & \to & \mathrm{Isom}(T_u) \\ h & \mapsto & (x \mapsto \varphi_n(h) \cdot x) \end{array} \right.$$
defines an action of $\Gamma \mathcal{G}$ on $\mathrm{Cone}(T_u):= \mathrm{Cone}(T_u, (1/\nu_n),o)$, then $\Gamma \mathcal{G}$ acts on the minimal invariant subtree of $\mathrm{Cone}(T_u)$ with arc-stabilisers in product subgroups and with $g$ as an elliptic isometry.
\end{lemma}

\begin{proof}
We begin by proving a few observations.

\begin{claim}\label{claim:gIsElliptic}
The element $g$ is elliptic in $\mathrm{Cone}(T_u)$.
\end{claim}

\noindent
It suffices to show that $g$ fixes a point in $\mathrm{Cone}(\QM):= \mathrm{Cone}(\QM,(1/\nu_n),o)$ with respect to the limit of the twisted actions
$$\left\{ \begin{array}{ccc} \Gamma \mathcal{G} & \to & \mathrm{Isom}(\QM) \\ h & \mapsto & (x \mapsto \varphi_n(h) \cdot x) \end{array} \right..$$
Indeed, it follows from Proposition \ref{prop:Eta} that $\eta$ induces a $\Gamma \mathcal{G}$-equivariant Lipschitz embedding
$$\mathrm{Cone}(\QM) \to \mathrm{Cone} \left( \prod\limits_{u \in V(\Gamma)} T_u, (1/\nu_n), \eta(o) \right) = \prod\limits_{u \in V(\Gamma)} \mathrm{Cone}(T_u).$$
Up to conjugating $g$ (and the automorphisms $\varphi_1, \varphi_2, \ldots$), we may suppose without loss of generality that $g$ is graphically cyclically reduced. According to Proposition \ref{prop:ContractingAxis}, there exist a bi-infinite geodesic $\ell \subset \QM$ on which $g$ acts as a translation of length $\tau$ and a hyperplane $J$ crossing $\ell$ such that $\{g^{2Dk} J \mid k \in \mathbb{Z} \}$ is a collection of pairwise strongly separated hyperplanes, where $D$ denotes the diameter of $\mathrm{supp}(g)$ in $\Gamma^\mathrm{opp}$. For every $n \geq 1$, let $r_n$ be such that $o_n$ lies between $g^{2Dr_n}J$ and $g^{2D(r_{n}+1)}J$. We know from Proposition \ref{prop:ContractingAxis} that there exist a vertex $p_n \in \ell$ and a vertex $q_n$ which belongs to a geodesic between $o_n$ and $g^{8D}o_n$ such that $d(p_n,q_n) \leq 6D|g|$. Notice that
$$d(o_n,p_n) \leq d(o_n,q_n)+ d(p_n,q_n)\leq d\left( o_n, g^{8D}o_n \right) + 6D|g|$$
for every $n \geq 1$. Because $(g^{8D}o_n)$ defines a point in $\mathrm{Cone}(\QM)$, i.e., the sequence $(d( o_n ,g^{8D}o_n)/\nu_n)$ is $\omega$-bounded, it follows that $p:=(p_n)$ defines a point in $\mathrm{Cone}(\QM)$, i.e., the sequence $(d(o_n,p_n)/ \nu_n)$ is $\omega$-bounded. We have
$$d_\mathrm{Cone}(p,g \cdot p )= \lim\limits_\omega \frac{1}{\nu_n} d(p_n, \underset{=g}{\underbrace{\varphi_n(g)}} \cdot p_n) = \lim\limits_\omega \frac{\tau}{\nu_n}=0,$$
i.e., $g$ fixes the point $p$ in $\mathrm{Cone}(\QM)$. We conclude that $g$ fixes a point in $\mathrm{Cone}(T_u)$ as desired.

\begin{claim}\label{claim:CommutatorStar}
Every commutator between two elements of an arc-stabiliser in $\mathrm{Cone}(T_u)$ lies in a star-subgroup.
\end{claim}

\noindent
Let $x,y \in \mathrm{Cone}(T_u)$ be two distinct points and $a,b \in \mathrm{stab}([x,y])$ two elements stabilising the arc $[x,y]$. Let $(x_n)$ and $(y_n)$ be two sequences of vertices of $T_u$ such that $x=(x_n)$ and $y=(y_n)$. Fix an $\epsilon>0$ satisfying $\epsilon<d/7$ where $d$ denotes the distance between $x$ and $y$ in $\mathrm{Cone}(T_u)$. For $\omega$-almost all $n$, we have 
$$\left\{ \begin{array}{l} d(x_n,y_n) \geq (d-\epsilon)\nu_n \\ d(x_n,ax_n), d(x_n,bx_n), d(y_n,ay_n),d(y_n,by_n) \leq \epsilon \nu_n \end{array} \right.$$
which implies that
$$d(x_n,ax_n), d(x_n,bx_n), d(y_n,ay_n),d(y_n,by_n) < \frac{1}{6} d(x_n,y_n).$$
As a consequence, there exists some $n \geq 1$ such that the commutator $[a,b]$ fixes pointwise the middle third of the segment $[x_n,y_n] \subset T_u$. Because this segment has positive length, it follows from the construction of $T_u$ that $[a,b]$ stabilises a hyperplane of $\QM$. So $[a,b]$ belongs to a star-subgroup, as desired.

\begin{claim}\label{claim:StabProductOrIrrCy}
An arc-stabiliser in $\mathrm{Cone}(T_u)$ lies in a product subgroup or is infinite cyclic generated by an irreducible element. 
\end{claim}

\noindent
Let $x,y \in \mathrm{Cone}(T_u)$ be two distinct points. If $H:= \mathrm{stab}([x,y])$ lies in a product subgroup, then there is nothing to prove, so from now we assume that the arc-stabiliser does not lie in a product subgroup. As a consequence of Proposition \ref{prop:IrrElementExist}, it must contain an irreducible element $h$. According to Proposition \ref{prop:IrreducibleWPD}, $H$ acts on the crossing graph $\crossing$, which is hyperbolic according to Theorem \ref{thm:Phyp}, with $h$ as a WPD element. 

\medskip \noindent
First, assume that $H$ does not stabilise the pair of points at infinity fixed by $h$. Then we deduce from \cite[Theorem 8.7]{DGO} (see Proposition \ref{prop:DGOWPD} below) that the normal closure (in $H$) of some power of $h$ is free non-abelian and purely loxodromic. Consequently, $H$ contains two elements whose commutator is loxodromic for the action on the crossing graph. Since star-subgroups are elliptic in the crossing graph, we get a contradiction with Claim \ref{claim:CommutatorStar}.

\medskip \noindent
Therefore, $H$ has to stabilise the pair of points at infinity fixed by $h$. It follows from Lemma \ref{lem:subgroupE} that $H \subset E(h)$. Notice that, if $g$ is an element such that $gh^ng^{-1}=h^{-n}$ for some $n \in \mathbb{Z}\backslash \{0\}$, then $[g,h^n]=gh^ng^{-1}h^{-n}=h^{-2n}$ cannot lie in a star-subgroup because $h$ is irreducible, so Claim \ref{claim:CommutatorStar} implies that such an element $g$ cannot belong to $H$. We deduce from Lemma~\ref{lem:subgroupE} that $H$ lies in the centraliser of a power of $h$. So $H$ is infinite cyclic, generated by an irreducible element, according to Proposition \ref{prop:centraliser}.

\begin{claim}\label{claim:IntervalStable}
Let $I,J \subset \mathrm{Cone}(T_u)$ be two arcs satisfying $I \subset J$. If $\mathrm{stab}(J)$ is generated by an irreducible element, then either $\mathrm{stab}(I)$ is trivial or it is also generated by an irreducible element.
\end{claim}

\noindent
Let $a \in \mathrm{stab}(J)$ be an irreducible generator and $b \in \mathrm{stab}(I)$ a non-trivial element. Let $x,y \in \mathrm{Cone}(T_u)$ be such that $I=[x,y]$, and fix two sequences $(x_n)$ and $(y_n)$ of vertices in $T_u$ such that $x=(x_n)$ and $y=(y_n)$. Finally, fix an $\epsilon>0$ which satisfies $\epsilon < d/ \max(7,18D+1)$ where $d$ denotes the distance between $x$ and $y$ in $\mathrm{Cone}(T_u)$ and $D$ the diameter of $\mathrm{supp}(g)$ in $\Gamma^\mathrm{opp}$. For $\omega$-almost all $n$, we have
$$\left\{ \begin{array}{l} d(x_n,y_n) \geq (d- \epsilon)\nu_n \\ d(x_n,ax_n),d(x_n,bx_n),d(y_n,ay_n),d(y_n,ay_n) \leq \epsilon \nu_n \end{array} \right..$$
Fix such an $n$ and let $M$ denote the middle third of $[x_n,y_n]$. Because 
$$d(x_n,ax_n),d(y_n,ay_n) < \frac{2}{3} d(x_n,y_n),$$
it follows that either $a$ is elliptic and it fixes $M$; or it is loxodromic of translation length $\leq \epsilon \nu_n$ and its axis passes through $M$. But $a$ cannot fix an edge of $T_u$, because otherwise it would stabilise a hyperplane in $\QM$ and so it would belong to a star-subgroup, contradicting the fact that it is irreducible. So $a$ is loxodromic in $T_u$ of translation length $\leq \epsilon \nu_n$ and its axis passes through $M$. Because $M$ has length $> 6D \epsilon \nu_n$, it must contain a hyperplane-vertex and its translate under $g^{6D}$, hence $\mathrm{stab}(M) \subset \mathrm{stab}(H) \cap \mathrm{stab}(g^{6D}H)$ for some hyperplane $H$. But $H$ and $g^{6D}H$ are strongly separated according to Corollary~\ref{cor:SkewerContracting}, so we deduce that $\mathrm{stab}(M)=\{1\}$. But, because we have
$$d(x_n,ax_n), d(x_n,bx_n), d(y_n,ay_n),d(y_n,by_n) < \frac{1}{6} d(x_n,y_n).$$
we know that the commutator $[a,b]$ fixes $M$, so $a$ and $b$ commute.

\medskip \noindent
Thus, we have proved that $\mathrm{stab}(I)$ lies in the centraliser of $a$, which is a cyclic group generated by an irreducible element. This concludes the proof of our claim.

\medskip \noindent
Now, we are ready to turn to the proof of our lemma. Let $S$ denote the set of all the conjugates of the elements in $\bigcup_{u \in V(\Gamma)} G_u$ minus the trivial element. Observe that:

\begin{fact}\label{fact:comgraphconnected}
The commutation graph of $S$ is connected, i.e. two elements in $S$ can always be connected by a sequence of elements in $S$ in which any two consecutive terms commute. 
\end{fact}

\noindent
Fixing an $s \in S$, we want to construct a path from $1$ to $s$ in the commutation graph. This will be sufficient in order to conclude the proof of our fact. By definition of $S$, there exist some $g \in \Gamma \mathcal{G}$, some $u \in V(\Gamma)$ and some $r \in G_u$ such that $s=grg^{-1}$. Write $g$ as a product $g_1 \cdots g_n$ of elements belonging to vertex-groups, i.e. for every $1 \leq i \leq n$ there exists some $u_i \in V(\Gamma)$ such that $g_i \in G_{u_i}$. Because $\Gamma$ is connected, there exists a path in $\Gamma$ from $u$ to $u_n$, and we can associate to such a path a path in the commutation graph from $grg^{-1}$ to $gg_ng^{-1}=g_1 \cdots g_{n-1} \cdot g_n \cdot (g_1 \cdots g_{n-1})^{-1}$. Similarly, we construct a path in the commutation graph from $g_1 \cdots g_{n-1} \cdot g_n \cdot (g_1 \cdots g_{n-1})^{-1}$ to $g_1 \cdots g_{n-2} \cdot g_{n-1} \cdot (g_1 \cdots g_{n-2})^{-1}$, and so on. After $n$ interation, we eventually get $1$, concluding the proof of our fact.

\medskip \noindent
For every $s \in S$, let $\mathrm{Char}(s) \subset \mathrm{Cone}(T_u)$ denote the axis of $s$ if $s$ is loxodromic and $\mathrm{Fix}(s)$ if $s$ is elliptic. Clearly, if $s_1,s_2 \in S$ commute then $\mathrm{Char}(s_1)$ and $\mathrm{Char}(s_2)$ intersect. Thus, Fact \ref{fact:comgraphconnected} implies that the union 
$$T:= \bigcup\limits_{s \in S} \mathrm{Char}(s) \subset \mathrm{Cone}(T_u)$$
is connected. Since $T$ is clearly $\Gamma \mathcal{G}$-invariant, necessarily it contains the minimal $\Gamma \mathcal{G}$-invariant subtree in $\mathrm{Cone}(T_u)$. Therefore, in order to prove that $\Gamma \mathcal{G}$ acts on the minimal invariant subtree in $\mathrm{Cone}(T_u)$ with arc-stabilisers in product subgroups, it suffices to show that the stabiliser of an arc in $T$ lies in a product subgroups, or, according to Claim~\ref{claim:StabProductOrIrrCy}, that it is not generated by an irreducible element.  

\medskip \noindent
Assume for contradiction that there exists an arc $I \subset T$ whose stabiliser is generated by an irreducible element. Because $S$ is countable, it follows from the definition of $T$ that there exists an $s \in S$ such that $J:=\mathrm{Char}(s) \cap I$ has positive length. According to Claim \ref{claim:IntervalStable}, the stabiliser of $J$ is either trivial or generated by an irreducible element. Because an element of a vertex-group is not a power of an irreducible element, it follows that $s$ cannot be elliptic. So $s$ is a loxodromic isometry of $T$ and its axis $\gamma$ contains $J$. If $r \in S$ commutes with $s$, then it has to stabilise $\gamma$. More precisely, either $r$ is elliptic and it fixes $\gamma$ pointwise; or $r$ is loxodromic and $\gamma$ is its axis. But $r$ cannot fix $\gamma$ as it contains $J$. By iterating the argument, it follows from Fact~\ref{fact:comgraphconnected} that all the elements of $S$ are loxodromic with $\gamma$ as their axis. But, because $\Gamma$ is connected, not a single vertex and not complete, it contains a vertex $u$ whose link is not complete, i.e., there exist $v,w \in \mathrm{link}(u)$ such that $v$ and $w$ are not adjacent. We know from the previous observation that, if $a \in G_v$ and $b \in G_w$, then $[a,b]$ defines a non-trivial element which is not irreducible (indeed, it belongs to $\langle \mathrm{star}(u) \rangle$) and which fixes $\gamma$. This contradicts the fact that the stabiliser of $J \subset \gamma$ is either trivial or generated by an irreducible element.

\medskip \noindent
Thus, we have proved that $\Gamma \mathcal{G}$ acts on the minimal invariant subtree in $\mathrm{Cone}(T_u)$ with arc-stabilisers in product subgroups and, according to Claim \ref{claim:gIsElliptic}, with $g$ as an elliptic isometry.
\end{proof}

\begin{proof}[Proof of Theorem \ref{thm:relativesplittings}.]
Let $\varphi_1, \varphi_2, \ldots \in \mathrm{Aut}(\Gamma \mathcal{G})$ be a collection of automorphisms fixing $g$ which have pairwise distinct images in $\mathrm{Out}(\Gamma \mathcal{G})$. For every $n \geq 1$, set
$$\lambda_n := \min\limits_{x \in \QM} \max\limits_{s \in S} \delta(x, \varphi_n(s) \cdot x).$$
As a consequence of Fact \ref{fact:lambdainfinity}, up to extracting a subsequence we may suppose without loss of generality that $\lambda_n \underset{n \to + \infty}{\longrightarrow} + \infty$. Also, set 
$$\mu_n := \min\limits_{x \in \QM} \max\limits_{s \in S} d(x, \varphi_n(s) \cdot x).$$
If $\mu_n \underset{\omega}{\longrightarrow} + \infty$, then it follows from Fact \ref{fact:fixedpointfree} that the twisted actions
$$\left\{ \begin{array}{ccc} \Gamma \mathcal{G} & \to & \mathrm{Isom} \left(\QM \right) \\ h & \mapsto & (x \mapsto \varphi_n(h) \cdot x) \end{array} \right.$$
converge to a fixed-point free action of $\Gamma \mathcal{G}$ on $\mathrm{Cone}:=\mathrm{Cone}(\QM, (1/\mu_n),o)$ where $o=(o_n)$ is a sequence of vertices satisfying $\max\limits_{s \in S} d(o_n, \varphi_n(s) \cdot o_n)= \mu_n$ for every $n \geq 1$. As a consequence of Proposition \ref{prop:Eta}, $\eta$ induces a $\Gamma \mathcal{G}$-equivariant Lipschitz embedding
$$\mathrm{Cone} \overset{\eta_\infty}{\longrightarrow} \mathrm{Cone} \left( \prod\limits_{u \in V(\Gamma)} T_u, (1/\mu_n), \eta(o) \right) = \prod\limits_{u \in V(\Gamma)} \mathrm{Cone}(T_u,(1/\mu_n),\eta_u(o)),$$
where the action of $\Gamma \mathcal{G}$ on each $\mathrm{Cone}(T_u):= \mathrm{Cone}(T_u,(1/\mu_n),\eta_u(o))$ is the limit of the twisted actions
$$\left\{ \begin{array}{ccc} \Gamma \mathcal{G} & \to & \mathrm{Isom} \left(T_u \right) \\ h & \mapsto & (x \mapsto \varphi_n(h) \cdot x) \end{array} \right..$$
Because $\Gamma \mathcal{G}$ acts on $\mathrm{Cone}$ fixed-point freely, it follows from Proposition \ref{prop:FixedPointConeQM} that its orbits are unbounded. As a consequence, there must exist some $u \in V(\Gamma)$ such that $\Gamma \mathcal{G}$ acts on $\mathrm{Cone}(T_u)$ with unbounded orbits, and a fortiori fixed-point freely. Then, the desired conclusion follows from Lemma \ref{lem:RealTree}.

\medskip \noindent
From now on, we assume that the sequence $(\mu_n)$ is $\omega$-bounded. Up to extracting a subsequence, we suppose that there exists some $R \geq 0$ such that $\mu_n \leq R$ for every $n \geq 1$. We know from Fact \ref{fact:fixedpointfree} that the twisted actions
$$\left\{ \begin{array}{ccc} \Gamma \mathcal{G} & \to & \mathrm{Isom}(\QM, \delta) \\ h & \mapsto & (x \mapsto \varphi_n(h) \cdot x) \end{array} \right.$$
converge to a fixed-point free action of $\Gamma \mathcal{G}$ on $\mathrm{Cone}(\Gamma \mathcal{G}):= \mathrm{Cone}((\QM,\delta),(1/\lambda_n),o)$ where $o=(o_n)$ is a sequence of vertices in $\QM$ satisfying $\max\limits_{s \in S} \delta(o_n, \varphi_n(s) \cdot o_n)= \lambda_n$ for every $n \geq 1$. As a consequence of Proposition \ref{prop:Pi}, the map $\pi$ induces a $\Gamma \mathcal{G}$-equivariant Lipschitz embedding
$$\mathrm{Cone}(\Gamma \mathcal{G}) \overset{\pi_\infty}{\longrightarrow} \mathrm{Cone} \left( \prod\limits_{u \in V(\Gamma)} TS_u, (1/\lambda_n), \pi(o) \right) = \prod\limits_{u \in V(\Gamma)} \mathrm{Cone} \left( TS_u, (1/\lambda_n), \pi_u(o) \right).$$
where the action of $\Gamma \mathcal{G}$ on each $\mathrm{Cone}(TS_u):= \mathrm{Cone}(TS_u,(1/\lambda_n),\pi_u(o))$ is the limit of the twisted actions
$$\left\{ \begin{array}{ccc} \Gamma \mathcal{G} & \to & \mathrm{Isom} \left(TS_u \right) \\ h & \mapsto & (x \mapsto \varphi_n(h) \cdot x) \end{array} \right..$$
Because $\Gamma \mathcal{G}$ acts on $\mathrm{Cone}(\Gamma \mathcal{G})$ fixed-point freely, we deduce from Theorem \ref{thm:FixedPointCone} that its orbits are unbounded. As a consequence, there must exist some $u \in V(\Gamma)$ such that $\Gamma \mathcal{G}$ acts on $\mathrm{Cone}(TS_u)$ with unbounded orbits, and a fortiori fixed-point freely. 

\medskip \noindent
\textbf{Case 1:} $G_u$ is finite.

\medskip \noindent
Let $\rho_u$ denote the canonical projection $TS_u \to T_u$ (which sends a vertex of $TS_u$ labelled by a fiber of a hyperplane $J$ to the vertex of $T_u$ labelled by the hyperplane $J$). Because $G_u$ is finite, $\rho_u$ defines a $\Gamma \mathcal{G}$-equivariant quasi-isometry $TS_u \to T_u$. So we get a $\Gamma \mathcal{G}$-equivariant isometry $\mathrm{Cone}(TS_u) \to \mathrm{Cone}(T_u, (1/\lambda_n), p(\pi_u(o)))$, where $\Gamma \mathcal{G}$ acts on the two asymptotic cones through limits of actions twisted by $\varphi_1, \varphi_2, \ldots$ as above. The desired conclusion follows from Lemma \ref{lem:RealTree}.

\medskip \noindent
\textbf{Case 2:} $G_u$ is infinite. 

\medskip \noindent
Recall that we refer to a piece in $\mathrm{Cone}(TS_u)$ as the ultralimit of a sequence of pieces in $TS_u$. Notice that, because every vertex of $TS_u$ is at distance at most one from some piece, every point of $\mathrm{Cone}(TS_u)$ belongs to a piece. Let $\mathscr{G}$ denote the graph whose vertices are the points of $\mathrm{Cone}(TS_u)$ and its pieces; and whose edges link a point to the pieces it belongs. We claim that the component $\mathscr{T}$ of $\mathscr{G}$ which contains $(\pi_u(o_n))$ is a tree and is $\Gamma \mathcal{G}$-invariant.

\medskip \noindent
Recall from Fact \ref{fact:TStreegraded} that $TS_u$ is tree-graded with respect to its pieces. As a consequence of Lemma \ref{lem:limitTreeGraded}, $\mathrm{Cone}(TS_u)$ is also tree-graded with respect to its pieces. So every topological arc in $\mathrm{Cone}(TS_u)$ between two distinct points which belong to a common piece must lie in this piece \cite[Corollary~2.10]{DrutuSapirTreeGraded}, which implies that $\mathscr{G}$ has to be a forest, and a fortiori that $\mathscr{T}$ has to be a tree.

\medskip \noindent
In order to show that $\mathscr{T}$ is $\Gamma \mathcal{G}$-invariant, it suffices to prove that every element of $\Gamma \mathcal{G}$ sends the vertex of $\mathscr{T}$ corresponding to $(\pi_u(o_n))$ in $\mathscr{T}$, as $\Gamma \mathcal{G}$ acts on $\mathscr{G}$ by permuting its connected components. Fix an element $h \in \Gamma \mathcal{G}$ and an $n \geq 1$. Notice that
$$d_{\QM}( o_n, \varphi_n(h) \cdot o_n) \leq |h| \cdot \max\limits_{s \in S} d_{\QM}(o_n, \varphi_n(s) \cdot o_n) = |h| \cdot \mu_n \leq R |h|.$$
Therefore, there exist at most $R |h|$ hyperplanes labelled by $u$ separating $o_n$ and $\varphi_n(h) \cdot o_n$ in $\QM$, say $J_1^n, \ldots, J_{r(n)}^n$. Up to reindexing our hyperplanes, we assume that $J_i^n$ separates $J_{i-1}^n$ and $J_{i+1}^n$ for every $2 \leq i \leq r(n)-1$ and that $J_1^n$ separates $o_n$ and $J_2^n$. For every $1 \leq i \leq r(n)$, let $L_i^n$ (resp. $R_i^n$) denote the fiber of $J_i^n$ contained in the sector which contains $o_n$ (resp. $\varphi_n(h) \cdot o_n$); and let $P_i^n$ denote the piece of $TS_u$ corresponding to the hyperplane $J_i^n$. Notice that, along the following sequence of vertices in $TS_u$:
$$o_n, \ L_1^n, \ R_1^n, \ L_2^n, \ R_2^n, \ldots, \ L_{r(n)}^n, \ R_{r(n)}^n, \ \varphi_n(h) \cdot o_n,$$
we have that $o_n$ is adjacent to $L_1^n$; $L_i^n$ and $R_i^n$ belong to the same piece $P_i^n$ for every $1 \leq i \leq r(n)$; $R_i^n$ and $L_{i+1}^n$ are at distance at most two for every $1 \leq i \leq r(n)-1$; and $\varphi_n(h) \cdot o_n$ is adjacent to $R_{r(n)}^n$. Therefore, this sequence converges to a sequence
$$o=x_1, \ x_2, \ldots, \ x_{k-1}, \ x_k= h \cdot o$$
of points in $\mathrm{Cone}(TS_u)$ of length $\leq R|h|$ such that, for every $1 \leq i \leq k-1$, $x_i$ and $x_{i+1}$ belong to a common piece, say $Q_i$. Then $x_1, Q_1, x_2, Q_2, \ldots, x_{k-1},Q_{k-1}, x_k$ defines a path from $o$ to $h \cdot o$ in $\mathscr{T}$, concluding the proof of our claim.

\medskip \noindent
Thus, we have constructed an action of $\Gamma \mathcal{G}$ on a simplicial tree $\mathscr{T}$. Notice that this action does not have a global fixed point. Indeed, we already know that $\Gamma \mathcal{G}$ does not fix a point in $\mathrm{Cone}(TS_u)$, and it follows from our next observation that it does not stabilise a piece either:

\begin{claim}\label{claim:PieceStab}
In $\mathrm{Cone}(TS_u)$, the stabiliser of a piece lies in a product subgroup.
\end{claim}

\noindent
Let $P$ be a piece in $\mathrm{Cone}(TS_u)$. Fix a sequence of pieces $(P_n)$ in $TS_u$ such that $P$ is the limit of $(P_n)$. If $h \in \Gamma \mathcal{G}$ stabilises $P$, then $(\varphi_n(h)P_n)$ also converges to $P$. But the fact that $G_u$ is infinite implies that each piece in $TS_u$ has infinite diameter, so the Hausdorff distance between any two distinct pieces in $TS_u$ must be infinite. It follows that $\varphi_n(h)P_n=P_n$ for $\omega$-almost all $n$. Because the stabiliser of a piece in $TS_u$ coincides with the stabiliser of a hyperplane in $\Gamma \mathcal{G}$, it follows that every element in the stabiliser of $P$ belongs to a star-subgroup. Proposition \ref{prop:IrrElementExist} provides the desired conclusion. 

\medskip \noindent
In order to conclude the proof of our theorem, it remains to show that $\Gamma \mathcal{G}$ acts on $\mathscr{T}$ with arc-stabilisers in product subgroups and with $g$ as an elliptic element. The former assertion follows from Claim \ref{claim:PieceStab} as edge-stabilisers lie in piece-stabilisers. From now on, we focus on the latter assertion. 

\medskip \noindent
Up to conjugating $g$ (and the automorphisms $\varphi_1, \varphi_2, \ldots$), we may suppose without loss of generality that $g$ is graphically cyclically reduced. According to Proposition \ref{prop:ContractingAxis}, there exist a bi-infinite geodesic $\ell \subset \QM$ on which $g$ acts as a translation and a hyperplane $J$ crossing $\ell$ such that $\{g^{2Dk} J \mid k \in \mathbb{Z} \}$ is a collection of pairwise strongly separated hyperplanes, where $D$ denotes the diameter of $\mathrm{supp}(g)$ in $\Gamma^\mathrm{opp}$. For every $n \geq 1$, let $r_n$ be such that $o_n$ lies between $g^{2Dr_n}J$ and $g^{2D(r_{n}+1)}J$. We know from Proposition \ref{prop:ContractingAxis} that there exist a vertex $p_n \in \ell$ and a vertex $q_n$ which belongs to a geodesic between $o_n$ and $g^{8D}o_n$ such that $\delta(p_n,q_n) \leq 6D|g|$. Notice that
$$\delta(o_n,p_n) \leq \delta(o_n,q_n)+ \delta(p_n,q_n)\leq \delta\left( o_n, g^{8D}o_n \right) + 6D|g|$$
for every $n \geq 1$. Because $(g^{8D}o_n)$ defines a point in $\mathrm{Cone}(\Gamma \mathcal{G})$, i.e., the sequence $(\delta( o_n ,g^{8D}o_n)/\lambda_n)$ is $\omega$-bounded, it follows that $p:=(p_n)$ defines a point in $\mathrm{Cone}(\Gamma \mathcal{G})$, i.e., the sequence $(\delta(o_n,p_n)/ \lambda_n)$ is $\omega$-bounded. We have
$$d_\mathrm{Cone}(p,g \cdot p )= \lim\limits_\omega \frac{1}{\lambda_n} \delta(p_n, \underset{=g}{\underbrace{\varphi_n(g)}} \cdot p_n) = \lim\limits_\omega \frac{|g|}{\lambda_n}=0,$$
i.e., $g$ fixes the point $p$ in $\mathrm{Cone}(\Gamma \mathcal{G})$. This implies that $\langle g \rangle$ has bounded orbits in $\mathrm{Cone}(TS_u)$. On the other hand, if $g$ is not an elliptic isometry of $\mathscr{T}$, then it acts on bi-infinite geodesic as a non-trivial translation. So there exist pairwise distinct pieces $\ldots, A_{-1},A_0,A_1, \ldots$ and pairwise distinct points $\ldots, x_{-1},x_0,x_1, \ldots$ in $\mathrm{Cone}(TS_u)$ such that $A_i \cap A_{i+1}= \{x_i\}$ for every $i \in \mathbb{Z}$ and such that there exists some $\tau \geq 1$ such that $gx_i=x_{i+\tau}$ for every $i \in \mathbb{Z}$. This implies that $\langle g \rangle$ has unbounded orbits in $\mathrm{Cone}(TS_u)$, contradicting the previous observation. Thus, $g$ has to be an elliptic isometry of $\mathscr{T}$. 
\end{proof}

\section{Step 4: Fixed point property on real trees}\label{section:step3}

\noindent
In this section, our goal is to prove that graph products satisfy relative fixed-point properties with respect to actions on real trees. The specific property we are interested in is the following: 

\begin{definition}
Let $G$ be a group, $H \leq G$ a subgroup and $\mathcal{H}$ a collection of subgroups of $G$. The group $G$ satisfies the \emph{property $\mathrm{F}\mathbb{R}(\mathcal{H})$ relative to $H$}\index{Property $\mathrm{F}\mathbb{R}(\mathcal{H})$} if, for every isometric action of $G$ on a real tree with arc-stabilisers in $\mathcal{H}$, the existence of fixed point for $H$ implies the existence of a global fixed-point for $G$.
\end{definition}

\noindent
This section is dedicated to the proof of the following statement:

\begin{thm}\label{thm:FRH}
Let $\Gamma$ be a finite simplicial graph and $\mathcal{G}$ a collection of groups indexed by $V(\Gamma)$. Let $\mathcal{H}$ denote the collection of the subgroups in proper parabolic subgroups of $\Gamma \mathcal{G}$. There exists an element $g \in \Gamma \mathcal{G}$ of full support such that $\Gamma \mathcal{G}$ satisfies the property $\mathrm{F}\mathbb{R}(\mathcal{H})$ relative to $\langle g \rangle$. 
\end{thm}

\noindent
The theorem is proved by applying small cancellation arguments to a specific action on a hyperbolic space. Such an action is constructed in Subsection \ref{section:FreeSubParabolic}, based on a hyperbolicity criterion proved in Subsection \ref{section:coningoffQM}. Theorem \ref{thm:FRH} is finally proved in Subsection~\ref{section:GPrelativeSplitting}.

\subsection{Hyperbolic cone-offs}\label{section:coningoffQM}

\noindent
In this subsection, given a simplicial graph $\Gamma$ and a collection of groups $\mathcal{G}$ indexed by $V(\Gamma)$, our goal is to construct hyperbolic graphs from $\QM$ by \emph{coning-off} some of its subspaces. Before stating our main result in this direction, let us specify the definition of \emph{cone-off} used in the sequel.

\begin{definition}
Let $X$ be a graph and $\mathcal{P}$ a collection of subgraphs. The \emph{cone-off of $X$ over $\mathcal{P}$}\index{Cone-off of a graph} is the graph obtained from $X$ by adding an edge between any two vertices which belong to the same subgraph in $\mathcal{P}$. 
\end{definition}

\noindent
The rest of the subsection is dedicated to the proof of the following statement:

\begin{prop}\label{prop:HypConeOff}
Let $\Gamma$ be a simplicial graph and $\mathcal{G}$ a collection of groups indexed by $V(\Gamma)$. Fix a collection $\mathscr{C}$ of subgraphs of $\Gamma$ and assume that every join subgraph in $\Gamma$ lies in a subgraph in $\mathscr{C}$. Then the cone-off of $\QM$ over $\{ g \langle \Lambda \rangle \mid g \in \Gamma \mathcal{G}, \Lambda \in \mathscr{C} \}$ is hyperbolic. 
\end{prop}

\noindent
We emphasize that our proposition is far from being optimal. For instance, if $\Gamma$ is square-free and contains at least one edge, then $\QM$ is already hyperbolic but the proposition does not apply (with $\mathscr{C}= \emptyset$). Nevertheless, Proposition \ref{prop:HypConeOff} will be sufficient to deduce the hyperbolicity of the graph constructed in the next subsection. The argument below follows closely \cite{coningoff}.

\medskip \noindent
During the proof of Proposition \ref{prop:HypConeOff}, the following criterion, proved by Bowditch in \cite[Proposition 3.1]{Bowditchcriterion}, will be used:

\begin{lemma}\label{lem:Bowditchcriterion}
Let $T$ be a graph and $D \geq 0$. Suppose that a connected subgraph $\eta(x,y)$, containing $x$ and $y$, is associated to each pair of vertices $(x,y) \in T^2$ such that
\begin{itemize}
	\item for all vertices $x,y \in T$, $d(x,y) \leq 1$ implies that $\mathrm{diam}(\eta(x,y)) \leq D$;
	\item for all vertices $x,y,z \in T$, we have $\eta(x,y) \subset ( \eta(x,z) \cup \eta(z,y))^{+D}$.
\end{itemize}
Then $T$ is $\delta$-hyperbolic for some $\delta$ depending only on $D$. 
\end{lemma}

\noindent
Here, given a subset $S$ in a metric space $X$ and a constant $K \geq 0$, we denote by $S^{+K}$ the \emph{$K$-neighborhood} of $S$, i.e. $\{x \in X \mid d(x,S) \leq K\}$. We are now ready to prove our proposition.

\begin{proof}[Proof of Proposition \ref{prop:HypConeOff}.]
Let $Y$ denote our cone-off. Fix two vertices $x,y \in Y$, two geodesics $\alpha,\beta$ in $\QM$ between $x$ and $y$, and $z \in \alpha$ a third vertex. Let $w \in \beta$ denote the vertex at the same distance from $x$ than $z$ with respect to the metric of $\QM$. We claim that $d_Y(z,w) \leq 1$.
\begin{figure}
\begin{center}
\includegraphics[scale=0.4]{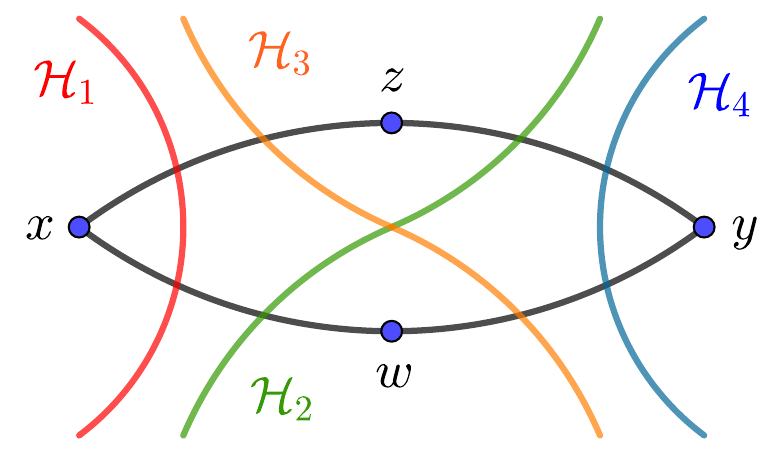}
\caption{Configuration from the proof of Proposition \ref{prop:HypConeOff}.}
\label{Bigon}
\end{center}
\end{figure}

\medskip \noindent
Decompose the set $\mathcal{H}$ of all the hyperplanes in $\QM$ separating $x$ and $y$ into four collections: the set $\mathcal{H}_1$ of the hyperplanes separating $x$ from $\{z,w,y\}$; the set $\mathcal{H}_2$ of the hyperplanes separating $\{x,z\}$ and $\{w,y\}$; the set $\mathcal{H}_3$ of the hyperplanes separating $\{x,w\}$ and $\{z,y\}$; and the set $\mathcal{H}_4$ of the hyperplanes separating $\{x,z,w\}$ and $y$. See Figure \ref{Bigon}. The fact that a geodesic crosses each hyperplane at most once implies that $\mathcal{H}= \mathcal{H}_1 \sqcup \mathcal{H}_2 \sqcup \mathcal{H}_3 \sqcup \mathcal{H}_4$. Notice that every hyperplane in $\mathcal{H}_2$ is transverse to every hyperplane in $\mathcal{H}_3$ and that $\# \mathcal{H}_2= \# \mathcal{H}_3$ because
$$\# \mathcal{H}_1 + \# \mathcal{H}_2 = d(x,w)=d(x,z)= \# \mathcal{H}_1 + \# \mathcal{H}_3.$$
Moreover, because $\mathcal{H}_2 \cup \mathcal{H}_3$ coincides with the hyperplanes in $\QM$ separating $z$ and $w$, the edges of a geodesic from $z$ to $w$ are labelled by the vertices of $\Gamma$ labelling the hyperplanes in $\mathcal{H}_2 \cup \mathcal{H}_3$. Therefore, $w \in z \langle \Lambda \rangle$ where $\Lambda \subset \Gamma$ decomposes as the join of the subgraph generated by the vertices labelling the hyperplanes in $\mathcal{H}_2$ with the subgraph generated by the vertices labelling the hyperplanes in $\mathcal{H}_3$. We conclude that $w$ and $z$ are at distance at most one in $Y$.

\medskip \noindent
Thus, we have proved the following assertion:

\begin{fact}\label{fact:Bigon}
The Hausdorff distance in $Y$ between any two geodesics in $\QM$ having the same endpoints is at most one.
\end{fact}

\noindent
Now, we are ready to apply Lemma \ref{lem:Bowditchcriterion}. For all vertices $x,y \in Y$, let $\eta(x,y) \subset Y$ denote the union of all the geodesics in $\QM$ between $x$ and $y$.

\medskip \noindent
Let us verify the first condition of Bowditch's criterion. If $x,y \in Y$ are two vertices satisfying $d_Y(x,y) \leq 1$, three cases may happen. First, $x$ and $y$ may be identical, so that $\eta(x,y)$ is reduced to a single vertex. Next, $x$ and $y$ may be linked by an edge in $\QM$, so that $\eta(x,y)$ is reduced to the unique edge containing $x$ and $y$. Finally, $x$ and $y$ may be linked by an edge of $Y$ which does not belong to $\QM$. So there exist an element $g \in \Gamma \mathcal{G}$ and a subgraph $\Lambda \in \mathscr{C}$ such that $g \langle \Lambda \rangle$ contains both $x$ and $y$. Notice that, because $g \langle \Lambda \rangle$ is convex in $\QM$, it has to contain $\eta(x,y)$, so the vertices of $\eta(x,y)$ are pairwise linked by an edge in $Y$. Thus, we have proved that $d_Y(x,y) \leq 1$ implies that $\mathrm{diam}_Y ~ \eta(x,y) \leq 1$. 

\medskip \noindent
Now, let us verify the second condition of Bowditch's criterion. So let $x,y,z \in Y$ be three vertices, and let $(x',y',z')$ denote their median triangle in $\QM$. Fix a vertex $w \in \eta(x,y)$. So there exists a geodesic $[x,y]$ in $\QM$ passing through $w$. Fix three geodesics $[x,x']$, $[x',y']$ and $[y',y]$ in $\QM$. By definition of a median triangle, the concatenation $[x,x'] \cup [x',y'] \cup [y',y]$ is a geodesic in $\QM$ from $x$ to $y$. As a consequence of Fact \ref{fact:Bigon}, there exists a vertex $w' \in [x,x'] \cup [x',y'] \cup [y',y]$ such that $d_Y(w,w') \leq 1$. We distinguish two cases. First, suppose that $w'$ belongs to $[x,x'] \cup [y',y]$. Up to switching $x$ and $y$, we may suppose without loss of generality that $w' \in [x,x']$. Fix two geodesics $[x',z']$ and $[z',z]$ in $\QM$, and notice that $[x,x'] \cup [x',z'] \cup [z',z]$ is a geodesic in $\QM$. We deduce again from Fact \ref{fact:Bigon} that there exists a vertex $w'' \in [x,x'] \cup [x',z'] \cup [z',z]$ satisfying $d_Y(w',w'') \leq 1$. Consequently,
$$d_Y(w, \eta(x,z)) \leq d_Y(w,w'') \leq d_Y(w,w')+d_Y(w',w'') \leq 2.$$
Now, suppose that $w'$ belongs to $[x',y']$. According to Proposition \ref{prop:MedianTriangleQM}, the vertices $x',y',z'$ lie in a common prism $P$, and $w'$ must belong to it by convexity. Because $y'$ belongs to $\eta(y,z)$, we deduce that
$$\begin{array}{lcl} d_Y(w,\eta(y,z)) & \leq & d_Y(w,y') \leq d_Y(w,w')+ d_Y(w',y') \leq 1+ \mathrm{diam}_Y(P) \\ \\ & \leq & 1+ \mathrm{diam}_{\QM}(P) \leq 1+ \mathrm{clique}(\Gamma) \end{array}$$
where $\mathrm{clique}(\Gamma)$ denotes the maximal size of a complete subgraph in $\Gamma$. Thus, we have proved that $\eta(x,y) \subset ( \eta(x,z) \cup \eta(z,y) )^{+ (2+ \mathrm{clique}(\Gamma))}$.

\medskip \noindent
Therefore, Bowditch's criterion applies and proves the desired conclusion.
\end{proof}

\subsection{Free subgroups avoiding parabolic subgroups}\label{section:FreeSubParabolic}

\noindent
The goal of this subsection is to prove the following statement:

\begin{thm}\label{thm:freesubgroupvsparabolic}
Let $\Gamma$ be a finite simplicial graph which contains at least two vertices and which is not a join; and let $\mathcal{G}$ be a collection of groups indexed by $V(\Gamma)$. If $g \in \Gamma \mathcal{G}$ has full support, then there exists some $n \geq 0$ such that the normal closure $\langle \langle g^n \rangle \rangle$ is a free subgroup of $\Gamma \mathcal{G}$ intersecting trivially every proper parabolic subgroup. 
\end{thm}

\noindent
The main tool which we will use to prove this statement is the following small cancellation theorem \cite[Theorem 8.7]{DGO}:

\begin{prop}\label{prop:DGOWPD}
Let $G$ be a group acting on a hyperbolic space $Y$. If $g \in G$ is a WPD element, then there exists some $n \geq 1$ such that the normal closure $\langle \langle g^n \rangle \rangle$ is a free subgroup all of whose non-trivial elements are loxodromic. 
\end{prop}

\noindent
So the strategy to prove Theorem \ref{thm:freesubgroupvsparabolic} is to construct a hyperbolic space on which our graph product acts such that proper parabolic subgroups are elliptic and such that elements of full support are WPD.

\begin{proof}[Proof of Theorem \ref{thm:freesubgroupvsparabolic}.]
Let $\mathcal{P}$ be the following collection of subgraphs of $\QM$:
$$\{ g \langle \Lambda \rangle \mid g \in \Gamma \mathcal{G}, \ \Lambda \subsetneq \Gamma \}.$$
Notice that, as a consequence of Proposition \ref{prop:HypConeOff}, the cone-off $Y$ of $\QM$ over $\mathcal{P}$ is hyperbolic. Our goal now is to show that that every element of $\Gamma \mathcal{G}$ with full support defines a WPD element with respect to the action $\Gamma \mathcal{G} \curvearrowright Y$. As a consequence of Proposition \ref{prop:DGOWPD}, and since proper parabolic subgroups are clearly elliptic in $Y$, the conclusion of our theorem follows. 

\medskip \noindent
We begin by introducing some terminology. 
\begin{itemize}
	\item A \emph{block} in $\QM$ is a collection of hyperplanes which contains at least one hyperplane labelled by each vertex of $\Gamma$. 
	\item A block \emph{separates} two subspaces $A,B \subset \QM$ if each of its hyperplanes contains $A$ and $B$ in two distinct sectors. 
	\item A \emph{chain of blocks} is a finite sequence of blocks $B_1, \ldots, B_n$ such that $B_i$ separates $B_{i-1}$ and $B_{i+1}$ for every $2 \leq i \leq n-1$. 
\end{itemize}
Our goal is to find a lower estimate of the metric in $Y$. We begin by the following observation:

\begin{claim}\label{claim:noblock}
Two vertices of $\QM$ which are adjacent in $Y$ cannot be separated by a block.
\end{claim}

\noindent
Suppose for contradiction that there exist two vertices $x,y \in \QM$ which are adjacent in $Y$ and which are separated by a block $B$. By construction of $Y$, there must exist $g \in \Gamma \mathcal{G}$ and $\Lambda \subsetneq \Gamma$ such that $x$ and $y$ both belong to $g \langle \Lambda \rangle$. Since every hyperplane separating $x$ and $y$ must intersect $g \langle \Lambda \rangle$, we deduce that the hyperplanes of $B$ must be labelled by vertices of $\Lambda$. But $B$ has to contain a hyperplane labelled by a vertex which does not belong to $\Lambda$ by definition of a block and because $\Lambda$ is a proper subgraph of $\Gamma$. We get our contradiction, proving the claim.

\medskip \noindent
Now we are ready to prove our lower estimate of the metric in $Y$. 

\begin{claim}\label{claim:distanceY}
For all vertices $x,y \in \QM$ and integer $N \geq 0$, if there exists a chain of $N$ blocks separating $x$ and $y$ in $\QM$, then $d_Y(x,y) \geq N+2$.
\end{claim}

\noindent
Consider a chain of blocks $B_1, \ldots, B_N$ separating two vertices $x,y \in \QM$. For every $1 \leq i \leq N$, let $B_i^+$ denote the connected component of $\bigcap\limits_{J \in B_i} \QM \backslash \backslash J$ which contains $y$. We also fix a geodesic $x_1, \ldots, x_n$ from $x$ to $y$ in $Y$. Notice that, as a consequence of Claim \ref{claim:noblock}, for every $1 \leq i \leq n-1$ the vertices $x_i$ and $x_{i+1}$ cannot be separated by a block. Therefore, $x_2$ cannot belong to $B_1^+$, since otherwise $B_1$ would separate $x_1$ and $x_2$. Similarly, $x_3$ cannot belong to $B_2^+$, since otherwise $B_2$ would separate $x_2$ and $x_3$. By iterating the argument, we conclude that, for every $1 \leq i \leq N$, the vertex $x_{i+1}$ does not belong to $B_i^+$. Therefore, since $x_n=y$ belongs to $B_N^+$, it follows that $n \geq N+2$. This concludes the proof of our claim.

\medskip \noindent
Before proving that an element of $\Gamma \mathcal{G}$ is WPD if it has full support, we need to show a preliminary claim. 

\begin{claim}
Let $g \in \Gamma \mathcal{G}$ be an element with full support. There exists a block $B$ and a power $s \geq 1$ such that, for every $k \in \mathbb{Z}$, the block $g^{sk}B$ separates $g^{s(k-1)}B$ and $g^{s(k+1)}B$; and every two consecutive blocks are separated by at least two strongly separated hyperplanes.
\end{claim}

\noindent
Up to conjugating $g$, we may suppose without loss of generality that $g$ is proper cyclically reduced. Because $\Gamma$ is not the star of one of its vertices, it follows that the product $g \cdots g$ of $n$ copies of $n$ is graphically reduced for every $n \in \mathbb{Z} \backslash \{0\}$. As a consequence, if we fix a geodesic $[1,g]$ from $1$ to $g$ in $\QM$, then 
$$\gamma = \bigcup\limits_{n \in \mathbb{Z}} g^n \cdot [1,g]$$
is a bi-infinite geodesic on which $\langle g \rangle$ acts by translations. Let $B$ denote the collection of the hyperplanes of $\QM$ separating $1$ and $g$. Because $g$ has full support, $B$ has to be a block. According to Proposition \ref{prop:ContractingAxis}, there exist a hyperplane $J$ crossing $\gamma$ and a power $r \geq 1$ such that $\{g^{rk} J \mid k \in \mathbb{Z}\}$ is a collection of pairwise strongly separated hyperplanes. Moreover, $g^{rk}J$ separates $g^{r(k-1)}J$ and $g^{r(k+1)}J$ for every $k \in \mathbb{Z}$. Fix a sufficiently large multiple $s$ of $r$ so that, for every $k \in \mathbb{Z}$, the blocks $g^{sk}B$ and $g^{s(k+1)}B$ are separated by at least three pairwise strongly separated hyperplanes of the collection $\{g^{rk} J \mid k \in \mathbb{Z}\}$; notice that, as a consequence, two hyperplanes belonging to $g^{sk}B$ and $g^{s(k+1)}B$ respectively cannot be transverse. We conclude that  $g^{sk}B$ separates $g^{s(k-1)}B$ and $g^{s(k+1)}B$ for every $k \in \mathbb{Z}$.

\begin{claim}\label{claim:WPDfullsupport}
Let $g \in \Gamma \mathcal{G}$ be an element with full support. Then $g$ is a WPD element with respect to the action $\Gamma \mathcal{G} \curvearrowright Y$.
\end{claim}

\noindent
According to the previous claim, there exist a block $B$ and power $s \geq 1$ such that $g^{sk}B$ separates $g^{s(k-1)}B$ and $g^{s(k+1)}B$ for every $k \in \mathbb{Z}$; and such that any two consecutive blocks are separated by at least two strongly separated hyperplanes. Fix some $\epsilon \geq 0$ and $x \in g^{-s}B$. We set $N=8(\epsilon+2(1+\delta))$, where $\delta$ is such that $(Y,d_Y)$ is a $\delta$-hyperbolic space, and $n = (N+2)s$. In order to deduce that $g$ is a WPD element, we need to show that the set
$$S:= \{ h \in G \mid d(x,hx), d(g^nx,hg^nx) \leq \epsilon \}$$
is finite. Notice that $x$ and $g^{n}x$ are separated by the blocks $B_0:=B$, $B_1:=g^sB$, $\ldots$, $B_N:=g^{sN}B$. Fix four indices $0 \leq a \leq b \leq c \leq d \leq N$ satisfying:
\begin{itemize}
	\item $a \geq \epsilon$ and $N-d \geq \epsilon$;
	\item $b-a \geq \epsilon +4(1+2 \delta)$ and $d-c \geq \epsilon+4(1+2 \delta)$;
	\item $c \geq d+1$.
\end{itemize}
Notice that such indices exist because we chose $N$ sufficiently large. Finally, fix some hyperplane $J$ separating the blocks $B_b$ and $B_c$. For reader's convenience, we decompose the proof of our claim into a few facts.

\begin{fact}
If $[x,g^nx]$ is a geodesic between $x$ and $g^nx$ in $Y$, then $N(J)$ is contained in the $2$-neighborhood of $[x,g^nx]$ in $Y$. 
\end{fact} 

\noindent
Let $p \in [x,g^nx]$ denote the last point of the geodesic which belongs to the sector delimited by $J$ containing $x$ and let $q$ denote the next vertex along $[x,g^nx]$. If $p$ and $q$ are adjacent in $\QM$, then $p,q \in N(J)$. Otherwise, there exists some $P \in \mathcal{P}$ containing both $p$ and $q$. Because $J$ separates $p$ and $q$ in $\QM$, necessarily $J$ intersects the subgraph $P$. In either case, we deduce that $d_Y(p,N(J)) \leq 1$ and $d_Y(q,N(J)) \leq 1$. But $N(J)$ is contained in a subgraph which belongs to $\mathcal{P}$, so it has diameter at most one in $Y$. In other words, every vertex of $N(J)$ is at distance at most two from the points $p$ and $q$ of $[x,g^nx]$. 

\begin{fact}\label{fact:transhypbyh}
If $h \in S$ and $y \in N(J)$, then $d_Y(y,hy) \leq \epsilon+4(1+2\delta)$.
\end{fact}

\noindent
Fix a geodesic $[x,g^nx]$ between $x$ and $g^nx$ in $Y$. As a consequence of the previous fact, there exists some $z \in [x,g^nx]$ such that $d_Y(y,z) \leq 2$. Next, according to \cite[Corollary 5.3]{CDP}, the distance in a $\delta$-hyperbolic space is $8\delta$-convex, hence $d_Y(z,hz) \leq \epsilon+8 \delta$. Consequently,
$$d_Y(y,hy) \leq d_Y(y,z)+d_Y(z,hz)+d_Y(hz,hy) \leq 4+\epsilon+8\delta$$
concluding the proof of our fact. 

\begin{fact}
If $h \in S$, then $hJ$ separates the blocks $B_a$ and $B_d$. 
\end{fact}

\noindent
If $N(hJ)$ intersects the subspace delimited by $B_{a+1}$ which contains $x$, then there exists some $y \in N(J)$ such that $y$ and $hy$ are separated by $B_{a+1}, \ldots, B_{b}$, hence $d_Y(y,hy) \geq b-a+2$ according to Claim \ref{claim:distanceY}. But we chose $a$ and $b$ so that $b-a \geq \epsilon+4(1+2 \delta)$, which contradicts Fact \ref{fact:transhypbyh}. Therefore, $N(hJ)$ must be included in the subspace delimited by $B_a$ which contains $g^nx$. Similarly, one show that $N(hJ)$ is included in the subspace delimited by $B_d$ which contains $x$. 

\medskip \noindent
Consequently, if $hJ$ does not separate $B_a$ and $B_d$, then $h$ has to send either $x$ or $g^nx$ in the subspace delimited by $B_a$ and $B_d$. In the former case, $B_0,\ldots, B_a$ separate $x$ and $hx$, hence $d_Y(x,hx) \geq a+3$ according to Claim \ref{claim:distanceY}; and in the latter case, $B_{d+1}, \ldots, B_N$ separate $g^nx$ and $hg^nx$, hence $d_Y(g^nx,hg^nx) \geq N-d+2$ according to Claim \ref{claim:distanceY}. But we chose $a$ and $d$ so that $a \geq \epsilon$ and $N-d \geq \epsilon$, contradicting the definition of $S$. This concludes the proof of our fact.

\medskip \noindent
So far, we have proved that every element of $S$ defines a map from the set of hyperplanes separating $B_b$ and $B_c$ to the set of hyperplanes separating $B_a$ and $B_d$. Because these two collections of hyperplanes are necessarily finite, we conclude that $S$ must be covered by finitely many cosets of 
$$\bigcap\limits_{\text{$J$ separates $B_b$ and $B_c$}} \mathrm{stab}(J).$$ 
But, by construction, there exist two strongly separated hyperplanes separating $B_b$ and $B_c$, implying that the intersection above must be trivial. We deduce that $S$ must be finite, and finally that $g$ is a WPD element, concluding the proof of Claim \ref{claim:WPDfullsupport}.

\medskip \noindent
Finally, Theorem \ref{thm:freesubgroupvsparabolic} follows immediately from Proposition \ref{prop:DGOWPD}.
\end{proof}

\noindent
Let us conclude this section with a formal interpretation of the results proved above.

\begin{definition}
Let $G$ be a group. A collection $\mathcal{H}$ of subgroups of $G$ is \emph{simultaneously elliptic}\index{Simultaneously elliptic collection of subgroups} if there exists a hyperbolic space on which $G$ acts with at least one WPD element such that all the subgroups in $\mathcal{H}$ are elliptic. An element $g \in G$ is a \emph{generalised loxodromic element relative to $\mathcal{H}$}\index{Relative generalised loxodromic elements} if there exists a hyperbolic space on which $G$ acts such that $g$ is WPD and such that every subgroup in $\mathcal{H}$ is elliptic.
\end{definition}

\noindent
The proof of Theorem \ref{thm:freesubgroupvsparabolic} consists essentially in showing that the elements of our graph product which have full support (or equivalently, which do not belong to a proper parabolic subgroup) are generalised loxodromic elements relative to the collection of proper parabolic subgroups. We record this fact for future use:

\begin{prop}\label{prop:ParabolicElliptic}
Let $\Gamma$ be a finite simplicial graph which contains at least two vertices and which is not a join; and let $\mathcal{G}$ be a collection of groups indexed by $V(\Gamma)$. An element of $\Gamma \mathcal{G}$ is a generalised loxodromic element relative to the collection of proper parabolic subgroups if and only if it has full support. In particular, the collection of proper parabolic subgroups is simultaneously elliptic.
\end{prop}

\noindent
In fact, we proved something stronger. We proved that all the elements of full support can be simultaneously realised as WPD isometries for a single action on a hyperbolic space for which proper parabolic subgroups are elliptic. The next step in the proof of Theorem \ref{thm:freesubgroupvsparabolic} was to apply the following immediate consequence of Proposition \ref{prop:DGOWPD}:

\begin{cor}\label{cor:NormalFree}
Let $G$ be a group and $\mathcal{H}$ a collection of subgroups. If $\mathcal{H}$ is simultaneously elliptic, then there exists some $g \in G$ such that $\langle \langle g \rangle \rangle$ is a free subgroup intersecting trivially each subgroup of $\mathcal{H}$. 
\end{cor}

\noindent
This formal point of view will allow us to state and prove some of the results of the next subsection in a more general setting.

\subsection{Relative fixed point property}\label{section:GPrelativeSplitting}

\noindent
Now, we are ready to state and prove the main result of the section, namely:

\begin{thm}\label{thm:GPrelativesplitting}
Let $\Gamma$ be a finite simplicial graph and $\mathcal{G}$ a collection of groups indexed by $V(\Gamma)$. Let $\mathcal{H}$ denote the collection of the subgroups in proper parabolic subgroups of $\Gamma \mathcal{G}$. There exists an element $g \in \Gamma \mathcal{G}$ of full support such that $\Gamma \mathcal{G}$ satisfies the property $\mathrm{F}\mathbb{R}(\mathcal{H})$ relative to $\langle g \rangle$. 
\end{thm}

\noindent
We begin by proving the following statement in a more general setting:

\begin{prop}\label{prop:RelativeSplittingCollection}
Let $G$ be a finitely generated group and $\mathcal{H}$ a collection of subgroups. If $\mathcal{H}$ is simultaneously elliptic, then there exists some $h \in G$ which does not belong to any subgroup in $\mathcal{H}$ and such that $G$ satisfies the property $\mathrm{F}\mathbb{R}(\mathcal{H})$ relative to $\langle h \rangle$. 
\end{prop}

\noindent
The strategy is to apply Corollary \ref{cor:NormalFree} in order to deduce from an action of our group on a real tree an action of a free subgroup with trivial stabilisers, and next to apply the following statement: 

\begin{lemma}\label{lem:FreeSplit}
Let $F$ be a finitely generated free group. There exists some $g \in F$ such that $F$ satisfies the property $\mathrm{F}\mathbb{R}(\{1\})$ relative to $\langle g \rangle$.
\end{lemma}

\begin{proof}
If $F$ is cyclic, there is nothing to prove, so from now on we assume that $F$ has rank at least two. According to \cite{MR2585579} (see also \cite{Solie, MR3367520, GuirardelLevittRandomSplit}), there exists some $g \in F$ such that $F$ does not split relative to $g$ over a cyclic subgroup. Now, assume that $F$ acts on a real tree $T$ with trivial arc-stabilisers and with $g$ as an elliptic isometry. Without loss of generality, we assume that $F$ acts minimally on $T$. We claim that the hypotheses of \cite[Theorem 5.1]{GuirardelRealTree} hold. Indeed, the first assumption is an immediate consequence of the fact that arc-stabilisers are trivial; and the second assumption follows from \cite[Lemma~1.14]{GuirardelRealTree} and from the observation that $T$ has no unstable arc. Consequently, \cite[Corollary~5.2]{GuirardelRealTree} applies and show that either $T$ is a line or $F$ splits over a cyclic subgroup relative to $\langle g \rangle$. The former case cannot happen since otherwise there would exist a morphism $F \to \mathbb{R}$ with trivial kernel (since arc-stabilisers are trivial), contradicting the assumption that $F$ has rank at least two. Thus, the desired conclusion follows.
\end{proof}

\begin{proof}[Proof of Proposition \ref{prop:RelativeSplittingCollection}.]
According to Corollary \ref{cor:NormalFree}, there exists an element $g \in G$ such that $\langle \langle g \rangle \rangle$ is a free subgroup of $G$ intersecting trivially every subgroup in $\mathcal{H}$. Set the subgroup 
$$S = \left\langle xgx^{-1}, \ x \in R \cup \{ 1 \} \right\rangle$$ 
where $R$ is a fixed finite generating set of $G$. Notice that, as a consequence of our choice of $g$, $S$ is a finitely generated free subgroup of $G$ intersecting trivially every subgroup in $\mathcal{H}$ since it is contained in $\langle \langle g \rangle \rangle$. By applying Lemma \ref{lem:FreeSplit}, we know that $S$ contains some element $h$ such that $S$ satisfies the property $F\mathbb{R}(\{1\})$ relative to $\langle h \rangle$. Notice that $h$ does not belong to any subgroup in $\mathcal{H}$.

\medskip \noindent
Fix a real tree $T$ on which $G$ acts with arc-stabilisers in $\mathcal{H}$. Our goal is to show that $h$ must be a loxodromic isometry of $T$.

\medskip \noindent
First of all, we claim that the induced action $S \curvearrowright T$ is fixed-point free.

\medskip \noindent
If $g$ is a loxodromic isometry of $T$, then the conclusion is clear. So suppose that $g$ is elliptic. Notice that $g$ cannot stabilise an arc since it does not belong to a subgroup in $\mathcal{H}$, so the fixed-point set of $g$ is reduced to a single point, say $p \in T$. Now, we distinguish two cases.

\medskip \noindent
Suppose first that there exists $x \in R$ such that $x$ is a loxodromic isometry of $T$. Then $g$ and $xgx^{-1}$ are two elliptic isometries whose fixed-point sets are disjoint, so that it follows from \cite[Lemma 3.2.2]{MR1851337} that $g \cdot xgx^{-1}$ is a loxodromic element of $T$. Because this element belongs to $S$, we conclude that $S$ does not fix a point of $T$.

\medskip \noindent
Now, suppose that every element of $R$ is elliptic. Since $G$ does not fix a point of $T$ and that $R$ generates $G$, we know that there exists $x \in R$ which does not fix the vertex $p$ fixed by $g$. Let $q \in T$ denote a vertex fixed by $x$. Notice that, because $g$ does not stabilise an arc, the vertex $p$ separates $q$ and $g \cdot q$. As a consequence, if the fixed-point sets of $x$ and $gx^{-1}g^{-1}$ intersect, then $p$ must be fixed by either $x$ or $gx^{-1}g^{-1}$. We already know that $x$ does not fix $p$. But, in the latter case, we deduce that
$$x^{-1} \cdot p = g^{-1} \cdot gx^{-1}g^{-1} \cdot g \cdot p= g^{-1} \cdot gx^{-1}g^{-1} \cdot p = g^{-1} \cdot p = p,$$
which is impossible. Therefore, the fixed-point sets of $x$ and $gx^{-1}g^{-1}$ must be disjoint. It follows from \cite[Lemma 3.2.2]{MR1851337} that $x \cdot gx^{-1}g^{-1}$ is a loxodromic element of $T$. Because this element belongs to $S$, we conclude that $S$ does not fix a point of $T$.

\medskip \noindent
This concludes the proof of our claim, i.e., $S$ acts fixed-point freely on $T$. Moreover, since $S$ intersects trivially every subgroup in $\mathcal{H}$, $S$ acts on $T$ with trivial arc-stabilisers. From the definition of $h$, we deduce that $h$ must be a loxodromic isometry of $T$. 

\medskip \noindent
Thus, we have proved that $h$ defines a loxodromic isometry of every real tree on which $G$ acts with arc-stabilisers in $\mathcal{H}$. This precisely means that $G$ satisfies the property $F\mathbb{R}(\mathcal{H})$ relative to $\langle h \rangle$, concluding the proof of our theorem.
\end{proof}

\begin{proof}[Proof of Theorem \ref{thm:GPrelativesplitting}.]
According to Proposition \ref{prop:ParabolicElliptic}, the collection of proper parabolic subgroups of $\Gamma \mathcal{G}$ is simultaneously elliptic. Therefore, Proposition \ref{prop:RelativeSplittingCollection} applies, and the desired conclusion follows.
\end{proof}

\section{Proofs of the main theorem and its corollaries}

\noindent
By combining the theorems proved in Sections \ref{section:StepTwo}, \ref{section:step2} and \ref{section:step3}, we are now able to prove our main criterion of acylindrical hyperbolicity:

\begin{thm}\label{thm:MainAcylHyp}
Let $\Gamma$ be a finite simplicial graph and $\mathcal{G}$ a collection of finitely generated irreducible groups indexed by $V(\Gamma)$. Assume that $\Gamma$ is not a join and that it contains at least two vertices. Either $\Gamma$ is a pair of two isolated vertices both labelled by $\mathbb{Z}_2$, and $\Gamma \mathcal{G}$ is infinite dihedral; or $\mathrm{Aut}(\Gamma \mathcal{G})$ is acylindrically hyperbolic.
\end{thm}

\noindent
Before turning to the proof of the theorem, we need the following elementary observation:

\begin{lemma}\label{lem:ForGraphicallyReduced}
For every finitely generated irreducible group $G$, there exists a finite simplicial graph $\Gamma$ and a collection of groups $\mathcal{G}$ indexed by $V(\Gamma)$ such that $G$ is isomorphic to $\Gamma \mathcal{G}$ and such that the groups in $\mathcal{G}$ are finitely generated and graphically irreducible.
\end{lemma}

\begin{proof}
If $G$ is graphically reduced, then we set $\Gamma= \{u\}$ and $\mathcal{G}= \{G_u=G\}$, and we are done. Otherwise, if $G$ is not graphically irreducible, then $G$ is isomorphic to $\Phi \mathcal{H}$ for some finite simplicial graph $\Phi$ which is not a complete and for some collection $\mathcal{H}$ of groups indexed by $V(\Gamma)$. Notice that, because $G$ is irreducible, $\Phi$ is not a join. As a consequence, if $u \in V(\Phi)$ then there exists some $v \in V(\Phi)$ which is not adjacent to $u$. It follows that $G$ surjects onto $\langle u,v \rangle \simeq H_u \ast H_v$, which implies that $\mathrm{rank}(H_u)< \mathrm{rank}(G)$. By induction over the rank, each vertex-group $H_u$ decomposes as a graph product $\Gamma_u \mathcal{G}_u$ where $\Gamma_u$ is a finite simplicial graph and where $\mathcal{G}_u$ is a collection of finitely generated graphically irreducible groups indexed by $V(\Gamma_u)$. Then $G$ decomposes as the graph product $\Gamma \mathcal{G}$ where
\begin{itemize}
	\item $V(\Gamma)= \bigsqcup\limits_{u \in V(\Phi)} V(\Gamma_u)$;
	\item two vertices $a \in V(\Gamma_u)$ and $b \in V(\Gamma_v)$ are adjacent in $\Gamma$ if and only if either $u=v$ and $a,b$ are adjacent in $\Gamma_u$ or $u$ and $v$ are adjacent in $\Phi$;
	\item $\mathcal{G} = \bigsqcup\limits_{u \in V(\Phi)} \mathcal{G}_u$.
\end{itemize}
Clearly, $\Gamma$ is finite and the groups in $\mathcal{G}$ are finitely generated and graphically irreducible, as desired.
\end{proof}

\begin{proof}[Proof of Theorem \ref{thm:MainAcylHyp}.]
If $\Gamma$ is not connected then $\Gamma \mathcal{G}$ decomposes as a free product, and it follows from \cite{AutFreeProductsHyp} that $\mathrm{Aut}(\Gamma \mathcal{G})$ is acylindrically hyperbolic unless $\Gamma \mathcal{G}$ is virtually cyclic, which happens precisely when $\Gamma$ is a pair of two isolated vertices both labelled by $\mathbb{Z}_2$. From now on, we assume that $\Gamma$ is connected.

\medskip \noindent
According to Lemma \ref{lem:ForGraphicallyReduced}, there exist a finite simplicial graph $\Phi$ and a collection of groups $\mathcal{H}$ indexed by $V(\Phi)$ such that $\Gamma\mathcal{G}$ is isomorphic to $\Phi \mathcal{H}$ and such that the groups in $\mathcal{H}$ are finitely generated and graphically irreducible. Notice that, because $\Gamma$ is not a join, $\Gamma \mathcal{G} \simeq \Phi \mathcal{H}$ is irreducible, so $\Phi$ is not a join either. Also, because $\Gamma$ is connected, $\Gamma \mathcal{G} \simeq \Phi \mathcal{H}$ is freely irreducible, so $\Phi$ has to be connected. Finally, notice that, because $\Gamma$ is not complete and because the groups in $\mathcal{H}$ are graphically irreducible, $\Phi$ cannot be a single vertex.

\medskip \noindent
According to Theorem \ref{thm:FRH}, there exists an element $g\in \Phi \mathcal{H}$ which has full support such that $\Phi \mathcal{H}$ satisfies the property $\mathrm{F} \mathbb{R}(\mathcal{K})$ relative to $\langle g \rangle$, where $\mathcal{K}$ denotes the collection of all the subgroups in proper parabolic subgroups. By combining Theorems~\ref{thm:WhenGenLox} and~\ref{thm:ActionRealTree}, it follows that the inner automorphism $\iota_g$ defines a generalised loxodromic element of $\mathrm{Aut}(\Gamma \mathcal{G}) \simeq \mathrm{Aut}(\Phi \mathcal{H})$. Moreover, because $\Gamma$ is connected, is not a join and contains at least two vertices, we know that $\Gamma \mathcal{G}$ is not virtually cyclic and that its center is trivial, so $\mathrm{Aut}(\Gamma \mathcal{G})$ cannot be virtually cyclic. We conclude that $\mathrm{Aut}(\Gamma \mathcal{G})$ is acylindrically hyperbolic.
\end{proof}

\noindent
Now, we turn to the proof of the main theorem of the article, namely Theorem \ref{BigThmIntro}. It will be a direct consequence of Theorem \ref{thm:MainAcylHyp} above and the next proposition:

\begin{prop}\label{prop:StructureAut}
Let $\Gamma$ be a finite simplicial graph and $\mathcal{G}$ a collection of graphically irreducible groups indexed by $V(\Gamma)$. Decompose $\Gamma$ as a join $\Gamma_0 \ast \Gamma_1 \ast \cdots \ast \Gamma_n$ such that $\Gamma_0$ is complete and such that each graph among $\Gamma_1,\ldots, \Gamma_n$ contains at least two vertices and is not a join. Then
$$\mathrm{Aut}(\Gamma \mathcal{G}) \simeq \mathrm{Hom} \left( \bigoplus\limits_{i=1}^n \langle \Gamma_i \rangle \to Z(\langle \Gamma_0 \rangle) \right) \rtimes \left( \mathrm{Aut}(\langle \Gamma_0 \rangle) \oplus \left[ \left( \bigoplus\limits_{i=1}^n \mathrm{Aut}(\langle \Gamma_i \rangle) \right) \rtimes S \right] \right)$$
where $S$ is a finite group permuting the isomorphic factors of $\langle \Gamma_1 \rangle \oplus \cdots \oplus \langle \Gamma_n \rangle$.
\end{prop}

\begin{proof}
Let $A$ denote the subgroup of $\mathrm{Aut}(\Gamma \mathcal{G})$ consisting of the automorphisms
$$\tau_\phi : a \mapsto \left\{ \begin{array}{cl} a & \text{if $a \in \langle \Gamma_0 \rangle$} \\ \phi(a)a & \text{if $a \in \langle \Gamma_1 \cup \cdots \cup \Gamma_n \rangle$} \end{array} \right., \ \phi \in \mathrm{Hom}\left( \langle \Gamma_1 \cup \cdots \cup \Gamma_n \rangle \to Z(\langle \Gamma_0 \rangle) \right).$$
Also, let $B$ denote the subgroup $\mathrm{Aut}(\langle \Gamma_0 \rangle) \leq \mathrm{Aut}(\Gamma \mathcal{G})$ and $C$ the subgroup of $\mathrm{Aut}(\Gamma \mathcal{G})$ containing the automorphisms $\varphi$ such that 
\begin{itemize}
	\item $\varphi_{|\langle \Gamma_0 \rangle} = \mathrm{Id}_{|\langle \Gamma_0 \rangle}$;
	\item and there exists a bijection $\sigma : \{1, \ldots, n \} \to \{1, \ldots, n\}$ such that $\varphi(\langle \Gamma_i \rangle)= \langle \Gamma_{\sigma(i)} \rangle$ for every $1 \leq i \leq n$.
\end{itemize}
Clearly, $\langle B,C \rangle = B \oplus C$ and 
$$A \simeq \mathrm{Hom} \left( \bigoplus\limits_{i=1}^n \langle \Gamma_i \rangle \to Z(\langle \Gamma_0 \rangle) \right) \text{ and } C \simeq \left( \bigoplus\limits_{i=1}^n \mathrm{Aut}(\langle \Gamma_i \rangle) \right) \rtimes S$$ 
where $S$ is a finite group which permutes the isomorphic factors of $\langle \Gamma_1 \rangle \oplus \cdots \oplus \langle \Gamma_n \rangle$. Also, it follows from Proposition~\ref{prop:Product} that $\langle A,B,C \rangle = \mathrm{Aut}(\Gamma \mathcal{G})$. Therefore, in order to conclude the proof of our proposition, it suffices to show that $B$ and $C$ normalise $A$. It is straightforward to verify that
$$b \circ \tau_\phi \circ b^{-1} = \tau_{b \circ \phi} \text{ and } c \circ \tau_\phi \circ c^{-1} = \tau_{\phi \circ c}$$
for every morphism $\phi \in \mathrm{Hom}\left( \langle \Gamma_1 \cup \cdots \cup \Gamma_n \rangle \to Z(\langle \Gamma_0 \rangle) \right)$ and all automorphisms $a \in A$ and $b \in B$.
\end{proof}

\begin{proof}[Proof of Theorem \ref{BigThmIntro}.]
The theorem is a direct consequence of Theorem \ref{thm:MainAcylHyp} and Proposition \ref{prop:StructureAut}.
\end{proof}

\paragraph{Vastness properties.} Following \cite{OutRAAGlarge}, we define vastness properties as follows:

\begin{definition}
A group property $\mathcal{P}$ is a \emph{vastness property}\index{Vastness property} if it satisfies the following conditions:
\begin{itemize}
	\item If a group $G$ has a quotient satisfying $\mathcal{P}$, then $G$ satisfies $\mathcal{P}$.
	\item For every group $G$ and every finite-index subgroup $H \leq G$, $G$ satisfies $\mathcal{P}$ if and only if $H$ does.
	\item If $1 \to N \to G \to Q \to 1$ is an exact sequence where $N$ and $Q$ do not have $\mathcal{P}$, then $G$ does not have $\mathcal{P}$.
\end{itemize}
\end{definition}

\noindent
Our results regarding vastness properties of automorphism groups of graph products are consequences of the combination of Theorem \ref{BigThmIntro} with the next statement:

\begin{prop}\label{prop:Vast}
Being SQ-universal, virtually having many quasimorphisms, not being boundedly generated and involving all finite groups are vastness properties. Moreover, the first three properties are satisfied by acylindrically hyperbolic groups.
\end{prop}

\begin{proof}
The four properties are vastness properties according to \cite[Corollaries 1.6 and 1.13, Propositions 1.18 and 1.22]{OutRAAGlarge}. The second assertion of our proposition is proved in \cite[Theorem 2.33]{DGO}.
\end{proof}

\begin{proof}[Proof of Theorem \ref{thm:VastProp}.]
Decompose $\Gamma$ as a join $\Gamma_0 \ast \Gamma_1 \ast \cdots \ast \Gamma_n$ where $\Gamma_0$ is complete and where each graph among $\Gamma_1, \ldots, \Gamma_n$ contains at least two vertices and is not a join. By assumption, there exists some $1 \leq i \leq n$ such that $\Gamma_i$ is not a pair of isolated vertices both labelled by $\mathbb{Z}_2$. It follows from Proposition \ref{prop:StructureAut} that $\mathrm{Aut}(\Gamma \mathcal{G})$ contains a finite-index subgroup which surjects onto $\mathrm{Aut}(\langle \Gamma_i \rangle)$, and it follows from Theorem \ref{thm:MainAcylHyp} that $\mathrm{Aut}(\langle \Gamma_i \rangle)$ is acylindrically hyperbolic. The desired conclusion follows from Proposition \ref{prop:Vast}.
\end{proof}

\begin{proof}[Proof of Corollary \ref{cor:VastProp}.]
As observed in \cite[Proposition 1.7]{OutRAAGlarge}, automorphism groups of non-abelian right-angled Artin groups involve all finite groups. Similarly, if $G$ is a graph product of finite groups which is not virtually abelian, then $\mathrm{Inn}(G)$ is a normal subgroup of $\mathrm{Aut}(G)$ which surjects onto a free product of two finite groups different from $\mathbb{Z}_2 \ast \mathbb{Z}_2$, which is virtually free but not virtually cyclic, so it follows from \cite[Lemma 1.5]{OutRAAGlarge} that $\mathrm{Aut}(G)$ involves all finite groups. The rest of our corollary is a direct consequence of Theorem \ref{thm:VastProp}.
\end{proof}

\paragraph{Acylindrical hyperbolicity of automorphism groups.} Now, let us determine when the automorphism group of a graph product is acylindrically hyperbolic. In other words, we want to prove Theorem \ref{thm:WhenAutHyp}. During the proof, the following elementary properties of acylindrically hyperbolic groups will be used:

\begin{lemma}\label{lem:HypElementaryProp}
Let $G$ be an acylindrically hyperbolic group. Then:
\begin{itemize}
	\item Any infinite normal subgroup of $G$ is acylindrically hyperbolic as well. 
	\item If $G \simeq A \oplus B$ then $A$ or $B$ must be finite.
\end{itemize}
\end{lemma}

\noindent
We refer to \cite[Lemma 7.2 and Corollary 7.3]{OsinAcyl} for a proof. We also need to understand how the acylindrical hyperbolicity behaves under commensurability. Although this is an open question in full generality, the following statement \cite{MinasyanOsinErratum} will be sufficient for our purpose.

\begin{lemma}\label{lem:HypFI}
Let $A,B$ be two groups with $A$ finite. A semidirect product $A \rtimes B$ is acylindrically hyperbolic if and only if so is $B$.
\end{lemma}

\begin{proof}[Proof of Theorem \ref{thm:WhenAutHyp}.]
According to Proposition \ref{prop:StructureAut},
$$\mathrm{Aut}(\Gamma \mathcal{G}) \simeq \mathrm{Hom} \left( \bigoplus\limits_{i=1}^n \langle \Gamma_i \rangle \to Z(\langle \Gamma_0 \rangle) \right) \rtimes \left( \mathrm{Aut}(\langle \Gamma_0 \rangle) \oplus \left[ \left( \bigoplus\limits_{i=1}^n \mathrm{Aut}(\langle \Gamma_i \rangle) \right) \rtimes S \right] \right)$$
where $S$ is a finite group. First, we claim that $\mathrm{Aut}(\Gamma \mathcal{G})$ is acylindrically hyperbolic if and only if the following two conditions are satisfied:
\begin{itemize}
	\item[(i)] $\mathrm{Hom} \left( \bigoplus\limits_{i=1}^n \langle \Gamma_i \rangle \to Z(\langle \Gamma_0 \rangle) \right)$ and $\mathrm{Aut}(\langle \Gamma_0 \rangle)$ are finite;
	\item[(ii)] $n=1$ and $\Gamma_1$ is not a pair of isolated vertices both labelled by $\mathbb{Z}_2$.
\end{itemize}
Assume that $\mathrm{Aut}(\Gamma \mathcal{G})$ is acylindrically hyperbolic. According to Lemma \ref{lem:HypElementaryProp}, $\mathrm{Aut}(\Gamma \mathcal{G})$ cannot contain an infinite abelian normal subgroup, so $\mathrm{Hom} \left( \bigoplus\limits_{i=1}^n \langle \Gamma_i \rangle \to Z(\langle \Gamma_0 \rangle) \right)$ has to be finite. As a consequence, $\mathrm{Aut}(\langle \Gamma_0 \rangle) \oplus \left[ \left( \bigoplus\limits_{i=1}^n \mathrm{Aut}(\langle \Gamma_i \rangle) \right) \rtimes S \right]$ must be acylindrically hyperbolic, and, as we know by assumption that $n \geq 1$, it follows from Lemma \ref{lem:HypElementaryProp} that $\mathrm{Aut}(\langle \Gamma_0 \rangle)$ must be finite as well. Consequently, $\bigoplus\limits_{i=1}^n \mathrm{Aut}(\langle \Gamma_i \rangle)$ has to be acylindrically hyperbolic, and we deduce again from Lemma \ref{lem:HypElementaryProp} that $n=1$ since otherwise this group would decompose as the direct sum of two infinite groups. Finally, because $\mathrm{Aut}(\langle \Gamma_1 \rangle)$ must be acylindrically hyperbolic, $\langle \Gamma_1 \rangle$ cannot be virtually cyclic, which amounts to saying, since $\Gamma_1$ contains at least two vertices and is not a join, that $\Gamma_1$ is not a pair of isolated vertices both labelled by $\mathbb{Z}_2$.

\medskip \noindent
Conversely, assume that $(i)$ and $(ii)$ hold. It follows from Lemma \ref{lem:HypFI} that $\mathrm{Aut}(\Gamma \mathcal{G})$ is acylindrically hyperbolic if and only if so is $\mathrm{Aut}(\langle \Gamma_1 \rangle)$. The latter assertion is a consequence of Theorem \ref{thm:MainAcylHyp}, concluding the proof of our claim.

\medskip \noindent
Next, we claim that $(i)$ is equivalent to the following assertion: $G_u$ is finite for every $u \in V(\Gamma_0)$, or $\mathrm{Aut}(\langle \Gamma_0 \rangle)$ is finite and $G_u$ has a finite abelianisation for every vertex $u \in V(\Gamma_1 \cup \cdots \cup \Gamma_n)$. 

\medskip \noindent
Assume that $(i)$ holds and that $G_u$ is infinite for some $u \in V(\Gamma_0)$. The fact that $\mathrm{Aut}(\langle \Gamma_0 \rangle)$ is finite implies that $\langle \Gamma_0 \rangle$ contains its center as a finite-index subgroup. So, because $\langle \Gamma_0 \rangle$ is an infinite group, its center must be an infinite (finitely generated) abelian group. Consequently, $\mathrm{Hom} \left( \bigoplus\limits_{i=1}^n \langle \Gamma_i \rangle \to Z(\langle \Gamma_0 \rangle) \right)$ is finite if and only if $G_v$ has a finite abelianisation for every $v \in V(\Gamma_1 \cup \cdots \cup \Gamma_n)$. 

\medskip \noindent
Conversely, if $G_u$ is finite for every $u \in V(\Gamma_0)$ or if $\mathrm{Aut}(\langle \Gamma_0 \rangle)$ is finite and $G_u$ has a finite abelianisation for every $u \in V(\Gamma_1 \cup \cdots \cup \Gamma_n)$, then $(i)$ clearly holds. This concludes the proof of our theorem.
\end{proof}

\begin{proof}[Proof of Corollary \ref{cor:WhenAutHyp}.]
The corollary is a direct consequence of Theorem \ref{thm:WhenAutHyp}.
\end{proof}

\paragraph{Applications to extensions.} We conclude this section by examining some extensions of graph products. Our main tool is the following statement proved in \cite{AutHypAcyl}:

\begin{prop}\label{prop:ExtensionHypAcylGeneral}
Let $G$ be a group whose center is finite and whose automorphism group is acylindrically hyperbolic. Fix a group $H$ and a morphism $\varphi : H \to \mathrm{Aut}(G)$. The semidirect product $G \rtimes_\varphi H$ is acylindrically hyperbolic if and only if the kernel of $H \overset{\varphi}{\to} \mathrm{Aut}(G) \to \mathrm{Out}(G)$ is finite.
\end{prop}

\begin{proof}[Proof of Theorem \ref{thm:Extension}.]
Assume first that $\Gamma \mathcal{G} \rtimes_\varphi H$ is acylindrically hyperbolic. As an infinite normal subgroup, $\Gamma \mathcal{G}$ has to be acylindrically hyperbolic as well according to Lemma \ref{lem:HypElementaryProp}. In particular, $\Gamma \mathcal{G}$ cannot decompose as the direct sum of two infinite groups according to the same lemma, hence $n=1$ and $G_u$ finite for every $u \in V(\Gamma_0)$. Also, $\Gamma \mathcal{G}$ cannot be virtually cyclic, so $\Gamma_1$ is not a pair of isolated vertices both labelled by $\mathbb{Z}_2$. Notice that the center of $\Gamma \mathcal{G}$ is then finite, and that Theorem \ref{thm:WhenAutHyp} implies the acylindrical hyperbolicity of $\mathrm{Aut}(\Gamma \mathcal{G})$. We conclude from Proposition \ref{prop:ExtensionHypAcylGeneral} that the kernel of $H \overset{\varphi}{\to} \mathrm{Aut}(\Gamma \mathcal{G}) \to \mathrm{Out}(\Gamma \mathcal{G})$ is finite.

\medskip \noindent
Conversely, assume that the three conditions given by our theorem are satisfied. Then $\Gamma \mathcal{G}$ has a finite center, and, according to Theorem \ref{thm:WhenAutHyp}, its automorphism group is acylindrically hyperbolic. We conclude from Proposition \ref{prop:ExtensionHypAcylGeneral} that $\Gamma \mathcal{G} \rtimes_\varphi H$ is acylindrically hyperbolic.
\end{proof}

\begin{proof}[Proof of Corollary \ref{cor:Extension}.]
The corollary is a direct consequence of Theorem \ref{thm:Extension}.
\end{proof}

\begin{proof}[Proof of Theorem \ref{thm:ExtensionVast}.]
As a consequence of Proposition \ref{prop:Vast}, it suffices to show that $\Gamma \mathcal{G} \rtimes_\varphi \mathbb{Z}$ contains a finite-index subgroup which surjects onto an acylindrically hyperbolic groups. Decompose $\Gamma$ as a join $\Gamma_0 \ast \Gamma_1 \ast \cdots \ast \Gamma_n$ where $\Gamma_0$ is complete and where each graph among $\Gamma_1, \ldots, \Gamma_n$ contains at least two vertices and is not a join. By assumption, there exists some $1 \leq j \leq n$ such that $\Gamma_j$ is not a pair of isolated vertices both labelled by $\mathbb{Z}_2$.

\medskip \noindent
First, notice that $\Gamma \mathcal{G} \rtimes_\varphi \mathbb{Z}$ surjects onto $\left( \Gamma \mathcal{G} / Z(\Gamma \mathcal{G}) \right) \rtimes_\varphi \mathbb{Z}$ where, for simplicity, we still denote $\varphi$ the automorphism of $\Gamma \mathcal{G}/Z(\Gamma \mathcal{G})$ induced by $\varphi : \Gamma \mathcal{G} \to \Gamma \mathcal{G}$. We distinguish two cases.

\medskip \noindent
First, assume that $\varphi$ has finite order in $\mathrm{Out}(\Gamma \mathcal{G}/Z(\Gamma \mathcal{G}))$. Then $(\Gamma \mathcal{G}/Z(\Gamma \mathcal{G})) \rtimes_\varphi \mathbb{Z}$ has a finite-index subgroup isomorphic to 
$$(\Gamma \mathcal{G}/Z(\Gamma \mathcal{G})) \oplus \mathbb{Z} \simeq \left( \langle \Gamma_0 \rangle /Z(\langle \Gamma_0 \rangle) \right) \oplus \langle \Gamma_1 \rangle \oplus \cdots \oplus \langle \Gamma_n \rangle \oplus \mathbb{Z}$$ 
which surjects onto the acylindrical hyperbolic group $\langle \Gamma_j \rangle$ \cite[Corollary 2.13]{MinasyanOsin} (see also Proposition \ref{prop:IrreducibleWPD}). 

\medskip \noindent
Next, assume that $\varphi$ has infinite order in $\mathrm{Out}(\Gamma \mathcal{G}/Z(\Gamma \mathcal{G}))$. As a consequence, the semidirect product $\left( \Gamma \mathcal{G}/Z(\Gamma \mathcal{G}) \right) \rtimes_\varphi \mathbb{Z}$ is isomorphic to the subgroup $\langle \mathrm{Inn}(\Gamma \mathcal{G}), \varphi \rangle$ of $\mathrm{Aut}(\Gamma \mathcal{G})$. According to Proposition \ref{prop:StructureAut}, 
$$\mathrm{Aut}(\Gamma \mathcal{G}) \simeq \mathrm{Hom} \left( \bigoplus\limits_{i=1}^n \langle \Gamma_i \rangle \to Z(\langle \Gamma_0 \rangle) \right) \rtimes \left( \mathrm{Aut}(\langle \Gamma_0 \rangle) \oplus \left[ \left( \bigoplus\limits_{i=1}^n \mathrm{Aut}(\langle \Gamma_i \rangle) \right) \rtimes S \right] \right)$$
where $S$ is finite. Consequently, $\left( \Gamma \mathcal{G}/Z(\Gamma \mathcal{G}) \right) \rtimes_\varphi \mathbb{Z}$ contains a finite-index subgroup which maps to $\mathrm{Aut}(\langle \Gamma_j \rangle)$ such that the image $H$ contains $\mathrm{Inn}(\langle \Gamma_j \rangle)$. 

\medskip \noindent
Notice that $H$ is acylindrically hyperbolic. Indeed, we know from Theorem \ref{thm:MainAcylHyp} that $\mathrm{Aut}(\langle \Gamma_j \rangle)$ is acylindrically hyperbolic. So, because $\mathrm{Inn}(\langle \Gamma_j \rangle)$ is an infinite normal subgroup of $\mathrm{Aut}(\langle \Gamma_j \rangle)$, it follows from \cite[Lemma 7.2]{OsinAcyl} that there exists a generalised loxodromic element of $\mathrm{Aut}(\langle \Gamma_j \rangle)$ which lies in $\mathrm{Inn}(\langle \Gamma_j \rangle)$ and a fortiori in $H$. As $H$ is not virtually cyclic (indeed, it contains $\mathrm{Inn}(\langle \Gamma_j \rangle) \simeq \langle \Gamma_j \rangle$), it follows that $H$ must be acylindrically hyperbolic. This concludes the proof of our theorem.
\end{proof}

\begin{proof}[Proof of Corollary \ref{cor:ExtensionVast}.]
Since they surject onto non-abelian free groups, non-abelian right-angled Artin groups involve all finite groups. Similarly, because they surject onto free products of finite groups distinct from $\mathbb{Z}_2 \ast \mathbb{Z}_2$, not virtually abelian graph products of finite groups involve all finite groups. It follows from \cite[Lemma 1.5]{OutRAAGlarge} that extensions of non-abelian right-angled Artin groups and of not virtually abelian graph products of finite groups always involve all finite groups. The rest of the corollary is a direct consequence of Theorem \ref{thm:ExtensionVast}.
\end{proof}

\section{A few other applications and some open questions}\label{section:Other}

\noindent
In this final section, we show that some of the intermediate statements proved in the article have applications of independent interest. We also discuss possible extensions of some of our results and set open questions.

\paragraph{Isomorphism problem and automorphism groups.} The rigidity provided by Theorem \ref{thm:IntroStepOne} is a progress towards the solution of the isomorphism problem among graph products of groups. For instance, when combined with \cite[Theorem 3.11]{ConjAut}, we immediately obtain the following partial solution:

\begin{thm}\label{thm:IsomorphismProblem}
Let $\Phi,\Psi$ be two finite simplicial graphs and $\mathcal{G}, \mathcal{H}$ two collections of graphically irreducible groups respectively indexed by $V(\Phi), V(\Psi)$. Assume that no two distinct vertices in $\Phi,\Psi$ are $\prec$-related. For every isomorphism $\varphi : \Phi \mathcal{G} \to \Psi \mathcal{H}$, there exists a conjugating automorphism $\alpha \in \mathrm{Aut}(\Phi \mathcal{G})$ and an isometry $s : \Phi \to \Psi$ such that $\varphi \circ \alpha$ sends $G_u$ isomorphically onto $H_{s(u)}$ for every $u \in V(\Phi)$.
\end{thm}

\noindent
Recall from \cite{ConjAut} that an automorphism is \emph{conjugating}\index{Conjugating automorphism} if it sends each vertex-group to a conjugate of a (possibly different) vertex-group. By setting $\Phi=\Psi$ and $\mathcal{G}= \mathcal{H}$ in Theorem \ref{thm:IsomorphismProblem}, we deduce from \cite[Theorem 3.4]{ConjAut} that:

\begin{cor}\label{cor:GeneratingSetAut}
Let $\Gamma$ be a finite simplicial graph and $\mathcal{G}$ a collection of graphically irreducible groups indexed by $V(\Gamma)$. Assume that no two distinct vertices in $\Gamma$ are $\prec$-related. Then $\mathrm{Aut}(\Gamma \mathcal{G})$ is generated by partial conjugations and local automorphisms.
\end{cor}

\noindent
Recall from \cite{ConjAut} that:
\begin{itemize}
	\item a \emph{local automorphism} is an automorphism $\Gamma \mathcal{G} \to \Gamma \mathcal{G}$ induced by $$\left\{ \begin{array}{ccc} \bigcup\limits_{u \in V(\Gamma)} G_u & \to & \Gamma \mathcal{G} \\ g & \mapsto & \text{$\varphi_u(g)$ if $g \in G_u$} \end{array} \right. $$ where $\sigma : \Gamma \to \Gamma$ is an isometry and $\Phi = \{ \varphi_u : G_u \to G_{\sigma(u)} \mid u \in V(\Gamma) \}$ a collection of isomorphisms;
	\item a \emph{partial conjugation} is an automorphism $\Gamma \mathcal{G} \to \Gamma \mathcal{G}$ induced by $$\left\{ \begin{array}{ccc} \bigcup\limits_{u \in V(\Gamma)} G_u & \to & \Gamma \mathcal{G} \\ g & \mapsto & \left\{ \begin{array}{cl} g & \text{if $g \notin \langle \Lambda \rangle$} \\ hgh^{-1} & \text{if $g \in \langle \Lambda \rangle$} \end{array} \right. \end{array} \right. $$ where $u \in V(\Gamma)$ is a vertex, $\Lambda$ a connected component of $\Gamma \backslash \mathrm{star}(u)$ and $h \in G_u$ an element.
\end{itemize}
We emphasize that finding a simple generating set for the automorphism group of a graph product is still an open problem in full generality. 

\medskip \noindent
We think that Lemma \ref{lem:Product} is also of independent interest and should have applications to products of acylindrically hyperbolic groups. More precisely, we expect a positive answer to the following question:

\begin{question}\label{QuestionAcyl}
Let $G$ be a finitely generated acylindrically hyperbolic group whose finite radical is trivial. Does $G$ admit a finite generating set of generalised loxodromic elements which pairwise do not commute?
\end{question}

\noindent
Combined with Lemma \ref{lem:Product}, a positive answer to this question would lead to the following statement: let $A_1, \ldots, A_n, B_1 ,\ldots, B_m$ be finitely generated acylindrically hyperbolic groups with trivial finite radical. For every isomorphism $\varphi : A_1 \oplus \cdots \oplus A_n \to B_1 \oplus \cdots \oplus B_m$, there exists a bijection $\sigma : \{1, \ldots, n\} \to \{1, \ldots, m\}$ such that $\varphi$ sends $A_i$ isomorphically onto $B_{\sigma(i)}$ for every $1 \leq i \leq n$.

\begin{remark}\label{MinasyanI}
In private communication, A. Minasyan informed us that Question~\ref{QuestionAcyl} has indeed a positive answer, as a consequence of \cite[Proposition~5.14]{MR3605030}. 
\end{remark}

\paragraph{Kazhdan's property (T).} Notice that, as immediate consequence of Corollary \ref{cor:GeneratingSetAut} and \cite[Theorem A]{MR4070718}, one gets:

\begin{cor}
Let $\Gamma$ be a finite simplicial graph and $\mathcal{G}$ a collection of graphically irreducible groups indexed by $V(\Gamma)$. Assume that no two distinct vertices in $\Gamma$ are $\prec$-related. If $\Gamma$ is not complete, then $\mathrm{Aut}(\Gamma \mathcal{G})$ does not satisfy Kazhdan's property (T).
\end{cor}

\noindent
The idea is that, as a consequence of Theorem \ref{thm:IntroStepOne}, $\mathrm{Aut}(\Gamma \mathcal{G})$ contains a finite-index subgroup $H$ which stabilises each conjugacy class of vertex-groups. Now, because $\Gamma$ is not complete, there exist two non-adjacent vertices $u,v \in V(\Gamma)$, and $H$ stabilises the normal closure $N:= \langle \langle G_w, \ w \in V(\Gamma) \backslash \{ u,v \} \rangle \rangle$, so the quotient map $\Gamma \mathcal{G} \twoheadrightarrow \Gamma \mathcal{G}/ N \simeq G_u \ast G_v$ induces a morphism $H \to \mathrm{ConjAut}(G_u \ast G_v)$ whose image contains the inner automorphisms, where $\mathrm{ConjAut}(\cdot)$ denotes the group of conjugating automorphisms. But the action of the free product $G_u \ast G_v$ on its Bass-Serre tree extends to an action of $H$, so $H$ does not satisfy Serre's property (FA). So we do not only know that $\mathrm{Aut}(\Gamma \mathcal{G})$ does not satisfy Kazhdan's property (T), we know that it does not have hereditary property (FA).

\medskip \noindent
Determining which automorphism groups of graph products satisfy the property (T) is an open problem. For graph products of finite groups, a solution can be found in \cite{MR4070718} (see also \cite{OutRACGlarge} for right-angled Coxeter groups). For right-angled Artin groups, partial solutions can be found in \cite{H1AUTRAAG, OutRAAGlarge, MR4023374, AutFreeT, AutFreeTbis}, but a complete solution is still unknown. For general graph products, only a few cases are known, see \cite{ConjAut, MR4070718}. 

\medskip \noindent
It is worth noticing that Theorem \ref{thm:IntroStepOne} can used to prove that many automorphism groups of graph products do not have hereditary property (FA), and so do not satisfy Kazhdan's property (T). As an illustration, we prove the next statement:

\begin{prop}
Let $\Gamma$ be a finite tree and $\mathcal{G}$ a collection of groups indexed by $V(\Gamma)$. If $\Gamma$ is not a star, then $\mathrm{Aut}(\Gamma \mathcal{G})$ does not satisfy Kazhdan's property (T).
\end{prop}

\begin{proof}
Notice that the leaves are the only vertices in $\Gamma$ which are not $\prec$-maximal. As a consequence of Theorem \ref{thm:IntroStepOne}, $\mathrm{Aut}(\Gamma \mathcal{G})$ contains a finite-index subgroup $H$ such that $H$ stabilises each conjugacy class of vertex-groups indexed by $\prec$-maximal vertices in $\Gamma$. Consequently, $H$ stabilises the normal closure $N := \langle \langle G_u, \text{ $u$ not a leaf} \rangle \rangle$ and the quotient map $\Gamma \mathcal{G} \twoheadrightarrow \Gamma \mathcal{G}/N$ induces a morphism $\varphi : H \to \mathrm{Aut}(\Gamma \mathcal{G}/N)$ whose image contains the inner automorphisms. Notice that $\Gamma \mathcal{G}/N$ decomposes as the free product $\langle [u_1] \rangle \ast \cdots \ast \langle [u_k ] \rangle$ where $u_1, \ldots, u_k$ are representatives of leaves modulo the relation ``having the same link''. Notice also that $k \geq 2$ because $\Gamma$ is not a star. For every $1 \leq i \leq k$, let $v_i$ denote the common neighbors of the vertices in $[u_i]$. Because $H$ stabilises the conjugacy class of $\langle v_i \rangle$ and that the normaliser in $\Gamma \mathcal{G}$ of $\langle v_i \rangle$ is $\langle v_i \rangle \oplus \langle [u_i] \rangle$, necessarily $\varphi(H)$ stabilises the conjugacy class of the free factor $\langle [u_i] \rangle$ of $\Gamma \mathcal{G}/N$. In other words, the image of $\varphi$ lies in $\mathrm{ConjAut}( \langle [u_1] \rangle \ast \cdots \ast \langle [u_k] \rangle)$. We conclude from \cite[Theorem~A]{MR4070718} that $\mathrm{Aut}(\Gamma \mathcal{G})$ does not satisfy Kazhdan's property (T). 
\end{proof}

\paragraph{Asymptotic invariants.} In Section \ref{section:TreeSpaces}, we saw that a graph product embeds quasi-isometrically into a tree of spaces whose vertex-spaces are copies of vertex-groups. This construction may be useful in order to show that some \emph{geometric properties} (i.e., properties of metric spaces stable under quasi-isometries) are preserved under graph products. More precisely:

\begin{prop}\label{prop:GeometricPropForGP}
Let $\mathcal{P}$ be a geometric property. Assume that:
\begin{itemize}
	\item A tree of spaces which all satisfy $\mathcal{P}$ also satisfies $\mathcal{P}$.
	\item The Cartesian product of finitely many spaces which satisfy $\mathcal{P}$ also satisfies $\mathcal{P}$.
	\item If a space satisfies $\mathcal{P}$, so do its quasi-isometrically embedded subspaces.
\end{itemize}
For every finite simplicial graph $\Gamma$ and every collection $\mathcal{G}$ indexed by $V(\Gamma)$ of finitely generated groups which all satisfy $\mathcal{P}$, the graph product $\Gamma \mathcal{G}$ satisfies $\mathcal{P}$.
\end{prop}

\noindent
As examples of applications, we recover several results available in the literature:

\begin{cor}
Let $\Gamma$ be a finite simplicial graph and $\mathcal{G}$ a collection of finitely generated groups indexed by $V(\Gamma)$. The following assertions are satisfied:
\begin{itemize}
	\item \cite{AntolinDreesen} If every group in $\mathcal{G}$ has finite asymptotic dimension, then so does $\Gamma \mathcal{G}$.
	\item \cite{Bell} If every group in $\mathcal{G}$ satisfies Yu's property A, then so does $\Gamma \mathcal{G}$.
	\item \cite{MeierDehn, CohenGP, AlonsoDehnGP} If every group in $\mathcal{G}$ is finitely presented and satisfies a polynomial isoperimetric inequality, then so does $\Gamma \mathcal{G}$.
	\item \cite{AntolinDreesen} If every group in $\mathcal{G}$ embeds coarsely into a Hilbert space, then so does $\Gamma \mathcal{G}$.
\end{itemize}
\end{cor}

\noindent
Actually, this strategy allows us to estimate the asymptotic dimension and the Dehn function of a graph product. We find that
$$\mathrm{asdim}(\Gamma \mathcal{G}) \leq \sum\limits_{u \in V(\Gamma)} \max \left( 1, \mathrm{asdim}(G_u) \right),$$
and that
$$\delta_{\Gamma \mathcal{G}} \prec \prod\limits_{u \in V(\Gamma)} \max \left( n \mapsto n, \delta_{G_u} \right)$$
if the groups in $\mathcal{G}$ are finitely presented. Also, the (equivariant) Hilbert space compression can be estimated. However, these estimates are far from being optimal in general. Compare with \cite{AntolinDreesen} and \cite{MeierDehn, CohenGP, AlonsoDehnGP}.

\medskip \noindent
Observe that the number of factors in our product of trees of spaces can be easily decreased. Let $\chi : V(\Gamma) \to C$ be a coloring of $\Gamma$ (i.e., $\chi(u) \neq \chi(v)$ for all adjacent vertices $u,v \in V(\Gamma)$) and fix a generating set $S_u$ of $G_u$ for every $u \in V(\Gamma)$. For every color $c \in C$, define $TS_c$ as the graph whose vertices are the fibers of the hyperplanes in $\QM$ labelled by a vertex in $\chi^{-1}(c)$ and the connected components of $\bigcap\limits_{\text{$J$ labelled by $\chi^{-1}(c)$}} \QM \backslash \backslash J$; and whose edges link a fiber to the component it belongs and two fibers of a hyperplane $J$ if they are at $\delta_J$-distance one. Because any two hyperplanes labelled by vertices of the same color cannot be transverse, the proofs of Fact \ref{fact:TStreegraded} and Proposition \ref{prop:Pi} can be adapted, and we show that $\Gamma \mathcal{G}$ embeds quasi-isometrically (and equivariantly) into the product $\prod\limits_{c \in C} TS_c$ where each $TS_c$ is a tree of spaces whose vertex-spaces are copies of vertex-groups.

\medskip \noindent
We expect this strategy of embedding graph products into products of trees of spaces to be useful in order to prove new results. As a first illustration, let us prove the following statement:

\begin{thm}\label{thm:CoarseMedian}
Let $\Gamma$ be a finite simplicial graph and $\mathcal{G}$ a collection of finitely generated groups indexed by $V(\Gamma)$. If the groups in $\mathcal{G}$ are coarse median, then so is the graph product $\Gamma \mathcal{G}$.
\end{thm}

\noindent
Coarse median spaces have been introduced in \cite{MR3037559}. For simplicity, we give the alternative definition proposed in \cite{MR4002223}.

\begin{definition}\label{def:CoarseMedian}
Let $X$ be a metric space. A ternary operation $\mu : X^3 \to X$ is a \emph{coarse median}\index{Coarse median space} if there exists some $\kappa \geq 0$ such that the following conditions are satisfied:
\begin{itemize}
	\item[(C0)] For all points $a_1,a_2,a_3 \in X$ and every permutation $\sigma$ of $\{1,2,3\}$, $$\mu(a_1,a_1,a_2) \sim_\kappa a_1 \text{ and } \mu(a_{\sigma(1)},a_{\sigma(2)}, a_{\sigma(3)}) \sim_\kappa \mu(a_1,a_2,a_3).$$
	\item[(C1)] For all points $a,a',b,c \in X$, $$d(\mu(a,b,c),\mu(a',b,c)) \leq \kappa d(a,a') + \kappa.$$
	\item[(C2)] For all points $a,b,c,d \in X$, $$\mu(\mu(a,b,c),b,d) \sim_\kappa \mu(a,b,\mu(c,b,d)).$$
\end{itemize}
For all points $x,y \in X$, the notation $x \sim_\kappa y$ means that $d(x,y) \leq \kappa$. 
A finitely generated group is \emph{coarse median}\index{Coarse median group} if it admits a coarse median operation when endowed with the word length associated to some (or, equivalently, any) finite generating set.
\end{definition}

\noindent
We refer to \cite{MR3037559} and \cite{MR3966604} for more information on coarse median spaces and groups.

\begin{proof}[Proof of Theorem \ref{thm:CoarseMedian}.]
For every $u \in V(\Gamma)$, fix a finite generating set $S_u \subset G_u$. Recall from Facts \ref{fact:TStreegraded} and \ref{fact:UnfoldClique} that $TS_u$ is a tree of spaces whose vertex-spaces are copies of $\mathrm{Cayl}(G_u,S_u)$. Consequently, each piece $P$ of $TS_u$ is endowed with a coarse median operation $\nu_P$. For all vertices $x,y,z \in TS_u$, define 
$$\nu_u(x,y,z):= \left\{ \begin{array}{cl} p & \text{if $P$ is a single vertex $p$} \\ \nu_P(x',y',z') & \text{if $P$ is a piece} \end{array} \right.$$
where $P$ denotes the median-set of $\{x,y,z\}$ in $TS_u$ and where $x',y',z'$ denote the projections of $x,y,z$ onto $P$. We emphasize that, by construction, $\nu_u(x,y,z)$ belongs to the median-set of $\{x,y,z\}$. As a consequence of \cite[Section 3]{RHcoarsemedian}, $(TS_u,\nu_u)$ is a coarse median space. It follows that
$$\left( \prod\limits_{u \in V(\Gamma)} TS_u, \nu \right) \text{ where } \nu : \left( (x_u), (y_u), (z_u) \right) \mapsto \left( \nu_u(x_u,y_u,z_u) \right)$$
is also a coarse median space. 

\medskip \noindent
Recall from Proposition \ref{prop:Pi} that, when $\Gamma \mathcal{G}$ is endowed with the finite generating set $S:= \bigcup\limits_{u \in V(\Gamma)} S_u$, there is a natural quasi-isometric embedding $\pi : \Gamma \mathcal{G} \to \prod\limits_{u \in V(\Gamma)} TS_u$. Fix three points $x,y,z \in \Gamma \mathcal{G}$. Because $\nu(\pi(x),\pi(y),\pi(z))$ belongs to the median-set of $\{\pi(x), \pi(y), \pi(z) \}$, it follows from Proposition \ref{prop:Pi} that there exists a point $\zeta(x,y,z)$ such that $\pi(\zeta(x,y,z))$ is at distance at most $C$ from $\nu(\pi(x),\pi(y),\pi(z))$, where $C$ is a constant which does not depend on $x,y,z$. Verifying that $\zeta$ defines a coarse median operation on $\Gamma \mathcal{G}$ is now straightforward.

\medskip \noindent
First, we need to fix some notation. Let $A>0$ and $B \geq 1$ be two constants such that
$$\frac{1}{A} d(x,y) -B \leq d(\pi(x),\pi(y)) \leq A d(x,y)+B$$
for all $x,y \in \Gamma \mathcal{G}$. Also, let $\kappa$ denote the constant coming from Definition \ref{def:CoarseMedian} for $\nu$.

\medskip \noindent
For all $a,b \in \Gamma \mathcal{G}$, we have
$$\begin{array}{lcl} d(\zeta(a,a,b),a) & \leq & A. d(\pi(\zeta(a,a,b)), \pi(a)) +AB \\ \\ & \leq & A. d(\nu(\pi(a),\pi(a),\pi(b)), \pi(a)) +A(B+C) \leq A(B+C+ \kappa) \end{array}$$
and, for all points $a_1,a_2,a_3 \in \Gamma \mathcal{G}$ and every permutation $\sigma$ of $\{1,2,3\}$, we have
$$\begin{array}{lcl} d(\zeta(a_{\sigma(1)},a_{\sigma(2)}, a_{\sigma(3)}), \zeta(a_1,a_2,a_3)) & \leq & A . d(\pi(\zeta(a_{\sigma(1)},a_{\sigma(2)},a_{\sigma(3)})), \pi(\zeta(a_1,a_2,a_3))) +B \\ \\ & \leq & A. d( \nu(a_{\sigma(1)},a_{\sigma(2)},a_{\sigma(3)}), \nu(a_1,a_2,a_3)) +A(B+2C) \\ \\ & \leq & A(B+2C+\kappa) \end{array}$$
Thus, the condition (C0) is verified. Next, for all $a,a',b,c \in \Gamma \mathcal{G}$, we have
$$\begin{array}{lcl} d(\zeta(a,b,c),\zeta(a',b,c)) & \leq & A. d(\pi (\zeta(a,b,c)), \pi(\zeta(a',b,c))) +AB \\ \\ & \leq & A.d(\nu(\pi(a),\pi(b),\pi(c)), \nu(\pi(a'),\pi(b),\pi(c))) +A(B+2C) \\ \\ & \leq & A \kappa. d(\pi(a),\pi(a')) + A(2C+B+ \kappa) \\ \\ & \leq & A^2 \kappa . d(a,a') + A(2C+B + \kappa+B\kappa) \end{array}$$
Thus, the condition (C1) is verified. Finally, for all $a,b,c,d \in \Gamma \mathcal{G}$, we have
$$d(\zeta(\zeta(a,b,c),b,d), \zeta(a,b,\zeta(c,b,d))) \leq A. d(\pi(\zeta(\zeta(a,b,c),b,d)), \pi(\zeta(a,b,\zeta(c,b,d)))) +AB$$ 
\vspace{-1.2cm}
$$\begin{array}{lcl} \ \\ \\ & \leq & A.d( \nu(\pi(\zeta(a,b,c)),\pi(b),\pi(d)), \nu(\pi(a), \pi(b), \pi(\zeta(c,b,d)))) + A(2C+B) \\ \\ & \leq & A. d(\nu(\pi (\zeta(a,b,c)), \pi(b),\pi(d)), \nu(\nu(\pi(a),\pi(b),\pi(c)), \pi(b), \pi(d)) \\ \\ & & + A.d(\nu(\nu(\pi(a),\pi(b),\pi(c)), \pi(b),\pi(d)), \nu(\pi(a),\pi(b), \nu(\pi(c),\pi(b),\pi(d)))) \\ \\ & & + A.d(\nu(\pi(a),\pi(b), \nu(\pi(c),\pi(b),\pi(d))), \nu( \pi(a), \pi(b), \pi(\zeta(c,b,d)))) +A(2C+B) \\ \\ & \leq & A \kappa . d(\pi(\zeta(a,b,c)), \nu(\pi(a), \pi(b),\pi(c))) + A \kappa . d( \nu(\pi(c),\pi(b),\pi(d)), \pi( \zeta(c,b,d))) \\ \\ & & + A(2C+B+ 4\kappa) \\ \\ & \leq &  A(2C(1+\kappa)+B+ 4\kappa) \end{array}$$
So the condition (C2) is satisfied, concluding the proof of our theorem.
\end{proof}

\noindent
Inspired by Theorem \ref{thm:CoarseMedian}, a natural question is:

\begin{question}
Can we prove similarly that hierarchically hyperbolic groups are stable under graph products (over finite graphs)?
\end{question}

\noindent
As a second illustration of the strategy inspired by Proposition~\ref{prop:GeometricPropForGP}, we can prove that the family of \emph{strongly shortcut groups} is closed under graph products over finite graphs. As defined in \cite{ShortcutI}, a graph is \emph{strongly shortcut} if there exists some $K \geq 1$ such that its $K$-biLipschitz cycles have uniformly bounded length; and a group $G$ is \emph{strongly shortcut} if it admits a proper and cocompact action on a strongly shortcut graph, or, equivalently (see \cite[Theorem~C]{ShortcutII}), if $G$ admits a finite generating set $S$ such that $\mathrm{Cayl}(G,S)$ is a strongly shortcut graph. We refer to \cite{ShortcutI, ShortcutII} for more information on strongly shortcut graphs and groups. 

\medskip \noindent
Proposition~\ref{prop:GeometricPropForGP} cannot be applied directly in order to show that a graph product of strongly shortcut groups is necessarily strongly shortcut. The reason is that the property of being strongly shortcut is not stable under taking quasi-isometrically embedded subgraphs. Therefore, we need a slight modification of Proposition~\ref{prop:GeometricPropForGP}:

\begin{prop}\label{prop:MetricPropForGP}
Fix an $r \geq 1$, a subset $Q \subset \mathbb{R}^r$, and a collection of families of graphs $\{ \mathcal{A}_{\bar{n}}, \bar{n} \in Q\}$ indexed by $Q$. Assume that:
\begin{itemize}
	\item[(i)] for all $\bar{n}, \bar{m} \in Q$, there exists some $\bar{p} \in Q$ such that $\mathcal{A}_{\bar{n}} \cup \mathcal{A}_{\bar{m}} \subset \mathcal{A}_{\bar{p}}$;
	\item[(ii)] if $X \in \mathcal{A}_{\bar{n}}$ for some $\bar{n} \in Q$, then, for every isometrically embedded subgraph $Y \subset X$, there exists some $\bar{m} \in Q$ such that $Y \in \mathcal{A}_{\bar{m}}$;	
	\item[(iii)] if $X,Y \in \mathcal{A}_{\bar{n}}$ for some $\bar{n} \in Q$, then $X \times Y \in \mathcal{A}_{\bar{m}}$ for some $\bar{m} \in Q$;
	\item[(iv)] if $X$ is a tree-graded with pieces all in $\mathcal{A}_{\bar{n}}$ for some $\bar{n} \in Q$, then $X \in \mathcal{A}_{\bar{m}}$ for some $\bar{m} \in Q$.
\end{itemize}
Let $\Gamma$ be a finite simplicial graph and $\mathcal{G}$ a collection of groups indexed by $V(\Gamma)$. For every $u \in V(\Gamma)$, fix a generating set $S_u \subset G_u$. If there exists some $\bar{n} \in Q$ such that $\mathrm{Cayl}(G_u,S_u) \in \mathcal{A}_{\bar{n}}$ for every $u \in V(\Gamma)$, then $\mathrm{Cayl}(\Gamma \mathcal{G}, S) \in \mathcal{A}_{\bar{m}}$ for some $\bar{m} \in Q$ where $S:= \bigcup_{u \in V(\Gamma)} S_u \backslash \{1\}$.
\end{prop}

\noindent
Typically, $\mathcal{A}_{\bar{n}}$ will correspond to the family of graphs satisfying some property $\mathcal{P}$ that depends on $r$ parameters $\bar{n}$ taking values in $Q$. Then, the assumptions of Proposition~\ref{prop:MetricPropForGP} assert that various combinations of graphs satisfying our property also satisfy the property but possibly with worse parameters.

\begin{proof}[Proof of Proposition \ref{prop:MetricPropForGP}.]
According to Lemma~\ref{lem:DeltaCayley}, $\mathrm{Cayl}(\Gamma \mathcal{G},S)$ is isometric to the graph $(\mathrm{QM}(\Gamma, \mathcal{G}),\delta)$. Because $\delta$ coincides with $\sum_{u \in V(\Gamma)} \delta_u$ by definition, we have an isometric embedding 
$$(\mathrm{QM}(\Gamma,\mathcal{G}),\delta) \hookrightarrow \prod\limits_{u \in V(\Gamma)} A_u,$$
where $A_u$ denotes the graph obtained from $(\mathrm{QM}(\Gamma,\mathcal{G}),\delta)$ by identifying any two vertices $x$ and $y$ satisfying $\delta_u(x,y)=0$. As a consequence of our assumptions $(i)-(iii)$, it is sufficient to show that each $A_u$ belongs to some $\mathcal{A}_{\bar{n}}$ in order to conclude the proof of the proposition. For every hyperplane $J$ in $\mathrm{QM}(\Gamma, \mathcal{G})$ labelled by $u$, fix a clique $C_J \subset J$ and let $P_J$ denote the image of $C_J$ in $A_u$; notice that, according to Fact~\ref{fact:UnfoldClique}, $P_J$ is isometric to $\mathrm{Cayl}(G_u,S_u)$. As a consequence of our assumption $(iv)$, it suffices to show that $A_u$ is tree-graded with respect to $\mathcal{P}_u:= \{ P_J, \text{ $J$ labelled by $u$}\}$ in order to conclude the proof of our proposition.

\medskip \noindent
Let $P_1,P_2 \in \mathcal{P}_u$ be two distinct pieces. Assume that $P_1 \cap P_2 \neq \emptyset$. By definition, $P_1$ and $P_2$ are respectively images of two cliques $C_1$ and $C_2$ lying in two hyperplanes $J_1$ and $J_2$ both labelled by $u$. If there exists a hyperplane labelled by $u$ separating $J_1$ and $J_2$ then $\delta_u(x,y) \geq 1$ for all $x \in C_1$ and $y \in C_2$, hence $P_1 \cap P_2 = \emptyset$. Otherwise, for all $x \in C_1$ and $y \in C_2$, the equality $\delta_u(x,y)=0$ holds if and only if $x$ is the projection of $C_2$ onto $C_1$ and $y$ the projection of $C_1$ onto $C_2$, so $P_1 \cap P_2$ must be reduced to a single point. Moreover, $P_1 \cap P_2$ separates $P_1 \backslash P_1 \cap P_2$ and $P_2 \backslash P_1 \cap P_2$ in $A_u$. Indeed, a path in $A_u$ corresponds to a sequence of connected components in $\bigcap_{\text{$J$ labelled by $u$}} \mathrm{QM}(\Gamma, \mathcal{G}) \backslash \backslash J$ such that any two consecutive components are separated by a single hyperplane labelled by $u$. But, because no two hyperplanes labelled by $u$ can be transverse, every such sequence of components starting from the component corresponding to a vertex of $P_1 \backslash P_1 \cap P_2$ (i.e. containing a vertex in $C_1 \backslash \mathrm{proj}_{C_1}(C_2)$) and ending to the component corresponding to a vertex of $P_2 \backslash P_1 \cap P_2$ (i.e. containing a vertex in $C_2 \backslash \mathrm{proj}_{C_2}(C_1)$) has to pass through the component corresponding to $P_1 \cap P_2$ (i.e. containing $\mathrm{proj}_{C_1}(C_2)$ and $\mathrm{proj}_{C_2}(C_1)$). We conclude from the previous observations that $A_u$ is tree-graded with respect to $\mathcal{P}_u$. 
\end{proof}

\begin{cor}\label{cor:ShortcutGP}
Let $\Gamma$ be a finite simplicial graph and $\mathcal{G}$ a collection of groups indexed by $V(\Gamma)$. If $G_u$ is a strongly shortcut group for every $u \in V(\Gamma)$, then so is $\Gamma \mathcal{G}$. 
\end{cor}

\begin{proof}
For all $r,s >0$, let $\mathcal{A}_{(r,s)}$ denote the family of graphs all of whose $r$-biLipschitz cycles have length $\leq s$. For every $u \in V(\Gamma)$, fix a finite generating set $S_u \subset G_u$ such that $\mathrm{Cayl}(G_u,S_u)$ is strongly shortcut, i.e. belongs to $\mathcal{A}_{(n,m)}$ for some $n,m >0$. Notice that the assumptions $(ii)$ and $(iv)$ of Proposition~\ref{prop:MetricPropForGP} are clearly sastisfied; that $(i)$ holds since $\mathcal{A}_{(r,s)} \cup \mathcal{A}_{(p,q)} \subset \mathcal{A}_{(\min(r,p),\max(s,q))}$ for all $r,s,p,q > 0$; and that $(iii)$ follows from \cite[Theorem~5.1]{ShortcutI}. Therefore, Proposition~\ref{prop:MetricPropForGP} applies and shows that $\mathrm{Cayl}(\Gamma \mathcal{G},S)$ belongs to $\mathcal{A}_{(n,m)}$ for some $n,m > 0$, i.e. is strongly shortcut. 
\end{proof}

\paragraph{Acylindrically hyperbolic automorphism groups.} As mentioned in the introduction, it is still unknown whether automorphism groups of finitely generated acylindrically hyperbolic groups are necessarily acylindrically hyperbolic (see Question \ref{QuestionHyp}). We expect that the arguments used to prove the positive answer provided by Theorems \ref{thm:WhenAutHyp} for graph products can be adapted to Haglund and Wise's special groups \cite{MR2377497} thanks to the formalism introduced in \cite{MoiGraphBraid}.

\begin{conj}
Let $G$ be a virtually cocompact special group. If $G$ is acylindrically hyperbolic, then so is its automorphism group $\mathrm{Aut}(G)$. 
\end{conj}

\paragraph{Virtually cyclic centralisers in automorphism groups.} A weak version of Question \ref{QuestionHyp} is the following:

\begin{question}\label{QuestionCentraliser}
Let $G$ be a finitely generated acylindrically hyperbolic group. Does there exist an element $g \in G$ such that $\{ \varphi \in \mathrm{Aut}(G) \mid \varphi(g)=g\}$ is virtually cyclic ?
\end{question}

\noindent
Notice that a positive answer to Question \ref{QuestionHyp} implies a positive answer to Question \ref{QuestionCentraliser} because generalised loxodromic elements have virtually cyclic centralisers and that $\{ \varphi \in \mathrm{Aut}(G) \mid \varphi(g)=g\}$ coincides with the centraliser in $\mathrm{Aut}(G)$ of the inner automorphism associated to $g$. Moreover, it seems that answering Question \ref{QuestionCentraliser} is often a necessary step in order to answer to Question \ref{QuestionHyp}. Also, containing an element whose centraliser is virtually cyclic implies some of the elementary properties satisfied by acylindrically hyperbolic groups (see Lemma \ref{lem:HypElementaryProp}) and so may be of independent interest.

\medskip \noindent
Interestingly, the strategy followed in Sections \ref{section:step2} and \ref{section:step3} in order to answer Question \ref{QuestionCentraliser} for graph products seems to apply to many other groups, namely groups which embed quasi-isometrically and equivariantly into products of hyperbolic spaces. For instance:

\begin{problem}
Let $G$ be a Coxeter group or a cocompact special groups (as defined in \cite{MR2377497}). If $G$ is acylindrically hyperbolic, shows that $\mathrm{Aut}(G)$ contains elements with virtually cyclic centralisers.
\end{problem}

\begin{remark}\label{MinasyanII}
In private communication, A. Minasyan informed us that a positive answer to Question~\ref{QuestionCentraliser} can be deduced from the preprint \cite{TestWords} for acylindrically hyperbolic groups with trivial finite radicals. More precisely, given a finitely generated acylindrically hyperbolic groups $G$ with no non-trivial finite normal subgroup, deduce from \cite{MR3605030} that $G$ admits a finite generating set $S:=\{a_1, \ldots, a_k\}$ of special elements that are pairwise non-commensurable. Let $W$ be an $(a_1,...,a_k,1,...,1)$-test word as given by \cite[Proposition~12.4]{TestWords}, and set $w=W(a_1,...,a_k,1,...,1)$. If $\varphi$ is an automorphism of $G$ satisfying $\varphi(w)=w$, then there exists a power $g$ of $w$ such that $\varphi(a_i)=ga_ig^{-1}$. It follows that $\varphi \in \langle \iota_w\rangle$, as desired.
\end{remark}

\appendix

\section{A general hyperbolic model: the graph of maximal products}

\noindent
We saw in Section \ref{section:StepTwo} that the small crossing graph can be used as a geometric model for the automorphism group of a graph product, but only under an additional restriction on the simplicial graph defining the graph product: we need the graph not to contain two vertices with the same link or the same star. Below, we show how to construct a geometric model without this extra condition.

\begin{definition}
Let $\Gamma$ be a simplicial graph and $\mathcal{G}$ a collection of groups indexed by $V(\Gamma)$. The \emph{graph of maximal products}\index{Graph of maximal products $\mathscr{M}(\Gamma \mathcal{G})$} $\mathscr{M}(\Gamma \mathcal{G})$ is the graph whose vertices are the maximal product subgroups of $\Gamma \mathcal{G}$ and whose edges link two subgroups whenever their intersection is non-trivial. 
\end{definition}

\noindent
Notice that the natural action 
$$\left\{ \begin{array}{ccc} \Gamma \mathcal{G} & \to & \mathrm{Isom}(\mathscr{M}(\Gamma\mathcal{G})) \\ g & \mapsto & (P \mapsto gPg^{-1}) \end{array} \right.$$
extends to the action
$$\left\{ \begin{array}{ccc} \mathrm{Aut}(\Gamma \mathcal{G}) & \to & \mathrm{Isom}(\mathscr{M}(\Gamma\mathcal{G})) \\ \varphi & \mapsto & (P \mapsto \varphi(P)) \end{array} \right..$$
The rest of the appendix is dedicated to the following theorem, which motivates the fact that the graph of maximal products is a relevant geometric model for the study of automorphisms of graph products of groups.

\begin{thm}\label{thm:GeneralHyp}
Let $\Gamma$ be a finite connected simplicial graph and $\mathcal{G}$ a collection of finitely generated irreducible groups indexed by $V(\Gamma)$. Assume that $\Gamma$ is not a join and contains at least two vertices. 
\begin{itemize}
	\item The graph $\mathscr{M}(\Gamma \mathcal{G})$ is a quasi-tree on which $\Gamma \mathcal{G}$ acts acylindrically. 
	\item An element of $\Gamma \mathcal{G}$ is loxodromic if and only if its centraliser is infinite cyclic; otherwise, it fixes a vertex. 
	\item The action of $\Gamma \mathcal{G}$ extends to an action of $\mathrm{Aut}(\Gamma \mathcal{G})$ which contains WPD elements, namely the inner automorphisms $\iota_g$ where $g \in \Gamma \mathcal{G}$ is an element of full support such that $\{ \varphi \in \mathrm{Aut}(\Gamma \mathcal{G}) \mid \varphi(g)=g\}$ is virtually cyclic. 
\end{itemize} 
\end{thm}

\noindent
The theorem will be deduced from our work on the small crossing graph thanks to the following statement, which shows that there exists a quasi-isometry between the graph of maximal products and the small crossing graph which is almost equivariant with respect to the automorphism group.

\begin{lemma}\label{lem:QIequi}
Let $\Gamma$ be a finite connected simplicial graph and $\mathcal{G}$ a collection of graphically irreducible groups indexed by $V(\Gamma)$. Assume that no two vertices in $\Gamma$ have the same star or the same link, and that, for every automorphism $\varphi \in \mathrm{Aut}(\Gamma \mathcal{G})$ and every $\prec$-maximal vertex $u \in V(\Gamma)$, there exist an element $g \in \Gamma \mathcal{G}$ and a $\prec$-maximal vertex $v \in V(\Gamma)$ such that $\varphi(\langle u \rangle)= g \langle v \rangle g^{-1}$. There exists a quasi-isometry $q:\ST \to \mathscr{M}(\Gamma \mathcal{G})$ such that
$$\left| d_{\mathscr{M}} ( q( \varphi J_1) , q( \varphi J_2) ) - d_{\mathscr{M}} ( \varphi q(J_1), \varphi q(J_2) ) \right| \leq 2$$
for all $J_1,J_2 \in \ST$ and every $\varphi \in \mathrm{Aut}(\Gamma \mathcal{G})$. 
\end{lemma}

\begin{proof}
Because $\Gamma$ is connected and not reduced to a single vertex, the stabiliser of every hyperplane in $\QM$ decomposes non-trivially as a product. Therefore, for every hyperplane $J \in \ST$ there exists a maximal product subgroup $q(J) \in \mathscr{M}(\Gamma \mathcal{G})$ such that $\mathrm{stab}(J) \subset q(J)$. 

\medskip \noindent
Next, given a maximal product subgroup $P \in \mathscr{M}(\Gamma \mathcal{G})$, there exist a unique maximal join $\Lambda_P \subset \Gamma$ and a unique element $g_P \in \Gamma \mathcal{G}$ satisfying $\mathrm{tail}(g_P) \cap \Lambda_P=\emptyset$ such that $P$ coincides with $g_P \langle \Lambda_P \rangle g_P^{-1}$. Set $r(P)= g_P J_u$ where $u$ is a $\prec$-maximal vertex of $\Gamma$ we fix once for all (and which does not depend on $P$). 

\medskip \noindent
Fix a hyperplane $J \in ST$, and write the maximal product subgroup $q(J)$ as $g_J \langle \Lambda_J \rangle g_J^{-1}$, where $\Lambda_J \subset \Gamma$ is a maximal join and where $g_J \in \Gamma \mathcal{G}$ satisfies $\mathrm{tail}(g_J) \cap \Lambda_J= \emptyset$. By construction, $r(q(J))= g_J J_u$. Fix a geodesic $v_1, \ldots, v_k$ in $\Gamma$ from $u$ to a vertex $w$ of $\Lambda_J$. So 
$$g_JJ_u=g_JJ_{v_1}, \ g_J J_{v_2}, \ldots, \ g_J J_{v_{k-1}}, \ g_J J_{v_k}=g_JJ_w$$
defines a path in the crossing graph. Moreover, 
$g_J \langle \Lambda_J \rangle g_J^{-1}$ contains the rotative-stabiliser of $J$ (since it contains $\mathrm{stab}(J)$) and the rotative-stabiliser of $g_J J_w$ (because $w$ belongs to $\Lambda_J$), so these two hyperplanes must cross the product subgraph $g_J \langle \Lambda_J \rangle$. We conclude that
$$d_{\mathrm{ST}}(J,r(q(J))) \leq d_{\mathrm{ST}}(J, g_J J_w) + d_{\mathrm{ST}}(g_JJ_w, g_JJ_u) \leq 2+ \mathrm{diam}(\Gamma).$$
Now, fix a maximal product subgroup $P \in \mathscr{M}(\Gamma \mathcal{G})$. Write $P= g_P \langle \Lambda_P \rangle g_P^{-1}$ where $\Lambda_P \subset \Gamma$ is a maximal join and where $g_P \in \Gamma \mathcal{G}$ satisfies $\mathrm{tail}(g_P) \cap \Lambda_P= \emptyset$. By construction, $r(P)=g_PJ_u$ and $q(r(P))$ is a maximal product subgroup $Q$ containing $g_P \langle \mathrm{star}(u) \rangle g_P^{-1}$. Let $v_1,\ldots, v_k$ be a geodesic in $\Gamma$ from $u$ to a vertex in $\Lambda_P$. For every $1 \leq i \leq k-1$, let $\Xi_i \subset \Gamma$ be a maximal join containing $v_i$ and $v_{i+1}$; also, set $\Xi_k= \Lambda_p$. We have
$$d_{\mathscr{M}}(q(r(P)),P) \leq d_{\mathscr{M}}(Q, g_P \langle \Xi_1 \rangle g_P^{-1}) + \sum\limits_{i=1}^{k-1} d_{\mathscr{M}}(g_P \langle \Xi_i \rangle g_P^{-1}, g_P \langle \Xi_{i+1} \rangle g_P^{-1}).$$
Notice that $g_P \langle u \rangle g_P^{-1}= g_P \langle v_1 \rangle g_P^{-1}$ lies in both $Q$ and $g_P \langle \Xi_1 \rangle g_P^{-1}$, and that, for every index $1 \leq i \leq k-1$, the intersection $g_P \langle \Xi_i \rangle g_P^{-1} \cap g_P \langle \Xi_{i+1} \rangle g_P^{-1}$ contains $g_P \langle v_{i+1} \rangle g_P^{-1}$. Consequently,
$$d_{\mathscr{M}}(P,q(r(P))) \leq 1+k \leq 1+ \mathrm{diam}(\Gamma).$$
We conclude that $q$ is a quasi-isometry with $r$ as a quasi-inverse.

\medskip \noindent
Now, let $J_1,J_2 \in \ST$ be two hyperplanes. Notice that, for $i=1,2$, $q(\varphi J_i)$ and $\varphi(q(J_i))$ are two maximal product subgroups containing $\varphi(\mathrm{stab}(J_i))$. Therefore, the difference
$$\left| d_{\mathscr{M}} ( q( \varphi J_1) , q( \varphi J_2) ) - d_{\mathscr{M}} ( \varphi q(J_1), \varphi q(J_2) ) \right|$$
is bounded above by 
$$d_{\mathscr{M}}(q(\varphi J_1), \varphi q(J_1)) + d_{\mathscr{M}}(q(\varphi J_2), \varphi q(J_2)) \leq 2,$$
concluding the proof of our lemma.
\end{proof}

\begin{proof}[Proof of Theorem \ref{thm:GeneralHyp}.]
As a consequence of Lemma \ref{lem:ForGraphicallyReduced} and Proposition \ref{prop:SameLinkStar}, we may suppose without loss of generality that no two vertices in $\Gamma$ have the same star or the same link, and that, for every automorphism $\varphi \in \mathrm{Aut}(\Gamma \mathcal{G})$ and every $\prec$-maximal vertex $u \in V(\Gamma)$, there exist an element $g \in \Gamma \mathcal{G}$ and a $\prec$-maximal vertex $v \in V(\Gamma)$ such that $\varphi(\langle u \rangle)= g \langle v \rangle g^{-1}$. 

\medskip \noindent
As a consequence of Lemma \ref{lem:QIequi} and Theorem \ref{thm:Phyp}, $\mathscr{M}(\Gamma \mathcal{G})$ is a quasi-tree. We know from \cite[Theorem B]{QMacylindrical} that $\Gamma \mathcal{G}$ acts acylindrically on $\crossing$, so, as a consequence of Lemmas \ref{lem:SmallGeodesicIn} and \ref{lem:QIequi}, $\Gamma \mathcal{G}$ acts acylindrically on $\mathscr{M}(\Gamma \mathcal{G})$. This proves the first point of our theorem.

\medskip \noindent
Let $g \in \Gamma \mathcal{G}$ be an element. Either $g$ belongs to product subgroup, and it fixes a vertex in $\mathscr{M}(\Gamma \mathcal{G})$; or it is irreducible, and it follows from Proposition \ref{prop:IrreducibleWPD} and Lemma \ref{lem:QIequi} that $g$ induces a loxodromic isometry of $\mathscr{M}(\Gamma \mathcal{G})$. The fact that $g$ is irreducible if and only if its centraliser is infinite cyclic follows from \ref{prop:centraliser}. This proves the second point of our theorem.

\medskip \noindent
The characterisation of WPD isometries in $\mathscr{M}(\Gamma \mathcal{G})$ follows from Lemma~\ref{lem:QIequi} and Proposition \ref{prop:InnerWPD}, and their existence follows from Theorems \ref{thm:ActionRealTree} and \ref{thm:FRH}. This proves the third point of our theorem.
\end{proof}

\addcontentsline{toc}{section}{References}

\bibliographystyle{alpha}
{\footnotesize\bibliography{RAAGauto}}

\Address

\addcontentsline{toc}{section}{Index}

\printindex

\end{document}